\documentclass[12pt,a4paper]{article}
\usepackage[left=1in,right=1in,top=1.2in,bottom=1.2in]{geometry}
\usepackage{setspace}
\usepackage{changepage}
\usepackage{array}
\usepackage{color}
\usepackage{amsmath, amsthm, amssymb}
\usepackage{fmtcount}
\usepackage[english]{babel}
\usepackage[utf8]{inputenc}
\usepackage{tikz}
\usepackage{kotex}
\usepackage{amsfonts}
\usepackage{mathrsfs}
\usepackage{mathtools}
\usepackage{enumitem}
\usepackage{algorithm}
\usepackage{algorithmic}
\usepackage{adjustbox}
\usepackage{pgfplots}
\pgfplotsset{compat=1.18}
\allowdisplaybreaks 
\usepackage{multicol}
\usepackage{multirow}
\usepackage{arydshln}
\usepackage{url}
\usepackage{mathabx}
\usepackage{subfiles}
\usepackage[round]{natbib}    
\usepackage{accents}  
\usepackage{tocloft}
\usepackage{csvsimple, booktabs}
\usepackage{bm}
\usepackage{caption}
\usepackage{subcaption}
\usepackage[T1]{fontenc}  
\usepackage[all]{xy}
\usepackage{tikz-cd}
\usepackage{xcite}
\usepackage{xr}     
\usepackage{xr-hyper}       
\usepackage[
    colorlinks=true,
    linkcolor=blue,
    citecolor=blue,
    urlcolor=blue,
    filecolor=blue
]{hyperref} 
\makeatletter
\newcommand*{\addFileDependency}[1]{
  \typeout{(#1)}
  \@addtofilelist{#1}
  \IfFileExists{#1}{}{\typeout{No file #1.}}
}
\makeatother

\usepackage{scalerel,stackengine}
\stackMath
\newcommand\reallywidehat[1]{%
\savestack{\tmpbox}{\stretchto{%
  \scaleto{%
    \scalerel*[\widthof{\ensuremath{#1}}]{\kern-.6pt\bigwedge\kern-.6pt}%
    {\rule[-\textheight/2]{1ex}{\textheight}}
  }{\textheight}%
}{0.5ex}}%
\stackon[1pt]{#1}{\tmpbox}%
}
\usetikzlibrary{calc,arrows.meta,positioning,shapes.geometric}
\newcolumntype{L}[1]{>{\raggedright\let\newline\\\arraybackslash\hspace{0pt}}m{#1}}
\newcolumntype{C}[1]{>{\centering\let\newline\\\arraybackslash\hspace{0pt}}m{#1}}
\newcolumntype{R}[1]{>{\raggedleft\let\newline\\\arraybackslash\hspace{0pt}}m{#1}}

\newcommand{\ie}{{\it i.e.}}

\newcommand{\iid}{{\it i.i.d.}}

\newcommand{\Ac}{\mathcal{A}}
\newcommand{\Bc}{\mathcal{B}}
\newcommand{\Cc}{\mathcal{C}}

\newcommand{\Ec}{\mathcal{E}}
\newcommand{\Fc}{\mathcal{F}}

\newcommand{\Hc}{\mathcal{H}}

\newcommand{\Kc}{\mathcal{K}}

\newcommand{\Mc}{\mathcal{M}}
\newcommand{\Nc}{\mathcal{N}}

\newcommand{\Pc}{\mathcal{P}}

\newcommand{\Rc}{\mathcal{R}}
\newcommand{\Sc}{\mathcal{S}}
\newcommand{\Tc}{\mathcal{T}}

\newcommand{\Xc}{\mathcal{X}}
\newcommand{\Yc}{\mathcal{Y}}

\newcommand{\Db}{\mathbb{D}}
\newcommand{\Eb}{\mathbb{E}}

\newcommand{\Ib}{\mathbb{I}}

\newcommand{\Mb}{\mathbb{M}}
\newcommand{\Nb}{\mathbb{N}}

\newcommand{\Pb}{\mathbb{P}}

\newcommand{\Rb}{\mathbb{R}}
\newcommand{\Sb}{\mathbb{S}}

\newcommand{\Xb}{\mathbb{X}}
\newcommand{\Yb}{\mathbb{Y}}
\newcommand{\Zb}{\mathbb{Z}}

\def\cX{\mathcal X}

\newcommand{\bx}{{\bf x}}

\newcommand{\bz}{{\bf z}}

\newcommand{\tr}{{\mathrm{tr}}}
\newcommand{\op}{{\mathrm{op}}}
\newcommand{\pconv}{\xrightarrow{P}}

\newcommand{\HS}{\mathrm{HS}}

\newcommand{\E}{\mathbb{E}}

\newcommand{\V}{\mbox{{\rm Var}}}

\newcommand{\bc}{\begin{center}}
\newcommand{\ec}{\end{center}}
\newcommand{\be}{\begin{equation}}
\newcommand{\ee}{\end{equation}}
\newcommand{\been}{\begin{equation*}}
\newcommand{\eeen}{\end{equation*}}
\newcommand{\ba}{\begin{array}}
\newcommand{\ea}{\end{array}}
\newcommand{\bean}{\setlength\arraycolsep{2pt}\begin{eqnarray*}}
\newcommand{\eean}{\end{eqnarray*}}
\newcommand{\bea}{\setlength\arraycolsep{2pt}\begin{eqnarray}}
\newcommand{\eea}{\end{eqnarray}}
\newcommand{\ben}{\begin{enumerate}}
\newcommand{\een}{\end{enumerate}}
\newcommand{\bed}{\begin{itemize}}
\newcommand{\eed}{\end{itemize}}

\newtheorem{thm}{Theorem}
\numberwithin{thm}{section}
\newtheorem{cor}[thm]{Corollary}
\newtheorem{lem}[thm]{Lemma}
\newtheorem{prop}[thm]{Proposition}
\newtheorem{defn}{Definition}
\newtheorem*{claim*}{Claim}
\newtheorem{claim}{Claim}

\newtheorem{rmk}{Remark}[section]

\usepackage{graphicx,psfrag,epsf}

\usepackage{url} 

\newcommand{\blind}{0}

\def\spacingset#1{\doublespacing}

\begin{document}

\def\spacingset#1{\doublespacing}


\if0\blind
{
  \title{A Two-Sample Test on Weighted Persistence Intensity Functions in Topological Data Analysis}
\author{
Yeongung Han\\
Institute of Basic Sciences, Seoul National University
\and
Ilmun Kim\\
\parbox{0.8\textwidth}{\centering
Department of Mathematical Sciences, Korea Advanced Institute of Science and Technology}
\and
Jisu Kim \thanks{Jisu Kim  (jkim82133@snu.ac.kr) is the corresponding author.}\\
Department of Statistics, Seoul National University
}

\date{}
\maketitle

} \fi

\if1\blind
{
  \bigskip
  \bigskip
  \bigskip
  \begin{center}
    {\LARGE A Two-Sample Test on Weighted Persistence Intensity Functions in Topological Data Analysis}
\end{center}
  \medskip
} \fi

\bigskip
\let\arxivoriginalnewpage\newpage
\newif\ifarxivmainfrontmatter
\arxivmainfrontmattertrue

\renewcommand{\newpage}{%
  \ifarxivmainfrontmatter
    \arxivmainfrontmatterfalse
    \arxivoriginalnewpage
    {\fontsize{12pt}{15.5pt}\selectfont
  \hypersetup{hidelinks}

  \setlength{\cftbeforesecskip}{20pt}
  \setlength{\cftbeforesubsecskip}{15pt}
  \setlength{\cftbeforesubsubsecskip}{15pt}

  \tableofcontents

  \hypersetup{linkcolor=blue}
}
    \arxivoriginalnewpage
    \let\newpage\arxivoriginalnewpage
  \else
    \arxivoriginalnewpage
  \fi
}

\begin{abstract}
Persistence intensity functions provide interpretable and informative first-order summaries of random persistence diagram distributions.
We study two-sample testing for equality of persistence intensity functions, allowing the underlying diagram distributions to differ under the null.
We construct a weighted-kernel statistic as an unbiased estimator of the squared reproducing-kernel Hilbert space distance between the corresponding weighted intensity embeddings, and calibrate it by studentization. For shrinking bandwidths, we establish uniform asymptotic normality under the null and thereby obtain asymptotic Type I error control.
Its power is characterized in terms of the $L^2$ discrepancy between the weighted intensity functions. To accommodate persistence diagrams with possibly unbounded cardinality, we introduce regularity conditions that control the effect of cardinality variation and yield the desired moment bounds for the test statistic. We further show that every probability density on $\{(x,y)\in\mathbb{R}^2:y>x > 0\}$ can be realized as the persistence intensity function of a random  diagram. Using these results, we establish that the proposed test attains minimax-optimal separation rates over anisotropic Sobolev balls. 
Lastly, since the optimal bandwidth is not directly accessible in practice, we adapt a bandwidth aggregation framework. 
\end{abstract}

\noindent%
\vfill

\newpage
\spacingset{1.8} 

\section{Introduction}\label{sec::intro}
When point clouds are repeatedly observed from two populations, a natural question is whether their topological features occur at the same scales and with the same expected frequency. For example, does one population tend to produce more one-dimensional features? Do such features tend to appear and disappear over the same ranges of scales? These questions concern the expected abundance of topological features across scales, a natural target for statistical inference in topological data analysis (TDA).

Persistence diagrams record the features involved in these questions. Each point in a diagram represents a topological feature, and its coordinates give the scales at which that feature appears and disappears. A diagram therefore records both the number and scales of topological features in a point cloud. However, the distribution of a random persistence diagram describes the possible configurations of its birth–death pairs, including the dependence among the pairs. This level of detail exceeds what is needed when the question concerns the expected abundance of topological features across scales. For such questions, the expectation of the random diagram provides the relevant first-order summary.


To formalize this first-order summary, we consider a persistence diagram $D$ as a random counting measure on the birth–death plane. Its expectation assigns to each region the expected number of diagram points in that region. If this expectation measure is absolutely continuous with respect to Lebesgue measure, its density is called the \emph{persistence intensity function}. The existence and estimation of persistence intensity functions have been studied \citep{Chazal2019,conf/icml/DivolL21,wu2024estimation}.

However, relatively little is known about two-sample testing for equality of persistence intensity functions. A central difficulty is that equal intensities do not imply identical diagram distributions: the two populations may differ in the variability of diagram cardinality and in within-diagram dependence, and hence have different covariance structures under the null. To address this difficulty, we propose a studentized kernel test that directly targets differences between weighted persistence intensity functions. With shrinking bandwidths, we establish uniform asymptotic normality of the test statistic under equality of intensities, yielding uniform asymptotic Type I error control.

The proposed statistic compares weighted persistence intensity functions through their embeddings in a reproducing kernel Hilbert space. To analyze its power, we derive a bandwidth-dependent upper bound on the uniform separation rate over anisotropic Sobolev balls of weighted intensity differences. Optimizing this bound identifies a bandwidth that attains the fastest rate. We establish a matching minimax lower bound using alternatives realized by \v{C}ech persistence diagrams of point clouds on circles, showing that this rate is optimal within the persistence diagram model. Because the optimal bandwidth depends on unknown smoothness parameters, we aggregate studentized tests over a collection of bandwidths. The resulting test controls asymptotic Type I error and attains the minimax separation rate up to an iterated logarithmic factor without knowing the smoothness parameters.

Finally, we use numerical simulations to examine finite-sample Type I error control under equality of persistence intensities, including cases in which the diagram distributions differ, and to compare power with existing two-sample tests for persistence diagrams. In the alternatives considered, the proposed test exhibits higher empirical power than the competing methods.

\subsection{Related work}
In this subsection, we briefly review existing works related to this paper. Our work lies at the intersection of the following three fields
\paragraph*{Analysis of persistence intensity functions} A persistence diagram can be viewed as a random measure on $\Rb^2$. The persistence intensity function quantifies the expected number of diagram points contained in a given subset of $\Rb^2$. Therefore, analyzing and estimating the intensity function plays a crucial role in understanding the expected behavior of persistence diagrams. Sufficient conditions for the existence and the smoothness of intensity functions are studied in \citet{Chazal2019}. Estimation for the intensity function of persistence diagrams has been studied in \citet{conf/icml/DivolL21,wu2024estimation}, which focus on statistical consistency and convergence properties.
\citet{chen2015statisticalanalysispersistenceintensity} also consider a two-sample comparison based on persistence intensity estimates with permutation calibration. However, equality of persistence intensities does not ensure exchangeability of persistence diagrams, so permutation validity is not generally guaranteed under their null. Moreover, theoretical guarantees for the power of the test were not developed.
\paragraph*{Testing methods using persistence diagrams} 
In TDA, several methods have been proposed for two-sample testing of persistence diagrams \citep{Bubenik2015,RobinsonTurner2017,kusano2018expectation,you2022comparingmultiplelatentspace,MoonLazar2023}. 
These approaches typically focus on specific representations of diagrams. 
In \citet{MoonLazar2023}, the persistence image representation is employed, where each pixel of the image is treated as a random variable, and a two-stage (filtering followed by testing) multiple testing procedure is studied. 
The work of \citet{kusano2018expectation} investigates a vectorization method for persistence diagrams using kernels, and proposes a two-sample test based on a kernel-induced distance between diagrams. 
Permutation-based two-sample tests are suggested in \citet{Bubenik2015,RobinsonTurner2017,you2022comparingmultiplelatentspace}, where \citet{Bubenik2015,you2022comparingmultiplelatentspace} employ the distance based on persistence landscapes, while \citet{RobinsonTurner2017} uses the distance defined directly on persistence diagrams.

\paragraph*{Kernel-based two-sample tests} 
Kernel-based methods, especially using the Maximum Mean Discrepancy (MMD), have been widely used in nonparametric two-sample testing \citep{NIPS2006_e9fb2eda,JMLR:v13:gretton12a,schrab2023mmdagg}. Kernel-based hypothesis testing methods use the distance between two distributions of interest that are embedded in a reproducing kernel Hilbert space (RKHS) generated by a kernel. 

Recent work in this field has incorporated bandwidth or kernel aggregation to improve robustness with respect to tuning-parameter choices \citep{SchrabKSD,schrab2023mmdagg}, and has been extended to non-Euclidean data, such as manifold data \citep{Cheng2024}, functional data \citep{WynneDuncan2022}, and persistence diagrams \citep{kusano2018expectation}. 

\subsection{Our contributions}

Our first contribution is a studentized kernel test for equality of persistence intensity functions under a null that does not require identical diagram distributions. The underlying $U$-statistic is degenerate when the intensities are identical. For shrinking bandwidths, we establish uniform asymptotic normality of its studentized form while allowing the two diagram distributions to have different covariance operators. This yields uniform asymptotic Type I error control under equality of persistence intensities.

Our second contribution is a minimax analysis of power over anisotropic Sobolev balls for differences of weighted persistence intensity functions. We derive a bandwidth-dependent upper bound on the uniform separation rate and optimize it over the bandwidth. A matching minimax lower bound establishes the optimality of the resulting rate. Since the optimal bandwidth depends on unknown smoothness parameters, we aggregate tests over a collection of bandwidths and show that the resulting test remains asymptotically valid and attains the minimax rate up to an iterated logarithmic factor.

Our analysis builds on the approaches of \citet{kim2022minimax,schrab2023mmdagg,li2019optimality}, but requires additional arguments for random persistence diagrams. To accommodate potentially unbounded diagram cardinality, we formulate regularity conditions that provide the moment bounds needed for the studentized limit and the power analysis. For the minimax lower bound, we characterize the one-dimensional \v{C}ech persistence diagrams of finite point clouds on a circle. This characterization allows us to realize a hard-to-distinguish family of density alternatives as intensity functions of genuine random persistence diagrams.

\subsection{Organization} In Section \ref{Sec::Preliminaries}, we introduce the basic notions of TDA and collect the assumptions and notation used throughout the paper. In Section \ref{sec::Test Statistic}, we define a distance between weighted intensity functions and its unbiased estimator. In Section \ref{Sec::Testing method}, the proposed test is introduced and its asymptotic validity and power are analyzed. In Section \ref{sec::Topological properties of cech on the circle}, we study the \v{C}ech complex on the circle for the minimax lower bound. Sections \ref{Sec:: Bandwidth-Aggregated Two Sample Test} and \ref{Sec::Simulation Study} present the bandwidth-aggregated test and simulation studies, respectively.

\section{Preliminaries}\label{Sec::Preliminaries}
In this section, we introduce the necessary preliminaries, notation, and assumptions.
\subsection{Persistent homology}

Homology provides an algebraic description of structural features such as
connected components and higher-dimensional holes. In TDA, homology is
typically applied to simplicial complexes constructed from observed data,
often over a range of scale parameters. By tracking how the homology of these
complexes changes with the scale, one can capture not only topological
features of the data but also geometric information encoded in the scales at
which those features appear and disappear.
We briefly introduce the notions needed in this
paper; additional background on homology and persistent homology is provided
in Section~\ref{sec::persistent homology} of the Supplementary Material.

Let $\Xc$ be a subset of a metric space $(\Mc,d)$. A common way to construct a
simplicial complex from $\Xc$ is to connect points that are sufficiently close
to one another. For example, for a scale parameter $r>0$, the \v{C}ech complex
of $\Xc$ is defined by
\[
\text{\v{C}ech}_r(\Xc)
:=
\left\{
\sigma\subseteq\Xc:
\bigcap_{x\in\sigma} B(x,r)\neq\varnothing
\right\},
\]
where $\Bc(x,r) =\left\{y\in \Mc  : d(x,y)\leq r \right\}$.
The Vietoris--Rips complex is another commonly used construction in TDA and
is employed in our numerical experiments.

The topology of the resulting simplicial complex depends on the choice of
the scale parameter $r$. At small scales, only nearby points are connected,
whereas increasing $r$ gradually creates and fills topological features.
Persistent homology tracks these changes across scales rather than examining
the topology at a single fixed value of $r$.

More generally, let $\Kc(\Xc,r)$ denote a simplicial complex constructed from
$\Xc$ at scale $r$. The collection $
\Kc(\Xc):=\{\Kc(\Xc,r):r>0\}
$
is called a filtration if
$
\Kc(\Xc,r_1)\subseteq\Kc(\Xc,r_2)
$ whenever $ r_1\le r_2.
$
For $k\geq0$, we denote by
$
\mathrm{PH}_k(\Kc(\Xc))
$
the $k$-th persistent homology associated with this filtration. It records
how $k$-dimensional homological features emerge and disappear as the scale
parameter varies. For each such feature, the scale at which it appears is
called its \emph{birth time}, and the scale at which it disappears is called
its \emph{death time}.

\subsection{Persistence diagram}\label{sec::Persistence diagram}
A persistence diagram is a useful tool to describe persistent homology. For a fixed $k\in\{0,1,2,\dots\}$, consider the $k$-th dimensional topological feature represented by an element of $\mathrm{PH}_k(\Kc(\Xc))$ with birth and death times \(b\) and \(d\), respectively. This feature can be identified with the point $(b,d)$ in $\Omega:=\{(x,y)\in\Rb^2 : y> x\geq0\}$. The multiset of all birth-death pairs in $\Omega$ of $k$-th dimensional topological features in $\mathrm{PH}_k(\Kc(\Xc))$, denoted by $\mathrm{PD}_k(\Kc(\Xc))$, is called the $k$-th \emph{persistence diagram}; see Figure \ref{fig:S^2example} for an example. 

\begin{figure}[t]
    \centering
    \includegraphics[width=1.0\textwidth]{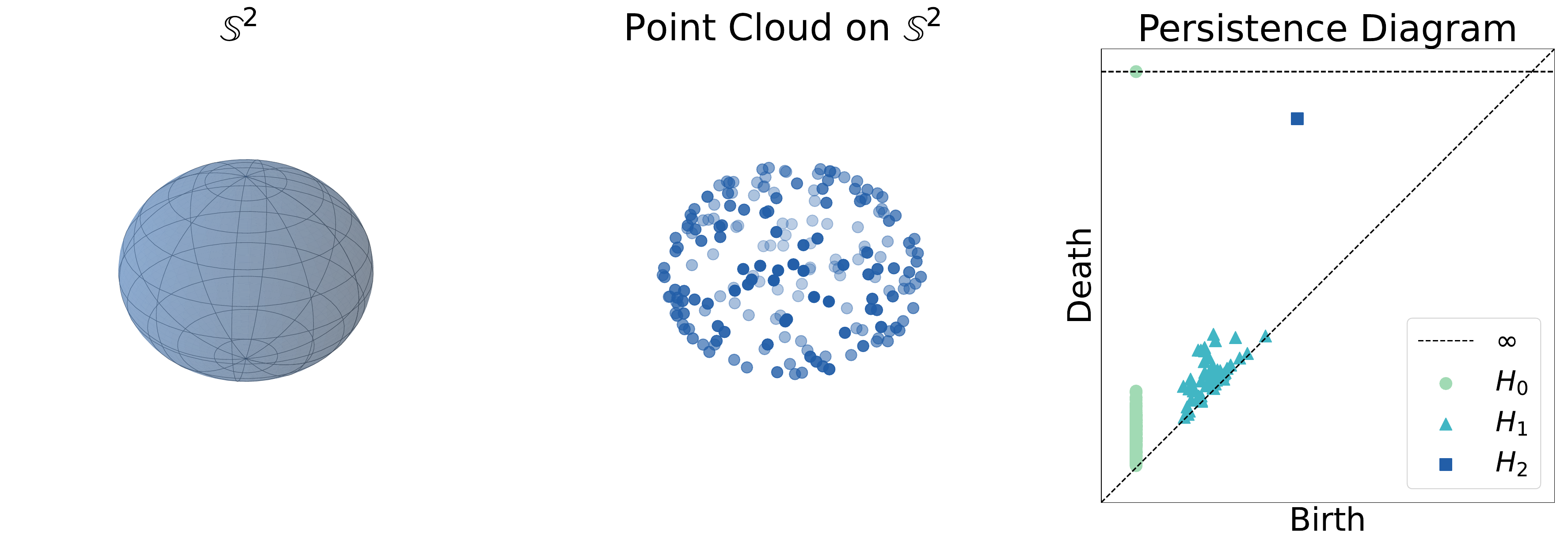}
    \caption{(Left) the sphere $\mathbb{S}^2$, (Middle) a point cloud on $\mathbb{S}^2$, and (Right) the persistence diagram obtained using the Vietoris-Rips filtration.}
    \label{fig:S^2example}
\end{figure}

\paragraph*{Randomness of persistence diagrams}
Consider a random point cloud $\Xc$ on an underlying topological space $\Mb$.
We regard a persistence diagram as a locally finite counting measure on $\Omega$.
We denote by $\mathbf{PD}$ the space of locally finite integer-valued
measures on $\Omega$, interpreted as persistence diagrams. We equip
$\mathbf{PD}$ with the sigma-field \citep{hiraoka2018limit},
\[
\mathcal B(\mathbf{PD})
\coloneq
\sigma\{D\mapsto D(B):B\in\Bc(\Omega),\ 
B\text{ relatively compact}\},
\]
where $\Bc(\Omega)$ is the Borel sigma-field of $\Omega$.
For any homological degree $k$ and filtration rule $\Kc$, the persistence
diagram $\mathrm{PD}_k(\Kc(\Xc))$ is regarded as an element of $\mathbf{PD}$.
We assume that the map
$
\Xc\mapsto \mathrm{PD}_k(\Kc(\Xc))
$
is measurable with respect to the $\sigma$-field. Then
$
D=\mathrm{PD}_k(\Kc(\Xc))$
is a random persistence diagram. Moreover, if $\Xc$ is generated by a stochastic process on $\Mb$, such as
a Poisson process, then the cardinality of $D$ is also a random variable.

\paragraph*{Weight functions on a persistence diagram}
Persistence diagrams often contain topological noise arising from the construction of simplicial complexes on point clouds.
Such noisy features usually have short persistence and appear near the diagonal $y=x$.
To emphasize more persistent, and hence more informative, features, many TDA methods assign weights to diagram points through a weight function.
This idea is used in several vectorization methods, such as persistence images \citep{persistenceimage} and the persistence weighted Gaussian vector \citep{kusano2016persistence,kusano2017kernel}. 

\subsection{Persistence intensity functions}\label{sec::Persistence intensity function} 
A persistence diagram $D$ is a multiset in $\Omega$. This multiset $D$ can be regarded as a discrete measure on $\Omega$, expressed as a sum of the Dirac measures: 
$\sum_{\bx \in D}\delta_\bx$.
As discussed in Section \ref{sec::Persistence diagram}, we treat a persistence diagram as a random object. Hence, it can be considered as a random measure. 
The expectation of the random measure $D$ is a deterministic measure on $\Omega$, denoted by $\Eb\left(D\right)$, and defined by 
\begin{align*}
    \Eb\left(D\right)\left(B\right)\coloneq\Eb(D(B)) =\Eb\left(\sum_{\mathbf{x}\in D}\delta_\mathbf{x}\left(B\right)\right),
\end{align*}
for any Borel set $B\subseteq \Omega$. This deterministic measure represents the expected number of points in $D$ contained in the given $B$. Under suitable conditions, the expectation measure $\Eb(D)$ is absolutely continuous with respect to the Lebesgue measure on $\Omega$. Consequently, there exists a Radon-Nikodym derivative $\frac{d\Eb(D)}{d\text{Leb}_2}$ satisfying that for every Borel set $B$,
\begin{align}\label{eq::intensity_func}
    \int_{B} \frac{d\Eb(D)}{d\text{Leb}_2} d\text{Leb}_2 = \Eb(D)(B),
\end{align}
where $\text{Leb}_2$ is the Lebesgue measure on $\Omega$.

In TDA, this density function is called the persistence intensity function (or simply the intensity function). For the sake of completeness, the sufficient conditions for the existence of intensity functions studied in \citet{Chazal2019} are presented in Section \ref{Appendix::A-sufficient-condition-for-existence-of-intensity-function} of the Supplementary Material.

The persistence intensity function plays a central role in TDA. For a random persistence diagram $D$, the persistence intensity function quantifies the expected number of points of $D$ contained in a Borel set $B$, as given in \eqref{eq::intensity_func}. Therefore, the intensity function represents the expected behavior of $D$ and provides a first-order summary of its underlying probability distribution. Furthermore, this representation is interpretable and allows for the straightforward calculation of linear transformations of $D$. 
Figure~\ref{fig::intensity_func_example} provides a visual illustration of the persistence intensity function.

However, persistence intensity functions do not capture all information contained in persistence diagrams. They do not capture dependencies among the topological features. Nevertheless, the information that persistence intensity functions do capture—namely, the existence and expected scales of topological features—is often of primary interest in TDA research \citep{Topaz2015,amorphoussolids,Wilding2021}. 
These observations highlight the role of persistence intensity functions as informative first-order summaries of random persistence diagram distributions.

\begin{figure}[t]
    \centering
    \includegraphics[width=1.0\textwidth]{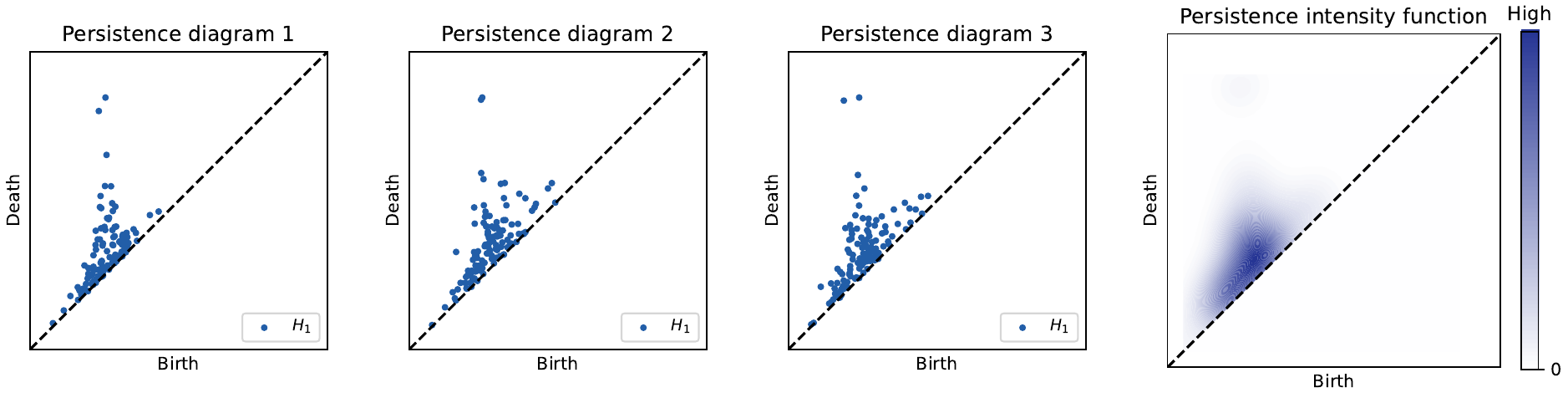}
    \caption{(left) Three persistence diagrams randomly sampled from a random persistence diagram $D$. (right) The estimated persistence intensity function of $D$, which is constructed using the three samples.}
    \label{fig::intensity_func_example}
\end{figure}

\subsection{Assumptions}\label{sec::Assumptions and Notations} 
We first define notations used throughout the paper. 
We denote the two sample sizes by $n$ and $m$.
For nonnegative arrays $a_{n,m}$ and $b_{n,m}$, we write $a_{n,m}\lesssim b_{n,m}$ if $a_{n,m}\leq Cb_{n,m}$ for some constant $C>0$ and all sufficiently large $n,m$, and define $a_{n,m}\gtrsim b_{n,m}$ analogously. We write $a_{n,m}\asymp b_{n,m}$ when both relations hold. Moreover, $a_{n,m}=O(b_{n,m})$ means $|a_{n,m}|\lesssim b_{n,m}$, whereas $a_{n,m}=o(b_{n,m})$ means $a_{n,m}/b_{n,m}\to0$, as $n,m\to \infty$. 
For random variables $Z_{n,m}$ and $Z$, $Z_{n,m}\pconv Z$ and \(Z_{n,m}\xrightarrow{d}Z\) denote convergence in probability and in distribution, respectively. We use \(C(\vartheta_1,\ldots,\vartheta_r)\) to denote a generic finite positive constant depending only on the listed quantities. Lastly, we let $\Nb_0:= \Nb \cup \{0\}.$

We assume that the sample sizes satisfy $n\asymp m$. We next state the assumptions on kernel functions, weight functions, and probability distributions.
\paragraph*{Conditions on the kernels}
Let $k^i : \Rb\rightarrow \Rb_{\geq0} $ be a positive definite continuous function in $L^1(\Rb)\cap L^2(\Rb)\cap L^4(\Rb)$ such that $k^i(x)=k^i(-x)$, $\int k^i(x)dx =1$ for $i=1,2$, where $\Rb_{\geq t}:=\{x\in \Rb:x\geq t\}$, for $t\in \Rb$. For bandwidths $\lambda = (\lambda_1,\lambda_2) \in (0,\infty)^2$, we define the kernel function $k_\lambda :\Rb^2\times \Rb^2 \rightarrow \Rb_{\geq 0}$ by 
\begin{align}\label{kernel def}
    k_\lambda(x,y) := \prod_{i=1}^2 \frac{1}{\lambda_i}k^i\left(\frac{x_i-y_i}{\lambda_i}\right) ,
\end{align}
which satisfies that 
\begin{align*}
    \int_{\Rb^2}k_\lambda(x,y)dx=1 ,\quad \int _{\Rb^2}k_\lambda(x,y)^2dx = \frac{\kappa_2}{\lambda_1\lambda_2},
\end{align*}
where $\kappa_2 := \prod_{i=1}^2 \int_{\Rb}k^i(x_i)^2dx_i <\infty$. 
Examples satisfying the above conditions include the Gaussian kernel with $k^i(u) =\frac{1}{\sqrt{\pi}}\exp(-u^2)$ and the Laplace kernel with $k^i(u) = \frac{1}{2}\exp(-|u|)$.
For notational simplicity, we also write for $u=(u_1,u_2)\in \Rb^2$,
\begin{align*}
    \varphi_\lambda (u) :=\prod_{i=1}^2 \frac{1}{\lambda_i}k^i\left(\frac{u_i}{\lambda_i}\right),
\end{align*}
so that $k_\lambda(x,y) = \varphi_\lambda(x-y)$, for $x,y\in \Rb^2$.
Throughout the remainder of this paper, we restrict our attention to such kernel functions.
\paragraph*{Conditions on the weight functions} We restrict the weight function $w(x,y)$ to be a positive and smooth function that is nondecreasing in $y-x$. Under Assumption \ref{assump::Bounded persistence condition}, defined later, we restrict the domain of $w$ to
\[
\Omega(M) \coloneq \{(x,y) \in \Omega : y < M\},
\]
so that $\|w^2\|_\infty <\infty.$
When it is necessary to view the weight function as a function on $\mathbb{R}^2$, 
we assume that $w$ admits a compactly supported smooth extension to $\mathbb{R}^2$, 
still denoted by $w$.
Examples of admissible weight functions include: 
\begin{enumerate}
    \item A polynomial function on $\Omega(M)$: $w(x,y) = (y-x)^n$, $n \in \mathbb{N}$,
    \item An arctangent function on $\Omega(M)$: $w(x,y) = \arctan(y-x)$.
\end{enumerate}
Throughout the remainder of this paper, we restrict our attention to such weight functions.

\paragraph*{Conditions on the probability distributions}
For regularity constants $M,N,L>0$, we define the class $\Pc:=\Pc(M,N,L)$ of probability distributions on $\mathbf{PD}$ as the set of all distributions $P$ satisfying the following assumptions. In the following, let $D$ be a random persistence diagram whose probability distribution is $P\in \Pc$.

\begin{enumerate}
\renewcommand{\labelenumi}{\theenumi} \renewcommand{\theenumi}{(A\arabic{enumi})}
  \item\label{assump::randomness of diagram}
  $D = \mathrm{PD}_k(\Kc(\Xc))$ for some $k\geq 0$, filtration $\Kc$, 
  and random point cloud $\Xc$.

  \item\label{assump::Bounded persistence condition} 
  For any $(x,y)\in D$, $y<M$. 
  \item\label{assump::non degeneracy}
    If $p$ is the persistence intensity function of $D$, then $\|w^2p\|_{L^2}\geq L$.
  \item\label{assump::core assumption} There exists a  $\Nb_0$-valued discrete random variable $Z$  and a function 
$h:\Nb_0\to \Rb_{\geq 1}$ such that,
 conditioned on $Z=\ell$, $|D| \leq h(\ell)$, 
 and the following hold:

  \begin{itemize} 
      \item For each $\ell \in \Nb_0$, there exists a density function 
      \[
      p_\ell = \frac{d \mathbb{E}[D \mid Z = \ell]}{d \mathrm{Leb}_2},
      \]
      \item The conditional densities satisfy the following weighted summability:
      \[\sum_{\ell=0}^\infty h(\ell)^3\|w^2\cdot p_{\ell}\|_\infty \Pb(Z=\ell)\leq N.
      \]
  \end{itemize}
\end{enumerate}
\begin{rmk}[Concrete formulations of Assumption \ref{assump::core assumption}]\label{Realizations of Assumption}
We describe two useful reformulations of Assumption \ref{assump::core assumption}, depending on the choice of $Z$.
When $Z$ is taken to be constant, the assumption reduces to:  
\begin{enumerate}
\renewcommand{\labelenumi}{\theenumi} \renewcommand{\theenumi}{(B\arabic{enumi})}
  \item\label{assump::Bounded cardinality condition} There exists a constant $B>0$ such that 
   $|D|\leq B$,  
  \item\label{assump::Existence of intensity function}
  There exists a density function $\frac{d\Eb(D)}{d\text{Leb}_2}=p$, and it satisfies $\|w^2\cdot p\|_\infty \le  N/B^3$.
\end{enumerate}
When $Z=|\Xc|$, Assumption \ref{assump::core assumption} is satisfied if the following holds.
\begin{enumerate}
\renewcommand{\labelenumi}{\theenumi} \renewcommand{\theenumi}{(C\arabic{enumi})}
  \item\label{assump::Control of the diagram cardinality}
  $|D|(|\Xc|) \le 2^{|\mathcal{X}|}$.
  \item\label{assump::existence of conditional intensity}
  For each $\ell \in \Nb_0$, there exists a density function $p_\ell = \frac{d \mathbb{E}[D \mid |\mathcal{X}| = \ell]}{d \mathrm{Leb}_2},$ and a constant $U>0$, independent of $P$, such that $\|w^2\cdot p_\ell\|_\infty \leq U^{\ell}.$

  \item\label{assump::Tail behavior} 
  $\sum_{\ell=0}^{\infty} 8^\ell\cdot U^{\ell}
  \cdot \Pb(|\mathcal{X}| = \ell) \le N$.
\end{enumerate}
\end{rmk}

\begin{rmk}[Brief discussion of the assumptions] In this remark, we briefly discuss the above assumptions; see Section~\ref{Appendix::Description-of-the-assumptions} of the Supplementary Material for detailed discussions.
\begin{itemize}
\item Binomial and Poisson point processes are standard models for random point clouds in TDA, while the Čech and Vietoris--Rips filtrations are among the most widely used filtration. When $k\geq 1$, our model assumptions hold for persistence diagrams generated by binomial point processes under either the Čech or Vietoris--Rips filtration, and by Poisson point processes under the Vietoris--Rips filtration.

\item The boundedness conditions on $\|w^2\cdot p\|_\infty$ and $\|w^2\cdot p_{\ell}\|_\infty$ are weaker than those on $\|p\|_\infty$ and $\|p_{\ell}\|_\infty$, respectively. Depending on the filtration employed, intensity functions may diverge near the diagonal $y=x$, while the products $w^2\cdot p$ and $w^2\cdot p_{\ell}$ can remain uniformly bounded under a suitable choice of weight function.

\item Assumption \ref{assump::core assumption} is introduced to cover models in which the cardinality of a diagram is unbounded. When the cardinality is bounded, as in \ref{assump::Bounded cardinality condition}, the setting becomes straightforward. However, when the cardinality is unbounded, one must control its effect in order to bound the moments of the test statistic and this can be achieved 
under Assumption \ref{assump::core assumption}, more concretely under \ref{assump::Control of the diagram cardinality}-\ref{assump::Tail behavior}. In this sense, Assumption \ref{assump::core assumption} is crucial in our analysis.
\end{itemize}
\end{rmk}

We have defined the class $\Pc$ of probability distributions. We now introduce notation related to $\Pc$.
Let $\Pc^{\otimes 2}$ denote the Cartesian product $\Pc\times \Pc$. 
For a family $\Fc$ of pairs of intensity functions, we define the subclass $\Pc^{\otimes 2}_{\Fc}$ of $\Pc^{\otimes 2}$ by 
\begin{align}\label{eq::sub class of distributions}
    \Pc^{\otimes 2}_{\Fc}\coloneqq\left\{(P,Q)\in\Pc^{\otimes 2}: (p,q) \in \Fc, \, \mathbb{E}(P)=p,\, \mathbb{E}(Q)=q\right\}.  
\end{align}
The notation $\E(P)=p$ indicates that the intensity function of the distribution $P$ is $p$.
For the null and alternative classes, we define 
\begin{align}
    &\Pc^{\otimes 2}_0\coloneqq\{(P,Q)\in \Pc^{\otimes 2} : p= q , \, \E(P)=p, \, \E(Q)=q\}, \label{eq::nullsublcass} \\
    &\Pc^{\otimes 2}_1\coloneqq
    \Pc^{\otimes 2} \setminus \Pc_0^{\otimes 2}.
    \label{eq::alternative sublcass}
\end{align}
\subsection{Minimax testing framework}
In this subsection, we introduce the minimax testing framework used in our theoretical analysis.
A brief review of minimax testing theory is provided in
Section~\ref{Appendixsecminimaxrateoftesting} of the Supplementary Material.

Let $\Cc$ be a function class. Since our object of interest is the weighted
difference of persistence intensity functions, we define
\begin{align*}
\mathcal{F}_{\rho}(\mathcal{C},w) := \Bigl\{ (p,q) :   
 w \cdot (p-q) \in \mathcal{C},
 \| w \cdot (p-q) \|_2 \geq \rho 
\Bigr\}.
\end{align*}
The use of the weighted $L^2$-distance reflects the fact that, in many TDA
applications, differences occurring farther from the diagonal $y=x$ are regarded
as more significant.

For $n,m\in \Nb$ and $P,Q\in \Pc$,
let $\Pb_{P,Q}^{n,m}=P^{\otimes n}\otimes Q^{\otimes m}$ denote the joint distribution of two independent samples of sizes $n$ and $m$ drawn from $P$ and $Q$, respectively.
For a test $\Delta$, its uniform separation rate over $\Cc$ is defined by
\begin{align*}
\rho_{n,m}(\Delta,\Cc,\beta,w,M,N,L)
:=
\inf\left\{
\widetilde{\rho}>0:
\sup_{(P,Q)\in
\Pc^{\otimes 2}_{\Fc_{\widetilde{\rho}}(\Cc,w)}}
\Pb^{n,m}_{P,Q}
\bigl(
\Delta(\Xb_n,\Yb_m)=0
\bigr)
\leq \beta
\right\}.
\end{align*}
Let $\Psi_{\alpha}$ denote the class of all sequences of tests that are asymptotically of level $\alpha$: 
\begin{align}\label{eq::asymptotic-test-class}
    \Psi_{\alpha}
    &:=
    \left\{
       \mathbf{\Delta}=(\Delta_{n,m}):
       \limsup_{n,m\rightarrow \infty} \sup_{(P,Q)\in\Pc_0^{\otimes 2}}
        \Pb^{n,m}_{P,Q}
        \left(
            \Delta_{n,m}(\Xb_n,\Yb_m)=1
        \right)
        \leq
        \alpha
    \right\}.
\end{align}
We call a positive sequence $(\rho_{n,m}^\dagger)_{n,m\geq1}$ satisfying $\rho_{n,m}^{\dagger}\to 0$ a minimax separation rate if
\begin{align}\label{Eq::asymptotic-minimax-radius}
    0<\inf_{(\Delta_{n,m})\in \Psi_{\alpha}}\liminf_{n,m\to \infty} \frac{\rho_{n,m}(\Delta_{n,m})}{\rho_{n,m}^{\dagger}}\leq \inf_{(\Delta_{n,m})\in \Psi_{\alpha}}\limsup_{n,m\to \infty} \frac{\rho_{n,m}(\Delta_{n,m})}{\rho_{n,m}^{\dagger}}<\infty,
\end{align}
Here, for notational simplicity, we suppress the dependence of $\rho_{n,m}$ and $\rho_{n,m}^{\dagger}$ on $\mathcal C$, $\alpha$, $\beta$, $w$, $M$, $N$, and $L$.

The function space $\Cc$ of interest is the anisotropic Sobolev ball defined by
\begin{align}\label{Sobolev ball}
    \Sc_2^{s_1,s_2}(R)\coloneq\left\{f\in L^1(\Rb^2)\cap L^2(\Rb^2) \mid \int_{\Rb^2} \left(|\xi_1|^{2s_1}+|\xi_2|^{2s_2}\right)|\hat{f}(\xi)|^2d\xi \leq (2\pi)^2R^2 \right\},
\end{align}
for $s_1,s_2>0$ and $R>0$, where $\widehat{f}$ is the Fourier transform of an integrable function $f$, with convention $\widehat{f}(\xi)=\int e^{-i\langle x,\xi\rangle}f(x)dx.$ 
Note that the anisotropic Sobolev ball extends the standard Sobolev ball by allowing the smoothness parameters $s_i$ to differ, thereby providing additional flexibility and yielding novel aspects in the analysis.

\section{Test Statistic}\label{sec::Test Statistic}
In this section, we define a distance between two intensity functions embedded in a reproducing kernel Hilbert space (RKHS) and obtain its unbiased estimator; see Section \ref{Appendix::Reproducing-kernel-Hilbert-spaces} of the Supplementary Material for the definition of the RKHS.

Let $P$ and $Q$ be probability distributions in $\Pc$. We have two groups of samples $X_1,\dots,$ $ X_n \overset{\iid}{\sim}P$ and $Y_1,\dots,Y_m \overset{\iid}{\sim}Q$. 
Let $w$ be a weight function and $k_\lambda$ be a kernel function. Consider the RKHS $\Hc_{k_{w,\lambda}}$ generated by the weighted kernel function $k_{w,\lambda}(x,y)=w(x)w(y)k_\lambda(x,y)$, which was first introduced in \citet{kusano2018expectation}. We embed an intensity function into this RKHS as follows. For an intensity function $p$, define 
\begin{align}\label{Eq::mean-embedding}
    \mu_p = \int  p(x)w(\cdot)w(x)k_\lambda(\cdot,x)dx \in \Hc_{k_{w,\lambda}}. 
\end{align}
For the two intensity functions $p$ and $q$, the squared $\Hc_{k_{w,\lambda}}$-norm $\|\mu_p-\mu_q\|_{\Hc_{k_{w,\lambda}}}^2$ of $\mu_p-\mu_q$ can be estimated by the following unbiased estimator: 
\begin{align}\label{U-stat}
    \widehat{\|\mu_p-\mu_q\|}_{\Hc_{k_{w,\lambda}}}^2 &= \frac{1}{n(n-1)m(m-1)}\sum_{1\leq i\neq i'\leq n}\sum_{1\leq j\neq j'\leq m} h_\lambda(X_i,X_{i'},Y_j,Y_{j'}),
\end{align}
where $h_\lambda(X,X',Y,Y') := K_{\lambda}(X,X')+K_{\lambda}(Y,Y')-K_{\lambda}(X,Y')-K_{\lambda}(X',Y)$ for $X,X',Y,Y' \in \mathbf{PD} $ and 
\begin{align*}
    K_{\lambda}(X,Y):=\left\langle \Phi_{\lambda}(X), \Phi_{\lambda}(Y)\right\rangle_{\Hc_{k_{w,\lambda}}},\quad \Phi_{\lambda}(X):=\sum_{z\in X}k_{w,\lambda}(\cdot,z),
\end{align*}
see Section \ref{Appendix::Analysis-for-test-statistic} of the Supplementary Material for the details.
\begin{rmk}\label{rmk::mmd}
In this remark, we briefly introduce the statistical inference method based on the maximum mean discrepancy $(\mathrm{MMD})$. For convenience,  consider two probability densities $f$ and $g$ on $\mathbb{R}^2$, and let $X_1,\dots, X_n \overset{\iid}{\sim} f $ , $Y_1,\dots, Y_m \overset{\iid}{\sim} g $. The $\mathrm{MMD}$ is a distance between $f$ and $g$ defined by
\begin{align*}
    \mathrm{MMD}(f,g,\mathcal{H}_k) := \sup_{h \in \mathcal{H}_k : \|h\|_{\mathcal{H}_k} \leq 1} \Big| \mathbb{E}_{X \sim f}[h(X)] - \mathbb{E}_{Y \sim g}[h(Y)] \Big|,
\end{align*}
where $\Hc_k$ is the RKHS generated by a kernel function $k$ on $\Rb^2\times \Rb^2$ \citep{JMLR:v13:gretton12a}.
An unbiased estimator for the squared MMD is given by
\begin{align}\label{eq::mmd_statistic}
     \frac{1}{n(n-1)} \sum_{i=1}^n \sum_{j \neq i} k(X_i, X_j)
    + \frac{1}{m(m-1)} \sum_{i=1}^m \sum_{j \neq i} k(Y_i, Y_j)
     - \frac{2}{nm} \sum_{i=1}^n \sum_{j=1}^m k(X_i, Y_j), 
\end{align}
that can be used for two-sample testing to assess whether $f = g$. 
\end{rmk}
\begin{rmk}\label{rmk::our_data_trait}
When each persistence diagram consists of a single point, our setting reduces to that of Remark~\ref{rmk::mmd}. 
In contrast, in our framework each observation is a persistence diagram, that is, a random multiset. 
As a consequence, computing kernel values requires pairwise evaluations over all points in the diagrams, which differs from the Euclidean data setting. 
This difference necessitates controlling the cardinality of persistence diagrams in order to bound the moments of the test statistic in \eqref{U-stat}, and explains why existing techniques, such as those of \citet{kim2022minimax,schrab2023mmdagg,li2019optimality}, cannot be directly applied.
\end{rmk}

Denote $U_{n,m,\lambda}:=\widehat{\|\mu_p-\mu_q\|}_{\Hc_{k_{w,\lambda}}}^2$.
We next construct an estimator of the null variance of
$U_{n,m,\lambda}$. Under $\mathrm{H}_0:p=q$, write
$\mu:=\mu_p=\mu_q$ and define
\begin{align*}
    A_i:=\Phi_{\lambda}(X_i)-\mu,\quad
    B_j:=\Phi_{\lambda}(Y_j)-\mu.
\end{align*}
The corresponding covariance operators are defined by 
\begin{align*}
    \Sigma_{P,\lambda}
    :=
    \Eb[A_1\otimes A_1],\quad
    \Sigma_{Q,\lambda}
    :=
    \Eb[B_1\otimes B_1],
\end{align*}
where
$(u\otimes v)f:=\langle v,f\rangle_{\Hc_{k_{w,\lambda}}}u$, for $u,v,f\in \Hc_{k_{w,\lambda}}$.
Define
\begin{align}
    a_{P,\lambda}
    :=
    \tr\big(\Sigma_{P,\lambda}^{2}\big),\quad
    a_{Q,\lambda}
    :=
    \tr\big(\Sigma_{Q,\lambda}^{2}\big),\quad
    a_{PQ,\lambda}
    :=
\tr\big(\Sigma_{P,\lambda}\Sigma_{Q,\lambda}\big),
    \label{Eq::variance-components}
\end{align}
where $\tr(T)$ denotes the trace of a trace-class operator $T$; see
Section~\ref{Appendix::Hilbert-space-operators} of the Supplementary
Material for the formal definition.
Then, the null variance of $U_{n,m,\lambda}$ is 
\begin{align}\label{Eq::null-variance}
    \sigma_{n,m}^{2}
    :=
    \V_{P,Q}\big(U_{n,m,\lambda}\big)
    =
    \frac{2a_{P,\lambda}}{n(n-1)}
    +\frac{2a_{Q,\lambda}}{m(m-1)}
    +\frac{4a_{PQ,\lambda}}{nm}.
\end{align}
Let $(r)_s:=r(r-1)\cdots(r-s+1)$, and let $\overset{\neq}{\sum}$ indicate that
the indices in that sum are mutually distinct.  Define
\begin{align*}
    \widehat a_{P,\lambda}
    &:=
    \frac{1}{4(n)_4}
    \sum^{\neq}_{\substack{1\leq i_1,i_2,i_3,i_4\leq n }}
    h_\lambda
    (X_{i_1},X_{i_2},X_{i_3},X_{i_4})^2,\\
    \widehat a_{Q,\lambda}
    &:=
    \frac{1}{4(m)_4}
    \sum^{\neq}_{\substack{1\leq j_1,j_2,j_3,j_4\leq m}}
    h_\lambda
    (Y_{j_1},Y_{j_2},Y_{j_3},Y_{j_4})^2,\\
    \widehat a_{PQ,\lambda}
    &:=
    \frac{1}{4(n)_2(m)_2}
    \sum_{\substack{1\leq i_1\neq i_2\leq n\\
                    1\leq j_1\neq j_2\leq m}}
    h_\lambda
    (X_{i_1},Y_{j_1},X_{i_2},Y_{j_2})^2.
\end{align*}
These estimators satisfy
\begin{align*}
    \Eb[\widehat a_{P,\lambda}]=a_{P,\lambda},\quad
    \Eb[\widehat a_{Q,\lambda}]=a_{Q,\lambda},\quad
    \Eb[\widehat a_{PQ,\lambda}]=a_{PQ,\lambda}.
\end{align*}
We consequently estimate \eqref{Eq::null-variance} by
\begin{align}
    \widehat{\sigma}_{n,m,\lambda}^{2}
    &:=
    \frac{2\widehat a_{P,\lambda}}{n(n-1)}
    +
    \frac{2\widehat a_{Q,\lambda}}{m(m-1)}
    +
    \frac{4\widehat a_{PQ,\lambda}}{nm}.
    \label{Eq::variance-estimator}
\end{align}
The studentized test statistic is defined by
\begin{align}
    T_{n,m,\lambda}
    &:=
    \frac{U_{n,m,\lambda}}
         {\widehat{\sigma}_{n,m,\lambda}}.
    \label{Eq::studentized-test-statistic}
\end{align}
On the event $\{\widehat{\sigma}_{n,m,\lambda}=0\}$, we set
$T_{n,m,\lambda}=0$.

\section{Testing procedure}\label{Sec::Testing method}
In this section, we define the proposed test, establish its asymptotic validity under the null, and analyze its power characterized by the $L^2$ distance between weighted persistence intensity functions. The proofs of the results in this section are provided in Section~\ref{Appendix::proof test} of the Supplementary Material.
\subsection{Hypotheses and test function}\label{Hypothesis} 
Consider two independent samples of persistence diagrams
$X_1,\dots,X_n \overset{\iid}{\sim} P$ and $Y_1,\dots,$ $Y_m\overset{\iid}{\sim} Q$,
where $P,Q \in \Pc$, and let $p$ and $q$ denote their corresponding intensity functions.
Under this setting, we consider the hypotheses
\begin{align*}
\mathrm{H}_0:p=q \quad \text{versus} \quad \mathrm{H}_1:p\neq q.
\end{align*}
For a prescribed significance level $\alpha\in(0,1)$, we define the test function by 
\begin{align}\label{testfunction}
    \Delta_{\alpha,\lambda}:= \mathbb{I}\left(T_{n,m,\lambda}>z_{1-\alpha}\right),
\end{align}
where $z_{1-\alpha}$ denotes the $(1-\alpha)$-quantile of the standard normal distribution $\Nc(0,1)$. 
\subsection{Asymptotic validity} 
In this section, we establish that the Type I error of the test $\Delta_{\alpha,\lambda}$ is asymptotically controlled by $\alpha$. The following theorem formally states it. 
\begin{thm}[Asymptotic validity theorem]
\label{Theorem::uniform-validity}
Suppose that the bandwidth sequence
$\lambda_{n,m}$$=(\lambda_{1,n,m},\lambda_{2,n,m})$ satisfies, for $\nu_{n,m}:=\lambda_{1,n,m}\lambda_{2,n,m}$,
\begin{align}
    \max(\lambda_{1,n,m},
    \lambda_{2,n,m})\to 0,
    \quad
    (n+m)^{2}\nu_{n,m}\to\infty,\quad \text{as}\quad n+m\to \infty.
    \label{Eq::validity-bandwidth-condition}
\end{align}
Then, it holds that
\begin{align}\label{Eq::uniform-null-normality}
    \sup_{(P,Q)\in\Pc_0^{\otimes 2}}
    \sup_{t\in\Rb}
    \left|
        \Pb\left(T_{n,m,\lambda}\leq t\right)
        -
        \Phi(t)
    \right|
    \longrightarrow0,
\end{align}
where $\Phi(t)$ is the cdf of $\Nc(0,1)$.
Consequently, the test $\Delta_{\alpha,\lambda}$ in  \ref{testfunction} is asymptotically valid:
\begin{align}\label{Eq::uniform-type-I-error}
    \limsup_{n,m\to \infty}\sup_{(P,Q)\in\Pc_0^{\otimes 2}}
\Pb\left(\Delta_{\alpha,\lambda}=1\right) \leq \alpha. 
\end{align}
\end{thm}
\begin{figure}[t]
    \centering
    \includegraphics[width=1.0\textwidth]{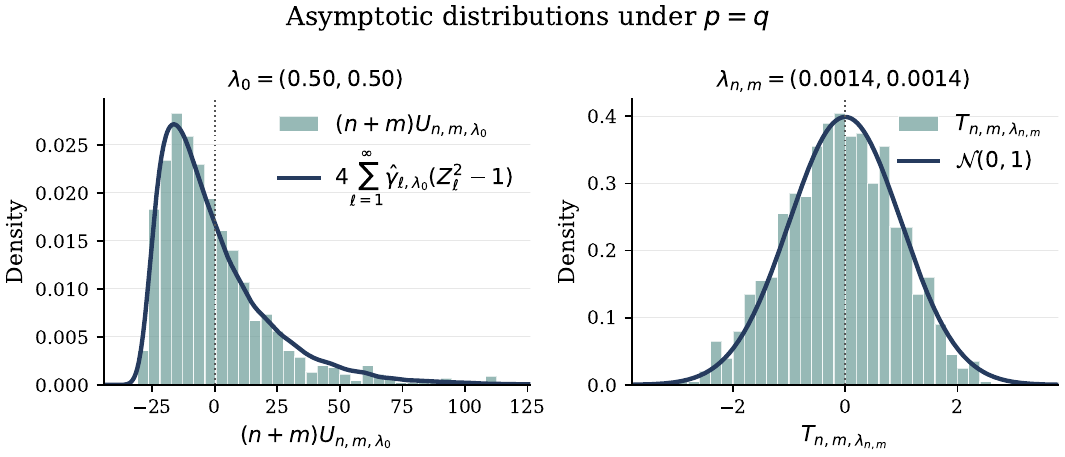}
    \caption{
Empirical and asymptotic null distributions when their persistence intensity functions
coincide. The histograms are based on
$1{,}000$ Monte Carlo replications with $n=m=500$. The left panel
compares $(n+m)U_{n,m,\lambda_0}$ at a fixed bandwidth with its
estimated weighted centered chi-square limit. The right panel compares
the studentized statistic $T_{n,m,\lambda_N}$ at a shrinking bandwidth
with the standard normal density.
}
    \label{fig::comparison of asymptotic dists}
\end{figure}
\begin{rmk}[Gaussianization under a shrinking bandwidth]
\label{Remark::Gaussianization}
The shrinking-bandwidth condition plays an essential role in the
Gaussian limit in Theorem~\ref{Theorem::uniform-validity}. To see this,
suppose first that $\lambda$ is fixed and that
$n/(n+m)\to\eta\in(0,1)$. 
Let
\begin{align*}
    \Gamma_{\eta,\lambda}
    &:=
    (1-\eta)\Sigma_{P,\lambda}
    +
    \eta\Sigma_{Q,\lambda},
\end{align*}
and let $\{\gamma_{\ell,\lambda}\}_{\ell\geq1}$ denote its eigenvalues.
Under the null hypothesis, the conventional limit of the degenerate
$U$-statistic is
\begin{align} \label{Eq::fixed-bandwidth-limit}
    (n+m)U_{n,m,\lambda}
    \ \xrightarrow{d}\ \frac{1}{\eta(1-\eta)}
    \sum_{\ell=1}^{\infty}
    \gamma_{\ell,\lambda}(Z_\ell^2-1),
\end{align}
where $Z_1,Z_2,\ldots$ are independent standard normal random
variables. 
Consequently, studentization alone does not generally produce a
Gaussian limit when the bandwidth is fixed.

However, the situation changes when $\lambda=\lambda_{n,m}\to0$.
For each fixed bandwidth $\lambda$, studentization of \eqref{Eq::fixed-bandwidth-limit} yields  
\begin{align}
    T_{n,m,\lambda}
    \ \xrightarrow{d} W_{\lambda}:=\
    \frac{
        \sum_{\ell=1}^{\infty}
        \gamma_{\ell,\lambda}(Z_\ell^2-1)
    }{
        \left\{
            2\sum_{\ell=1}^{\infty}
            \gamma_{\ell,\lambda}^2
        \right\}^{1/2}
    }.
    \label{Eq::fixed-bandwidth-studentized-limit}
\end{align}
The distribution of $W_\lambda$ generally depends on the unknown
spectrum of $\Gamma_{\eta,\lambda}$. We now consider this family of
fixed-bandwidth limit distributions as
$\max(\lambda_1,\lambda_2)\to0$.  
Our assumptions imply
\begin{align}
    \frac{
        \tr\big(\Gamma_{\eta,\lambda}^{4}\big)
    }{
        \tr\big(\Gamma_{\eta,\lambda}^{2}\big)^2
    }
    =
    O(\lambda_1\lambda_2)
    \longrightarrow0.
    \label{Eq::spectral-dispersion}
\end{align}
Equivalently, no individual eigenvalue contributes a non-negligible
fraction of the total sum. In particular,
\begin{align*}
    \max_{\ell\geq1}
    \frac{\gamma_{\ell,\lambda}^{2}}
         {\sum_{r\geq1}\gamma_{r,\lambda}^{2}}
    \longrightarrow0.
\end{align*}
The standardized weighted sum of centered chi-square variables is
therefore asymptotically Gaussian:
\begin{align*}
    \frac{
        \sum_{\ell\geq1}
        \gamma_{\ell,\lambda}(Z_\ell^2-1)
    }{
        \left\{
            2\sum_{\ell\geq1}\gamma_{\ell,\lambda}^2
        \right\}^{1/2}
    }
    \ \xrightarrow{d}\
    \Nc(0,1),\quad \text{as}\quad \max(\lambda_1,\lambda_2)\to 0.
\end{align*}
Thus, the normal limit in
Theorem~\ref{Theorem::uniform-validity} can be interpreted as a
Gaussianization of the usual weighted chi-square limit caused by the
dispersion of the eigenvalues as the bandwidth shrinks.
Figure \ref{fig::comparison of asymptotic dists} provides an illustration of asymptotic distributions when the bandwidth is fixed and shrinks, respectively.

A similar Gaussianization phenomenon was established by \citet{li2019optimality} when the two underlying distributions coincide. In contrast, Theorem~\ref{Theorem::uniform-validity} establishes uniform asymptotic normality under the weaker null hypothesis of equal persistence intensity functions, even when the diagram distributions have different covariance operators. 
\end{rmk}
\subsection{Power analysis}\label{sec::Power analysis} 
We next investigate the power of the test
$\Delta_{\alpha,\lambda}$ over the class
$\mathcal P_1^{\otimes 2}$.
The $\mathrm{Type\ II}$ error is given by
\begin{align}\label{power function}
    \sup_{(P,Q)\in \Pc^{\otimes 2}_1}\Pb_{P\times Q}\left(\Delta_{\alpha,\lambda}(\Xb_n,\Yb_m)=0\right),
\end{align}
where $\Pc^{\otimes 2}_{1}$ is defined in \eqref{eq::alternative sublcass}.
In the power analysis, for a given $\beta > 0$, we aim to derive a sufficient condition under which the $\mathrm{Type\ II}$ error is bounded by $\beta$, and to express this condition in terms of the $L^2$ distance between the two weighted intensity functions. We also analyze this condition from a minimax perspective.

\paragraph*{Roadmap for the power analysis} 
We split the power analysis into two parts: deriving (1) an upper bound on the uniform separation rate of the proposed test over anisotropic Sobolev balls, and (2) a lower bound of the minimax rate of testing over the balls. By combining the results in the two parts, we conclude the proposed test achieves the minimax optimal rate over anisotropic Sobolev balls. 

\subsubsection{Upper bound for the uniform separation rate over anisotropic Sobolev balls}
The uniform separation rate of the proposed test over anisotropic Sobolev balls is formally stated in the following theorem. 
\begin{thm}\label{thm::upperbound_unifseparation_rate}
    Let $\alpha \in(0,1/2)$, $\beta\in(0,1)$, $s_1,s_2>0$, and $R>0$. Set 
    \begin{align*}
        \bar{s}:=\frac{2}{1/s_1+1/s_2},\quad \tau:=\left(1+\frac{1}{4s_1}+\frac{1}{4s_2}\right)^{-1},\quad\lambda_i^\star:= (n+m)^{-\tau/(2s_i)},\quad i=1,2.
    \end{align*}
         Then, the uniform separation rate of the test $\Delta_{\alpha,\lambda^\star}$ over the anisotropic Sobolev ball $\Sc_2^{s_1,s_2}(R)$ can be upper bounded as follows: there exists $C^{\star}(M,N,w,k,s_1,s_2,R,\alpha,\beta)>0$ such that,     for all sufficiently large $n,m$,
    \begin{align*}
        \rho_{n,m}\left(\Delta_{\alpha,\lambda^*},\Sc_2^{s_1,s_2}(R),\beta,w,M,N,L\right) \leq C^{\star}(M,N,w,k,s_1,s_2,R,\alpha,\beta)(n+m)^{-\bar{s}/(2\bar{s}+1)}.
    \end{align*}
\end{thm}
\begin{rmk}[Some technical details] 
    In deriving an upper bound on the separation rate, we employ the uniform bound for $\V(U_{n,m,\lambda})$:
\begin{align*}
    \V_{P\times Q}\left(U_{n,m,\lambda}\right) \lesssim\left(\frac{\|(w\cdot \psi )*\varphi_\lambda\|^2_2}{n+m} + \frac{1}{(n+m)^2\lambda_1\lambda_2}\right),
\end{align*}
for any $(P,Q)\in \Pc^{\otimes 2}$. Here, $\psi= p-q $ for persistence intensity functions $p$ and $q$ of $P$ and $Q$, respectively.
Deriving the variance bound requires
controlling the following quantity: 
    \begin{align}\label{eq::variance_explanation}
        \Eb_X\left(|X|\int_{\Rb^2}w(x)^2\left[(w\cdot\psi)*\varphi_\lambda(x)\right]^2dX(x)\right).
    \end{align}
     When $|X|$ is bounded as in Condition \ref{assump::Bounded cardinality condition}, it is straightforward to control \eqref{eq::variance_explanation}. However, $|X|$ is not always bounded; for instance, this is the case when $|\Xc|$ follows a Poisson distribution, where $\Xc$ is a random point cloud that induces the random persistence diagram $X$ \citep{Chazal2019,GLUZBERG2023107216}. To handle this case as well, we decompose \eqref{eq::variance_explanation} as a weighted sum of conditional expectations: 
    \begin{align*}
        \sum_{\ell=0}^\infty \Eb\left(|X| \int_{\Rb^2}w(x)^2\left[(w\cdot\psi)*\varphi_\lambda(x)\right]^2dX(x)\mid Z = \ell \right) P(Z=\ell),
    \end{align*}
    where $Z$ is a discrete random variable, which may be taken as $Z=|\Xc|$.
    Conditioning on $Z = \ell$ allows us to control $|X|$ and $\|w^2\cdot p_\ell\|_{\infty}$ as a function of $\ell$. Using this conditional control, we obtain that \eqref{eq::variance_explanation} is bounded by 
        $C \cdot  \|(w\cdot \psi) *\varphi_\lambda\|_2^2$,
    for some constant $C$ independent of $\lambda$.

\end{rmk}

This completes the upper bound part in our power analysis.
We now turn to the corresponding minimax lower bound.
\subsubsection{Lower bound for the minimax rate over anisotropic Sobolev balls} 
The central message of this lower bound analysis is that the minimax hardness of Euclidean nonparametric two-sample testing persists in the persistence diagram setting. We establish this by an explicit embedding argument. Specifically, we construct a collection of probability distributions on $\mathbf{PD}$ that are difficult to distinguish, while their intensity functions are well separated in the target metric. 
The construction is based on hard-to-distinguish probability densities
\(f_0\) and \(f_\theta\) on \(\Omega\), adapted from the density
construction of \citet{AlbertLaurentMarrelMeynaoui2022}, which in turn follows the classical Ingster-type perturbation scheme. We embed these density functions into the class $\Pc$ in such a way that the intensity functions of the embedded distributions are precisely $f_0$ and $f_\theta$, respectively.

\begin{figure}[t]
    \centering
    \includegraphics[width=1.0\textwidth]{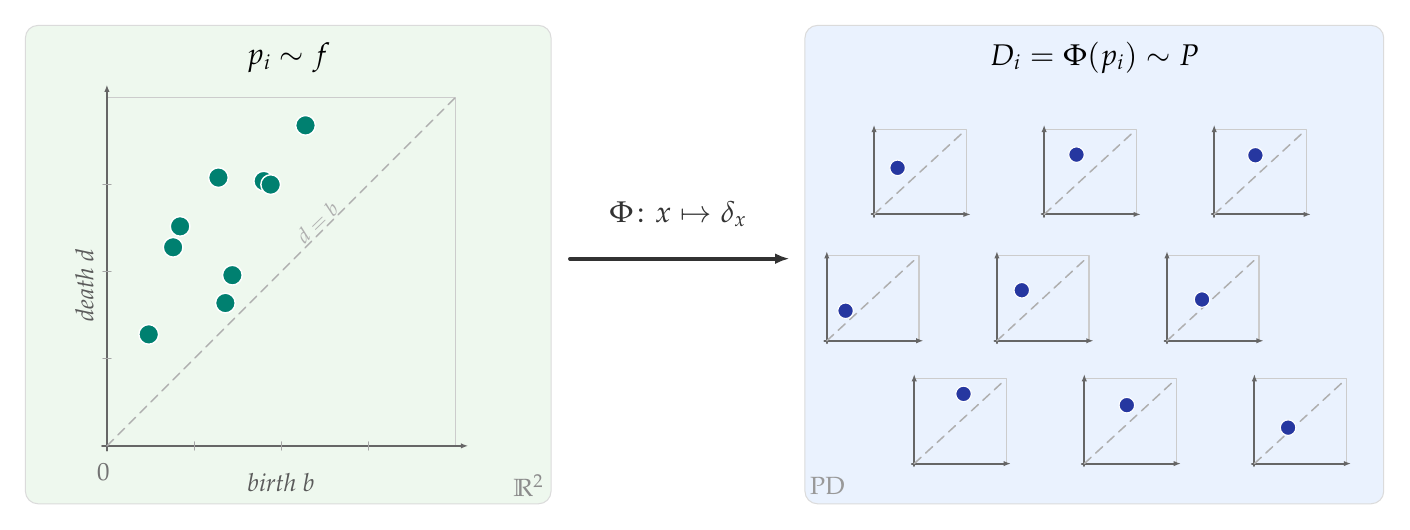}
    \caption{An example of $\Phi$. Each point in $\Omega$ is mapped to the persistence diagram consisting of that point, inducing a distribution $P$ from the probability density $f$ on $\Omega$.}
    \label{fig::exampleofPhi}
\end{figure}
We briefly explain how density functions on $\Omega$ are embedded into the class $\Pc$. 
 Define a measurable map $\Phi:\Omega \rightarrow \mathbf{PD}$ by $\Phi(x)\coloneq \delta_{x}$.
For a probability density function $f$ on $\Omega$,
we push forward $f\,dx$ to obtain a probability 
distribution \(P\) on \(\mathbf{PD}\), namely 
\begin{align*}
    P(A)\coloneq\Phi_{\#}(f\,dx)(A)= \int_{\Phi^{-1}(A)} f(x)dx,
\end{align*}
for any measurable set $A\subseteq \mathbf{PD}$. By construction, the intensity function of $P$ is $f$.
Indeed, for every Borel set \(B\subset \Omega\),
\[
    \mathbb{E}_{D\sim P}[D(B)]
    =
    \int_{\mathbf{PD}} D(B)\,dP(D)
    =
    \int_{\Omega} \delta_x(B) f(x)\,dx
    =
    \int_B f(x)\,dx.
\]
Moreover, this construction is realizable by genuine persistence diagrams:
by Proposition \ref{prop::phofcircle}, for each \(x=(b,d)\in\Omega\) with $b>0$, there exists a point cloud
\(\Xc(b,d)\subset S^1(d)\) such that
\[
    \mathrm{PD}_1\left(\textrm{\v{C}ech}(\Xc(b,d))\right)=\{(b,d)\}.
\]
Hence the singleton diagram \(\delta_x\) is not merely an abstract measure on
\(\Omega\), but can be realised as the persistence diagram of a Čech complex. Figure~\ref{fig::exampleofPhi} provides a visual example of $\Phi$.

Through this embedding, we show that the class of probability measures on $\mathbf{PD}$ is sufficiently broad to 
 contain an embedded copy of the class of probability distributions on $\Omega$. This idea, together with a careful analysis, yields the following minimax lower bound over
anisotropic Sobolev balls.
\begin{thm}\label{minimax lower bound} 
Let $\alpha,\beta,\gamma\in (0,1)$ satisfy $\alpha+\beta+\gamma<1$, and let $s_1,s_2>0$, $R>0$. Let $w$ satisfy 
$\underline{w}:=\inf_{(x,y) \in \Rc(M)} w(x,y) >0$, where $\Rc(M):=\left(\frac{M}{4},\frac{M}{2}\right)\times \left(\frac{3M}{4},M\right)$.
Suppose that 
    $N>\frac{16\|w^2\|_\infty}{M^2}$, $\frac{4\underline{w}^2}{M}\geq L.$
  Then, there exists a constant $C_{\star}(M,N,L,w,s_1,s_2,R,\alpha,\beta,\gamma)>0$ such that, for all sufficiently large $n$ and $m$,
    \begin{align*}
        \rho^{\dagger}_{n,m}(\Sc_2^{s_1,s_2}(R),\alpha,\beta,w,M,N,L)  \geq C_{\star}(M,N,L,w,s_1,s_2,R,\alpha,\beta,\gamma)(n+m)^{-\bar{s}/(2\bar{s}+1)}.
    \end{align*}
\end{thm}
Combining Theorem \ref{thm::upperbound_unifseparation_rate} and Theorem \ref{minimax lower bound} yields matching upper and lower bounds up to
multiplicative constants independent of $n$ and $m$.  Consequently, the proposed test achieves the minimax optimal separation rate over the anisotropic Sobolev ball.

\section{Topological properties of \texorpdfstring{\v{C}ech}{Cech} complex on the circle}\label{sec::Topological properties of cech on the circle}
In the proof of Theorem \ref{minimax lower bound}, we verify that the single-point measure $\Phi(x)$
 can be realised as a genuine persistence diagram. To this end, we compute the homology of the \v{C}ech complex built on a point cloud on the circle. 
Since this characterization is also of independent interest, we present it in a separate section. 
The proofs of the results in this section are provided in Section~\ref{Appendix::proof topology} of the Supplementary Material.

In the intersection of topological data analysis (TDA) and computational
geometry, the relationship between an underlying geometric space and
the simplicial complexes built from its discrete samples has been
extensively studied. Previous literature in this area generally falls
into two broad categories:

\begin{enumerate}
\item Topological reconstruction of manifolds: 
These works identify 
 the geometric conditions (e.g., density, reach)
under which the {\v C}ech or Vietoris-Rips complex built from a point cloud correctly reconstructs the homotopy or homology of the underlying manifold with high confidence 
\citep{NiyogiSmaleWeinberger2008}.
\item Homotopy types on the circle:
These works investigate
the exact homotopy types
of {\v C}ech and Vietoris-Rips complexes formed by points sampled
from a circle $S^{1}$ at various scale parameters.
\end{enumerate}
Motivated by these perspectives,
we provide a complete topological characterization of 
the \v{C}ech complex constructed from a point cloud $\mathcal{X}$
on a circle. 
We compute the
1-dimensional persistent homology $\mathrm{PH}_{1}$ of this complex and
discuss its geometric implications for manifold reconstruction 
in this section. The
explicit, stage-by-stage classification of the complex's homotopy
type is deferred to Section~\ref{Appendix::Homotopy equivalence Cech Circle} of the Supplementary Material.

A fundamental question in manifold reconstruction is whether a given
sample is 
distributed well enough 
to capture the global topological
signature of the underlying space.
To formalize our results, we first
define an adjacency relation for points in $\mathcal{X}$ on the circle.  Let \(S^1(r)\subseteq \Rb^2\) denote the circle centered at the origin with radius \(r>0\).
\begin{defn}\label{defn::adjacent}
For any $x\neq y\in\mathcal{X}\subset S^{1}(r)$, we write $x\sim y$
if and only if either path-connected component of $S^{1}(r)\setminus\{x,y\}$
has an empty intersection with $\mathcal{X}$. In other words, $x$
and $y$ are adjacent points  with no other points
of $\mathcal{X}$ lying on one of the arcs connecting them.
\end{defn}
Using this relation, the following proposition establishes
a sharp geometric threshold for this reconstruction on $S^{1}(r)$.
\begin{prop}\label{prop::phofcircle}
Fix $r>0$, and let $\mathcal{X}=\{x_1,\dots,x_n\}\subseteq S^{1}(r)$ be a nonempty finite subset with $n\geq 3$. Then, we have that

\begin{enumerate}
\item If $\mathcal{X}$ is contained in some open hemicircle lying on $S^{1}(r)$,
then 
\[
\mathrm{PD}_{1}\left(\mathrm{\check{C}ech}(\mathcal{X})\right)=\emptyset.
\]
\item If $\mathcal{X}$ is not contained in any open hemicircle lying on $S^{1}(r)$, then 
\[
\mathrm{PD}_{1}\left(\mathrm{\check{C}ech}(\mathcal{X})\right)=\left\{ \left(b_{\Xc},r\right)\right\}\setminus \Delta , 
\]
where  $b_{\Xc}:= \frac{1}{2}\sup_{x,y\in\mathcal{X},x\sim y}\left\Vert x-y\right\Vert _{2}$, $\Delta:=\{(t,t):t\in \Rb\}$.
\end{enumerate}
\end{prop}
Proposition~\ref{prop::phofcircle} gives an exact criterion for recovering the one-dimensional topology of the circle. If $\mathcal{X}$ is contained in an open hemicircle, the sample is too localized to generate a one-dimensional persistent feature, and hence $\mathrm{PD}_1\left(\mathrm{\check{C}ech}(\mathcal{X})\right)=\emptyset$. Otherwise, $b_{\Xc}\leq r$. If $b_{\Xc}<r$, the \v{C}ech filtration has a unique one-dimensional persistent feature. It is born at $b_{\Xc}$, half the largest distance between adjacent sample points, and dies at $r$, the radius of the underlying circle. If $b_{\Xc}=r$, no one-dimensional feature persists over a nonzero interval, and the persistence diagram is empty. Indeed, when the filtration radius reaches $r$, all \v{C}ech balls contain the origin, so the \v{C}ech complex is a full simplex and has trivial one-dimensional homology.

\section{Bandwidth-aggregated two-sample test}\label{Sec:: Bandwidth-Aggregated Two Sample Test}
In Section \ref{sec::Power analysis}, we have obtained the optimal bandwidth. However, this oracle-bandwidth test is not implementable in practice, as it depends on the unknown parameters $s_1$ and $s_2$.
To avoid selecting a single bandwidth in practice, we adopt the multi-bandwidth approach using Bonferroni-type calibration.

Let us define the Bonferroni test. 
For $a>5$, let $s_{n,m}:=\left(\log_2 (n+m)\right)^{a}\wedge (n+m)^{1/4}$,
and let 
\begin{align*}
    L_{n,m}:= \lceil \log_2 s_{n,m}\rceil,\quad K_{n,m}:=\left\lfloor 2\log_2 \left(\frac{n+m}{s_{n,m}}\right)\right\rfloor.
\end{align*}

Define the collection $\Lambda$ of bandwidths by 
\begin{align}\label{eq::bandwidth collection}
\Lambda:=\left\{(2^{-k_1},2^{-k_2}):k_1,k_2\in \Zb , k_1,k_2 \geq L_{n,m}, k_1+k_2\leq K_{n,m} \right\}.
\end{align}
We define the Bonferroni test $\Delta_{\alpha,\Lambda}^\mathrm{Bonf},$ across the bandwidth collection $\Lambda $, by 
\begin{align*}
   \Delta_{\alpha,\Lambda}^\mathrm{Bonf}= \mathbb{I}\left\{\max_{\lambda \in \Lambda }T_{n,m,\lambda} >z_{1-\alpha/|\Lambda|} \right\}.
\end{align*}
The following proposition establishes the asymptotic validity of $\Delta_{\alpha,\Lambda}^\mathrm{Bonf}$ and shows that the test attains the optimal separation rate up to an iterated-logarithmic factor.
\begin{prop}\label{prop::aggtest properties} Let $\alpha,\beta \in (0,1)$, $s_1,s_2>0$, and $R>0$.
    The Bonferroni test $\Delta_{\alpha,\Lambda}^\mathrm{Bonf}$ satisfies 
    \begin{align*}
        &\limsup_{n,m\to \infty}\sup_{(P,Q)\in \Pc_{0}^{\otimes 2}}\Pb\left(\Delta^{\mathrm{Bonf}}_{\alpha,\Lambda}=1 \right) \leq\alpha,\\
    &\rho_{n,m}
\left(
\Delta_{\alpha,\Lambda}^{\mathrm{Bonf}},
\Sc_2^{s_1,s_2}(R),
\beta,w,M,N,L
\right)
\lesssim
\left(
\frac{\sqrt{\log\log(n+m)}}{n+m}
\right)^{\frac{\bar{s}}{2\bar{s}+1}}.
\end{align*}
\end{prop}
\section{Simulation study}\label{Sec::Simulation Study}
In this section, we conduct simulation studies. 
We consider three data sets: uniform intensities, torus, and ORBIT5K. Results for ORBIT5K and some simulations using the torus data are deferred to Section~\ref{Appendix::Additional Simulation Study} of the Supplementary Material.
The source code and the data used in this paper are available in the  
\href{https://github.com/yeongunghan919/A-Two-Sample-Test-for-Random-Persistence-Diagrams}{GitHub repository}.

We compare the proposed Bonferroni test with three commonly used testing procedures in TDA: 
(i) the permutation test based on the distance between persistence diagrams (PD) 
\citep{RobinsonTurner2017}, 
(ii) the permutation test based on the distance between persistence landscapes (PL) 
\citep{Bubenik2015}, and 
(iii) the two-stage test based on persistence images (PI) \citep{MoonLazar2023}. 
The hyperparameter settings for each method are given in Section~\ref{Appendix::Additional Simulation Study} of the Supplementary Material. For the proposed test, these settings include a data-adaptive bandwidth collection that reflects the scale of the observed data, which we use instead of the collection $\Lambda$ defined in \eqref{eq::bandwidth collection}.

\begin{figure}[t]
    \centering
    
    \includegraphics[]{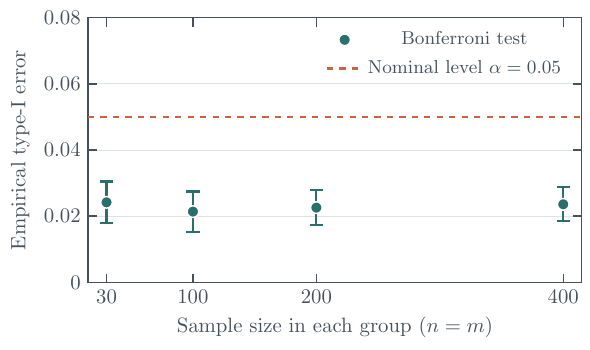}
    \caption{
        Empirical type-I error of the Bonferroni test under the
        equal-intensity null. Points represent rejection proportions
        based on $5{,}000$ Monte Carlo replications, and error bars
        indicate 95\% Monte Carlo confidence intervals. The dashed
        horizontal line denotes the nominal level $\alpha=0.05$.
    }

    \label{fig::equal-intensity-validity}
\end{figure}
\paragraph*{Equal intensity simulation}
To examine the finite-sample validity of the Bonferroni test under $p=q$ but $P\neq Q$, let $F$ be the uniform distribution on $[1.5,2.5]\times[4.5,5.5]$. Under $P$, a persistence diagram is given by $\{Z,Z\}$ for $Z\sim F$, whereas under $Q$, it is given by $\{Z_1,Z_2\}$ for independent $Z_1,Z_2\sim F$. Both distributions therefore have the same intensity function $2f$, where $f$ is the density of $F$, but have different within-diagram dependence structures.

We applied the Bonferroni test at level $\alpha=0.05$ with the linear weight $w(b,d)=d-b$. Figure~\ref{fig::equal-intensity-validity} presents the empirical type-I errors based on $5{,}000$ Monte Carlo replications. For $n=m=30,100,200,$ and $400$, the empirical type-I errors were $0.0242$, $0.0214$, $0.0226$, and $0.0236$, respectively. The error bars represent approximate 95\% Monte Carlo confidence intervals, calculated as
\begin{align*}
\widehat{\pi}\pm 1.96
\left\{
\frac{\widehat{\pi}(1-\widehat{\pi})}{5{,}000}
\right\}^{1/2},
\end{align*}
where $\widehat{\pi}$ denotes the empirical rejection proportion.

\paragraph*{Torus data simulation}
In the simulation using the torus data, each point cloud is generated from a homogeneous Poisson point process with intensity 5 per unit area on a torus; see Figure~\ref{Fig::Torus_example} for an example of the 
torus.
For each expanding level
$\beta \in \{0.01, 0.016, 0.026, 0.036\}$, we conduct $500$
simulation replications. In each replication, two samples of $50$
point clouds each are independently generated.
The point clouds in the first group come from the unexpanded surface with radii $R=2$ and $r=1$,
whereas the second sample comes from the expanded torus with radii $R=2+\beta$ and $r=1$.
Persistence diagrams are constructed using the Vietoris–Rips complex, focusing on one-dimensional features. 
\begin{figure}[t]
    \centering    
    \includegraphics[width=0.9\textwidth,height=0.25\textheight]{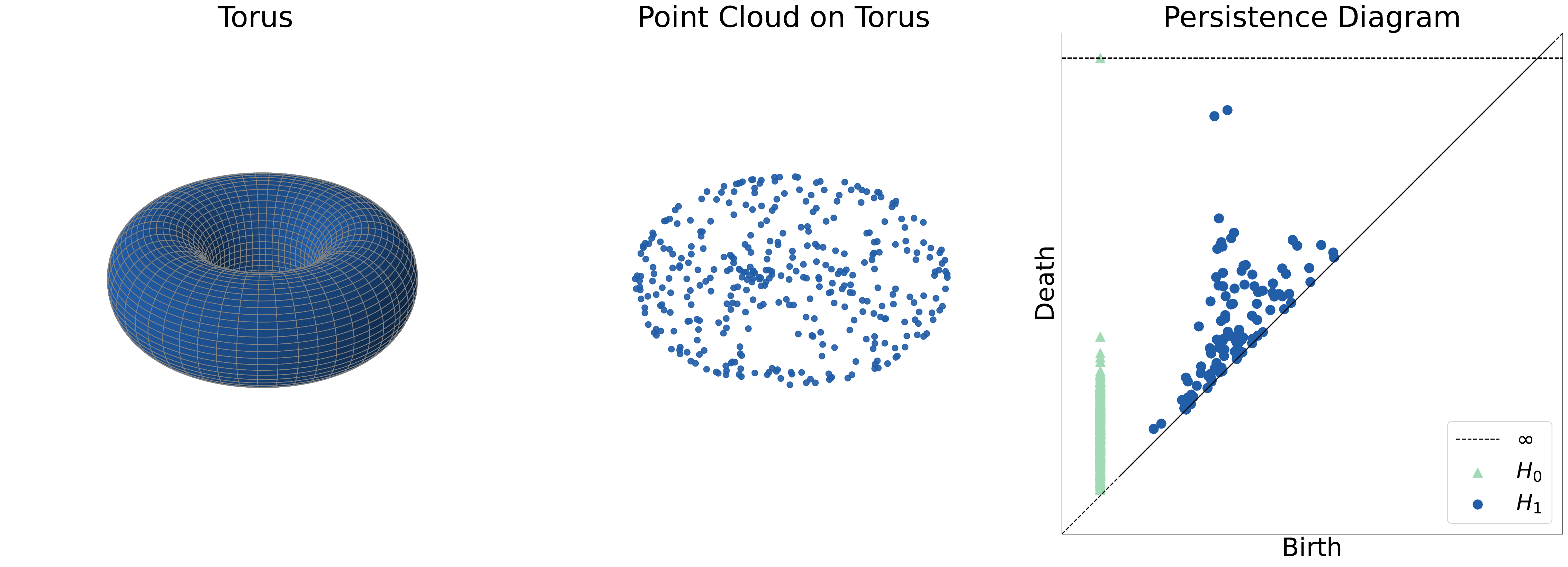}
    \caption{(left) A torus with radii $R=2,r=1$, (middle) a point cloud from the torus, and (right) the persistence diagram of zero- and one-dimensional homological features obtained from the point cloud using the  Vietoris-Rips filtration.}
    \label{Fig::Torus_example}
\end{figure}
The simulation results are presented in Figure~\ref{Figg:Torus_simulation}. 

A test is said to have greater power when its rejection probability increases more rapidly as the expanding level increases.
The left panel in Figure~\ref{Figg:Torus_simulation} shows that the proposed test with the square weight $(y-x)^2$ achieves the highest power.
We also investigate the effect of the weight function.
The proposed procedure is applied with three weight functions: square: $(y-x)^2$, linear: $y-x$, and constant: $1$.
The results are shown in the right panel in Figure~\ref{Figg:Torus_simulation}. 
Expanding the radius $R$ while keeping the expected number of points in unit area, affects the birth and 
death times of the two persistent one-dimensional features in each diagram. 
Assigning greater weight to persistence therefore enhances the ability of the test to 
distinguish between the intensity functions. 
Consistent with this intuition, a higher weight function leads to higher empirical power. Moreover, the empirical power under the constant weight  remains close to the nominal significance level across the examined values of $\beta$. This pattern suggests that noise features may obscure the difference of interest when persistent points are not given greater weight.
\begin{figure}[h]
    \centering
\includegraphics[width=\textwidth]{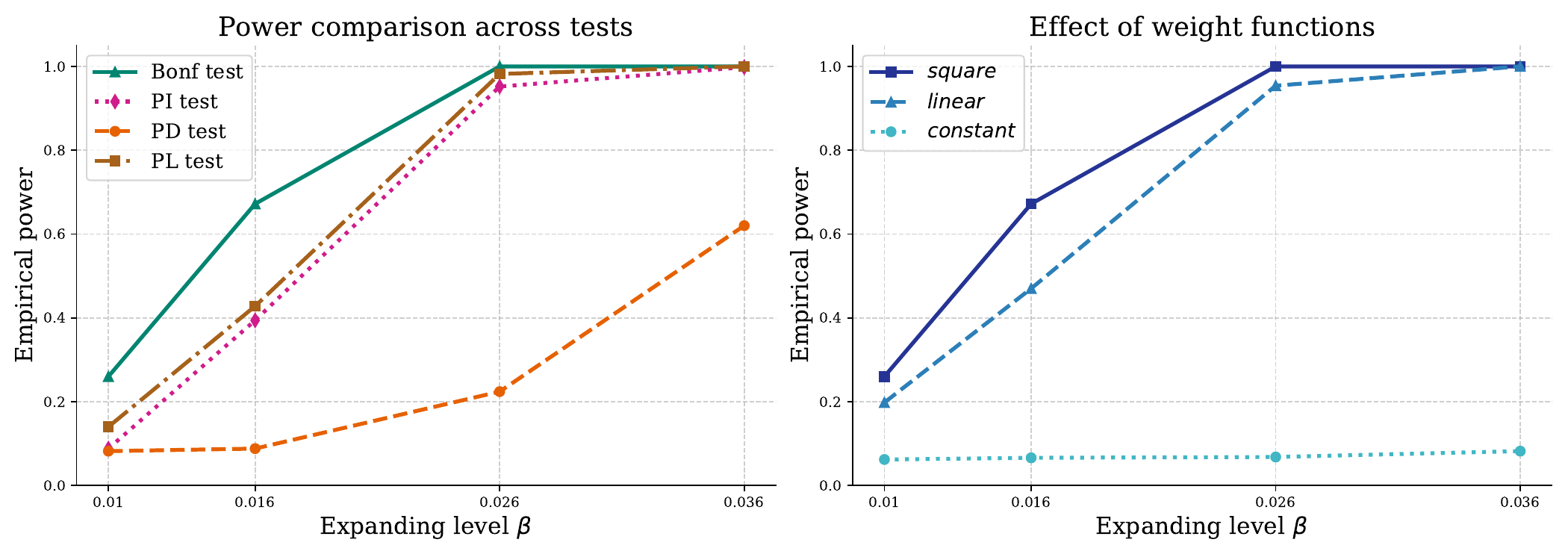}
    \caption{Results of the torus simulation: (left) empirical powers of tests over different expanding levels; (right) empirical powers of the proposed test using different weight functions over different expanding levels. The numerical values underlying this figure are provided in Section~\ref{Appendix::Additional torus data simulation} of the Supplementary Material.}
    \label{Figg:Torus_simulation}
\end{figure}

\section{Conclusion}

We proposed a studentized kernel two-sample test for the equality of persistence intensity functions. The proposed test is asymptotically valid under the null hypothesis $p=q$.
From power perspective, the  testing procedure attains the minimax optimal separation rate over anisotropic Sobolev balls of weighted intensity discrepancies. To the best of our knowledge, this is the first minimax separation-rate analysis of two-sample testing for equality of persistence intensity functions. 
These results extend kernel-based two-sample testing from Euclidean observations to random persistence diagrams viewed as random measures, thereby connecting statistical inference with topological data analysis.

\begin{center}
{\large\bf Supplementary Material}
\end{center}
\textbf{Supplementary Material:} The Supplementary Material for “A Two-Sample Test on Weighted Persistence Intensity Functions in Topological Data Analysis” contains background material, full theoretical details, descriptions of the testing algorithms, additional simulation studies, and real-data analysis.

\let\newpage\arxivoriginalnewpage

\clearpage

\thispagestyle{empty}
\appendix

\if0\blind
{
  \begin{center}
  {\LARGE Supplementary Material for \\"A Two-Sample Test on Weighted Persistence Intensity Functions in Topological Data Analysis"}\par
  {\large
  Yeongung Han\\
  Institute of Basic Sciences, Seoul National University
  }\par
  {\large
  Ilmun Kim\\
  {\centering
  Department of Mathematical Sciences, Korea Advanced\\ Institute of Science and Technology}
  }\par
  {\large
  Jisu Kim\footnote{Jisu Kim (email: jkim82133@snu.ac.kr) is the corresponding author.}\\
  Department of Statistics, Seoul National University
  }\par
  \end{center}
  
} \fi

\if1\blind
{
  \bigskip
  \bigskip
  \bigskip
  \begin{center}
    {\LARGE\bf Supplementary Material for ``A Two-Sample Test for Random Persistence Diagrams Against Alternatives Characterized by Persistence Intensity Differences''}
\end{center}
  \medskip
} \fi

\bigskip

\begin{abstract}
The Supplementary Material includes all proofs of the theoretical results, related work on intensity functions, computational algorithms, and additional simulations.
\end{abstract}

\noindent%
\vfill

\clearpage

\def\combinedarxiv{1}

\spacingset{1.8} 

\ifdefined\combinedarxiv
\else
  \newpage
  {\fontsize{12pt}{15.5pt}\selectfont
    \hypersetup{hidelinks}
    \tableofcontents
    \newpage
    \hypersetup{linkcolor=blue}
  }
\fi

\section{Background materials}
In this section, we provide the background materials necessary for understanding the theory in this paper.
\subsection{Minimax rate of testing}\label{Appendixsecminimaxrateoftesting} 
In this section, we provide a brief review of statistical inference theory, 
with a particular focus on minimax rate theory.

For convenience, suppose we have two data groups $X_1,\dots, X_n \overset{\iid}{\sim} f $ and  $Y_1,\dots, Y_m \overset{\iid}{\sim} g $ from a pair of probability density functions $f$ and $g$ on $\Rb^d$ that belong to a certain family of pairs of probability densities $\mathscr{P}$. Let $\Xb_n=(X_1,\dots, X_n)$ and $\Yb_m=(Y_1,\dots,Y_m)$. Consider the two-sample testing problem: $\mathrm{H}_0: f=g $ versus $\mathrm{H}_1: f\neq g$. Let $\alpha \in (0,1)$ and $\beta\in (0,1-\alpha)$, and let $\Phi_{\alpha}$ be the collection of all level $\alpha$ tests $\Delta_{\alpha}$, $\ie$,
\begin{align*}
    \Phi_{\alpha}=\left\{\Delta_{\alpha}:\sup_{(f,f)\in \mathscr{P}}\Pb_{f\times f}\!\left(\Delta_\alpha(\Xb_n,\Yb_m)=1\right)\leq \alpha \right\}.
\end{align*}

A fundamental question in nonparametric hypothesis testing is to determine the smallest value $\tilde{\rho}_{n+m}>0$ such that the test $\Delta_\alpha$ has power at least $1-\beta$ against all alternative hypotheses satisfying $f-g \in \Cc $, for a function class $\Cc$, and $\|f-g\|_2 > \tilde{\rho}_{n+m}$. The definition of \emph{uniform separation rate} is derived from this question, and it is defined by 
\begin{align*}
    \rho_{n+m}(\Delta_{\alpha},\Cc,\beta):=\inf\left\{\tilde{\rho}_{n+m}>0:\sup_{(f,g)\in \Fc_{\tilde{\rho}_{n+m}}(\Cc)}\Pb_{f\times g}(\Delta_\alpha(\Xb_n,\Yb_m)=0)\leq \beta\right\},
\end{align*}
where $\Fc_{\tilde{\rho}_{n+m}}(\Cc):=\{(f,g):f-g \in \Cc,\|f-g\|_2\geq\tilde{\rho}_{n+m}\}$. The next question is to determine the smallest value $\tilde{\rho}^\dagger_{n+m}$  such that there exists a test $\Delta_\alpha\in \Phi_{\alpha}$ which has power at least $1-\beta$ whenever $(f,g) \in \Fc_{\tilde{\rho}_{{n+m}}^\dagger}(\Cc)$. That is, if  $\rho<\tilde{\rho}^\dagger_{n+m}$ then there is no test $\Delta_\alpha\in \Phi_{\alpha}$ such that 
\begin{align*}
    \sup_{(f,g)\in \Fc_{\rho}(\Cc)}\Pb_{f\times g}(\Delta_{\alpha}(\Xb_n,\Yb_m)=0)\leq \beta.
\end{align*} This question motivates the definition of the \emph{minimax rate of testing}, and it is defined as
\begin{align*}
    \rho_{n+m}^{\dagger}(\Cc,\alpha,\beta) := \inf\left\{\rho_{n+m}(\Delta_\alpha,\Cc,\beta)\mid \Delta_{\alpha} \in \Phi_{\alpha}\right\}.
\end{align*}

A given test $\Delta_\alpha$ is called \emph{optimal in the minimax rate sense}
if its uniform separation rate is upper bounded, up to a multiplicative constant
independent of $n$ and $m$, by the minimax rate of testing,
\citep{IngsterSuslina2003,kim2022minimax,schrab2023mmdagg,li2019optimality}.

\subsection{Basic algebraic topology}\label{Appendix::Background_homotopy}
In this section, we introduce several notions from algebraic topology, focusing in particular on homotopy equivalence, which is utilised in Section \ref{Appendix::Homotopy equivalence Cech Circle}. Let $X$ and $Y$ be topological spaces. Consider two continuous maps $f_0,f_1 : X\rightarrow Y$. If there exists a continuous map $H:X\times [0,1]\rightarrow Y$ such that 
\begin{align*}
    H(\cdot,0)= f_0(\cdot),\quad \text{ and }\quad H(\cdot,1)=f_1(\cdot),
\end{align*} then we say that these maps are homotopic, and write as $f_0 \simeq f_1$. The map $H$ is called a homotopy between $f_0$ and $f_1$. In other words, the two continuous maps, which are homotopic, can be continuously deformed into each other via a homotopy. Figure \ref{Appendix::fig-An example of a homotopic pair} provides a visual example of a homotopic pair.
\begin{figure}[t]
    \centering
    \includegraphics[]{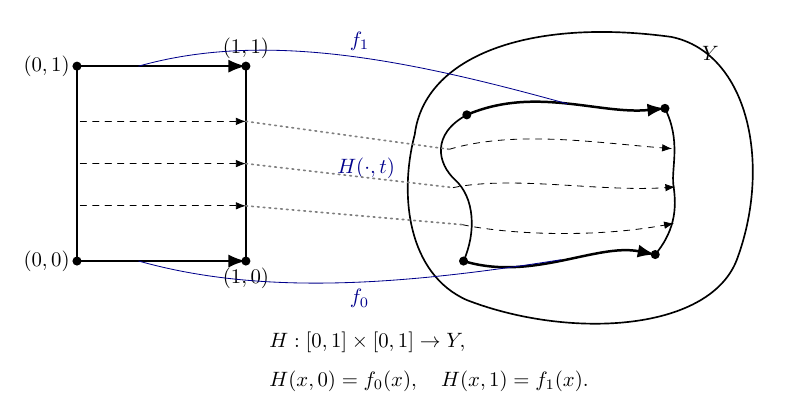}
    \caption{An example of a homotopic pair.}
    \label{Appendix::fig-An example of a homotopic pair}
\end{figure}

In topology, a fundamental question is whether two spaces are equivalent.  
The notion of homotopy between two continuous maps provides a criterion for the equivalence of topological spaces. A continuous map $f:X\rightarrow Y$ is called a homotopy equivalence if there exists a continuous map $g:Y\rightarrow X$ such that $f\circ g \simeq \mathrm{Id}_{Y}$, and $g\circ f \simeq \mathrm{Id}_{X}$, where $\mathrm{Id}_X: X\rightarrow X$ is the identity map. If there exists a homotopy equivalence between $X$ and $Y$, then the topological spaces $X$ and $Y $ are called homotopy equivalent and write as $X\simeq Y$. 

A simple way to show that two spaces are homotopy equivalent is via a deformation retraction. Consider a subspace $A$ of the space $X$. A continuous map $r : X\rightarrow X$ is called a retraction of $X$ onto $A$ if $r(X) =A$ and $r(a)= a$ for any $a\in A$. A deformation retraction of $X$ onto $A$ is a homotopy between the identity map of $X$ and a retraction of $X$ onto $A$. If there exists a deformation retraction of $X$ onto $A$, then $X\simeq A$, because the retraction $r$ is a homotopy equivalence between $X$ and $A$. For example, for any convex set $\Cc$ and any point $c_0\in \Cc$, one can show that $\Cc \simeq \{c_0\}$ via a deformation retraction.

One of the desirable properties of homotopy equivalence is that it ensures that homotopy equivalent spaces have the same homology, as stated in the following proposition. The notion of the homology is described in the next section.
\begin{prop}
    Let $X$ and $Y$ be homotopy equivalent topological spaces. Then, for any $n\in \{0,1,2,\dots\}$, the $n$-th homologies $\mathrm{H}_n(X)$ and $\mathrm{H}_n(Y)$ are isomorphic, $\ie$,  \[
        \mathrm{H}_n(X) \cong \mathrm{H}_n(Y).
    \]
\end{prop}
\noindent
By this proposition, to show that the homologies of $X$ and $Y$ are isomorphic, it suffices to show $X$ and $Y$ are homotopy equivalent. We close this section by referring the reader to \citet{Hatcher}. 

\subsection{Persistent homology}\label{sec::persistent homology}
In this section we briefly introduce homology and persistent homology.
Homology is an invariant of the topological features of a given topological space $\Tc$. The homology of $\Tc$ is a collection of vector spaces $\{\mathrm{H}_k(\Tc):k=0,1,\dots
\}$. The 0th homology $\mathrm{H}_0(\Tc)$ is generated by elements that represent 0-dimensional topological features of $\Tc$ (connected components). Likewise, the $k$-th homology $\mathrm{H}_k(\Tc)$ is generated by elements that represent $k$-dimensional topological features of $\Tc$. 

However, when a goal is to infer the support of an unknown distribution, one typically observes only a finite set of sample points. As a result, the $k$-th homology is trivial for all $k \ge 1$. To recover meaningful topological information, it is therefore necessary to construct a simplicial complex over the observed point cloud. Several constructions of simplicial complexes are available, among which the Vietoris-Rips and \v{C}ech complexes are widely used in TDA. Let $r>0$ be a scale parameter, and let $\Xc$ be a subset of a metric space $(\Mc,d)$.
Then the Vietoris-Rips complex of $\Xc$ at scale $r>0$ is the simplicial complex defined by
\begin{align*}
    \text{Rips}_r(\Xc) :=\left\{\sigma \subseteq \Xc : d(x,y)\leq r \text{ for all }x,y\in \sigma \right\}.
\end{align*}
The \v{C}ech complex of $\Xc$ at scale $r>0$ is defined by
\begin{align*}
    \textrm{\v{C}ech}_r(\Xc) := \left\{\sigma \subseteq \Xc: \bigcap_{x\in \sigma}\Bc(x,r)\neq \emptyset \right\},
\end{align*}
where $\Bc(x,r) =\left\{y \in \Mc \mid d(x,y)\leq r \right\}$.
When $\Xc$ is a given point cloud sampled from an unknown probability distribution, the homology of the support of the underlying distribution is equal to that of $\text{Rips}_r(\Xc)$ or $\text{\v{C}ech}_r(\Xc)$ with high probability for an appropriate choice of $r>0$ under regularity conditions; see~\citet{NiyogiSmaleWeinberger2008,ChazalCohenSteinerLieutier2009} for more details. However, selecting an appropriate scale parameter $r$ in practice remains a challenging problem.

Rather than choosing a single scale $r$, using various scales to extract richer information is the key motivation of persistent homology. Roughly speaking, persistent homology records how homological features appear and disappear across scales. Let $\Kc(\Xc,r)$ be a simplicial complex built on a point cloud $\Xc$ at scale $r$. The collection $\Kc(\Xc):=\{\Kc(\Xc,r)\}_{r>0}$ is called a filtration of $\Xc$ if $\Kc(\Xc,{r_1})\subset \Kc(\Xc,{r_2})$ whenever $r_1 \leq r_2$. The function $\Kc$ used to build the simplicial complex is called a filtration function. The persistent homology of the filtration $\Kc(\Xc)$ records a change in the homology of $\Kc(\Xc,t)$ as $t$ varies. To be specific, consider the inclusion maps $i_{r,s}:\Kc(\Xc,{r}) \hookrightarrow \Kc(\Xc,{s})$ for  $r\leq s$. These inclusion maps $i_{r,s}$ induce a map $i'_{r,s}:\mathrm{H}_k(\Kc(\Xc,{r}))\rightarrow \mathrm{H}_k(\Kc(\Xc,{s}))$ between homologies. The collection of vector spaces  $\{\mathrm{H}_k(\Kc(\Xc,{r})):r>0\}$ with the linear maps $\{i'_{r,s}:r<s\}$ is called the persistent homology of dimension $k$, denoted by $\mathrm{PH}_k(\Kc(\Xc))$; see \citet{EdelsbrunnerHarer2010} for more details.

An element \(\sigma\) of \(\mathrm{PH}_k(K(\mathcal X))\) that represents a \(k\)-dimensional topological feature in \(K(\mathcal X,r_1)\)
becomes trivial for the first time at some \(r_2>r_1\). Then $r_2$ is called the death time of $\sigma$. Moreover, if $\sigma$ is nontrivial precisely for \(r_1\le r\le r_2\), then $r_1$ is called the birth time of $\sigma$. Persistent homology records the birth and death times of all topological features of the space across different scales.

\subsection{Hilbert spaces}
\subsubsection{Reproducing kernel Hilbert spaces}\label{Appendix::Reproducing-kernel-Hilbert-spaces}
In this section, we introduce the definition of the reproducing kernel Hilbert space (RKHS) and its properties that we use to construct the test statistic. 

For a set $\Sc$, let $k:\Sc\times \Sc \rightarrow \Rb$ be a function, 
denoted by 
a kernel function. The kernel function $k$ is positive definite if $k$ is symmetric, $\ie$, $k(x,y)=k(y,x)$ for any $x,y \in \Sc$ and for any finite collection of $x_1,\dots, x_N \in \Sc$, the matrix $A=\left[k(x_i,x_j)\right]_{1\leq i,j\leq N}$ is positive semi-definite, $\ie$, $v^TAv \geq 0$ for all $v\in \Rb^N$. It is well-known from the following theorem that a positive definite kernel $k$ uniquely defines a separable Hilbert space $\Hc_k$ as a subspace of a real-valued function space on $\Sc$. The space $\Hc_k$ is called the reproducing kernel Hilbert space (RKHS).  

\begin{thm}[Moore--Aronszajn; see
\textnormal{\citealp{PaulsenRaghupathi2016}}]Let $\Sc$ be a set and $k$ be a positive definite kernel on $\Sc\times \Sc$. Then, there uniquely exists a reproducing kernel Hilbert space $\Hc_k$ satisfying the following:
\begin{enumerate}
    \item $k(\cdot, x)\in \Hc_k$ for any $x\in \Sc$,
    \item $\mathrm{Span}\{k(\cdot,x)\mid x\in \Sc\}$ is dense in $\Hc_k$ with the norm $\|\cdot \|_{\Hc_k}$,
    \item $\langle f,k(\cdot, x)\rangle_{\Hc_k}=f(x)$ for any $x \in \Sc $ and any $f\in \Hc_k$.
\end{enumerate}
\end{thm}
The third property of the above theorem is called the reproducing property. By this property, we have 
\begin{align*}
    \left\langle k(\cdot,x),k(\cdot,y)\right\rangle_{\Hc_k}=k(x,y),
\end{align*}
for any $x,y\in \Sc$. We use the definition and properties of RKHS in Section \ref{Appendix::Analysis-for-test-statistic}. 

\subsubsection{Operators on a Hilbert space}
\label{Appendix::Hilbert-space-operators}
Let $\Hc$ be a real separable Hilbert space with inner product
$\langle \cdot,\cdot\rangle_{\Hc}$ and norm $\|\cdot\|_{\Hc}$.
In our application, $\Hc=\Hc_{k_{w,\lambda}}$.
We collect below some standard definitions and properties of operators
on $\Hc$ that are used throughout the proofs. For more detailed content, we refer the reader to \citet{HsingEubank2015}.

Let $\mathcal L(\Hc)$ denote the space of bounded linear operators from
$\Hc$ to itself. For $T\in\mathcal L(\Hc)$, its operator norm is defined by
\begin{align*}
    \|T\|_{\mathrm{op}}
    &:=
    \sup_{\|u\|_{\Hc}\leq 1}
    \|Tu\|_{\Hc}.
\end{align*}
The adjoint of $T$ is denoted by $T^*$. An operator $T$ is called
self-adjoint if $T=T^*$. 
It is called
positive semidefinite, denoted by $T\succeq0$, if $T$ is self-adjoint and $\langle u,Tu\rangle_{\Hc}\geq0$
for every $u\in\Hc$.
For two
self-adjoint operators $S$ and $T$, we write $S\preceq T$ if
$T-S\succeq0$. In particular, if $0\preceq S\preceq T$, then 
\begin{align*}
    \|S\|_{\mathrm{op}}
    \leq
    \|T\|_{\mathrm{op}}.
\end{align*}
Moreover, for every positive bounded operator $T$,
\begin{align}
    0
    \leq
    \langle u,Tu\rangle_{\Hc}
    \leq
    \|T\|_{\mathrm{op}}\|u\|_{\Hc}^2.
    \label{Eq::positive-operator-bound}
\end{align}
For $u,v\in\Hc$, the outer-product operator $u\otimes v$ is defined by
\begin{align}
    (u\otimes v)f
    &:=
    \langle v,f\rangle_{\Hc}u,
    \qquad f\in\Hc.
    \label{Eq::outer-product-definition}
\end{align}
It is a rank-one bounded operator satisfying
\begin{align*}
    (u\otimes v)^*=v\otimes u,\quad
    \|u\otimes v\|_{\mathrm{op}}=\|u\|_{\Hc}\|v\|_{\Hc}.
\end{align*}
In particular, $u\otimes u$ is positive semidefinite.

\paragraph*{Trace-class and Hilbert--Schmidt operators}
Let $\{e_\ell\}_{\ell\geq1}$ be an orthonormal basis of $\Hc$. For
$T\in\mathcal L(\Hc)$, let
\begin{align*}
    |T|:=(T^*T)^{1/2}.
\end{align*}
The operator $T$ is called trace-class if
\begin{align*}
\|T\|_1:=\sum_{\ell=1}^{\infty}
    \langle |T|e_\ell,e_\ell\rangle_{\Hc}<\infty.
\end{align*}
For a trace-class operator $T$, its trace is defined by
\begin{align}
    \tr(T)
    &:=
    \sum_{\ell=1}^{\infty}
    \langle Te_\ell,e_\ell\rangle_{\Hc}.
    \label{Eq::trace-definition}
\end{align}
Both the trace-class property and the value of $\tr(T)$ are independent
of the choice of orthonormal basis. If $T\succeq0$, then
\begin{align*}
    \|T\|_1
    =
    \tr(T).
\end{align*}

The operator $T$ is called Hilbert--Schmidt if
\begin{align}\label{Eq::HS-norm-definition}
    \|T\|_{\mathrm{HS}}^2
    :=\sum_{\ell=1}^{\infty}
    \|Te_\ell\|_{\Hc}^2<\infty.
\end{align}
For two Hilbert--Schmidt operators $S$ and $T$, their
Hilbert--Schmidt inner product is
\begin{align}\label{Eq::HS-inner-product-definition}
    \langle S,T\rangle_{\mathrm{HS}}
    :=
    \sum_{\ell=1}^{\infty}
    \langle Se_\ell,Te_\ell\rangle_{\Hc}=\tr(S^*T).
\end{align}
The definitions in \eqref{Eq::HS-norm-definition} and
\eqref{Eq::HS-inner-product-definition} do not depend on the choice of
orthonormal basis. The operator norm, Hilbert--Schmidt norm, and
trace norm satisfy
\begin{align*}
    \|T\|_{\mathrm{op}}
    \leq
    \|T\|_{\mathrm{HS}}
    \leq
    \|T\|_1
\end{align*}
whenever the norms are finite.

For $u,v,x,y\in\Hc$, the outer-product operators satisfy
\begin{align}
    \langle u\otimes v,x\otimes y\rangle_{\mathrm{HS}}
    &=
    \langle u,x\rangle_{\Hc}
    \langle v,y\rangle_{\Hc}.
    \label{Eq::outer-product-HS-identity}
\end{align}
In particular,
\begin{align}
    \langle u\otimes u,v\otimes v\rangle_{\mathrm{HS}}
    &=
    \langle u,v\rangle_{\Hc}^2.
    \label{Eq::outer-product-square-identity}
\end{align}
Indeed, Parseval's identity yields
\begin{align*}
    \langle u\otimes v,x\otimes y\rangle_{\mathrm{HS}}
    &=
    \sum_{\ell=1}^{\infty}
    \langle
        (u\otimes v)e_\ell,
        (x\otimes y)e_\ell
    \rangle_{\Hc}\\
    &=
    \langle u,x\rangle_{\Hc}
    \sum_{\ell=1}^{\infty}
    \langle v,e_\ell\rangle_{\Hc}
    \langle y,e_\ell\rangle_{\Hc}\\
    &=
    \langle u,x\rangle_{\Hc}
    \langle v,y\rangle_{\Hc}.
\end{align*}

We use the ideal property of the Hilbert--Schmidt class.
If $R\in\mathcal L(\Hc)$ and $T$ is Hilbert--Schmidt, then $RT$ and
$TR$ are Hilbert--Schmidt and
\begin{align}
    \|RT\|_{\mathrm{HS}}
    \leq
\|R\|_{\mathrm{op}}\|T\|_{\mathrm{HS}},\quad
    \|TR\|_{\mathrm{HS}}
    \leq
\|R\|_{\mathrm{op}}\|T\|_{\mathrm{HS}}.
    \label{Eq::HS-ideal-property}
\end{align}
If $S$ and $T$ are Hilbert--Schmidt, then $ST$ and $TS$ are
trace-class, and
\begin{align*}
    \|ST\|_1
    \leq
\|S\|_{\mathrm{HS}}\|T\|_{\mathrm{HS}},\quad
    \tr(ST)=\tr(TS).
\end{align*}
The Cauchy--Schwarz inequality in the Hilbert--Schmidt class yields
\begin{align*}
    |\tr(S^*T)|=
    |\langle S,T\rangle_{\mathrm{HS}}|
    \leq
    \|S\|_{\mathrm{HS}}\|T\|_{\mathrm{HS}}.
\end{align*}
In particular, if $S$ and $T$ are self-adjoint Hilbert--Schmidt
operators, then
\begin{align*}
    \langle S,T\rangle_{\mathrm{HS}}=\tr(ST).
\end{align*}

\paragraph*{Second-moment and covariance operators.}

Let $U$ be an $\Hc$-valued random element satisfying
$\mathbb E\|U\|_{\Hc}^2<\infty$. Define its mean, uncentered
second-moment operator, and covariance operator by
\begin{align}\label{Eq::covariance-operator-definition}
    \mu_U
    :=
    \mathbb E[U],
    \quad
    C_U
    :=
    \mathbb E[U\otimes U],\quad
    \Sigma_U
    :=\mathbb E[ (U-\mu_U)\otimes(U-\mu_U)].
\end{align}
These operator-valued expectations are well-defined, since
\begin{align*}
    \|U\otimes U\|_1
    =\|U\|_{\Hc}^2.
\end{align*}
Both $C_U$ and $\Sigma_U$ are self-adjoint, positive, and trace-class.
Moreover,
\begin{align*}
    \Sigma_U=C_U-\mu_U\otimes\mu_U,\quad
    C_U-\Sigma_U=\mu_U\otimes\mu_U\succeq0.
\end{align*}

If $U$ is centered, $\ie$ $\Eb(U)=0$, and $T\in\mathcal L(\Hc)$, then
\begin{align*}
    \mathbb E\langle U,TU\rangle_{\Hc}=\tr(T\Sigma_U).
\end{align*}
If $U$ and $V$ are independent centered random elements with covariance
operators $\Sigma_U$ and $\Sigma_V$, respectively, then
\begin{align*}
    \mathbb E\langle U,V\rangle_{\Hc}^2
    &=\mathbb E
    \langle U\otimes U,V\otimes V\rangle_{\mathrm{HS}}\\
    &=
    \langle \Sigma_U,\Sigma_V\rangle_{\mathrm{HS}}\\
    &=\tr(\Sigma_U\Sigma_V).
\end{align*}
For example, if $A_1,A_2$ are independent copies of a centered random
element with covariance operator $\Sigma_P$, then
\begin{align*}
    \mathbb E\langle A_1,A_2\rangle_{\Hc}^2=
    \tr(\Sigma_P^2).
\end{align*}
Similarly, if $A_1$ and $B_1$ are independent with covariance
operators $\Sigma_P$ and $\Sigma_Q$, respectively, then
\begin{align*}
    \mathbb E\langle A_1,B_1\rangle_{\Hc}^2=
    \tr(\Sigma_P\Sigma_Q).
\end{align*}

The same identity explains the factors appearing in the estimators of
the variance components. Let
\begin{align*}
    U:=A_1-A_2,\quad V:=A_3-A_4,
\end{align*}
where $A_1,\ldots,A_4$ are independent and have covariance operator
$\Sigma_P$. Then
\begin{align*}
    \mathbb E[U\otimes U]
    =\mathbb E[V\otimes V]
    =2\Sigma_P,
\end{align*}
and we hence have
\begin{align*}
    \mathbb E\langle U,V\rangle_{\Hc}^2
    &=\left\langle
        \mathbb E[U\otimes U],
        \mathbb E[V\otimes V]
    \right\rangle_{\mathrm{HS}}\\
    &=4\tr(\Sigma_P^2).
\end{align*}
Likewise, if $U=A_1-A_2$ and $V=B_1-B_2$, then
\begin{align*}
    \mathbb E\langle U,V\rangle_{\Hc}^2=4\tr(\Sigma_P\Sigma_Q).
\end{align*}

\paragraph*{Spectral decomposition}
Let $T$ be a positive semidefinite trace-class operator on $\Hc$. Since $T$ is
compact and self-adjoint, the spectral theorem gives nonnegative
eigenvalues $\{\gamma_\ell\}_{\ell\geq1}$ and an orthonormal system
$\{e_\ell\}_{\ell\geq1}$ such that
\begin{align}
    T
    &=
    \sum_{\ell=1}^{\infty}
    \gamma_\ell e_\ell\otimes e_\ell.
    \label{Eq::spectral-decomposition}
\end{align}
The series in \eqref{Eq::spectral-decomposition} converges in trace
norm, and we can write
\begin{align*}
    \|T\|_{\mathrm{op}}=
    \sup_{\ell\geq1}\gamma_\ell,
    \quad
    \tr(T)=\sum_{\ell=1}^{\infty}\gamma_\ell,\quad
    \|T\|_{\mathrm{HS}}^2=\tr(T^2)=\sum_{\ell=1}^{\infty}\gamma_\ell^2,\quad
\tr(T^4)=\sum_{\ell=1}^{\infty}\gamma_\ell^4.
\end{align*}
In particular, we can write
$    \tr(T^4)\leq
\|T\|_{\mathrm{op}}^2\tr(T^2)$.

\section{Description of model assumptions}\label{Appendix::Description-of-the-assumptions}
In this section, we elaborate on the conditions on the probability distributions described in Remark
\ref{Realizations of Assumption} of the main text. We clarify the implications of each condition and provide justification. 

Assumption \ref{assump::randomness of diagram} ensures that the sampled diagrams $D_i$ cannot come from different homological dimensions, that is, there exists $k$ such that $D_i=\mathrm{PD}_k(\Kc(\Xc_i))$, where $\Xc_i$ is a random sample of $\Xc$, for any $1\leq i\leq n$. This assumption is natural in view of the inherent randomness of persistence diagrams.

Assumption \ref{assump::Bounded persistence condition} ensures that the $y$-axis of the diagram is bounded.  When the underlying manifolds have uniformly bounded diameter, and the filtration function $\Kc$ depends only on the location of points, such as Vietoris-Rips, and \v{C}ech filtrations, as formally stated in condition \ref{assump::K1} in Section \ref{Appendix::A-sufficient-condition-for-existence-of-intensity-function}, this assumption is satisfied. 

Assumption \ref{assump::non degeneracy} is imposed to ensure that the null variance is uniformly nondegenerate, which is needed to establish uniform asymptotic control of the type I error.

Assumption \ref{assump::core assumption} is introduced to incorporate the two groups of conditions \ref{assump::Bounded cardinality condition}-\ref{assump::Existence of intensity function} and \ref{assump::Control of the diagram cardinality}-\ref{assump::Tail behavior}. We now discuss the implications of each condition and the settings in which they are satisfied.

Conditions \ref{assump::Bounded cardinality condition} and  \ref{assump::Existence of intensity function} are intuitive but restrictive, as they exclude common models such as Poisson point processes, where the cardinality of $\Xc$ is random and unbounded. In such settings, the associated persistence diagram may also have unbounded cardinality, and the intensity function need not be bounded in the supremum norm. Such restrictions motivate us to develop Conditions \ref{assump::Control of the diagram cardinality}-\ref{assump::Tail behavior}. 

Condition~\ref{assump::existence of conditional intensity} consists of two components: the existence of a conditional intensity function and its uniform domination by $U^{\ell}$. 
Under the conditions of Theorem~\ref{Appendix::thm::existence-intensity},
the conditional intensity function exists for $k\geq1$; see Section~\ref{Appendix::A-sufficient-condition-for-existence-of-intensity-function} for details. 
The second component requires that there exists a constant $U>0$ such that for any $P\in \Pc$, if $p_\ell $ is the conditional intensity function of $P$ and $w$ is a weight function, then $\|w^2\cdot p_{\ell}\|_\infty \leq U^{\ell}$.  
This condition is satisfied under the Vietoris--Rips filtration with an admissible weight function, when the point process is a Poisson process whose intensity measure has a uniformly bounded density; see Remark~\ref{Appendix::rmk::the uniform domination remark} for details.
Moreover, a number of filtration functions with adequately chosen weight functions may satisfy the condition. Since this is beyond the scope of this paper, we leave it as future work.

Condition \ref{assump::Control of the diagram cardinality} holds for all but pathological filtration functions. For example, this assumption is satisfied when the Vietoris-Rips, and \v{C}ech filtrations are used.

Condition \ref{assump::Tail behavior} imposes the uniform summability condition: 
\[\sum_{\ell=0}^{\infty}(8U)^{\ell} \Pb(|\Xc|=\ell)\leq N,\]
for any point cloud random variable $\Xc$ under consideration.
This condition is satisfied, for example, when $\Xc$ follows a Poisson process with an intensity measure whose total mass  $\lambda$ is bounded by $\frac{\log N}{8U-1}$, assuming $8U>1$. Indeed, let $|\Xc| \sim \mathrm{Pois}(\lambda)$ with $\lambda \leq \frac{\log N}{8U-1}$. Then, we have 
\begin{align*}
    \sum_{\ell=0}^{\infty}(8U)^{\ell} \Pb(|\Xc|=\ell)=\sum_{\ell=0}^\infty (8U)^{\ell} \frac{\lambda^{\ell}}{\ell!}e^{-\lambda}\leq e^{(8U-1)\lambda} \leq N. 
\end{align*}


\section{Analysis for test statistic}\label{Appendix::Analysis-for-test-statistic}
In this section, we analyse our test statistic defined in Section \ref{sec::Test Statistic} of the main text. We first show how the test statistic is derived as an unbiased estimator of the squared RKHS norm $\|\mu_p-\mu_q\|^2_{\Hc_{k_w,\lambda}}$, defined in Section \ref{sec::Test Statistic} of the main text. Recall that the $\mu_p$ for an intensity function $p$ is defined by 
\begin{align}
    \mu_p = \int  p(x)w(\cdot)w(x)k_\lambda(\cdot,x)dx \in \Hc_{k_{w,\lambda}}. 
\end{align}
The squared $\Hc_{k_{w,\lambda}}$-norm can be written as 
\begin{align*}
    \|\mu_p-\mu_q\|_{\Hc_{k_{w,\lambda}}}^2 = \langle\mu_p,\mu_p \rangle_{\Hc_{k_{w,\lambda}}} + \langle\mu_q,\mu_q \rangle_{\Hc_{k_{w,\lambda}}} -2\langle\mu_p,\mu_q \rangle_{\Hc_{k_{w,\lambda}}}.
\end{align*}
Note that 
\begin{align*}
    \langle\mu_p,\mu_p \rangle_{\Hc_{k_{w,\lambda}}}= \left\langle \int w(\cdot)w(x)k_\lambda(\cdot,x)d\E(X_1)(x),\int w(\cdot)w(x)k_\lambda(\cdot,x)d\E(X_2)(x) \right\rangle_{\Hc_{k_{w,\lambda}}},
\end{align*}
which can be unbiasedly estimated as
\begin{align*}
    &\left\langle \int w(\cdot)w(x)k_\lambda(\cdot,x)\frac{1}{n}\sum_{i=1}^nd(\delta_{X_i})(x),\int w(\cdot)w(x)k_\lambda(\cdot,x)\frac{1}{n-1}\sum_{j\neq i}^nd(\delta_{X_j})(x)\right\rangle_{\Hc_{k_{w,\lambda}}}\\
    &=\frac{1}{n(n-1)}\sum_{i=1}^n\sum_{j\neq i}^{n}\left\langle \int w(\cdot)w(x)k_\lambda(\cdot,x)d(\delta_{X_i})(x),\int w(\cdot)w(x)k_\lambda(\cdot,x)d(\delta_{X_j})(x)\right\rangle_{\Hc_{k_{w,\lambda}}}\\
    &=\frac{1}{n(n-1)}\sum_{i=1}^n\sum_{j\neq i}^{n}\left\langle  \sum_{\mathbf{z}\in X_i}w(\cdot)w(\mathbf{z})k_\lambda(\cdot,\mathbf{z}),\sum_{\mathbf{z'}\in X_j} w(\cdot)w(\mathbf{z'})k_\lambda(\cdot,\mathbf{z'})\right\rangle_{\Hc_{k_{w,\lambda}}}\\
    &=\frac{1}{n(n-1)}\sum_{i=1}^n\sum_{j\neq i}^n\sum_{\mathbf{z}\in X_i}\sum_{\mathbf{z'}\in X_j}\left\langle w(\cdot)w(\mathbf{z})k_\lambda(\cdot,\mathbf{z}),w(\cdot)w(\mathbf{z'})k_\lambda(\cdot,\mathbf{z}')\right\rangle_{\Hc_{k_{w,\lambda}}}\\
    &=\frac{1}{n(n-1)}\sum_{i=1}^n\sum_{j\neq i}^n\sum_{\mathbf{z}\in X_i}\sum_{\mathbf{z'}\in X_j} w(\mathbf{z})w(\mathbf{z'})k_\lambda(\mathbf{z},\mathbf{z}'),
\end{align*}
where the last equality comes from the reproducing property of RKHS.
Applying the same argument to the remaining terms yields the unbiased estimator:
\begin{align*}
    \widehat{\|\mu_p-\mu_q\|}_{\Hc_{k_{w,\lambda}}}^2
    =& \frac{1}{n(n-1)}\sum_{i=1}^n\sum_{j\neq i}^n \sum_{\mathbf{z}\in X_i}\sum_{\mathbf{z'}\in X_j} w(\mathbf{z})w(\mathbf{z}')k_\lambda(\mathbf{z},\mathbf{z}') \notag\\
    &+\frac{1}{m(m-1)}\sum_{i=1}^m\sum_{j\neq i}^m \sum_{\mathbf{z}\in Y_i}\sum_{\mathbf{z'}\in Y_j} w(\mathbf{z})w(\mathbf{z}')k_\lambda(\mathbf{z},\mathbf{z}')\notag\\
    &-\frac{2}{nm}\sum_{i=1}^n\sum_{j=1}^m \sum_{\mathbf{z}\in X_i}\sum_{\mathbf{z'}\in Y_j} w(\mathbf{z})w(\mathbf{z}')k_\lambda(\mathbf{z},\mathbf{z}').\notag
\end{align*}
This unbiased estimator can be written as
\begin{align}\label{Appendix::eq::def of U_nm}
    \widehat{\|\mu_p-\mu_q\|}_{\Hc_{k_{w,\lambda}}}^2 &= \frac{1}{n(n-1)m(m-1)}\sum_{1\leq i\neq i'\leq n}\sum_{1\leq j\neq j'\leq m} h_\lambda(X_i,X_{i'},Y_j,Y_{j'})
\end{align}
where $h_\lambda(X,X',Y,Y') := K(X,X')+K(Y,Y')-K(X,Y')-K(X',Y)$ for $X,X',Y,Y' \in \mathbf{PD} $ and $K(X,Y):=\sum_{\bz\in X}\sum_{\bz'\in Y}w(\bz)w(\bz')k_\lambda(\bz,\bz')$ for $X,Y\in \mathbf{PD}$.

The estimator $\widehat{\|\mu_p-\mu_q\|}_{\Hc_{k_{w,\lambda}}}^2$ is a two-sample U-statistic. 
\begin{defn}[Two-sample U-statistic, \citep{lee1990ustatistics,vandervaart1998asymptotic}] Let $X_1,\dots,X_n$ and $Y_1,\dots ,Y_m$ be $\iid$ random samples from possibly different distributions. Let $h(x_1,\dots,$   $ x_r,y_1,\dots y_s)$ be a measurable function that is permutation symmetric in $x_1,\dots, x_r$ and $y_1,\dots y_s$ separately, for $r,s\in \Nb$. A two-sample U-statistic with kernel $h$ is defined as 
\begin{align*}
    U=\frac{1}{\binom{n}{r}}\frac{1}{\binom{m}{s}}\sum_{\alpha}\sum_{\beta}h(X_{\alpha_1},\dots,X_{\alpha_{r}},Y_{\beta_1},\dots, Y_{\beta_s}),
\end{align*}
    where $\alpha$ and $\beta$ range over the collections of all subsets of $r$ different elements from $\{1,\dots,n\}$ and  of $s$ different elements from $\{1,\dots, m\}$, respectively. 
\end{defn}

The estimator $\widehat{\|\mu_p-\mu_q\|}_{\Hc_{k_{w,\lambda}}}^2$ is a two-sample U-statistic with $r=2$, $s=2$. This fact is used repeatedly in several proofs throughout the paper to exploit fundamental properties of U-statistics.

\section{Homotopy equivalence of  \texorpdfstring{\v{C}ech}{Cech} complex on the circle}
\label{Appendix::Homotopy equivalence Cech Circle}

While we discuss the persistent homology of the \v{C}ech
complex (which is used in the proof of the lower bound in Section~\ref{sec::Power analysis}) in Section~\ref{sec::Topological properties of cech on the circle} of the main text, a more granular
analysis yields the exact homotopy type of the complex at any given
filtration parameter $t>0$. 

Let $n_{t}$ be the number of connected
components in $\bigcup_{x\in\mathcal{X}}\mathcal{B}(x,t)$, where $\mathcal{B}(x,t)$ is the closed ball of radius $t$ centred at $x$ with respect to $\|\cdot \|_2$-norm and let
$r_{0}\coloneqq\frac{1}{2}\underset{\substack{x\sim y \\x,y\in\mathcal{X}}}{\sup} \left\Vert x-y\right\Vert _{2}$; see Section \ref{sec::Topological properties of cech on the circle} of the main text for the definition of $x\sim y$.
The following proposition establishes the explicit homotopy equivalence
($\simeq$) for the \v{C}ech complex built on a point cloud sampled from a circle. For the definition of homotopy equivalence, see Section \ref{Appendix::Background_homotopy}.
\begin{prop}\label{Appendix::homotopy_cech_of_circle}
Fix $r>0$, and let $\mathcal{X}=\{x_1,\dots,x_n\}\subseteq S^{1}(r)$ be a nonempty finite subset with $n\geq 3$. 
\begin{enumerate}
\item If $\mathcal{X}$ is contained in some open hemicircle lying on $S^{1}(r)$,
then for each $t>0$, 
\[
\mathrm{\check{C}ech}(\mathcal{X},t)\simeq\{x_{1},\ldots,x_{n_{t}}\},
\]
\item If $\mathcal{X}$ is not contained in any open hemicircle of radius
$r$, then for each $t>0$, 
\[
\mathrm{\check{C}ech}(\mathcal{X},t)\simeq\begin{cases}
\{x_{1},\ldots,x_{n_{t}}\}, & \text{if }t< r_{0},\\
S^{1}, & \text{if }r_{0}\leq t< r,\\
\{x_{1}\}, & \text{if }t\geq r.
\end{cases}
\]
\end{enumerate}
\end{prop}
This complete characterization of the homotopy equivalence provides
distinct theoretical advantages when compared to the existing literature
discussed in Section \ref{sec::Topological properties of cech on the circle} of the main text.

\begin{rmk}
When compared with topological reconstruction of manifolds, classical
reconstruction theorems typically rely on geometric conditions (such
as the reach of a manifold and the sampling density) to define a conservative,
restricted interval of $t$ within which the homotopy equivalence
$\textrm{\v{C}ech}(\Xc,t)\simeq\mathcal{M}$ is strictly guaranteed.
In contrast, Proposition \ref{Appendix::homotopy_cech_of_circle} does not merely provide a conservative sufficient range; it exhaustively analyses the exact homotopy type for all $t>0$.
By identifying the precise phase transitions at $r_{0}$
and $r$, this result yields a significantly stronger and more complete
theoretical characterization for the specific case of a circle.
\end{rmk}

\begin{rmk}
When compared with homotopy types on the circle, much of the existing
literature that explicitly calculates the homotopy types of complexes
on $S^{1}$ focuses on either the Vietoris-Rips complex or the intrinsic
$\check{C}ech$ complex (which measures distance along the geodesic
arc of the circle). In those settings, the complexes often exhibit
highly complex topological behaviours at larger scale parameters, such
as the emergence of higher-dimensional odd spheres. Our proposition,
however, specifically analyses the ambient $\check{C}ech$ complex
constructed using the standard Euclidean metric in the ambient space.
This leads to a simpler topological evolution: transitioning
cleanly from discrete points, to $S^{1}$, and directly to a contractible
point. This structural simplicity makes the ambient $\check{C}ech$
complex highly practical for constructing straightforward theoretical
examples and proofs, such as the minimax lower bound framework utilised
in our main text.
\end{rmk}

\section{Auxiliary results}
In this section, we provide auxiliary results that support our assumptions and theoretical developments.   
\subsection{Persistence intensity function}\label{Appendix::Intensity-function}
Persistence intensity functions are important objects in TDA. They have been studied extensively. Here, we focus on two aspects: a sufficient condition for the existence of the persistence intensity function and its uniform boundedness. 
\subsubsection{A sufficient condition for the existence of intensity functions}\label{Appendix::A-sufficient-condition-for-existence-of-intensity-function}
In this section, we describe the conditions under which the persistence intensity function exists, as established in  \citet{Chazal2019}. For the underlying space $\Mc$, we recall the following definitions. 
\begin{defn}[Real Analytic Function]
Let \( U \subset \mathbb{R}^n \) be an open set. A function 
\[
f: U \to \mathbb{R}
\]
is called \textbf{real analytic} on \(U\) if for every point \(x_0 \in U\), there exists a sequence of real numbers \(\{c_\alpha\}_{\alpha \in \mathbb{N}^n}\) such that
\[
f(x) = \sum_{\alpha \in \mathbb{N}^n} c_\alpha (x - x_0)^\alpha,
\]
for all \(x\) in some neighborhood of \(x_0\), where
\[
(x - x_0)^\alpha := (x_1 - x_{0,1})^{\alpha_1} \cdots (x_n - x_{0,n})^{\alpha_n}.
\]

Equivalently, \(f\) is infinitely differentiable (\(C^\infty\)) and its Taylor series at each point converges to \(f\) in a neighborhood of that point.
\end{defn}
\begin{defn}[Real Analytic Manifold]
A \textbf{real analytic manifold} of dimension \(n\) is a topological manifold \(\Mc\) of dimension \(n\) together with an \emph{atlas} 
\[
\{(U_\alpha, \varphi_\alpha)\}_{\alpha \in A}
\] 
satisfying the following properties:

\begin{enumerate}
    \item Each \(U_\alpha \subset \Mc\) is an open set and \(\varphi_\alpha : U_\alpha \to \mathbb{R}^n\) is a homeomorphism onto its image.
    \item The collection \(\{U_\alpha\}_{\alpha \in A}\) covers \(\Mc\), i.e., \(\Mc = \bigcup_{\alpha \in A} U_\alpha\).
    \item For any \(\alpha, \beta \in A\) with \(U_\alpha \cap U_\beta \neq \emptyset\), the \emph{transition map}
    \[
    \varphi_\beta \circ \varphi_\alpha^{-1} : \varphi_\alpha(U_\alpha \cap U_\beta) \to \varphi_\beta(U_\alpha \cap U_\beta)
    \]
    is a \textbf{real analytic function} between open subsets of \(\mathbb{R}^n\).
\end{enumerate}

A manifold equipped with such an atlas is called a \emph{real analytic manifold}.
\end{defn}
For example, any open subset of $\Rb^n$ and $n$-dimensional sphere $\Sb^n \subset \Rb^{n+1}$ are real analytic manifolds. 

Now, we state the conditions in \citet{Chazal2019} for the filtration $\Kc(\Xc)$ built on a point cloud $\Xc$. Fix $n\in \Nb$. Let $\Mc$ be a topological space and let $\Fc_n$ be the collection of non-empty subsets of $\{1,\dots,n\}$. Let $\varphi^{(n)}=(\varphi^{(n)}[J])_{J\in \Fc_n}:\Mc^n\rightarrow \Rb^{\Fc_n}$ be a continuous function. A simplex $J$ is added in the filtration at the time $\varphi^{(n)}[J]$, so it is called a filtering function. If $x$ is a finite subset of $\Mc$, then $\Kc(x)$ is defined by the filtration associated with $\varphi^{(|x|)} $, where $|x|$ is the size of $x$. For $x=(x_1,\dots,x_n)\in \Mc^n$ and for a simplex $J$, let $x(J) := (x_j)_{j\in J}$. For $n\in \Nb$, the assumptions for $\varphi^{(n)}(=\varphi)$ are as follows: 

\begin{enumerate}
\renewcommand{\labelenumi}{\theenumi}
\renewcommand{\theenumi}{(K\arabic{enumi})}
  \item\label{assump::K1} \textit{Absence of interaction:} For $J \in \mathcal{F}_n$, 
  $\varphi[J](x)$ only depends on $x(J)$.

  \item\label{assump::K2} \textit{Invariance under permutations:} For $J \in \mathcal{F}_n$ and for 
  $(x_1, \ldots, x_n) \in \Mc^n$, if $\tau$ is a permutation of $\{1, \ldots, n\}$ 
  whose support is included in $J$, then
  \[
    \varphi[J](x_{\tau(1)}, \ldots, x_{\tau(n)}) 
    = \varphi[J](x_1, \ldots, x_n).
  \]

  \item\label{assump::K3} \textit{Monotonicity:} For $J \subset J' \in \mathcal{F}_n$, 
  $\varphi[J] \leq \varphi[J']$.

  \item\label{assump::K4} \textit{Compatibility:} For a simplex $J \in \mathcal{F}_n$ and for $j \in J$, 
  if $\varphi[J](x_1, \ldots, x_n)$ is not a function of $x_j$ on some open set 
  $U$ of $\Mc^n$, then 
  $\varphi[J] \equiv \varphi[J \setminus \{j\}]$ on $U$.

  \item\label{assump::K5} \textit{Smoothness:} The function $\varphi$ is subanalytic and the gradient 
  of each of its entries indexed by $J$ with $|J|>1 $ is non-vanishing $a.s$. Moreover, for a singleton $\{j\}$, $\varphi[\{j\}]=0$
\end{enumerate}
Assumption \ref{assump::K1} implies that the time at which a simplex $J$ is added to the filtration depends only on the locations of its vertices, not on the relative position between vertices. For a more detailed discussion of the assumptions, see \citep{Chazal2019}.

\begin{rmk} The Vietoris-Rips complex and \text{\v{C}ech} complex defined in Section \ref{sec::persistent homology} satisfy  Assumptions \ref{assump::K1}-\ref{assump::K5}. 
\end{rmk}

\begin{thm}[Existence of an intensity function \citep{Chazal2019}]\label{Appendix::thm::existence-intensity} Fix $n\geq 1$. Assume that $\Mc$ is a real analytic compact $d$-dimensional connected submanifold and that $\Xc$ is a random variable on $\Mc^n$ having a density with respect to the Hausdorff measure. Assume that the filtration $\Kc$ satisfies assumptions from \ref{assump::K1} to \ref{assump::K5}. Then, for $k\geq 1$, $\Eb[\mathrm{PD}_k\left(\Kc(\cX)\right)]$ has a density with respect to the Lebesgue measure on $\Omega:=\{(x,y)\in\Rb^2:0<x<y<\infty\}$. Moreover, $\Eb[\mathrm{PD}_0\left(\Kc(\cX)\right)]$ has a density with respect to the Lebesgue measure on the vertical line $\{0\}\times [0,\infty)$.
\end{thm}

\begin{rmk} When $\Xc$ follows a Poisson process on $\Mc$, the same result holds, \citep[Corollary 3.4, p. 133]{Chazal2019}.  
\end{rmk}
\begin{rmk}
   Assumption \ref{assump::K5} can be replaced by the following assumption:
    \begin{enumerate}
    \renewcommand{\labelenumi}{\theenumi}
    \renewcommand{\theenumi}{(K6)}
        \item\label{assump::K6} \textit{Smoothness:}  The function $\varphi$ is subanalytic and the gradient of each of its entries is non vanishing a.s.
    \end{enumerate}
    The Vietoris-Rips and \v{C}ech filtrations do not satisfy assumption \ref{assump::K6}, since they have $0$ value for any singleton $J$. For a filtration function satisfying \ref{assump::K6}, the density of the expected $0$-dimensional persistence diagram is defined on $\Omega$, rather than on $\{0\}\times [0,\infty)$.
\end{rmk}

\begin{rmk}
    Since the underlying connected manifold is assumed to be compact and assumption \ref{assump::K1} holds, the $y$-axis of the generated persistence diagram is bounded by $M$ for some $M$>0. 
    Therefore, one can replace $\Omega$ by $\Omega(M)$ and $[0,\infty)$ by $[0,M)$  in Theorem \ref{Appendix::thm::existence-intensity}, where $\Omega(M)= \{(x,y)\in \Omega : y<M\}$.
\end{rmk}
\subsubsection{Uniform bound on persistence intensity functions}\label{Appendix::sec::Uniform-bound-on-intensity-function}
In Condition \ref{assump::existence of conditional intensity}, we assume that $\|w^2\cdot p_\ell\|_\infty\leq L^\ell $. In this section, we discuss a sufficient condition for this assumption, using the result in \citet[p. 13]{wu2024estimation}. To do this, we need the following definitions.

\begin{defn} Let $f$ be a probability density function on a set $\Sc\subseteq \Rb^d$ and let $n\in \Nb$. A point process $\Xc$ consisting of $n$ independent and identically distributed points with common probability density $f$ is called a binomial point process of $n$ points in $\Sc$ with density $f$ and denoted by:
\begin{align*}
    \Xc\sim \mathrm{binomial}(\Sc,n,f).
\end{align*}
If $f$ is the probability density function of the uniform distribution on $\Sc$ with respect to the Lebesgue measure, then $\Xc$ is called a homogeneous binomial point process. 
\end{defn}
\begin{defn}
    A point process $\Xc$ on a set $\Sc\subseteq \Rb^d$ is a Poisson point process with an intensity measure $\mu$ (and a density function $\rho$) if the following two properties hold: 
    \begin{enumerate}[label=$(\alph*)$]
    \item For any $B \subseteq \Sc$ such that $\mu(B) < \infty$, 
    $N(B) \sim \mathrm{Pois}(\mu(B))$—the Poisson distribution with mean $\mu(B)$.

    \item For any $n \in \mathbb{N}$ and $B \subseteq S$ such that $0 < \mu(B) < \infty$,
    \[
        (\Xc\cap  B \mid N(B)=n) \sim \mathrm{binomial}\!\left(B, n, \frac{\rho(\cdot)}{\mu(B)}\right),
    \]
    where $N(B)$ is the cardinality of $\Xc\cap B$.
\end{enumerate}
If the binomial point process appearing in the second condition is a homogeneous binomial point process then $\Xc$ is called a homogeneous Poisson point process. 
\end{defn}

We are ready to state the following theorem.
\begin{thm}[{\citet[p. 13]{wu2024estimation}}]\label{Appendix::thm::uniform-bound-of-intensity}
    Let $\kappa$ be a probability density on $[0,1]^d$ such that $0<\inf \kappa \leq \sup \kappa <\infty $. Suppose that either $\Xc_N$ is a binomial process with parameters $N$ and $\kappa$ or a Poisson process with intensity measure $N\cdot\kappa$ in the cube $[0,1]^d$. Denote by $p(u)$ the persistence intensity function for the $k$-dimensional expected persistence measure induced by the Vietoris-Rips filtration. Then, there exists a  constant $C'$ depending only on $k$ such that 
    \begin{align*}
        \|p\|_\infty \leq C'\cdot \mathrm{poly}(N,d)\sup\kappa,
    \end{align*}
    where $\mathrm{poly}(N,d)=N^5d^3$.
\end{thm} 
\begin{rmk}\label{Appendix::rmk::the uniform domination remark}
In this remark, we discuss a sufficient condition for Condition \ref{assump::existence of conditional intensity}, using Theorem \ref{Appendix::thm::uniform-bound-of-intensity}. Let $\Xc$ follow a Poisson process on the cube $[0,1]^d$ with an intensity  density $\kappa$. With the Vietoris-Rips filtration $\Kc$ and some positive integer $k$, the random variable $D=\mathrm{PD}_{k}(\Kc(\Xc))$ defines the conditional intensity function $p_{\ell} $ by 
\[p_{\ell} = \frac{d\Eb(D\mid |\Xc|=\ell)}{d\mathrm{Leb_2}}.\]
Let $|\kappa| := \int \kappa(x)\,dx$ denote the total mass of the density $\kappa$ so that $\kappa/|\kappa|$ is a probability density function.
Since the conditional Poisson process given $|\Xc|=\ell$ is the binomial process with parameter $\ell $ and $\kappa/|\kappa|$, we have by Theorem~\ref{Appendix::thm::uniform-bound-of-intensity},
\begin{align*}
    \|p_{\ell}\|_{\infty} \leq C'(k)\ell^{5}d^3\sup (\kappa/|\kappa|), 
\end{align*}
for each $\ell\geq 1$ and for some constant $C'(k)$.

It is further required that the upper bound on 
$\|p_{\ell}\|_{\infty}$ is independent of $\kappa$, $d$, and $k$. 
To this end, we restrict the ambient dimension to 
$d \in \{1,\dots,K\}$ for some $K$, so that the dimension 
$k$ of topological features also satisfies $k \le K$. 
Define
\[
C'_{\mathrm{max}} := \max_{k\in\{1,\dots,K\}} C'(k).
\]
Denote by $\Fc$ the collection of all intensity densities $\kappa$ under consideration.
We also assume that there exists a constant $K'>0$ such that for any 
$\kappa \in \Fc$,
\[
\left\|\frac{\kappa}{|\kappa|}\right\|_\infty \le K'.
\]
Lastly, let 
\[
U := (\|w^2\|_\infty\cdot C'_{\mathrm{max}} \cdot K^3 \cdot K')  \vee 32.
\]
Then, we have that 
\[
\|w^2\cdot p_{\ell}\|_{\infty}\leq \|w^2\|_\infty \|p_\ell\|_\infty \le U\ell^5 \le U^{\ell},
\]
for each $\ell\in\Nb$. Consequently, this provides a sufficient 
condition for Condition~\ref{assump::existence of conditional intensity}.
\end{rmk}

\section{Proofs for Section~\ref{Sec::Testing method} and \ref{Sec:: Bandwidth-Aggregated Two Sample Test}}
In this section, we use some notions of functional analysis. For related definitions and formulas, see Section \ref{Appendix::Hilbert-space-operators}. 
\label{Appendix::proof test}
\subsection{Proof of Theorem  \ref{Theorem::uniform-validity}}
Recall that $A_i=\Phi_\lambda(X_i)-\mu$, $B_j=\Phi_\lambda(Y_j)-\mu$, 
and $\mu=\mu_p=\mu_q$. For notational simplicity, we let $\nu=\nu_{n,m}=\lambda_1\lambda_2$, 
\begin{align*}
\Sigma_P=\Sigma_{P,\lambda}=\Eb(A_1\otimes A_1),\quad \Sigma_Q=\Sigma_{Q,\lambda}=\Eb(B_1\otimes B_1).
\end{align*}
Moreover, we write the inner product $\langle,\rangle_{\Hc_{k_{w,\lambda}}}$ as $\langle,\rangle$ when there is no confusion, and denote by $\langle \cdot \rangle_{\mathrm{HS}}$ and 
$\|\cdot \|_{\mathrm{HS}}$ the Hilbert-Schmidt inner product on $\Hc_{k_{w,\lambda}}$ and its associated norm, respectively. For their formal definitions and properties used in this proof, see Section \ref{Appendix::Reproducing-kernel-Hilbert-spaces}.
For $\eta=\frac{n}{n+m}$, let $\Gamma_{\eta}=(1-\eta)\Sigma_P+\eta\Sigma_Q$.

In this proof, we use the following bounds, the proofs of which are deferred.
\begin{enumerate}\renewcommand{\theenumi}{M\arabic{enumi}}
\renewcommand{\labelenumi}{(\theenumi)}
    \item\label{M1:uniform bound of sum of operator norms}$  \sup_{(P,Q)\in \Pc_0^{\otimes 2}}\max\{\|\Sigma_P\|_{\op},\|\Sigma_Q\|_{\op}\}\leq C_{\mathrm{M1}}$,
    \item\label{M2:uniform bound of sum of trace}$\sup_{(P,Q)\in \Pc_0^{\otimes 2}}\max\left\{\mathrm{tr}\left(\Sigma_P^2\right),\mathrm{tr}\left(\Sigma_Q^2\right)\right\}\leq C_{\mathrm{M2}}\nu^{-1}$,
    \item\label{M3:asymp condition of trace}  $\inf_{(P,Q)\in \Pc_0^{\otimes 2}}  \tr\left(\Gamma_{\eta}^2\right)\geq c_{\mathrm{M3}}\nu^{-1},\quad 
    \sup_{(P,Q)\in \Pc_0^{\otimes 2}}\tr\left(\Gamma_{\eta}^2\right)\leq C_\mathrm{M3}\nu^{-1}$, for all sufficiently large $n,m$.
    \item\label{M4:uniform bound of sum of expectations} $\sup_{(P,Q)\in \Pc_0^{\otimes 2}}\max\{\Eb\langle A_1,\Sigma_{P}A_1\rangle^2 ,\Eb\langle B_1,\Sigma_{P}B_1\rangle^2,\Eb\langle A_1,\Sigma_{Q}A_1\rangle^2,\Eb\langle B_1,\Sigma_{Q}B_1\rangle^2\}\leq C_{\mathrm{M4}}\nu^{-2}$,
    \item\label{M5:uniform bound of fourth moment} $\sup_{(P,Q)\in \Pc_0^{\otimes 2}} \max\{\Eb\langle A_1,A_2\rangle ^4,\Eb\langle B_1,B_2\rangle ^4 ,\Eb\langle A_1,B_1\rangle ^4\}  \leq C_\mathrm{M5}\nu^{-3}$.
\end{enumerate}

The first step of this proof is to establish the uniform convergence:  
\begin{align}
    \sup_{(P,Q)\in\Pc_0^{\otimes 2}}
    \sup_{t\in\Rb}
    \left|
        \Pb\left(U_{n,m,\lambda}/\sigma_{n,m}\leq t\right)
        -
        \Phi(t)
    \right|
    \longrightarrow0,\quad \text{as}\quad n+m\to \infty.
\end{align}
To do this, we express the statistic $U_{n,m,\lambda}$ as the sum of the martingale differences and 
show that the conditions for the martingale difference central limit theorem uniformly hold over $\Pc_0^{\otimes 2}$.

Note that, under $p=q$, we can write 
\begin{align*}
    U_{n,m,\lambda}=\frac{2}{n(n-1)}\sum_{i<i'}\langle A_i,A_{i'}\rangle+\frac{2}{m(m-1)}\sum_{j<j'}\langle B_j,B_{j'}\rangle-\frac{2}{nm}\sum_{i=1}^n\sum_{j=1}^{m}\langle A_i,B_j\rangle.
\end{align*}
Define sigma fields as
\begin{align*}
\Fc_r:=\sigma(A_1,\dots ,A_r),\quad 1\leq r\leq n,\\
    \Fc_{n+j}:=\sigma(A_1,\dots A_n,B_1,\dots ,B_j),\quad 1\leq j \leq m,
\end{align*}
and let $\Fc_0$ be the trivial sigma field. 
Define partial sums $S_r^A$ and $S_j^B$ by 
\begin{align*}
    S_r^A:= \sum_{i=1}^r A_i,\quad S_j^{B}:=\sum_{\ell=1}^j B_{\ell},
\end{align*}
and define
\begin{align*}
    &D_r:= \frac{2}{n(n-1)}\langle S^A_{r-1},A_r\rangle,\quad 2\leq r\leq n,\\
    &D_1:=0,
\end{align*}
and 
\begin{align*}
    &D_{n+j}:= \left\langle \frac{2}{m(m-1)}S^B_{j-1}-\frac{2}{nm}S_n^A,B_j\right\rangle,\quad 1\leq j\leq m.
\end{align*}
Note that the sequence $\{(D_i,\Fc_i):1\leq i\leq n+m\}$ is a martingale difference array. Indeed, for $2\leq r\leq n$,
\begin{align*}
    \Eb\left[D_r\mid \Fc_{r-1}\right]&=\frac{2}{n(n-1)}\Eb \left[\langle S^A_{r-1},A_r\rangle\mid \Fc_{r-1}\right]\\
    &=\frac{2}{n(n-1)}\langle S^A_{r-1},\Eb \left[A_r\mid \Fc_{r-1}\right]\rangle\\
    &=0,
\end{align*}
since $A_r$ is independent of $\Fc_{r-1}$ and its mean is zero. Likewise, 
\begin{align*}
    \Eb\left[D_{n+j}\mid \Fc_{n+j-1}\right]&=\left\langle \frac{2}{m(m-1)}S^B_{j-1}-\frac{2}{nm}S_n^A,\Eb[B_j\mid \Fc_{n+j-1}]\right\rangle=0.
\end{align*}
The statistic $U_{n,m,\lambda}$ is written as the sum of the martingale differences:
\begin{align*}
    U_{n,m,\lambda}=\sum_{r=1}^{n+m}D_r.
\end{align*}

We now apply the martingale difference central limit theorem, \citet{hall1980martingale}, to $U_{n,m,\lambda}$. 
Recall $ \sigma_{n,m}^2 = \V(U_{n,m,\lambda})$. Denote by $\pconv$ convergence in probability.
If the following conditions hold
\begin{align}
    \frac{1}{\sigma_{n,m}^2}\sum_{r=1}^{n+m}\Eb[D^2_{r}\mid \Fc_{r-1}]\pconv 1 \label{mtg first condition}\\
    \text{for any }\varepsilon>0,\ \frac{1}{\sigma_{n,m}^2}\sum_{r=1}^{n+m}\Eb\left[D^2_r\mathbb{I}\left\{|D_r|>\varepsilon\sigma_{n,m}\right\}\mid F_{r-1}\right]\pconv 0, \label{mtg second condition}
\end{align}
then, 
\begin{align}
    \frac{U_{n,m,\lambda}}{\sigma_{n,m}}\xrightarrow{d} \Nc(0,1).
\end{align}

Let us prove the two conditions. 
Put 
\begin{align}
    a_P:=\Eb\langle A_1,A_2\rangle^2=\tr\left(\Sigma_P^2\right),\quad a_Q:= \tr\left(\Sigma_Q^2\right),\quad \text{and}\quad a_{PQ}:=\tr (\Sigma_P\Sigma_Q).
\end{align}
Under $p=q$, we can easily compute 
\begin{align*}
    \sigma_{n,m}^2=\frac{2a_P}{n(n-1)}+\frac{2a_Q}{m(m-1)}+\frac{4a_{PQ}}{nm}.
\end{align*}
Define the effective sample size 
\begin{align*}
    N_{\mathrm{eff}}:= \frac{nm}{n+m}.
\end{align*}
Then, we have 
\begin{align*}
    \frac{\Gamma_{\eta}}{N_{\mathrm{eff}}}=\frac{\Sigma_P}{n}+\frac{\Sigma_Q}{m}.
\end{align*}
This implies that
\begin{align*}
    \frac{1}{N_{\mathrm{eff}}^2}\tr\left(\Gamma_\eta^2\right)&=\tr \left(\frac{\Sigma_P}{n}+\frac{\Sigma_Q}{m}\right)^2\\
    &=\frac{\tr(\Sigma_P^2)}{n^2}+\frac{\tr(\Sigma_Q^2)}{m^2}+\frac{\tr(\Sigma_P\Sigma_Q)+\tr(\Sigma_Q\Sigma_P)}{nm}\\
    &=\frac{a_P}{n^2}+\frac{a_Q}{m^2}+\frac{2a_{PQ}}{nm}.
\end{align*}
Therefore, we obtain 
\begin{align*}
    \sigma^2_{n,m}=\frac{2}{N_{\mathrm{eff}}^2}\tr\left(\Gamma_\eta^2\right)+\frac{2a_P}{n^2(n-1)}+\frac{2a_Q}{m^2(m-1)}.
\end{align*}
This formulation yields
\begin{align}
    \sigma_{n,m}^4\asymp \frac{1}{(n+m)^4 \nu^2}.
\end{align}
Indeed, since $n\asymp m$, $N_{\mathrm{eff}}\asymp n+m$, and by the condition \ref{M3:asymp condition of trace}, 
\begin{align*}
    \frac{2}{N_{\mathrm{eff}}^2}\tr\left(\Gamma_{\eta}^2\right)\asymp \frac{1}{(n+m)^2\nu}.
\end{align*}
By the condition \ref{M2:uniform bound of sum of trace}, the remaining terms are negligible. Indeed,
\begin{align*}
    \frac{2a_P}{n^2(n-1)}+\frac{2a_Q}{m^2(m-1)} \lesssim \frac{1}{(n+m)^3\nu}.
\end{align*}
Therefore, we have 
\begin{align}\label{variance order here}
    \sigma_{n,m}^2 \asymp \frac{1}{(n+m)^2\nu}.
\end{align}

Let us verify the condition \eqref{mtg first condition}.
Define
\begin{align*}
    V_{n,m}:= \sum_{r=1}^{n+m} \Eb\left[D_r^2\mid \Fc_{r-1} \right].    
\end{align*}
The quantity $V_{n,m}$ can be decomposed as follows:
\begin{align*}
    V_{n,m}=V_P+V_Q+V_C+V_{PQ},
\end{align*}
where
\begin{align*}
    &V_P :=\frac{4}{n^2(n-1)^2}\sum_{r=2}^n \langle S_{r-1}^A,\Sigma_P S_{r-1}^A\rangle\\
    &V_Q :=\frac{4}{m^2(m-1)^2}\sum_{j=1}^m \langle S_{j-1}^B,\Sigma_Q S_{j-1}^B\rangle\\
    &V_C:= \frac{-8}{nm^2(m-1)}\sum_{j=1}^m \langle S_{j-1}^B,\Sigma_Q S_n^A\rangle\\
    &V_{PQ}:= \frac{4}{n^2m}\langle S_n^A, \Sigma_Q S_n^A\rangle.
\end{align*}
One can easily compute 
\begin{align*}
    \Eb[V_P]=\frac{2a_P}{n(n-1)},\quad \Eb[V_Q]=\frac{2a_Q}{m(m-1)},\quad \Eb[V_C]=0,\quad \Eb[V_{PQ}]=\frac{4a_{PQ}}{nm},
\end{align*}
so that
\begin{align*}
    \Eb(V_{n,m})= \sigma_{n,m}^2.
\end{align*}
We now establish the upper bound for the variance of $V_{n,m}$ as follows:
\begin{align*}
    \V(V_{n,m}) \leq C\left(\frac{1}{(n+m)^5\nu^2}+\frac{1}{(n+m)^4\nu}\right),
\end{align*}
for some constant $C>0$.
First, consider the variance of $V_P$. The term $V_P$ can be decomposed as 
\begin{align*}
    V_P= \frac{4}{n^2(n-1)^2}\left[\sum_{i=1}^{n-1}(n-i)\langle A_i,\Sigma_P A_i\rangle +2\sum_{1\leq i<k\leq n-1}(n-k)\langle A_i,\Sigma_PA_k\rangle\right].
\end{align*}
Define 
\begin{align*}
    R_P:=\Eb\langle A_1,\Sigma_P A_1 \rangle ^2, \quad R_Q:=\Eb\langle B_1,\Sigma_Q B_1 \rangle ^2.
\end{align*}
Then, by simple computation, we can obtain 
\begin{align}
    \V\left(V_P\right)\leq \frac{16}{3n^4(n-1)^4}\left(n^3R_P+n^4\tr (\Sigma_P^4)\right).
\end{align}
By the condition \ref{M4:uniform bound of sum of expectations}, 
\begin{align*}
    R_P \leq C_{\mathrm{M4}} \nu^{-2}. 
\end{align*}
Moreover, by the conditions \ref{M1:uniform bound of sum of operator norms} and \ref{M2:uniform bound of sum of trace}, we have 
\begin{align*}
    \tr(\Sigma_P^4)\leq \|\Sigma_P\|_{\mathrm{op}}^2 \tr(\Sigma_P^2)\leq C_{\mathrm{M1}}^2C_{\mathrm{M2}}\nu^{-1}.
\end{align*}
Therefore, 
\begin{align}\label{variance of VP}
    \V(V_P)\leq C'\left(\frac{1}{(n+m)^5\nu^2}+\frac{1}{(n+m)^4\nu}\right),
\end{align}
for some constant $C'>0$.
Following the same arguments used in the case of $V_P$, 
we can obtain 
\begin{align}\label{variance of VQ}
    \V(V_Q)\leq C'\left(\frac{1}{(n+m)^5\nu^2}+\frac{1}{(n+m)^4\nu}\right).
\end{align}

We now consider the variance of $V_{PQ}$. Define 
\begin{align*}
    R_{PQ}:= \Eb\left[\langle A_1,\Sigma_Q A_1 \rangle^2\right],\quad a_{PQ}:= \Eb\left[\langle A_1, \Sigma_Q A_1\rangle \right],\quad\text{and}\quad\ W_{PQ}:= \Eb\left[\langle A_1,\Sigma_Q A_2 \rangle^2\right].
\end{align*}
Note that 
\begin{align*}
    \langle S_n^A, \Sigma_Q S_n^A \rangle =\sum_{i=1}^n \langle A_i,\Sigma_Q A_i\rangle +2\sum_{1\leq i<k\leq n}\langle A_i,\Sigma_Q A_k\rangle. 
\end{align*}
By degeneracy, the covariance vanishes for both off-diagonal pairs and diagonal–off-diagonal pairs. Moreover, the diagonal terms are independent. Therefore,  
\begin{align*}
    \V \left(\langle S_n^A, \Sigma_Q S_n^A\right)= n(R_{PQ}-a_{PQ}^2 )+2n(n-1)W_{PQ}.
\end{align*}
This yields 
\begin{align*}
    \V(V_{PQ})\leq \frac{32}{n^4m^2} \left(nR_{PQ}+n^2 W_{PQ}\right).
\end{align*}
Since 
\begin{align*}
    \Eb[\langle A_1, \Sigma_Q A_2\rangle ^2 \mid A_1 ]&=\Eb[\langle \Sigma_Q A_1, A_2\rangle ^2 \mid A_1 ]\\
    &=\langle \Sigma_Q A_1, \Sigma_P\Sigma_QA_1\rangle,
\end{align*}
we can write $W_{PQ}$ as 
\begin{align*}
    W_{PQ}&=\Eb\langle A_1, \Sigma_Q\Sigma_P\Sigma_QA_1\rangle\\
    &=\tr (\Sigma_Q\Sigma_P \Sigma_Q\Sigma_P)\\
    &=\tr \left(\left(\Sigma_P\Sigma_Q \right)^2\right).
\end{align*}
By Cauchy-Schwarz inequality and the ideal property of the Hilbert–Schmidt class,
\begin{align*}
    \tr \left(\left(\Sigma_P\Sigma_Q \right)^2\right) \leq \|\Sigma_P \Sigma_Q \|^2_{\HS}\leq \|\Sigma_P\|_{\HS}^2 \|\Sigma_Q\|_{\op}^2=\tr(\Sigma_P^2)\|\Sigma_Q\|_{\op}^2.
\end{align*}
Lastly, by the conditions \ref{M1:uniform bound of sum of operator norms} and \ref{M2:uniform bound of sum of trace}, 
\begin{align*}
    W_{PQ} \leq C_{\mathrm{M1}}^2C_{\mathrm{M2}} \nu^{-1}.
\end{align*}
Therefore, we obtain 
\begin{equation}\label{variance of VPQ}
\begin{aligned}
    \V(V_{PQ})&\leq \frac{32}{n^4m^2} \left(nC_{\mathrm{M4}}\nu^{-2}+n^2 C_{\mathrm{M1}}^2C_{\mathrm{M2}}\nu^{-1}\right)\\
    &\leq C''\left((n+m)^{-5}\nu^{-2}+(n+m)^{-4}\nu^{-1}\right),
\end{aligned}
\end{equation}
for some constant $C''>0$.

We now consider the variance of $V_{C}$.
By simple computation, we can obtain
\begin{align*}
    \V(V_C)=\frac{64}{m^4(m-1)^2n} \sum_{\ell=1}^{m-1}(m-\ell)^2\tr(\Sigma_P\Sigma_Q^3).
\end{align*}
By the Cauchy-Schwarz inequality and the conditions \ref{M1:uniform bound of sum of operator norms}, \ref{M2:uniform bound of sum of trace}, we have
\begin{align*}
    \tr(\Sigma_P\Sigma_Q^3) &\leq \|\Sigma_P\|_{\HS}\|\Sigma_Q^3\|_{\HS}\\
    &\leq \|\Sigma_Q\|_{\op}^2 \sqrt{\tr(\Sigma_P^2)}\sqrt{\tr(\Sigma_Q^2)}\\
    &\leq C_{\mathrm{M1}}^2C_{\mathrm{M2}}\nu^{-1}.
\end{align*}
Therefore, 
\begin{align}\label{variance of VC}
    \V(V_C) \leq \frac{C'''}{(n+m)^4\nu},
\end{align}
for some constant $C'''>0$.

Combining \eqref{variance of VP}, \eqref{variance of VQ}, \eqref{variance of VPQ}, \eqref{variance of VC}, and 
\begin{align*}
    \V(V_{n,m})\leq 4[\V(V_P)+\V(V_Q)+\V(V_{PQ})+\V(V_C)]
\end{align*}
yields 
\begin{align}\label{bound of variance of Vnm}
    \V(V_{n,m}) \leq C\left(\frac{1}{(n+m)^5\nu^2}+\frac{1}{(n+m)^4\nu}\right),
\end{align}
for some constant $C>0$.
Since 
\begin{align*}
    \sigma_{n,m}^2 \asymp \frac{1}{(n+m)^2\nu},
\end{align*}
it holds that, for some $\widetilde{C}>0$, 
\begin{align}\label{Vnmsigmanm Lsquare}
    \Eb\left[\frac{V_{n,m}}{\sigma_{n,m}^2}-1\right]^2 \leq \widetilde{C}\left(\frac{1}{n+m}+\nu\right),
\end{align}
which yields
\begin{align*}
    \frac{V_{n,m}}{\sigma_{n,m}^2} \pconv 1,\quad \text{as}\quad n+m\rightarrow \infty.
\end{align*}
The upper bound \eqref{bound of variance of Vnm} holds for any $(P,Q) \in \Pc_0^{\otimes 2}$, since the constant $\widetilde{C}$ does not depend on the pair $(P,Q)$. Consequently, there exists a constant $C_1$ such that 
\begin{align*}
    \sup_{(P,Q)\in \Pc_0^{\otimes 2}} \Pb_{P\times Q}\left(\left|\frac{V_{n,m}}{\sigma_{n,m}^2}-1\right|>\varepsilon\right)\leq \frac{C_1}{\varepsilon^2}\left(\frac{1}{(n+m)}+\nu\right)\rightarrow 0,\quad \text{as}\quad n+m\rightarrow \infty,
\end{align*}
for any $\varepsilon>0$.
This completes the proof of \eqref{mtg first condition}.

We now prove the condition \eqref{mtg second condition}.
Recall that, for $2\leq r\leq n$, 
\begin{align*}
    D_r=\frac{2}{n(n-1)}\sum_{i=1}^{r-1}\langle A_i, A_r\rangle.
\end{align*}
The summands 
$    \xi_{i,r}:= \langle A_i,A_r\rangle$,  $i<r$,
are conditionally independent when $A_r$ is fixed, and 
\begin{align*}
    \Eb[\xi_{i,r}\mid A_r] =\langle \Eb A_i,A_r\rangle =0.
\end{align*}
We hence have 
\begin{align*}
    \Eb\left[\left(\sum_{i=1}^{r-1}\xi_{i,r}\right)^4\mid A_r \right]\leq \sum_{i=1}^{r-1}\Eb[\xi_{i,r}^4\mid A_r]+3\left(\sum_{i=1}^{r-1}\Eb[\xi_{i,r}^2 \mid A_r]\right)^2.
\end{align*}
Then, it holds that
\begin{align*}
    \Eb[D_r^4]\leq \frac{2^4}{n^4(n-1)^4}\left[(r-1)\Eb\langle A_1,A_2\rangle^4+3(r-1)^2\Eb\langle A_1,\Sigma_P A_1\rangle^2\right].
\end{align*}
Furthermore, by 
\begin{align*}
    \Eb\langle A_1,A_2\rangle^4 \leq C_{\mathrm{M5}}\nu^{-3}\quad \text{and} \quad \Eb\langle A_1,\Sigma_P A_1\rangle^2 \leq C_{\mathrm{M4}}\nu^{-2},
\end{align*}
we obtain
\begin{align}\label{sum of fourth moments of r}
    \sum_{r=2}^n\Eb[D_r^4] = O\left((n+m)^{-6}\nu^{-3}+(n+m)^{-5}\nu^{-2}\right).
\end{align}
Recall that for $1\leq j\leq m$, 
\begin{align*}
    D_{n+j}=\frac{2}{m(m-1)}\sum_{\ell=1}^{j-1}\langle B_{\ell},B_j \rangle -\frac{2}{nm}\sum_{i=1}^n \langle A_i, B_j\rangle.
\end{align*}
Conditionally on $B_j$, the collections 
\begin{align*}
    \left\{\langle B_{\ell}, B_j\rangle\mid \ell <j\right\}\quad \text{and}\quad \left\{\langle A_{i}, B_j\rangle\mid 1 \leq i\leq n\right\}
\end{align*}
are independent, and every random variable in these collections has conditional mean zero. Therefore, by the same argument, we obtain 
\begin{align}\label{sum of fourth moments of nj}
    \sum_{j=1}^m \Eb[D^4_{n+j}]= O\left((n+m)^{-6}\nu^{-3}+(n+m)^{-5}\nu^{-2}\right).
\end{align}
Combining \eqref{variance order here}, \eqref{sum of fourth moments of r}, and \eqref{sum of fourth moments of nj} yields
\begin{align}\label{order of scaled forth moment sum}
    \frac{\sum_{r=}^{n+m} \Eb[D_r^4]}{\sigma^4_{n,m}} \leq C_2\left(\frac{1}{(n+m)^2\nu}+\frac{1}{n+m}\right),
\end{align}
for some constant $C_2>0$ which does not depend on the distributions $P,Q$ and samples sizes $n$ and $m$. 
Let 
\begin{align*}
   L_{n,m}(\varepsilon):=\frac{1}{\sigma_{n,m}^2}\sum_{r=1}^{n+m}\Eb\left[D^2_r\mathbb{I}\left\{|D_r|>\varepsilon\sigma_{n,m}\right\}\mid F_{r-1}\right].
\end{align*}
Then, by \eqref{order of scaled forth moment sum}, we have that, for any $(P,Q) \in \Pc_0^{\otimes 2}$, 
\begin{align*}
    \Eb[L_{n,m}(\varepsilon)]\leq \frac{1}{\varepsilon^2\sigma_{n,m}^4}\sum_{r=1}^{n+m}\Eb[D_r^4]\leq \frac{C_2}{\varepsilon^2}\left(\frac{1}{(n+m)^2\nu}+\frac{1}{n+m}\right) \longrightarrow 0,\quad \text{as}\quad n+m\to \infty,
\end{align*}
since $(n+m)^2\nu \to \infty$, as
$n+m\rightarrow \infty$.
By Markov's inequality, we obtain that, for any $\varepsilon'>0$
\begin{align*}
    \sup_{(P,Q)\in \Pc^{\otimes 2}_0} \Pb_{P,Q}\left(L_{n,m}(\varepsilon)>\varepsilon'\right)\longrightarrow 0,\quad \text{as}\quad n+m\to \infty.
\end{align*}
To establish the uniform convergence, consider an arbitrary sequence of sample sizes $(n_k,m_k)$ satisfying $n_k+m_k\to \infty $ and $n_k\asymp m_k$, and an arbitrary sequence $(P_k,Q_k)\subset \Pc_{0}^{\otimes 2}$. The preceding uniform bounds imply that \eqref{mtg first condition}, \eqref{mtg second condition} hold along this array. Hence, the martingale difference central limit theorem yields
\begin{align*}
    \sup_{t\in\Rb}
    \left|
        \Pb\left(U_{n_k,m_k,\lambda}/\sigma_{n_k,m_k}\leq t\right)
        -
        \Phi(t)
    \right|
    \longrightarrow0,\quad \text{as}\quad n_k+m_k\to \infty.
\end{align*}
Since the $(n_k,m_k)$ and $(P_k,Q_k)$ are arbitrary, the following uniform convergence follows:
\begin{align}\label{asymptotic normality result 1}
    \sup_{(P,Q)\in\Pc_0^{\otimes 2}}
    \sup_{t\in\Rb}
    \left|
        \Pb\left(U_{n,m,\lambda}/\sigma_{n,m}\leq t\right)
        -
        \Phi(t)
    \right|
    \longrightarrow0,\quad \text{as}\quad n+m\to \infty.
\end{align}

The second step of this proof is to show that the estimator 
\begin{align*}
    \hat{\sigma}^2_{n,m}=\frac{2\hat{a}_P}{n(n-1)}+\frac{2\hat{a}_Q}{m(m-1)}+\frac{4\hat{a}_{PQ}}{nm},
\end{align*} defined in Section \ref{sec::Test Statistic} of the main text,
is consistent for $\sigma^2_{n,m}$. In particular, our target is to show $\hat{\sigma}^2_{n,m}/\sigma_{n,m}^2 \pconv 1.$

For distinct indices $i_1,i_2,i_3,i_4\in \{1,\dots,n\}$, define 
\begin{align*}
    \Delta_{i_1,i_2,i_3,i_4}^P :=\langle \Phi_\lambda(X_{i_1})-\Phi_\lambda(X_{i_2}),\Phi_\lambda(X_{i_3})-\Phi_\lambda(X_{i_4}) \rangle.
\end{align*}
For $n,k\in \Nb$ with $k<n$, let $(n)_k:= n(n-1)(n-2)\cdots (n-k+1)$. Then, we can write 
\begin{align*}
    \hat{a}_P=\frac{1}{4(n)_4}\sum_{i_1,i_2,i_3,i_4}^{\neq} \left(\Delta_{i_1,i_2,i_3,i_4}^P\right)^2,
\end{align*}
Likewise, for
\begin{align*}
    &\Delta_{j_1,j_2,j_3,j_4}^Q :=\langle \Phi_\lambda(Y_{j_1})-\Phi_\lambda(Y_{j_2}),\Phi_\lambda(Y_{j_3})-\Phi_\lambda(Y_{j_4}) \rangle,\\
    &\Delta_{i_1,i_2,j_1,j_2}^{PQ} := \langle \Phi_\lambda(X_{i_1})-\Phi_\lambda(X_{i_2}),\Phi_\lambda(Y_{j_1})-\Phi_\lambda(Y_{j_2}) \rangle,
\end{align*}
we can write
\begin{align*}
    &\hat{a}_Q=\frac{1}{4(m)_4}\sum_{j_1,j_2,j_3,j_4}^{\neq} \left(\Delta_{j_1,j_2,j_3,j_4}^Q\right)^2,\\
    &\hat{a}_{PQ}=\frac{1}{4(n)_2(m)_2}\sum_{i_1\neq i_2}\sum_{j_1\neq j_2} \left(\Delta_{i_1,i_2,j_1,j_2}^{PQ}\right)^2.
\end{align*}
The estimators $\hat{a}_P$, $\hat{a}_Q$, and $\hat{a}_{PQ}$ are unbiased for $a_P$, $a_Q$, and $a_{PQ}$, respectively. Indeed, since $A_i=\Phi_\lambda(X_i)-\mu_\lambda$, it holds that 
\begin{align*}
    \Phi_{\lambda}(X_{i_1})- \Phi_{\lambda}(X_{i_2}) = A_{i_1}-A_{i_2}.
\end{align*}
Let 
\begin{align*}
    U= A_1-A_2,\quad V=A_3-A_4.
\end{align*}
Then, we have 
\begin{align*}
    \Eb[U\otimes U]= \Eb[V\otimes V]=2\Sigma_P.
\end{align*}
Therefore, it holds that
\begin{align*}
    \Eb\langle U,V\rangle^2 &=  \Eb\langle U\otimes U,V\otimes V\rangle_{\mathrm{HS}}\\
    &=\langle \Eb [U\otimes U],\Eb[V\otimes V]\rangle_{\mathrm{HS}}
    \\&= 4\langle \Sigma_P,\Sigma_P\rangle_{\mathrm{HS}}
    \\&=4\tr (\Sigma_{P}^2)=4a_P.
\end{align*}
This yields $\Eb[\hat{a}_P]=a_P$. By the same argument, we also have $\Eb[\hat{a}_Q]=a_Q$ and $\Eb[\hat{a}_{PQ}]=a_{PQ}$.

We now examine the variance of each estimator. 
Define 
\begin{align*}
   h(a_1,a_2,a_3,a_4):= \frac{1}{4!}\sum_{\pi \in \Sb_4}\frac{1}{4}\langle a_{\pi(1)}-a_{\pi(2)},a_{\pi(3)}-a_{\pi(4)}\rangle ^2,
\end{align*}
where $\Sb_4$ denotes the symmetric group on $\{1,2,3,4\}$.
Then, 
\begin{align*}
    \Eb[h(A_1,A_2,A_3,A_4)]=a_P,
\end{align*}
and 
\begin{align*}
    \hat{a}_P =  \binom{n}{4}^{-1}\sum_{\substack{I\subset \{1,\dots,n\}\\|I|=4}}h(A_I),
\end{align*}
where $A_I=(A_{i_1},A_{i_2},A_{i_3},A_{i_4})$, for $I=\{i_1,i_2,i_3,i_4\}.$
For $c=1,2,3,4$, define 
\begin{align*}
    h_c(a_1,\dots,a_c):= \Eb[h(a_1,\dots,a_c,A_{c+1},\dots,A_4)]
\end{align*}
and 
\begin{align*}
    \zeta_c := \V(h_c(A_1,\dots A_c)).
\end{align*}
By the Hoeffding decomposition for $U$-statistics \citep{Hoeffding1948}, we obtain 
\begin{align*}
    \V(\hat{a}_P)=\binom{n}{4}^{-1}\sum_{c=1}^4 \binom{4}{c}\binom{n-4}{4-c}\zeta_c.
\end{align*}
We refer the reader to \citet[Chapter 12]{vandervaart1998asymptotic} for the Hoeffding decomposition.

The conditional variance $\zeta_1$ can be written as 
\begin{align*}
    \zeta_1=\V\left(h_1(A_1\right)) =\frac{1}{4}\V \left(\langle A_1,\Sigma_P A_1\rangle \right),
\end{align*}
since $\Eb(h(A_1,A_2,A_3,A_4 )\mid A_1=a) = \frac{1}{2}\left(\langle a,\Sigma_P a\rangle +a_P\right)$. Therefore, we obtain
\begin{align*}
    \zeta_1 \leq \frac{1}{4}\Eb \left(\langle A_1, \Sigma_P A_1\rangle^2 \right)\leq \frac{C_{\mathrm{M4}}}{4}\nu^{-2}.
\end{align*}

For $c=2,3,4,$ we first compute an upper bound for $\Eb[h(A_1,A_2,A_3,A_4)^2]$. 
Define, for each $\pi \in \Sb_4$,
\begin{align*}
    G_{\pi}:= \frac{1}{4}\langle A_{\pi(1)}-A_{\pi(2)},A_{\pi(3)}-A_{\pi(4)}\rangle ^2 . 
\end{align*}
By the Jensen's inequality, we obtain 
\begin{align*}
    h(A_1,A_2,A_3,A_4)^2 \leq \frac{1}{4!}\sum_{\pi \in \Sb_4}G_\pi^2.
\end{align*}
Since $A_1,\dots, A_4$ are \iid, 
\begin{align*}
    \Eb[h(A_1,A_2,A_3,A_4)^2] &\leq \Eb[G_{\mathrm{id}}^2]\\
    &=\frac{1}{16}\Eb[\langle A_1-A_2,A_3-A_4\rangle^4].
\end{align*}
Put 
\begin{align*}
    L=\langle A_1-A_2,A_3-A_4\rangle.
\end{align*}
For any $x_1,\dots,x_4 \in \Rb$, it holds that 
\begin{align*}
    |x_1+x_2+x_3+x_4|^4 \leq 64 \sum_{r=1}^4 |x_r|^4.
\end{align*}
This implies 
\begin{align}\label{forth moment jensen bound}
    |L|^4 \leq 64 \left\{|\langle A_1,A_3\rangle|^4+|\langle A_1,A_4\rangle|^4+|\langle A_2,A_3\rangle|^4+|\langle A_2,A_4\rangle|^4\right\}.
\end{align}
By the condition \ref{M5:uniform bound of fourth moment}, we obtain
\begin{align*}
    \Eb[h(A_1,A_2,A_3,A_4)^2]\leq 16C_{\mathrm{M5}} \nu^{-3}.
\end{align*}

For $c=2$, note that 
\begin{align*}
    \zeta_2 =\V\left(\Eb[h(A_1,A_2,A_3,A_4)]\mid A_1,A_2\right).
\end{align*}
By the law of total variance, 
\begin{align*}
    \V\left(\Eb[h(A_1,A_2,A_3,A_4)]\mid A_1,A_2\right) \leq \V(h(A_1,A_2,A_3,A_4)).
\end{align*}
Therefore, 
\begin{align*}
    \zeta_2 \leq 16C_{\mathrm{M5}} \nu^{-3}.
\end{align*}
By the same argument, we can also obtain 
\begin{align*}
    \zeta_c \leq 16C_{\mathrm{M5}} \nu^{-3},\quad \text{for}\quad c=3,4.
\end{align*}
Combining each bound on $\zeta_c$ yields 
\begin{align*}
    \V\left(\hat{a}_P\right) &\leq \frac{4C_{\mathrm{M4}}}{n\nu^2}+\frac{1152 C_{\mathrm{M5}}}{n^2\nu^3} \\
    &= O\left(\frac{1}{n\nu^2}+\frac{1}{n^2\nu^3}\right).
\end{align*}
The same bound holds for $\hat{a}_Q$. 

We now consider $\hat{a}_{PQ}$. Define 
\begin{align*}
    g(a_1,a_2,b_1,b_2) := \frac{1}{4}\langle a_1-a_2, b_1-b_2\rangle^2. 
\end{align*}
Then, the estimator $\hat{a}_{PQ}$ can be written as a two-sample $U$-statistic:
\begin{align*}
    \hat{a}_{PQ}=\frac{1}{\binom{n}{2}\binom{m}{2}}\sum_{1\leq i_1<i_2\leq n}\sum_{1\leq j_1< j_2\leq m}g(A_{i_1},A_{i_2},B_{j_1},B_{j_2}).
\end{align*}
For $r,s\in \{0,1,2\}$, 
define 
\begin{align*}
    g_{r,s}(a_1,\dots,a_r,b_1,\dots,b_s):= \Eb\left[g(a_1,\dots,a_r,A_{r+1},\dots,A_2,b_1,\dots,b_s,B_{s+1},\dots,B_2)\right]
\end{align*}
and 
\begin{align*}
    \zeta_{r,s}:= \V\left(g_{r,s}(A_1,\dots,A_r,B_1,\dots,B_s)\right).
\end{align*}
By the Hoeffding decomposition for $U$-statistics \citep{Hoeffding1948}, we also obtain 
\begin{align}\label{Hoeffding two sample ustat decomposition}
    \V\left(\hat{a}_{PQ}\right) =\frac{1}{\binom{n}{2}\binom{m}{2}}\sum_{0\leq r,s\leq 2}\binom{2}{r}\binom{n-2}{2-r}\binom{2}{s}\binom{m-2}{2-s}\zeta_{r,s}.
\end{align}
We derive an upper bound for the variance of each $\zeta_{r,s}$. 
Firstly, consider the case of $r=1$, $s=0$. We can write 
\begin{align*}
    g_{1,0}(a)=\Eb\left(g(a,A_2,B_1,B_2\right)&=\frac{1}{4}\Eb\left(\langle a-A_2,(B_1-B_2)\otimes (B_1-B_2)(a-A_2))\right)\\
    &=\frac{1}{2}\left(\langle a-A_2,\Sigma_Q(a-A_2)\rangle\right)\\
    &=\frac{1}{2}\left(\langle a,\Sigma_Q a\rangle +\tr(\Sigma_P\Sigma_Q)\right),
\end{align*}
where $\langle x,y\rangle^2 = \langle x,(y\otimes y)x\rangle $ and $\Eb[(B_1-B_2)\otimes (B_1-B_2)]=2\Sigma_Q$ are used.
Therefore, by the condition \ref{M4:uniform bound of sum of expectations}, we have
\begin{align*}
    \zeta_{1,0}=\V\left(g_{1,0}(A_1)\right)\leq \frac{1}{4}\Eb\langle A_1,\Sigma_Q A_1\rangle ^2\leq \frac{1}{4}C_{\mathrm{M4}}\nu^{-2}.
\end{align*}
The same bound for $\zeta_{0,1}$ holds by the same argument. We next consider the case of $r+s\geq 2$. By the conditional Jensen inequality, we have 
\begin{align*}
    \zeta_{r,s}\leq \Eb\left[g(A_1,A_2,B_1,B_2)^2\right].
\end{align*}
Since 
\begin{align*}
    g(A_1,A_2,B_1,B_2)^2=\frac{1}{16}\langle A_1-A_2,B_1-B_2\rangle ^4,
\end{align*}
combining \eqref{forth moment jensen bound} and the condition \ref{M5:uniform bound of fourth moment} yields 
\begin{align*}
    \zeta_{r,s}\leq 16C_{\mathrm{M5}}\nu^{-3},\quad r+s\geq2.
\end{align*}
Substituting these bounds on $\zeta_{r,s}$ into \eqref{Hoeffding two sample ustat decomposition} yields the following bound: 
\begin{align*}
    \V\left(\hat{a}_{PQ}\right)\leq C_{\mathrm{M4}}\left(\frac{1}{n}+\frac{1}{m}\right)\nu^{-2} +576C_{\mathrm{M5}}\left(\frac{1}{n^2}+\frac{1}{nm}+\frac{1}{m^2}\right)\nu^{-3}.
\end{align*}
Therefore,
\begin{align*}
    \V\left(\hat{a}_{PQ}\right)=O\left(\frac{1}{(n+m)\nu^2}+\frac{1}{(n+m)^2\nu^3}\right).
\end{align*}

We are now ready to prove the consistency of $\hat{\sigma}^2_{n,m}$.
Let 
\begin{align*}
    \Delta_P := \hat{a}_P-a_P,\quad \Delta_Q := \hat{a}_Q-a_Q,\quad \Delta_{PQ} := \hat{a}_{PQ}-a_{PQ}.
\end{align*}
We hence have 
\begin{align*}
    \hat{\sigma}^2_{n,m}-\sigma_{n,m}^2 =\frac{2\Delta_P}{n(n-1)}+\frac{2\Delta_Q}{m(m-1)}+\frac{4\Delta_{PQ}}{nm},
\end{align*}
and this equation yields 
\begin{align*}
    \Eb\left[\left(\hat{\sigma}^2_{n,m}-\sigma_{n,m}^2\right)^2\right]\leq 3\left\{\frac{4}{n^2(n-1)^2}\Eb\Delta_P^2+\frac{4}{m^2(m-1)^2}\Eb\Delta_Q^2 +\frac{16}{n^2m^2}\Eb\Delta_{PQ}^2\right\}
\end{align*}
We have computed that the quantities 
\begin{align*}
    \Eb\Delta_{P}^2 =\V \left(\hat{a}_P\right),\quad \E\Delta_Q^2=\V \left(\hat{a}_Q\right),\quad \text{and}\quad \E\Delta_{PQ}^2=\V \left(\hat{a}_{PQ}\right)
\end{align*}
are all  
\begin{align*}
    O\left(\frac{1}{(n+m)\nu^2}+\frac{1}{(n+m)^2\nu^3}\right).
\end{align*}
This implies 
\begin{align*}
    \Eb\left[\left(\hat{\sigma}^2_{n,m}-\sigma_{n,m}^2\right)^2\right]=O\left(\frac{1}{(n+m)^5\nu^2}+\frac{1}{(n+m)^6\nu^3}\right).
\end{align*}
By \eqref{variance order here}, we finally obtain 
\begin{align}\label{RatioL2order}
    \Eb\left[\left(\frac{\hat{\sigma}^2_{n,m}}{\sigma_{n,m}^2}-1\right)^2\right]=O\left(\frac{1}{n+m}+\frac{1}{(n+m)^2\nu}\right)=o(1),\quad \text{as}\quad n+m\rightarrow\infty,
\end{align}
since $(n+m)^2\nu\rightarrow  \infty$, as $(n+m)\rightarrow \infty.$
Furthermore, since all the preceding bounds hold
with a constant $C>0$ independent of $(P,Q)$, we have
\begin{align*}
 \sup_{(P,Q)\in\Pc^{\otimes 2}_0}
 \Eb_{P,Q}\left[
 \left(
 \frac{\widehat{\sigma}_{n,m}^2}{\sigma_{n,m}^2}-1
 \right)^2
 \right]
 \leq
 C\left(
 \frac{1}{n+m}+\frac{1}{(n+m)^2\nu}
 \right)
 \longrightarrow 0,\quad \text{as}\quad n+m\rightarrow \infty. 
\end{align*}
It follows from Markov's inequality that, for every $\epsilon>0$,
\begin{align*}
 \sup_{(P,Q)\in\Pc_0^{\otimes 2}}
 \Pb_{P,Q}\left(
 \left|
 \frac{\widehat{\sigma}_{n,m}^2}{\sigma_{n,m}^2}-1
 \right|>\epsilon
 \right)
 \longrightarrow 0,\quad \text{as}\quad n+m\rightarrow \infty. 
\end{align*}
Thus, the variance estimator $\widehat{\sigma}_{n,m}^2$ is ratio consistent uniformly over the null class
$\Pc_0^{\otimes 2}$.
Combining this result with \eqref{asymptotic normality result 1} yields 
\begin{align*}
    \sup_{(P,Q)\in\Pc_0^{\otimes 2}}
    \sup_{t\in\Rb}
    \left|
        \Pb\left(T_{n,m,\lambda}\leq t\right)
        -
        \Phi(t)
    \right|
    \longrightarrow0,\quad \text{as}\quad n+m\to \infty.
\end{align*}

The final step of this proof is to verify \ref{M1:uniform bound of sum of operator norms}--\ref{M5:uniform bound of fourth moment}.
\begin{proof}[Proof of \ref{M1:uniform bound of sum of operator norms}]
Define 
    $C_{P}:= \Eb_{P}\left[\Phi_{\lambda}(X_1)\otimes\Phi_\lambda(X_1)\right]$.
Then, we have 
\begin{align*}
    C_{P}=\Sigma_{P}+\mu_p\otimes \mu_p,
\end{align*}
which implies 
  $ 0\preceq \Sigma_{P} \preceq C_{P}$,
where $A \preceq B$ denotes $B-A$ is positive semidefinite operator, and 
\begin{align*}
    \|\Sigma_P\|_{\op}\leq \|C_P\|_{\op}.
\end{align*}
Therefore, our target is to control $\|C_P\|_{\op}$. 
We claim that 
\begin{align}\label{fxplintegral bound claim}
    \int|f(x)|^2 p_{\ell}(x)dx \leq \|w^2 p_{\ell}\|_{\infty}\|f\|^2_{\Hc_{k_{w,\lambda}}}
\end{align}
holds.
Then, it holds that,
for any $f\in \Hc_{k_{w,\lambda}}$,
\begin{equation}\label{fCpf bound 1}
\begin{aligned}
    \langle f,C_P f\rangle =\Eb\left[\langle f,\Phi_\lambda(X_1)\rangle^2\right]&=\Eb\left[\left(\sum_{x\in X_1}f(x)\right)^2\right]\\
    &\leq\sum_{\ell=0}^\infty \Pb(Z=\ell)h(\ell)\int |f(x)|^2p_{\ell}(x)dx\\
    &\leq \left(\sum_{\ell=0}^{\infty}h(\ell)\|w^2p_{\ell}\|_\infty \Pb(Z=\ell)\right)\|f\|_{\Hc_{k_{w,\lambda}}}^2\\
    &\leq N\|f\|_{\Hc_{k_{w,\lambda}}}^2,
\end{aligned}
\end{equation}
where the reproducing property of the RKHS is used in the second equality, and $p_{\ell}=d\Eb(X_1\mid Z=\ell)/d\mathrm{Leb}_2$.
To show \eqref{fxplintegral bound claim}, we establish that, for $g\in \Hc_{k_\lambda}$ 
\begin{align}\label{L2 is smaller than rkhs norm}
    \|g\|_{L^2}\leq \|g\|_{\Hc_{k_\lambda}}.
\end{align}
Recall that 
\begin{align*}
    k_{\lambda}(x,y)=\varphi_\lambda(x-y)
\end{align*}
and, for $t=(t_1,t_2) \in \Rb^2$,
\begin{align*}
    \varphi_\lambda(t) =\frac{1}{\lambda_1\lambda_2}\varphi\left(\frac{t_1}{\lambda_1},\frac{t_2}{\lambda_2}\right).
\end{align*}
Since $\varphi_\lambda$ is continuous, integrable, and positive
definite, Bochner's theorem, \citep{Bochner1933}, implies 
\begin{align*}
    \widehat{\varphi_\lambda}(\xi)\geq 0,
    \qquad \xi\in\Rb^2.
\end{align*}
 Moreover, it holds that
\begin{align*}
    \widehat{\varphi_\lambda}(\xi) \leq \int|\varphi_\lambda(t)|dt =\int|\varphi(t)|dt=\|\varphi\|_{L^1}=1.
\end{align*}
Therefore, we have
\begin{align}\label{varphi hat is in 01}
    0\leq \widehat{\varphi_\lambda}(\xi)\leq 1,\quad \text{for any}\quad \xi \in \Rb^2.
\end{align}
Consider a finite linear combination $g$ of kernels, \ie,
\begin{align*}
    g(\cdot )=\sum_{j=1}^r \alpha_jk_{\lambda}(\cdot,x_j)=\sum_{j=1}^r \alpha_j\varphi_\lambda(\cdot- x_j),
\end{align*}
for some $\alpha_j\in \Rb$, $x_j\in \Rb^2$, and $r\in \Nb$. Using the Fourier transform, we have 
\begin{align*}
    \widehat{g}(\xi) =\widehat{\varphi_{\lambda}}(\xi) S(\xi),
\end{align*}
where 
\begin{align*}
    S(\xi):=\sum_{j=1}^r \alpha_j e^{-i\xi\cdot x_j}.
\end{align*}
By the Plancherel's identity, we obtain 
\begin{align*}
    \|g\|_{L^2}^2 &=\frac{1}{(2\pi)^2}\int_{\Rb^2}\widehat{\varphi_\lambda}(\xi)^2|S(\xi)|^2d\xi\\
    &\leq \frac{1}{(2\pi)^2}\int_{\Rb^2}\widehat{\varphi_\lambda}(\xi)|S(\xi)|^2d\xi,
\end{align*}
where the inequality holds by \eqref{varphi hat is in 01}. Furthermore, the last term is equal to $\|g\|_{\Hc_{k_{\lambda}}}^2$. Indeed, 
\begin{align*}
    \|g\|_{\Hc_{k_{\lambda}}}^2 &=\sum_{i=1}^r \sum_{j=1}^r\alpha_i \alpha_j \varphi_{\lambda}(x_i-x_j)\\
    &=\sum_{i=1}^r \sum_{j=1}^r\frac{\alpha_i \alpha_j}{(2\pi)^d}\int_{\Rb^2}e^{i\xi\cdot (x_i-x_j)}\widehat{\varphi_\lambda}(\xi)d\xi\\
    &=\frac{1}{(2\pi)^2}\int_{\Rb^2}\widehat{\varphi_{\lambda}}(\xi)\left|\sum_{j=1}^r \alpha_j e^{-i\xi\cdot x_j}\right|^2d\xi\\
    &=\frac{1}{(2\pi)^2}\int_{\Rb^2}\widehat{\varphi_{\lambda}}(\xi)\left|S(\xi)\right|^2d\xi.
\end{align*}
Therefore, we obtain
\begin{align}
    \|g\|_{L^2}^2\leq \|g\|_{\Hc_{k_{\lambda}}}^2.
\end{align}
Since finite linear combinations of $k_\lambda(\cdot,x)$ are dense in $\Hc_{k_{\lambda}}$, the relation holds for every $g\in \Hc_{k_{\lambda}}$. This observation yields \eqref{L2 is smaller than rkhs norm}.

We are now ready to prove \eqref{fxplintegral bound claim}. Let 
\begin{align*}
    f(\cdot)=\sum_{j=1}^r\alpha_{j}k_{w,\lambda}(\cdot,x_j)=\sum_{j=1}^r\alpha_j w(\cdot)w(x_j)k_{\lambda}(\cdot,x_j)\in \Hc_{k_{w,\lambda}}
\end{align*}
and 
\begin{align*}
    g(\cdot)=\sum_{j=1}^{r}\alpha_j w(x_j)k_\lambda(\cdot,x_j)\in \Hc_{k_\lambda}.
\end{align*}
The squared norms of $f$ and $g$ are same. Indeed, 
\begin{align*}
    \|g\|_{\Hc_{k_\lambda}}^2&=\sum_{i=1}^r\sum_{j=1}^r \alpha_i \alpha_j w(x_i)w(x_j) k_\lambda(x_i,x_j) \\
    &=\|f\|^2_{\Hc_{k_{w,\lambda}}}.
\end{align*}
Then, we obtain 
\begin{align*}
    \int |f(x)|^2 p_{\ell}(x)dx &= \int w(x)^2 g(x)^2 p_\ell(x)dx \\
    &\leq \|w^2\cdot p_{\ell}\|_\infty \|g\|_{L^2}^2 \\
    &\leq \|w^2\cdot p_{\ell}\|_\infty \|g\|_{\Hc_{k_\lambda}}^2\\
    &\leq \|w^2\cdot p_{\ell}\|_\infty \|f\|_{\Hc_{k_{w,\lambda}}}^2.
\end{align*}
Therefore, it holds that
\begin{align*}
    \int |f(x)|^2 p_{\ell}(x)dx \leq \|w^2\cdot p_{\ell}\|_\infty \|f\|_{\Hc_{k_{w,\lambda}}}^2.
\end{align*}
Since finite linear combinations of $k_{w,\lambda}(\cdot,x)$ are dense in $\Hc_{k_{w,\lambda}}$, the relation holds for every $f\in \Hc_{k_{w,\lambda}}$.
This yields the claim \eqref{fxplintegral bound claim}. 
 
Using \eqref{fxplintegral bound claim}, we obtain \eqref{fCpf bound 1}. 
Consequently, we obtain 
\begin{align*}
    \|\Sigma_P\|_{\op}\leq \|C_P\|_{\op} \leq N.
\end{align*}
We can obtain the same bound for $\|\Sigma_Q\|_{\op}$ by the same argument. Therefore, letting $C_{\mathrm{M1}}:= N$ yields the condition \ref{M1:uniform bound of sum of operator norms}.
\end{proof}
\begin{proof}[Proof of \ref{M2:uniform bound of sum of trace}]
Note that 
\begin{align*}
    \tr(C_P^2)&=\tr(\Sigma_P+\mu_p\otimes \mu_p)^2 \\
    &=\tr(\Sigma_P^2)+2\tr\left(\Sigma_P(\mu_p\otimes \mu_p)\right)+\tr\left((\mu_p\otimes \mu_p)^2\right)\\&=\tr(\Sigma_P^2)+2\langle \mu_p,\Sigma_P \mu_p\rangle + \|\mu_p\|^4\\
    &\geq \tr (\Sigma_P^2).
\end{align*}
Therefore, our target is to control $\tr(C_P^2)$.  
We can write 
\begin{align*}
    \tr(C_P^2) &= \|C_P\|_{\HS}^2 \\
    &= \Eb\left[\langle \Phi_\lambda(X_1)\otimes \Phi_\lambda(X_1),\Phi_\lambda(X_2)\otimes\Phi_\lambda(X_2)\rangle_{\HS}\right]\\
    &=\Eb \left[\langle \Phi_\lambda(X_1),\Phi_\lambda(X_2)\rangle^2\right].
\end{align*}
Therefore, we have
\begin{align*}
    \tr(C_P^2) =\Eb\left[K_{\lambda}(X_1,X_2)^2\right].
\end{align*}
By Jensen's inequality, we have 
\begin{align*}
    K_{\lambda}\left(X_1,X_2\right)^2\leq |X_1||X_2|\sum_{x\in X_1}\sum_{y\in X_2}w(x)^2w(y)^2 \varphi_\lambda(x-y)^2.
\end{align*}
Let $Z_1$ and $Z_2$ be the auxiliary variables, introduced in the model assumption \ref{assump::core assumption}, corresponding to $X_1$ and $X_2$, respectively. Then, we have 
\begin{align*}
    &\Eb \left[K_{\lambda}\left(X_1,X_2\right)^2\mid Z_1=\ell , Z_2=r\right]\\&\leq h(\ell)h(r)\int_{\Omega(M)}\int_{\Omega(M)} w(x)^2p_{\ell}(x)w(y)^2 p_{r}(y)\varphi_\lambda(x-y)^2dxdy\\
    &\leq h(\ell)h(r) \|w^2p_\ell\|_{\infty} \|w^2p_r\|_{\infty}\int_{\Omega(M)}\int_{\Omega(M)}\varphi_\lambda(x-y)^2dxdy\\
    &\leq h(\ell)h(r) \|w^2p_\ell\|_{\infty} \|w^2p_r\|_{\infty}\frac{|\Omega(M)|\kappa_2}{\nu}.
\end{align*}
Therefore, by \ref{assump::core assumption}, it holds that
\begin{align*}
    \tr(\Sigma_P^2)\leq \tr(C_P^2)=\Eb \left[K_{\lambda}\left(X_1,X_2\right)^2\right] \leq \frac{|\Omega(M)|\kappa_2N^2}{\nu}.
\end{align*}
The same bound holds for $\tr(\Sigma_Q^2)$ by the same argument. 
Letting $C_{\mathrm{M2}}=|\Omega(M)|\kappa_2N^2$ yields the condition \ref{M2:uniform bound of sum of trace}.
\end{proof}

\begin{proof}[Proof of \ref{M3:asymp condition of trace}]
Recall 
\begin{align*}
    \Gamma_{\eta}=(1-\eta)\Sigma_{P}+\eta \Sigma_{Q},\quad \text{and}\quad \eta=\frac{n}{n+m}.
\end{align*}

The upper bound in \ref{M3:asymp condition of trace} can be obtained as
\begin{align*}
    \tr\left(\Gamma_{\eta}^2\right) &=\left\|(1-\eta)\Sigma_P+\eta \Sigma_Q\right\|^2_{\HS}\\
    &\leq (1-\eta)\|\Sigma_P\|_{\HS}^2 +\eta \|\Sigma_Q\|_{\HS}^2\\
    &=(1-\eta)\tr\left(\Sigma_{P}^2\right) +\eta\tr\left(\Sigma_{Q}^2\right)\\
    &\leq C_{\mathrm{M2}} \nu^{-1},
\end{align*}
where the last inequality holds by \ref{M2:uniform bound of sum of trace}.

We now compute the lower bound in \ref{M3:asymp condition of trace}. To do this, define the mixture distribution $R_\eta$ by 
\begin{align*}
    R_\eta := (1-\eta)P+\eta Q,
\end{align*}
and let
\begin{align*}
    C_{\eta}:= (1-\eta)C_P+\eta C_Q.
\end{align*}
Then, it holds that
\begin{align}\label{Cgamma is Gammagamm muotimes}
    C_{\eta}=\Gamma_{\eta}+\mu \otimes \mu,
\end{align}
where $\mu=\mu_p=\mu_q$.
Let $D_1,D_2 \overset{\iid}{\sim }R_{\eta}$. Using $D_1, D_2$, we can write
\begin{align*}
    \tr \left(C_{\eta}^2\right)=\Eb\left[K_{\lambda}(D_1,D_2)^2\right].
\end{align*}
Since 
\begin{align*}
    K_\lambda (D_1,D_2)^2 &=\left\{\sum_{x\in D_1}\sum_{y\in D_2}w(x)w(y)\varphi_\lambda(x-y)\right\}^2\\
    &\geq \sum_{x\in D_1}\sum_{y\in D_2}w(x)^2w(y)^2\varphi_\lambda(x-y)^2,
\end{align*}
we obtain 
\begin{align*}
    \tr \left(C_{\eta}^2\right) \geq I_{\lambda}(p):=\int_{\Rb^2}\int_{\Rb^2}w(x)^2p(x)w(y)^2p(y) \varphi_\lambda(x-y)^2 dxdy.
\end{align*}
Define 
\begin{align}\label{definition of rholambda}
    \rho(u):= \frac{\varphi(u)^2}{\kappa_2},\quad \rho_{\lambda}(u):=\frac{1}{\nu}\rho\left(\frac{u_1}{\lambda_1},\frac{u_2}{\lambda_2}\right),\quad \text{where}\quad \kappa_2= \int_{\Rb^2}\varphi(u)^2du>0.
\end{align}
Then, it holds that
\begin{align*}
    \rho(u)\geq 0,\quad \int_{\Rb^2}\rho(u)du=1,\quad \text{and}\quad \varphi_\lambda(u)^2 =\frac{\kappa_2}{\nu}\rho_\lambda(u).
\end{align*}
Therefore, we can obtain
\begin{equation}\label{Ilambdap expression}
\begin{aligned}
    I_\lambda(p)&=\frac{\kappa_2}{\nu}\int_{\Rb^2}\int_{\Rb^2}w(x)^2p(x)w(y)^2p(y) \rho_\lambda(x-y) dxdy\\
    &=\frac{\kappa_2}{\nu} \langle w^2p,(w^2p)* \rho_\lambda \rangle_{L^2}. 
\end{aligned}
\end{equation}
It hence suffices to establish the lower bound for $\frac{\kappa_2}{\nu} \langle w^2p,(w^2p)* \rho_\lambda \rangle_{L^2}$.

Note that we can write 
\begin{align*}
    w^2p=\sum_{\ell=0}^\infty P(Z=\ell)w^2p_{\ell},
\end{align*}
which implies, since $h(\ell)\ge 1$, 
\begin{align}
    \|w^2p\|_{\infty }&\leq \sum_{\ell=0}^\infty h(\ell)^3P(Z=\ell)\|w^2p_{\ell}\|_\infty 
    \leq N.
\end{align}
Therefore, we have 
    $\sup_{(P,Q)\in \Pc_0}\|w^2 p\|_{\infty }\leq N$.
Furthermore, combining this upper bound with the model assumption \ref{assump::non degeneracy} yields 
\begin{align*}
    \|w^2p\|_{L^1 }\geq \frac{\|w^2p\|_{L^2}^2}{\|w^2p\|_{\infty}}\geq \frac{L^2}{N}=:L_*.
\end{align*}

In what follows, we extend $p$, and hence $w^2p$, by zero outside $\Omega(M)$.
We consider the Fourier transform of $w^2p$: 
\begin{align*}
    \widehat{w^2p}(\xi)=\int_{\Rb^2}e^{-i\langle \xi,x\rangle}w^2(x)p(x)dx.
\end{align*}
Since $w^2p\geq 0$, 
    $\widehat{w^2p}(0)= \|w^2p \|_{L^1}$.
Moreover, since $|e^{-iu}-1|\leq \|u\|$,
we have 
\begin{align*}
    \left|\widehat{w^2p}(\xi)-\widehat{w^2p}(0)\right|&\leq \int_{\Rb^2}\left|e^{-i\langle \xi,x\rangle}-1 \right|w^2(x)p(x)dx\\
    &\leq \|\xi \|_2\int_{\Rb^2} \|x\|_2 w^2(x)p(x)dx\\
    &\leq \Mc(M)\|\xi\|_2 \|w^2p\|_{L^1}, 
\end{align*}
where $\Mc(M):= \sup_{x \in \Omega(M)}\|x\|_2$.
Therefore, if $\|\xi\|_2 \leq r_0 := (2\Mc(M))^{-1}$, then 
\begin{align*}
    \left|\widehat{w^2p}(\xi)\right|&\geq \widehat{w^2p}(0) - \left|\widehat{w^2p}(\xi)-\widehat{w^2p}(0)\right| \\
    &\geq \frac{1}{2}\|w^2 p\|_{L^1}\\
    &\geq \frac{L_*}{2}.
\end{align*}

We now consider the Fourier transform $\hat{\rho}_\lambda$ of $\rho_\lambda$, which is defined in \eqref{definition of rholambda}. 
Since $\varphi$ is positive definite, so is $\varphi^2$.
Therefore, the fact that $\varphi^2 \in L^1$ and Bochner's theorem  \citep{Bochner1933} imply 
\begin{align*}
    \hat{\rho}(\xi)\geq 0,\quad \xi \in \Rb^2.
\end{align*}
It also holds that 
    $\hat{\rho}(0)=\int \rho(u)du =1$.
By the continuity of $\hat{\rho}$, there exists a $\delta^\dagger>0$ such that 
\begin{align*}
    \hat{\rho}(\xi) \geq \frac{1}{2},\quad \text{whenever}\quad \|\xi \|_2 \leq \delta^\dagger.
\end{align*}
Moreover, by the fact 
\begin{align*}
    \hat{\rho}_\lambda(\xi)=\hat{\rho}(\lambda_1\xi_1,\lambda_2\xi_2),
\end{align*}
there exists $N^\dagger >0$ such that 
\begin{align*}
    \hat{\rho}_\lambda(\xi)\geq \frac{1}{2},\quad \|\xi\|_2 \leq r_0,\quad \text{whenever}\quad n,m\geq N^\dagger,
\end{align*}
since $\lambda_1\to0$ and  $\lambda_2\to0$, as $n,m\longrightarrow \infty$. 

We are now ready to obtain a lower bound for $I_\lambda(p)$. According to \eqref{Ilambdap expression}, 
\begin{align*}
    I_\lambda(p)=\frac{\kappa_2}{\nu} \langle w^2 p,(w^2p)* \rho_\lambda
\rangle.
\end{align*}
By Plancherel's identity, it holds that 
\begin{align*}
    \langle w^2 p,(w^2p)* \rho_\lambda
\rangle =\frac{1}{(2\pi)^2 }\int_{\Rb^2}\left|\widehat{w^2p}(\xi)\right|^2 \hat{\rho}_\lambda(\xi)d\xi.
\end{align*}
Combining the lower bounds for $\left|\widehat{w^2p}(\xi)\right|$ and   $\hat{\rho}_\lambda(\xi)$ yields 
\begin{align*}
    \langle w^2 p,(w^2p)* \rho_\lambda
\rangle &\geq \frac{1}{(2\pi)^2 }\int_{\|\xi\|_2 \leq  r_0}\left|\widehat{w^2p}(\xi)\right|^2 \hat{\rho}_\lambda(\xi)d\xi\\&\geq  L'_{*} :=\frac{1}{(2\pi)^2 }\frac{L_*^2}{8}|B(0,r_0)|,
\end{align*}
for all sufficiently large $n$ and $m$.
Therefore, we obtain 
\begin{align*}
    \tr\left(C_\eta^2\right)\geq \frac{\kappa_2 L'_{*}}{\nu},
\end{align*}
for all sufficiently large $n$ and $m$.

In the proof of \ref{M1:uniform bound of sum of operator norms}, we proved 
\begin{align*}
    \|C_P\|_{\op}, \|C_Q\|_{\op} \leq N,
\end{align*}
and, by \eqref{Cgamma is Gammagamm muotimes},
\begin{align*}
    C_\eta - \mu \otimes \mu =\Gamma _\eta \succeq 0,
\end{align*}
where $\Gamma _\eta \succeq 0$ implies that $\Gamma _\eta$ is positive semidefinite. 
This implies $\mu \otimes \mu \preceq C_{\eta}$, which further implies 
\begin{align*}
    \|\mu\|^2 =\|\mu \otimes \mu \|_{\op}\leq\|C_{\eta}\|_{\mathrm{op}} \leq N.
\end{align*}
Therefore, we obtain 
\begin{align*}
    \tr\left(\Gamma_\eta^2 \right)&=\tr (C_\eta^2 )-2\langle \mu ,C_{\eta}\mu \rangle +\|\mu\|^4 \\
    &\geq \tr (C_\eta^2 )-2\|C_{\eta}\|_{\op}\|\mu\|^2 +\|\mu\|^4\\
    &\geq \tr (C_\eta^2 )-2N^2 \\
    &\geq \frac{\kappa_2L'_{*}}{\nu}-2N^2,
\end{align*}
where the last inequality holds for all sufficiently large $n$ and $m$. 
Moreover, since $\nu\rightarrow 0$, as $n,m\rightarrow \infty$,
for all sufficiently large $n$ and $m$, we obtain 
\begin{align*}
    2N^2 \leq \frac{\kappa_2L'_{*}}{2\nu}.
\end{align*}
Consequently, it holds that 
\begin{align*}
    \tr\left(\Gamma_\eta^2 \right) \geq \frac{\kappa_2L'_{*}}{2\nu},
\end{align*}
for all sufficiently small $\nu>0$.

Letting $C_\mathrm{M3}:= C_{\mathrm{M2}}$ and $c_{\mathrm{M3}}:= \frac{\kappa_2L'_{*}}{2}$ completes the proof of \ref{M3:asymp condition of trace}.
\end{proof}

\begin{proof}[Proof of \ref{M4:uniform bound of sum of expectations}]
To prove \ref{M4:uniform bound of sum of expectations}, we first show 
\begin{align*}
    \Eb\|A_1\|^4, \Eb \|B_1\|^4 = O\left(\frac{1}{\nu^2}\right).
\end{align*}
Recall that 
\begin{align*}
    \|\Phi_\lambda (X_1)\|^2 =\sum_{x\in X_1}\sum_{y\in X_1}w(x)w(y) \varphi_\lambda(x-y).
\end{align*}
Since $\varphi_\lambda$ is positive definite, $|\varphi_\lambda(x-y)|^2 \leq \varphi_\lambda(0)^2$.
This implies 
\begin{align*}
    \|\Phi_\lambda (X_1)\|^2&\leq \varphi_\lambda(0)\sum_{x\in X_1}\sum_{y\in X_1}w(x)w(y)\\
    &=\varphi_\lambda(0)\left(\sum_{x\in X_1}w(x)\right)^2\\
    &\leq\varphi_\lambda(0)|X_1|\sum_{x\in X_1}w(x)^2.
\end{align*}
Squaring both sides yields 
\begin{align*}
    \|\Phi_\lambda (X_1)\|^4&\leq\varphi_\lambda(0)^2|X_1|^2\left(\sum_{x\in X_1}w(x)^2\right)^2\\
    &\leq \varphi_\lambda(0)^2|X_1|^3\sum_{x\in X_1}w(x)^4.
\end{align*}
Note that 
\begin{align*}
    \Eb\left[\|\Phi_\lambda (X_1)\|^4\mid Z=\ell \right] \leq \varphi_\lambda(0)^2 h(\ell)^3 \|w^2 p_\ell\|_\infty W^2_{2,M},
\end{align*}
where 
\begin{align*}
    W_{2,M}^2 :=\int_{\Omega(M)}w(x)^2 dx<\infty .
\end{align*}
Therefore, by the assumption \ref{assump::core assumption}, we obtain 
\begin{align*}
    \Eb\left[\|\Phi_\lambda (X_1)\|^4 \right]\leq \frac{\varphi(0)^2W^2_{2,M}N}{\nu^2}.
\end{align*}
Since $A_1:=\Phi_\lambda(X_1)-\mu_p$ and $\Eb \|A_1\|^4 \leq 8(\Eb\|\Phi_\lambda(X_1)\|^4+\Eb\|\mu_p\|^4)\leq 16 \Eb\|\Phi_\lambda(X_1)\|^4$, we have 
\begin{align*}
    \Eb\left\|A_1\right\|^4 \leq \frac{16\varphi(0)^2W^2_{2,M}N}{\nu^2}.
\end{align*}
Following the same argument, we can show
\begin{align*}
    \Eb\left\|B_1\right\|^4 \leq \frac{16\varphi(0)^2W^2_{2,M}N}{\nu^2}.
\end{align*}

We are now ready to prove \ref{M4:uniform bound of sum of expectations}.
For every positive bounded operator $\Sigma$, it holds that
\begin{align*}
    0\leq \langle u,\Sigma u\rangle \leq \|\Sigma\|_{\op}\|u\|^2 ,
\end{align*}
and 
   $ \langle u,\Sigma u\rangle^2 \leq \|\Sigma\|_{\op}^2\|u\|^4$.
Applying this inequality to $\Sigma_P$ yields 
\begin{align*}
    \Eb\langle A_1, \Sigma_P A_1 \rangle ^2 &\leq \|\Sigma_P\|_{\op}^2\Eb \|A_1\|^4 \\
    &\leq \|\Sigma_P\|_{\op}^2 16\varphi(0)^2W^2_{2,M}N \nu^{-2}\\
    &\leq 16\varphi(0)^2W^2_{2,M}N^3 \nu^{-2},
\end{align*}
where the last inequality holds by \ref{M1:uniform bound of sum of operator norms}.
Following the same arguments, we can also obtain the same bound for $\Eb\langle A_1, \Sigma_Q A_1 \rangle ^2$, $\Eb\langle B_1, \Sigma_P B_1 \rangle ^2$, and $\Eb\langle B_1, \Sigma_Q B_1 \rangle ^2$. Therefore, 
putting $C_{\mathrm{M4}} := 16\varphi(0)^2W^2_{2,M}N^3$ completes the proof of \ref{M4:uniform bound of sum of expectations}.
\end{proof}

\begin{proof}[Proof of \ref{M5:uniform bound of fourth moment}]
In this proof, we use the assumption 
    $\kappa_4 := \int_{\Rb^2 }|\varphi(u)|^4 du <\infty$. 
To prove \ref{M5:uniform bound of fourth moment}, we first show that 
\begin{align*}
    \Eb\left[K_{\lambda}(X_1,X_2)^4\right],\Eb\left[K_{\lambda}(X_1,Y_1)^4\right],\Eb\left[K_{\lambda}(Y_1,Y_2)^4\right]=O\left(\frac{1}{\nu^3}\right).
\end{align*}
Note that 
\begin{align*}
    |K_\lambda(X_1,X_2)| ^4 \leq |X_1|^3 |X_2|^3 \sum_{x\in X_1}\sum_{y\in X_2}|w(x)|^4 |w(y)|^4 |\varphi_\lambda(x-y)|^4,
\end{align*}
and,  
\begin{align*}
    &\Eb \left[|K_{\lambda}(X_1,X_2)|^4 \mid Z_1=\ell , Z_2=r\right]\\ &\leq \|w\|_\infty^4 h(\ell)^3h(r)^3\|w^2 p_{\ell}\|_\infty \|w^2 p_r\|_{\infty }\int_{\Omega(M)}\int_{\Omega(M)}|\varphi_\lambda(x-y)|^4dxdy,\\&\leq \|w\|_\infty^4 h(\ell)^3h(r)^3\|w^2 p_{\ell}\|_\infty \|w^2 p_r\|_{\infty}\frac{|\Omega(M)|\kappa_4}{\nu^3},
\end{align*} where
$Z_1$ and $Z_2$ are the auxiliary variables corresponding to $X_1$ and $X_2$, respectively.
This yields that 
\begin{align*}
    \Eb \left[|K_{\lambda}(X_1,X_2)|^4\right]\leq \frac{\|w\|_\infty^4|\Omega(M)| N^2\kappa_4}{\nu^3}.
\end{align*}
Put $\widetilde{C}(M,N,w,k):=\|w\|_\infty^4|\Omega(M)| N^2\kappa_4.$
By the same argument, we can have 
\begin{align*}
    \max\left\{\Eb\left[K_{\lambda}(X_1,X_2)^4\right],\Eb\left[K_{\lambda}(X_1,Y_1)^4\right],\Eb\left[K_{\lambda}(Y_1,Y_2)^4\right]\right\}\leq \frac{\widetilde{C}(M,N,w,k)}{\nu^3}.
\end{align*}

We are now ready to prove \ref{M5:uniform bound of fourth moment}. 
By the definition of $A_1$ and $A_2$, 
\begin{align*}
    \langle A_1,A_2\rangle = K_{\lambda}(X_1,X_2) -\langle \Phi_\lambda(X_1),\mu_p\rangle-\langle \Phi_\lambda(X_2),\mu_p\rangle + \langle \mu_p,\mu_p\rangle.
\end{align*}
Since, by the conditional Jensen's inequality,
\begin{align*}
    \Eb\langle \Phi_\lambda(X_1),\mu_p\rangle^4, \Eb\langle \Phi_\lambda(X_2),\mu_p\rangle^4,\langle \mu_p,\mu_p\rangle^4 \leq\Eb K_\lambda(X_1,X_2)^4,  
\end{align*}
we have  
\begin{align*}
    \Eb\langle A_1,A_2\rangle ^4 \leq  \frac{256\widetilde{C}(M,N,w,k)}{\nu^3}.
\end{align*}
The same bounds for $\Eb\langle B_1,B_2\rangle ^4$ and $\Eb\langle A_1,B_1\rangle ^4$ hold by the same arguments. Therefore, putting $C_{\mathrm{M5}}:=256\widetilde{C}(M,N,w,k) $ completes the proof of \ref{M5:uniform bound of fourth moment}.
\end{proof}

\begin{proof}[Proof of Theorem \ref{Theorem::uniform-validity}]
We have proved that the uniform convergence holds under the conditions \ref{M1:uniform bound of sum of operator norms}--\ref{M5:uniform bound of fourth moment}, and the conditions hold under our model assumption. Therefore, 
the proof of Theorem \ref{Theorem::uniform-validity} is  completed. 
\end{proof}

\subsection{Proof of Theorem \ref{thm::upperbound_unifseparation_rate}}
Before proceeding the proof, we state three lemmas used in this proof. 
The following lemma states an upper bound on the variance of $U_{n,m,\lambda}$.
\begin{lem}\label{Variance upper bound} There exists a positive constant $C_{\mathrm{Var}}(M,N,w,k)$ such that 
\begin{align*} 
    \V_{P\times Q}\left(U_{n,m,\lambda}\right) \leq C_{\mathrm{Var}}(M,N,w,k)\left(\frac{\|(w\cdot \psi)*\varphi_\lambda\|^2_2}{n+m} + \frac{1}{(n+m)^2\lambda_1\lambda_2}\right),
\end{align*}
for any $(P,Q)\in \Pc^{\otimes 2}$, where $\psi = p - q$ with $p=\Eb(P)$ and $q=\Eb(Q)$.
\end{lem}
\begin{proof}[Proof of Lemma \ref{Variance upper bound}]
    Let $P$ and $Q$ be the probability distributions in $\Pc$ and denote by $p$ and $q$ their intensity functions, respectively.
Recall that $U_{n,m,\lambda}$ is  a two-sample U-statistic. We will use the following upper bound \eqref{Appendix::intermediate step variance upper bound} for two-sample U-statistics; for a proof see \citet[p. 38]{lee1990ustatistics}. There exists a positive constant $c_0$ such that 
\begin{align}\label{Appendix::intermediate step variance upper bound}
    \V_{P\times Q}\left(U_{n,m,\lambda}\right) \leq c_0\left(\frac{\sigma^2_{\lambda,1,0}}{n}+\frac{\sigma^2_{\lambda,0,1}}{m}+\left(\frac{1}{n}+\frac{1}{m}\right)^2\sigma_{\lambda,2,2}^2\right),
\end{align}
where 
\begin{align}\label{Appendix::intermediate step quantities}
    &\sigma^2_{\lambda,1,0} := \V_X\left(\Eb_{X',Y,Y'}[h_\lambda(X,X',Y,Y')]\right),\\
    &\sigma^2_{\lambda,0,1} := \V_Y\left(\Eb_{X,X',Y'}[h_\lambda(X,X',Y,Y')]\right),\nonumber \\
    &\sigma^2_{\lambda,2,2} := \V_{X,X',Y,Y'}\left(h_\lambda(X,X',Y,Y')\right).\nonumber
\end{align}
Here $X,X'\overset{\iid}{\sim} P$ and $Y,Y'\overset{\iid}{\sim} Q$ are all independent of each other. Since we assumed $n\asymp m$, there exists $c_0^{\dagger}$ such that
\begin{align*}
    \V_{P\times Q}\left(U_{n,m,\lambda}\right) \leq c_0^{\dagger}\left(\frac{\sigma^2_{\lambda,1,0}+\sigma^2_{\lambda,0,1}}{n+m}+\frac{\sigma_{\lambda,2,2}^2}{(n+m)^2}\right).
\end{align*}

\begin{lem}\label{Appendix::variance lemma} Assume the same conditions as in Lemma \ref{Variance upper bound}. The quantities $\sigma^2_{\lambda,1,0}$, $\sigma^2_{\lambda,0,1}$, $\sigma^2_{\lambda,2,2}$ defined in \eqref{Appendix::intermediate step quantities} satisfy the following upper bounds:
\begin{align*}
    &\sigma_{\lambda,1,0}^2 \leq N\|(w\cdot\psi)* \varphi_\lambda\|^2_2, \\
    &\sigma_{\lambda,0,1}^2 \leq N\|(w\cdot\psi)* \varphi_\lambda\|^2_2,\\
    &\sigma^2_{\lambda,2,2}\leq 16|\Omega(M)| N^2 \frac{\kappa_2}{\lambda_1\lambda_2}.
\end{align*}
\end{lem}
\noindent
\begin{proof}[Proof of Lemma \ref{Appendix::variance lemma}]
Note that by the definition of $h_\lambda$ in \eqref{U-stat} of the main text,
\begin{align*}
\Eb_{X',Y'}\left[h_\lambda(X,X',Y,Y')\right] =\Eb_{X',Y'}\left[K_\lambda(X,X')+K_\lambda(Y,Y')-K_\lambda(X,Y')-K_\lambda(X',Y)\right].
\end{align*}
Moreover, 
\begin{align*}
     \Eb_{X'}(K_\lambda(X,X'))&=\Eb_{X'}\left(\sum_{x\in X}\sum_{x' \in X'}w(x)w(x')k_\lambda(x,x')\right)\\
     &=\Eb_{X'}\left(\int_{\Rb^2}\sum_{x\in X}w(x)w(u)k_\lambda(x,u)dX'(u)\right)\\
     &=\int_{\Rb^2}\sum_{x\in X}w(x)w(u)k_\lambda(x,u)d\Eb X'(u)\\
     &=\int_{\Rb^2}\sum_{x\in X}w(x)w(u)k_\lambda(x,u)p(u)du\\
     &=\int_{\Rb^2}V_{k_\lambda,X}^w(u)w(u)p(u)du,
\end{align*}
where $V_{k_\lambda,X}^w(u) = \sum_{x\in X}w(x)k_\lambda(x,u)$.
For simplicity, we define 
\begin{align*}
    G(X,p) := \int_{\Rb^2}V_{k_\lambda,X}^w(u)w(u)p(u)du.
\end{align*}
By using this notation, we have that 
\begin{align*}
    \Eb_{X',Y'}\left[h_\lambda(X,X',Y,Y')\right] = G(X,p) +G(Y,q) -G(X,q) -G(Y,p).
\end{align*}
By the definition of $\sigma_{\lambda,1,0}^2$ in \eqref{Appendix::intermediate step quantities}, we obtain that
\begin{align*}
    \sigma_{\lambda,1,0}^2 
    &= \V_X\left(\Eb_Y\left(G(X,p)+G(Y,q)-G(X,q)-G(Y,p)\right)\right)\\
    &=\V_X\left(G(X,p)-G(X,q)+\Eb_Y\left(G(Y,q)-G(Y,p)\right)\right)\\
    &=\V_X\left(G(X,p)-G(X,q)\right)\\
    &\leq\Eb_X\left([G(X,p)-G(X,q)]^2\right)\\
    &=\Eb_X\left(\left[\int_{\Rb^2}V_{k_\lambda,X}^w(u)(p-q)(u)w(u)du\right]^2\right)\\
    &=\Eb_X\left(\left[\int_{\Rb^2}\int_{\Rb^2}w(x)k_\lambda(x,u)\psi(u)w(u)dudX(x)\right]^2\right) \\ 
    &=\Eb_X\left(\left[\int_{\Rb^2}w(x)((w\cdot\psi)*\varphi_\lambda)(x)dX(x)\right]^2\right)\\
    &\leq  \Eb_X\left(|X|\int_{\Rb^2}w(x)^2\left[(w\cdot\psi)*\varphi_\lambda(x)\right]^2dX(x)\right). 
\end{align*}
Note that $|X|$ need not be bounded. We control the last term using  Assumption \ref{assump::core assumption}. By Assumption \ref{assump::core assumption}, there exists a discrete random variable $Z$ and a function $h:\Nb_0\rightarrow \Rb$ such that, conditioned on $Z=\ell$, $|X| \leq h(\ell)$ and 
$\sum_{\ell=0}^{\infty}h(\ell)\|w^2\cdot p_\ell\|_\infty P(Z=\ell)\leq N$. Here, the discrete random variable $Z$ is associated with the random diagram $X$. Therefore, we have 
\begin{align*}
\sigma_{\lambda,1,0}^2 \leq &\Eb_X\left(|X|\int_{\Rb^2}w(x)^2\left[(w\cdot\psi)*\varphi_\lambda(x)\right]^2dX(x)\right)\\
\leq &\sum_{\ell=0}^\infty h(\ell)\cdot\Eb\left(\int_{\Rb^2}w(x)^2\left[(w\cdot\psi)*\varphi_\lambda(x)\right]^2dX(x)\mid Z = \ell \right)P(Z=\ell)\\
 =& \sum_{\ell=0}^\infty h(\ell) P(Z=\ell)\cdot  \int_{\Rb^2}w(x)^2\left[(w\cdot\psi)*\varphi_\lambda(x)\right]^2p_\ell(x)dx \\
    \leq & \sum_{\ell=0}^\infty h(\ell) P(Z=\ell) \|w^2\cdot p_\ell\|_\infty \|(w\cdot \psi) *\varphi_\lambda\|_2^2\\
    \leq & N \|(w\cdot \psi) *\varphi_\lambda\|_2^2.
\end{align*}
Thus, we can conclude 
\begin{align*}
    \sigma_{\lambda,1,0}^2  \leq  N \|(w\cdot \psi) *\varphi_\lambda\|_2^2.
\end{align*}
By the same argument, using the discrete random variable $Z'$ associated with $Y$, we have 
\begin{align*}
    \sigma_{\lambda,0,1}^2 \leq   N \|(w\cdot \psi) *\varphi_\lambda\|_2^2.
\end{align*} 
For the third term, we have 
\begin{align*}
    \sigma^2_{\lambda,2,2} &= \V_{X,X',Y,Y'}\left(h_\lambda(X,X',Y,Y')\right)\\
    &=\V_{X,X',Y,Y'}\left(K_\lambda(X,X')+K_\lambda(Y,Y')-K_\lambda(X,Y')-K_\lambda(X',Y)\right)\\
    &\leq 4(\Eb_{X,X'}[K_\lambda(X,X')^2] + \Eb_{Y,Y'}[K_\lambda(Y,Y')^2]+2\Eb_{X,Y}[K_\lambda(X,Y)^2]).
\end{align*}
Denote the squared moments as follows:
\begin{align*}
    &\Mc_{XX'}(\lambda):= \Eb_{X,X'}[K_\lambda(X,X')^2],\\
    &\Mc_{YY'}(\lambda):= \Eb_{Y,Y'}[K_\lambda(Y,Y')^2],\\
    &\Mc_{XY}(\lambda):= \Eb_{X,Y}[K_\lambda(X,Y)^2].
\end{align*}
We compute an upper bound for each of the squared moments, $\Mc_{XX'}(\lambda), \Mc_{YY'}(\lambda),$ and $\Mc_{XY}(\lambda)$. First, we consider the moment $\Mc_{XY}(\lambda)$. Note that
\begin{align*}
    \Mc_{XY}(\lambda)&= \Eb_{X,Y}\left[\left( \int_{\Rb^2}\int_{\Rb^2}w(x)w(y)k_\lambda(x,y)dX(x)dY(y)\right)^2\right] \\
    &\leq  \Eb_{X,Y}\left[ |X|\cdot |Y|\cdot \int_{\Rb^2}\int_{\Rb^2}w(x)^2w(y)^2k_\lambda(x,y)^2dX(x)dY(y)\right],
\end{align*}
where the inequality is derived from Cauchy–Schwarz inequality. As in the proof of $\sigma^2_{\lambda,1,0}$, we proceed using the conditional control assumption to obtain an upper bound for the last term. 
\begin{align*}
    &\Eb_{X,Y}\!\left[\, |X|\,|Y|
        \int_{\Rb^2}\!\!\int_{\Rb^2}
            w(x)^2 w(y)^2 k_\lambda(x,y)^2 \, dX(x)\, dY(y)
    \right]\\[0.5em]
    &\le 
    \sum_{\ell=0}^\infty \sum_{\ell'=0}^\infty 
        h(\ell)h(\ell')\,
        \Eb\!\left[
            \Eb\!\left[
                \int_{\Rb^2}\!\!\int_{\Rb^2}
                w(x)^2 w(y)^2 k_\lambda(x,y)^2 \, dX(x)\, dY(y)
                \,\middle|\, Z=\ell
            \right]
            \middle|\, Z'=\ell'
        \right] \\
    &\qquad\qquad \times
        P(Z=\ell)\, P(Z'=\ell') \\[0.5em]
    &=
    \sum_{\ell=0}^\infty \sum_{\ell'=0}^\infty 
        h(\ell)h(\ell')\,
        P(Z=\ell)\,P(Z'=\ell') \\
    &\qquad\qquad \times
        \int_{\Rb^2}\!\!\int_{\Rb^2}
            w(x)^2 w(y)^2 k_\lambda(x,y)^2
            p_\ell(x)\, q_{\ell'}(y)
        \, dx\, dy \\[0.5em]
    &\le 
    \sum_{\ell=0}^\infty\sum_{\ell'=0}^\infty
        h(\ell)h(\ell')
        P(Z=\ell)\,P(Z'=\ell')\|w^2\cdot p_\ell\|_\infty \\
    &\qquad\qquad \times
        \int_{\Rb^2}\!\!\int_{\Rb^2}
            w(y)^2 k_\lambda(x,y)^2 q_{\ell'}(y)
        \, dx\, dy \\[0.5em]
    &=
    \sum_{\ell=0}^\infty\sum_{\ell'=0}^\infty
        h(\ell)h(\ell')
        P(Z=\ell)\,P(Z'=\ell')\|w^2\cdot p_\ell\|_\infty
        \frac{\kappa_2}{\lambda_1\lambda_2}
        \int_{\Rb^2} w(y)^2 q_{\ell'}(y)\,dy .
\end{align*}
Note that 
\begin{align*}
    \int_{\Rb^2} w(y)^2q_{\ell'}(y) dy=\int_{\Omega(M)} w(y)^2q_{\ell'}(y) dy &\leq \|w^2\cdot q_{\ell'} \|_\infty \cdot |\Omega(M)|,
\end{align*}
where the first equality holds since the support of $q_{\ell'}$ is contained in the set $\Omega(M)$, and $|\Omega(M)|$ is the area of $\Omega(M).$
Thus, we have that 
\begin{align*}
    \Mc_{XY}(\lambda) &\leq \frac{\kappa_2}{\lambda_1\lambda_2} |\Omega(M)|\sum_{\ell=0}^\infty h(\ell)P(Z=\ell)\| w^2\cdot p_\ell\|_\infty
    \sum_{\ell'=0}^\infty h(\ell')P(Z'=\ell')\|w^2\cdot q_{\ell'}\|_\infty \\
    &\leq  \frac{\kappa_2}{\lambda_1\lambda_2}|\Omega(M)|N^2.
\end{align*}
Thus, we can conclude that 
\begin{align*}
    \Mc_{XY}(\lambda)  \leq \frac{\kappa_2}{\lambda_1\lambda_2}|\Omega(M)|N^2.
\end{align*}
Similarly, we can compute 
\begin{align*}
\Mc_{XX'}(\lambda),\Mc_{YY'}(\lambda)\leq  \frac{\kappa_2}{\lambda_1\lambda_2}|\Omega(M)|N^2.
\end{align*}
We have the following inequality 
\begin{align}\label{Appendix::ineq::upperboundsquaremoment}
    \mathfrak{M}(\lambda):=\max\left\{\Mc_{XX'}(\lambda),\Mc_{YY'}(\lambda),\Mc_{XY}(\lambda)\right\} \leq  \frac{\kappa_2}{\lambda_1\lambda_2}|\Omega(M)|N^2.
\end{align} 
Therefore, we have that 
\begin{align*}
    \sigma^2_{\lambda,2,2} \leq  16\frac{\kappa_2}{\lambda_1\lambda_2}|\Omega(M)|N^2.
\end{align*} 
This completes the proof of Lemma \ref{Appendix::variance lemma}.
\end{proof}
Letting $C_{\mathrm{Var}}(M,N,w,k) := \max \left\{4c_0^{\dagger}N,16c_0^{\dagger}|\Omega(M)|N^2\kappa_2\right\}$ and putting together the result of Lemma \ref{Appendix::variance lemma} with  \eqref{Appendix::intermediate step variance upper bound}, we obtain 
\begin{align*}
    \V_{P\times Q}\left(U_{n,m,\lambda}\right)\leq C_{\mathrm{Var}}(M,N,w,k)\left(\frac{\|(w\cdot\psi)*\varphi_\lambda\|^2_2}{n+m}+\frac{1}{(n+m)^2\lambda_1\lambda_2}\right). 
\end{align*}
This completes the proof of Lemma \ref{Variance upper bound}.
\end{proof}
We now state the relation between $\V(U_{n,m,\lambda})$ and $\hat{\sigma}^2_{n,m}$.
\begin{lem}\label{Lem::relation of VarU and Ehatsigma}
Let $(P,Q)\in\Pc^{\otimes 2}$, and let
\begin{align*}
    &\delta_\lambda=
    \mu_p-\mu_q,\\
    &\Sigma_{P}
    =
    \Eb_P
    \left[
        \bigl(\Phi_\lambda(X_1)-\mu_p\bigr)
        \otimes
        \bigl(\Phi_\lambda(X_1)-\mu_p\bigr)
    \right],\\
    &\Sigma_{Q}
    =
    \Eb_Q
    \left[
        \bigl(\Phi_\lambda(Y_1)-\mu_q\bigr)
        \otimes
        \bigl(\Phi_\lambda(Y_1)-\mu_q\bigr)
    \right].
\end{align*}
Then, we have 
\begin{align*}
    \V_{P,Q}\bigl(U_{n,m,\lambda}\bigr)
    =
    \Eb_{P,Q}
    \left[
        \widehat{\sigma}_{n,m}^{2}
    \right]
    +
    \frac{4}{n}
    \left\langle
        \delta_\lambda,
        \Sigma_{P}\delta_\lambda
    \right\rangle_{\Hc_{k_{w,\lambda}}}+
    \frac{4}{m}
    \left\langle
        \delta_\lambda,
        \Sigma_{Q}\delta_\lambda
    \right\rangle_{\Hc_{k_{w,\lambda}}}.
\end{align*}
In particular, it holds that
\begin{align*}
    \Eb_{P,Q}
    \left[
        \widehat{\sigma}_{n,m}^{2}
    \right]
    \leq
    \V_{P,Q}\bigl(U_{n,m,\lambda}\bigr).
\end{align*}
In particular, equality holds under the null hypothesis $p=q$.
\end{lem}

\begin{proof}
For notational simplicity, all inner products and norms below are taken in
$\Hc_{k_{w,\lambda}}$. Let
\begin{align*}
    A_i
    &:=
    \Phi_\lambda(X_i)-\mu_p,
    &
    B_j
    &:=
    \Phi_\lambda(Y_j)-\mu_q.
\end{align*}
Then, $\Eb_P[A_i]=0$ and $\Eb_Q[B_j]=0$. A direct expansion of
$U_{n,m,\lambda}$ gives
\begin{align*}
    U_{n,m,\lambda}
    &=
    \|\delta_\lambda\|^2
    +
    L_{n,m,\lambda}
    +
    R_{n,m,\lambda},
\end{align*}
where
\begin{align*}
    L_{n,m,\lambda}
    &:=
    \frac{2}{n}
    \sum_{i=1}^{n}
    \langle A_i,\delta_\lambda\rangle
    -
    \frac{2}{m}
    \sum_{j=1}^{m}
    \langle B_j,\delta_\lambda\rangle
\end{align*}
and
\begin{align*}
    R_{n,m,\lambda}
    &:=
    \frac{1}{n(n-1)}
    \sum_{i\neq i'}
    \langle A_i,A_{i'}\rangle
    +
    \frac{1}{m(m-1)}
    \sum_{j\neq j'}
    \langle B_j,B_{j'}\rangle\\
    &\quad-
    \frac{2}{nm}
    \sum_{i=1}^{n}
    \sum_{j=1}^{m}
    \langle A_i,B_j\rangle.
\end{align*}
Since the variables $A_i$ and $B_j$ are centered and mutually independent, $R_{n,m,\lambda}$ is  orthogonal in
$L^2$ to $L_{n,m,\lambda}$. Hence,
\begin{align*}
    \V_{P,Q}\bigl(U_{n,m,\lambda}\bigr)
    &=
    \V_{P,Q}\bigl(L_{n,m,\lambda}\bigr)
    +
    \V_{P,Q}\bigl(R_{n,m,\lambda}\bigr).
\end{align*}
Indeed, write $L_{n,m,\lambda}
    =
    L_{P,\lambda}
    +
    L_{Q,\lambda},
    R_{n,m,\lambda}
    =
    R_{P,\lambda}
    +
    R_{Q,\lambda}
    +
    R_{PQ,\lambda},
$
where
\begin{align*}
    L_{P,\lambda}
    :=
    \frac{2}{n}
    \sum_{i=1}^{n}
    \langle A_i,\delta_\lambda\rangle,\quad
    L_{Q,\lambda}
    :=
    -
    \frac{2}{m}
    \sum_{j=1}^{m}
    \langle B_j,\delta_\lambda\rangle,
\end{align*}
and
\begin{align*}
    R_{P,\lambda}
    :=
    \frac{1}{n(n-1)}
    \sum_{i\neq i'}
    \langle A_i,A_{i'}\rangle,\quad
    R_{Q,\lambda}
    :=
    \frac{1}{m(m-1)}
    \sum_{j\neq j'}
    \langle B_j,B_{j'}\rangle,\quad
    R_{PQ,\lambda}
    :=
    -
    \frac{2}{nm}
    \sum_{i=1}^{n}
    \sum_{j=1}^{m}
    \langle A_i,B_j\rangle.
\end{align*}
We verify the orthogonality term by term. For any $i,r,s$ with
$r\neq s$,
\begin{align*}
    \Eb\left[
        \langle A_i,\delta_\lambda\rangle
        \langle A_r,A_s\rangle
    \right]
    &=
    0.
\end{align*}
Indeed, if $i\notin\{r,s\}$, the result follows from independence and
centeredness. If $i=r$, then conditioning on $A_i$ yields
\begin{align*}
    \Eb\left[
        \langle A_i,\delta_\lambda\rangle
        \langle A_i,A_s\rangle
        \,\middle|\,
        A_i
    \right]
    =
    \langle A_i,\delta_\lambda\rangle
    \left\langle
        A_i,\Eb[A_s]
    \right\rangle
    =
    0,
\end{align*}
and the case $i=s$ is identical. Consequently, we have
\begin{align*}
    \Eb\left[
        L_{P,\lambda}R_{P,\lambda}
    \right]
    &=
    0.
\end{align*}
Moreover, for any $i,r,j$, independence of the two samples yields
\begin{align*}
    \Eb\left[
        \langle A_i,\delta_\lambda\rangle
        \langle A_r,B_j\rangle
    \right]
    =
    \Eb\left[
        \langle A_i,\delta_\lambda\rangle
        \left\langle
            A_r,\Eb[B_j\mid A_i,A_r]
        \right\rangle
    \right]
    =
    0.
\end{align*}
Hence, we also have
    $\Eb\left[ L_{P,\lambda}R_{PQ,\lambda}\right]=0$.
By the same argument, we have
\begin{align*}
    \Eb\left[
        L_{P,\lambda}R_{Q,\lambda}
    \right]=\Eb\left[
        L_{Q,\lambda}R_{Q,\lambda}
    \right]
    =
    \Eb\left[
        L_{Q,\lambda}R_{PQ,\lambda}
    \right]
    =
    \Eb\left[
        L_{Q,\lambda}R_{P,\lambda}
    \right]
    =
    0.
\end{align*}
Therefore,
\begin{align*}
    \operatorname{Cov}_{P,Q}
    \left(
        L_{n,m,\lambda},
        R_{n,m,\lambda}
    \right)
    &=
    \Eb\left[
        L_{n,m,\lambda}R_{n,m,\lambda}
    \right]
    =
    0,
\end{align*}
which proves the asserted orthogonality.

Independence of the two samples yields
\begin{align*}
    \V_{P,Q}\bigl(L_{n,m,\lambda}\bigr)
    &=
    \frac{4}{n}
    \Eb_P
    \left[
        \langle A_1,\delta_\lambda\rangle^2
    \right]
    +
    \frac{4}{m}
    \Eb_Q
    \left[
        \langle B_1,\delta_\lambda\rangle^2
    \right]\\
    &=
    \frac{4}{n}
    \left\langle
        \delta_\lambda,
        \Sigma_{P}\delta_\lambda
    \right\rangle
    +
    \frac{4}{m}
    \left\langle
        \delta_\lambda,
        \Sigma_{Q}\delta_\lambda
    \right\rangle.
\end{align*}
Note that, for 
$a_{P}
    =
    \tr\bigl(\Sigma_{P}^2\bigr)$,
    $a_{Q}
    =
    \tr\bigl(\Sigma_{Q}^2\bigr)$,
    $
    a_{PQ}
    =
\tr\bigl(\Sigma_{P}\Sigma_{Q}\bigr)$,
it holds that
\begin{align*}
    \V_{P,Q}\bigl(R_{n,m,\lambda}\bigr)
    =
    \frac{2a_{P}}{n(n-1)}
    +
    \frac{2a_{Q}}{m(m-1)}
    +
    \frac{4a_{PQ}}{nm}.
\end{align*}
It follows that
\begin{align*}
    \Eb_{P,Q}
    \left[
        \widehat{\sigma}_{n,m}^{2}
    \right]
    &=
    \frac{2a_{P}}{n(n-1)}
    +
    \frac{2a_{Q}}{m(m-1)}
    +
    \frac{4a_{PQ}}{nm}\\
    &=
    \V_{P,Q}\bigl(R_{n,m,\lambda}\bigr).
\end{align*}
Combining the preceding identities yields
\begin{align*}
    \V_{P,Q}\bigl(U_{n,m,\lambda}\bigr)
    =
    \Eb_{P,Q}
    \left[
        \widehat{\sigma}_{n,m}^{2}
    \right]+\frac{4}{n}
    \left\langle
        \delta_\lambda,
        \Sigma_{P}\delta_\lambda
    \right\rangle+\frac{4}{m}
    \left\langle
        \delta_\lambda,
        \Sigma_{Q}\delta_\lambda
    \right\rangle.
\end{align*}
Lastly, since $ \left\langle
        \delta_\lambda,
        \Sigma_{P}\delta_\lambda
    \right\rangle
    =
    \Eb_P
    \left[
        \langle A_1,\delta_\lambda\rangle^2
    \right]
    \geq 0$,
we have
\begin{align*}
    \Eb_{P,Q}
    \left[
        \widehat{\sigma}_{n,m}^{2}
    \right]
    \leq
    \V_{P,Q}\bigl(U_{n,m,\lambda}\bigr).
\end{align*}
Under $p=q$, we have $\delta_\lambda=0$, so equality holds. This completes the proof of Lemma \ref{Lem::relation of VarU and Ehatsigma}.
\end{proof}

\begin{lem}\label{lem::variance estimator zero prob conv} 
Suppose that $\lambda_1,\lambda_2\to0$ and $(n+m)^2\nu\to\infty$, as $n+m\to \infty$, where 
$\nu=\lambda_1\lambda_2$.
The estimator $\hat{\sigma}^{2}_{n,m}$ satisfies 
\begin{align}\label{Eq::nonzero-variance-estimator}
    \sup_{(P,Q)\in \Pc^{\otimes 2}}
    \Pb_{P,Q}\bigl(\hat{\sigma}_{n,m}=0\bigr)
    \longrightarrow 0,
    \qquad \text{as } n+m\to\infty.
\end{align}
\begin{proof}[Proof of Lemma \ref{lem::variance estimator zero prob conv}]

For $(P,Q)\in\Pc^{\otimes 2}$, define
\begin{align*}
    \widetilde{\sigma}_{n,m}^2
    &:=
    \Eb_{P,Q}\left[\hat{\sigma}_{n,m}^2\right].
\end{align*}
Then, we can write
\begin{align}
    \widetilde{\sigma}_{n,m}^2
    =
    \frac{2a_P}{n(n-1)}
    +
    \frac{2a_Q}{m(m-1)}
    +
    \frac{4a_{PQ}}{nm}.
    \label{Eq::mean-variance-estimator-alternative}
\end{align}
We first establish a uniform lower bound for this quantity. Since
$\Sigma_P$ and $\Sigma_Q$ are positive operators,
\begin{align*}
    a_{PQ}
    &=
    \tr(\Sigma_P\Sigma_Q)
    =
    \left\|
        \Sigma_P^{1/2}\Sigma_Q^{1/2}
    \right\|_{\mathrm{HS}}^2
    \geq 0.
\end{align*}
Moreover, \ref{M3:asymp condition of trace} in the proof of Theorem \ref{Theorem::uniform-validity} can be applied to the pair $(P,P)$.
Indeed, for this pair,
\begin{align*}
    \Gamma_\eta=(1-\eta)\Sigma_P+\eta\Sigma_P=\Sigma_P.
\end{align*}
Consequently, uniformly over $P\in\Pc$,
\begin{align*}
    a_P=\tr(\Sigma_P^2)\geq
    c_{\mathrm{M3}}\nu^{-1}
\end{align*}
for all sufficiently large $n$ and $m$. The same bound also holds for $a_Q$.
It follows from \eqref{Eq::mean-variance-estimator-alternative} that
\begin{align*}
    \widetilde{\sigma}_{n,m}^2
    &\geq
    \frac{2c_{\mathrm{M3}}}{\nu}
    \left\{\frac{1}{n(n-1)}+\frac{1}{m(m-1)}
    \right\} \\
    &\geq
    \frac{c_{\sigma}}{(n+m)^2\nu}
\end{align*}
for some constant $c_{\sigma}>0$ independent of $(P,Q)$ and all sufficiently large $n,m$.
Therefore, we have that, for all sufficiently large $n,m$,
\begin{align}
    \left(\widetilde{\sigma}_{n,m}^2\right)^2
    \geq
    \frac{c_{\sigma}^2}{(n+m)^4\nu^2}.
    \label{Eq::lower-mean-variance-estimator}
\end{align}

The proofs of \ref{M1:uniform bound of sum of operator norms}, \ref{M4:uniform bound of sum of expectations}, and \ref{M5:uniform bound of fourth moment}, as well as the variance calculations for
$\hat a_P$, $\hat a_Q$, and $\hat a_{PQ}$ in the proof of
Theorem~\ref{Theorem::uniform-validity} do not use the null condition $p=q$. 
Therefore, the same calculations yield
\begin{align} \label{Eq::variance-estimator-mse-alternative}
\sup_{(P,Q)\in\Pc^{\otimes 2}}
\Eb_{P,Q}\left[\left(\hat{\sigma}_{n,m}^2-\widetilde{\sigma}_{n,m}^2
        \right)^2\right]\leq C_{\sigma}\left\{
        \frac{1}{(n+m)^5\nu^2}+\frac{1}{(n+m)^6\nu^3}\right\}
\end{align}
for some constant $C_{\sigma}>0$ independent of $(P,Q)$.

Combining
\eqref{Eq::lower-mean-variance-estimator} and
\eqref{Eq::variance-estimator-mse-alternative} yields
\begin{align*}
\sup_{(P,Q)\in\Pc^{\otimes 2}}
    \Eb_{P,Q}\left[ \left(\frac{\hat{\sigma}_{n,m}^2}{\widetilde{\sigma}_{n,m}^2}-1\right)^2
    \right]\leq
    C\left\{\frac{1}{n+m}+\frac{1}{(n+m)^2\nu}\right\}
    \to 0,\quad \text{as}\quad n+m\to  \infty.
\end{align*}

Finally, since $\widetilde{\sigma}_{n,m}^2>0$ for all sufficiently
large $n$ and $m$, we have 
\begin{align*}
    \left\{\hat{\sigma}_{n,m}=0\right\}=\left\{\hat{\sigma}_{n,m}^2=0\right\} \subseteq
    \left\{\left|\frac{\hat{\sigma}_{n,m}^2}
{\widetilde{\sigma}_{n,m}^2}-1\right|\geq 1\right\}.
\end{align*}
By Markov's inequality, it holds that
\begin{align*}
    \sup_{(P,Q)\in\Pc^{\otimes 2}}\Pb_{P,Q}\bigl(\hat{\sigma}_{n,m}=0\bigr)
    &\leq\sup_{(P,Q)\in\Pc^{\otimes 2}}
    \Eb_{P,Q}
\left[\left(\frac{\hat{\sigma}_{n,m}^2}{\widetilde{\sigma}_{n,m}^2} -1\right)^2\right] \\
    &\leq C\left\{\frac{1}{n+m}+\frac{1}{(n+m)^2\nu}\right\}
    \to 0.
\end{align*}
This completes the proof of Lemma \ref{lem::variance estimator zero prob conv}
\end{proof}
\end{lem}

\begin{proof}[Proof of Theorem \ref{thm::upperbound_unifseparation_rate}]
We are now ready to prove the theorem. 
Let $(P,Q)\in \Pc^{\otimes 2}_1$ such that, for their persistence intensity functions $p$ and $q$, respectively, 
\begin{align*}
    w\cdot (p-q) \in \Sc_2^{s_1,s_2}(R).
\end{align*}
For notational simplicity, let
\begin{align*}
    f:=w(p-q),\quad \rho^2:= \|f\|_{L^2}^{2},\quad  \theta_\lambda := \|\mu_p-\mu_q \|_{\Hc_{k_{w,\lambda}}}^2.
\end{align*}
Then, we can write 
\begin{align*}
    \theta_\lambda = \langle f,f*\varphi_{\lambda}\rangle _{L^2}=\frac{1}{2}\left(\rho^2+\|f*\varphi_\lambda\|^2_{L^2}-\|f-f*\varphi_\lambda\|_{L^2}^2\right)\geq \frac{1}{2}\left(\rho^2 -\|f-f*\varphi_\lambda\|_{L^2}^2\right).
\end{align*}
Indeed, 
\begin{align*}
     \|\mu_p-\mu_q \|_{\Hc_{k_{w,\lambda}}}^2&=\left\langle\int_{\Rb^2} (p-q)(x)w(x)w(\cdot)k_{\lambda}(\cdot,x)dx,\int_{\Rb^2} (p-q)(x')w(x')w(\cdot)k_{\lambda}(\cdot,x')dx'\right\rangle_{\Hc_{k_{w,\lambda}}}\\
     &=\int_{\Rb^2}\int_{\Rb^2} k_\lambda(x,x')f(x)f(x')dxdx'\\
     &=\int_{\Rb^2}f(x)\int_{\Rb^2} \varphi_\lambda(x-x')f(x')dx
     'dx\\
     &= \int_{\Rb^2}f(x)(\varphi_\lambda * f)(x)dx\\
     &= \langle f,f*\varphi_{\lambda}\rangle _{L^2}\\
     &=\frac{1}{2}\left(\|f\|_{L^2}^2+\|f*\varphi_\lambda\|^2_{L^2}-\|f-f*\varphi_\lambda\|_{L^2}^2\right).
\end{align*}
The non-rejection region of $\Delta_{\alpha,\lambda}$ is $\{T_{n,m,\lambda}\leq z_{1-\alpha}\}$.
Then, we can write 
\begin{align*}
    \Pb\left(T_{n,m,\lambda}\leq z_{1-\alpha}\right)
    &\leq\Pb(\hat{\sigma}_{n,m}=0)+\Pb(U_{n,m,\lambda}\leq z_{1-\alpha}\hat{\sigma}_{n,m})\\
    &\leq \Pb(\hat{\sigma}_{n,m}=0)+\Pb\left(z_{1-\alpha}\hat{\sigma }_{n,m}
    \geq \frac{\theta_\lambda}{2}\right)+\Pb \left(U_{n,m,\lambda}\leq \frac{\theta_\lambda}{2}\right),
\end{align*}
since \begin{align*}
    \left\{U_{n,m,\lambda}\leq z_{1-\alpha}\hat{\sigma}_{n,m}\right\}\subset \left\{z_{1-\alpha}\hat{\sigma} _{n,m}\geq \frac{\theta_\lambda}{2}\right\} \cup \left\{U_{n,m,\lambda}\leq \frac{\theta_\lambda}{2}\right\}.
\end{align*}
We first consider the last two terms $\Pb\left(z_{1-\alpha}\hat{\sigma }_{n,m}
    \geq \frac{\theta_\lambda}{2}\right)$, $\Pb \left(U_{n,m,\lambda}\leq \frac{\theta_\lambda}{2}\right)$.
By Chebyshev inequality, we have
\begin{align*}
    \Pb\left(U_{n,m,\lambda}\leq \frac{\theta_\lambda}{2}\right)\leq \Pb\left(|U_{n,m,\lambda}-\theta_\lambda|\geq \frac{\theta_\lambda}{2}\right)\leq \frac{4\V(U_{n,m,\lambda})}{\theta_\lambda^2}
\end{align*}
and
\begin{align*}
    \Pb\left(z_{1-\alpha}\hat{\sigma}_{n,m}\geq \frac{\theta_\lambda}{2}\right)\leq \frac{4z_{1-\alpha}^2 \V(U_{n,m,\lambda})}{\theta_\lambda^2},
\end{align*}
since $\V(U_{n,m,\lambda})\geq \Eb\left[\hat{\sigma}^2_{n,m}\right]$, for any $(P,Q)\in \Pc^{\otimes 2}$, by Lemma \ref{Lem::relation of VarU and Ehatsigma}. 
Therefore, we have 
\begin{align*}
    \Pb\left(T_{n,m,\lambda}\leq z_{1-\alpha}\right)\leq 4(1+z_{1-\alpha}^2)\frac{\V(U_{n,m,\lambda})}{\theta_\lambda^2}+\Pb(\hat{\sigma}_{n,m}=0).
\end{align*}
By Lemma \ref{Variance upper bound}, we obtain 
\begin{align}
     \Pb\left(T_{n,m,\lambda}\leq z_{1-\alpha}\right)\leq 4(1+z_{1-\alpha}^2)C_{\mathrm{Var}}\left\{\frac{\|f*\varphi_\lambda\|_2^2}{(n+m)\theta_\lambda^2}+\frac{1}{(n+m)^2\lambda_1\lambda_2\theta_\lambda^2}\right\}+\Pb(\hat{\sigma}_{n,m}=0). 
\end{align}
Note that, since $\varphi_\lambda$ is positive definite and $\|\varphi_\lambda\|_{L^1}=1$, we have $0\leq \widehat{\varphi_{\lambda}}\leq 1$. This implies $\|f*\varphi_\lambda \|^2_2 \leq \langle f,f*\varphi\rangle_{L^2}=\theta_\lambda$. 
Indeed, by the Plancherel's identity,
\begin{align*}
    \|f*\varphi_\lambda\|_2^2 =\frac{1}{(2\pi)^2}\int_{\Rb^2} |\hat{f}(\xi)|^2|\widehat{\varphi_\lambda}(\xi)|^2 d\xi\leq \frac{1}{(2\pi)^2}\int_{\Rb^2}|\hat{f}(\xi)|^2 |\widehat{\varphi}_\lambda(\xi)|d\xi=\theta_\lambda.
\end{align*}
Therefore, we have 
\begin{align}\label{nonrejection prob upper bound}
     \Pb\left(T_{n,m,\lambda}\leq z_{1-\alpha}\right)\leq 4(1+z_{1-\alpha}^2)C_{\mathrm{Var}}\left\{\frac{1}{(n+m)\theta_\lambda}+\frac{1}{(n+m)^2\lambda_1\lambda_2\theta_\lambda^2}\right\}+\Pb(\hat{\sigma}_{n,m}=0). 
\end{align}
Let $A_{\alpha,\beta}>0$ be a positive constant satisfying 
\begin{align*}
    A_{\alpha,\beta}\geq 
    \frac{4}{\beta}\max\left\{4(1+z_{1-\alpha}^2)C_{\mathrm{Var}},\sqrt{4(1+z_{1-\alpha}^2)C_{\mathrm{Var}}} \right\}.
\end{align*}
Suppose that 
\begin{align*}
    \theta_\lambda \geq \frac{A_{\alpha,\beta}}{(n+m)\sqrt{\lambda_1\lambda_2}},
\end{align*}
for sufficiently large $n,m$ so that $\sqrt{\lambda_1\lambda_2}\leq 1$.
Since $\sqrt{\lambda_1\lambda_2}\leq 1$, we then have
\begin{align*}
    \frac{1}{(n+m)\theta_\lambda}\leq \frac{1}{A_{\alpha,\beta}},\quad \frac{1}{(n+m)^2\lambda_1\lambda_2\theta_\lambda^2}\leq \frac{1}{A_{\alpha,\beta}^2}.
\end{align*}
Using these inequalities and the definition of $A_{\alpha,\beta}$, we obtain 
\begin{align*}
    \Pb\left(T_{n,m,\lambda}\leq z_{1-\alpha}\right)\leq \frac{\beta}{2}+\Pb(\hat{\sigma}_{n,m}=0).
\end{align*}
By Lemma \ref{lem::variance estimator zero prob conv},
we have that, for all sufficiently large $n,m$, 
\begin{align*}
    \sup_{(P,Q)\in \Pc_1^{\otimes 2}} \Pb(\hat{\sigma}_{n,m}=0) \leq \frac{\beta}{2}.
\end{align*}
Therefore, the condition 
\begin{align*}
    \theta_\lambda \geq \frac{A_{\alpha,\beta}}{(n+m)\sqrt{\lambda_1\lambda_2}}
\end{align*}
is a sufficient condition for the Type II error to be less than $\beta$, for all sufficiently large $n,m$.

We now express the sufficient condition in terms of $\|f\|_{L^2}^2=\rho^2$. Recall that 
\begin{align*}
    \theta_\lambda \geq \frac{1}{2}(\rho^2 -\|f-f*\varphi_\lambda\|_{L^2}^2).
\end{align*}
By the following lemma, we can write $\|f-f*\varphi_\lambda\|_{L^2}^2$ in terms of $\|f\|_{L^2}$.
\begin{lem}[Anisotropic bias bound]\label{Appendix::lem::Anisotropic bias bound} Let $s_1,s_2>0$ and $R>0$. For $t>0$, define 
\begin{align*}
    S(t):= \sup_{|u_1|\leq t,|u_2|\leq t} \left|1-\widehat{k_1}(u_1)\widehat{k_2}(u_2)\right|.
\end{align*}
Since, $S(t)\rightarrow 0$ as $t \downarrow 0$, there exists $t_0>0$ such that $S(t_0)<1$.
For every $f\in \Sc_2^{s_1,s_2}(R)$ and every $\lambda=(\lambda_1,\lambda_2)\in (0,\infty)^2$, we have that
\begin{align}
    \|f-f*\varphi_\lambda\|_2^2 \leq S(t_0)^2\|f\|_2^2 +4R^2(t_0^{-2s_1}\lambda_1^{2s_1}+t_0^{-2s_2}\lambda_2^{2s_2}).
\end{align} 
\end{lem}
\begin{proof}
    Fix $f\in \Sc_2^{s_1,s_2}(R)$ and $\lambda=(\lambda_1,\lambda_2)\in (0,\infty)^2$. Define a measurable subset $A_{\lambda}(t_0)$ of $\Rb^2$ by
    \begin{align*}
        A_{\lambda}(t_0):=\left\{\xi \in \Rb^2 : |\lambda_i\xi_i|\leq t_0,\ i=1,2. \right\}.
    \end{align*}
By Plancherel's identity, 
\begin{align*}
    (2\pi)^2 \|f-f*\varphi_\lambda\|^2_2 &=\int_{\Rb^2}\left|1-\widehat{\varphi_\lambda}(\xi)\right|^2 \left|\widehat{f}(\xi)\right|^2d\xi\\
    &=\int_{\Rb^2}\left|1-\widehat{k_1}(\lambda_1\xi_1)\widehat{k_2}(\lambda_2\xi_2)\right|^2 \left|\widehat{f}(\xi)\right|^2d\xi.
\end{align*}
By splitting the last integral over $A_{\lambda}(t_0)$ and its complement, we obtain that 
\begin{align*}
    (2\pi)^2 \|f-f*\varphi_\lambda\|^2_2 &\leq S(t_0)^2 \int_{A_{\lambda}(t_0)} \left|\widehat{f}(\xi)\right|^2d\xi +\int_{A_{\lambda}(t_0)^c}\left|1-\widehat{k_1}(\lambda_1\xi_1)\widehat{k_2}(\lambda_2\xi_2)\right|^2 \left|\widehat{f}(\xi)\right|^2d\xi\\
    &\leq S(t_0)^2 (2\pi)^2 \left\|f\right\|^2_2 +4\int_{A_{\lambda}(t_0)^c} \left|\widehat{f}(\xi)\right|^2d\xi.
\end{align*}
In the last inequality, we use $|\widehat{k_i}(u)|\leq \|k_i\|_1=1$ and the Plancherel's identity.
Let us next consider the integral $4\int_{A_{\lambda}(t_0)^c} \left|\widehat{f}(\xi)\right|^2d\xi$ in the last term. If $\xi \in A_\lambda(t_0)^c$, then at least one of the inequalities $|\lambda_1\xi_1| >t_0 $ or $|\lambda_2\xi_2|>t_0$ holds. Therefore, 
\begin{align*}
    4\int_{A_{\lambda}(t_0)^c} \left|\widehat{f}(\xi)\right|^2d\xi &\leq 4t_0^{-2s_1}\lambda_1^{2s_1}\int_{\Rb^2}|\xi_1|^{2s_1}\left|\widehat{f}(\xi)\right|^2d\xi 
    +4t_0^{-2s_2}\lambda_2^{2s_2}\int_{\Rb^2}|\xi_2|^{2s_2}\left|\widehat{f}(\xi)\right|^2d\xi\\
    &\leq 4(2\pi)^2 R^2\left(t_0^{-2s_1}\lambda_1^{2s_1}+t_0^{-2s_2}\lambda_2^{2s_2}\right).
\end{align*}
In the last inequality, we use the definition of the anisotropic Sobolev ball. Consequently, we obtain that
\begin{align*}
    \|f-f*\varphi_\lambda\|^2_2 &\leq S(t_0)^2  \left\|f\right\|^2_2 + 4R^2\left(t_0^{-2s_1}\lambda_1^{2s_1}+t_0^{-2s_2}\lambda_2^{2s_2}\right).
\end{align*}
This completes the proof of Lemma \ref{Appendix::lem::Anisotropic bias bound}. 
\end{proof}
By Lemma \ref{Appendix::lem::Anisotropic bias bound}, we obtain 
\begin{align*}
    \theta_\lambda \geq \frac{1}{2}(\rho^2 -\|f-f*\varphi_\lambda\|_{L^2}^2)\geq \frac{1}{2}\left(\left(1-S(t_0)^2\right)\rho^2 - C(R,s_1,s_2,k)\left(\lambda_1^{2s_1}+\lambda_2^{2s_2}\right)\right),
\end{align*}
where $C(R,s_1,s_2,k):= 4R^2\max (t_0^{-2s_1},t_0^{-2s_2}).$ Therefore, if 
\begin{align*}
    \rho^2 \geq C_{\mathrm{sep}}(M,N,w,k,s_1,s_2,R,\alpha,\beta)\left(\lambda_1^{2s_1}+\lambda_2^{2s_2}+\frac{1}{(n+m)\sqrt{\lambda_1\lambda_2}}\right),
\end{align*}
then the Type II error is less than $\beta$, for all sufficiently large $n,m$ and
where 
\begin{align*}
    C_{\mathrm{sep}}(M,N,w,k,s_1,s_2,R,\alpha,\beta)= \frac{1}{1-S(t_0)^2}\max\left\{2A_{\alpha,\beta},C(R,s_1,s_2,k)\right\}.
\end{align*}
Since the constant $C_{\mathrm{sep}}$ does not depend on $(P,Q)$, this implies 
\begin{align*}
    \rho^2_{n,m}&\left(\Delta_{\alpha,\lambda},\Sc_2^{s_1,s_2}(R),\beta,w,M,N,L\right)\leq\\ &C_{\mathrm{sep}}(M,N,w,k,s_1,s_2,R,\alpha,\beta) \left(\lambda_1^{2s_1}+\lambda_2^{2s_2}+\frac{1}{(n+m)\sqrt{\lambda_1\lambda_2}}\right).
\end{align*}
Lastly, for the chosen bandwidths, we have 
\begin{align*}
    (\lambda_i^{\star})^{2s_i}=(n+m)^{-\tau},\quad i=1,2,
\end{align*}
and 
\begin{align*}
    \frac{1}{(n+m)\sqrt{\lambda_1^{\star}\lambda_2^{\star}}}=(n+m)^{-\tau}.
\end{align*}
Consequently, we have that, for all sufficiently large $n,m$, 
\begin{align*}
    \rho_{n,m}\left(\Delta_{\alpha,\lambda^{\star}},\Sc_2^{s_1,s_2}(R),\beta,w,M,N,L\right)&\leq C^{\star}(M,N,w,k,s_1,s_2,R,\alpha,\beta) (n+m)^{-\tau/2}\\
    &=C^{\star}(M,N,w,k,s_1,s_2,R,\alpha,\beta) (n+m)^{-\bar{s}/(2\bar{s}+1)},
\end{align*}
where
$C^{\star}(M,N,w,k,s_1,s_2,R,\alpha,\beta):= \sqrt{3C_{\mathrm{sep}}(M,N,w,k,s_1,s_2,R,\alpha,\beta)}$. The proof of Theorem \ref{thm::upperbound_unifseparation_rate} is completed. 
\end{proof}
\subsection{Proof of Theorem \ref{minimax lower bound}}\label{Appendix::Proof of Minimax lower bound} 
In this proof, we use the notation $X_i$ for samples in the two-dimensional Euclidean space $\mathbb{R}^2$ and $D_i$ for samples of persistence diagrams. For a probability distribution $P$, denote by $P^{\otimes n}$ the $n$-fold product measure of $P$.

The proof is based on the following lemma, which is a modified version of \citet[Lemma 5]{AlbertLaurentMarrelMeynaoui2022}. Its proof follows the similar arguments as in \citet[Lemma 5]{AlbertLaurentMarrelMeynaoui2022}.
\begin{lem}\label{Appendix::lem::base of minimax lowerbound} Let $\alpha$, $\beta$, and $\gamma$ be in $(0,1)$ such that $\alpha +\beta+\gamma <1$.
Let $(\rho^*_{n,m})$ be a sequence of positive real numbers. For each $n,m$, let $\kappa_{n,m}$ be a probability measure on $\Pc\times \Pc$ such that $\kappa_{n,m}\left(\Pc^{\otimes 2}_{\mathcal{F}_{\rho^{*}_{n,m}}(\mathcal{C},w)}\right)\geq 1-\gamma$. Define the associated probability measure $\Pb_{\kappa_{n,m}}$ by 
\begin{align*}
    \Pb_{\kappa_{n,m}}(\Ac) = \int_{\Pc\times\Pc}P^{\otimes n}\otimes Q^{\otimes m}(\Ac)d\kappa_{n,m}(P,Q),
\end{align*}
for any measurable set $\Ac$ in $\mathbf{PD}^n\times \mathbf{PD}^m$. Let $\mathbb{P}_{0,n,m}$ be a probability measure  on $\mathbf{PD}^n \times \mathbf{PD}^m$ such that
$\mathbb{P}_{0,n,m} = P^{\otimes n}\otimes Q^{\otimes m}$,
for a pair of probability measures $(P,Q)\in \Pc^{\otimes 2}_0$. Assume that $\mathbb{P}_{\kappa_{n,m}}$ is absolutely continuous with respect to $\mathbb{P}_{0,n,m}$, and that
\begin{align*}
\limsup_{n,m\rightarrow \infty}\Eb_{\mathbb{P}_{0,n,m}}\left[L^2_{\kappa_{n,m}}(\Db_{n+m})\right] 
< 1 + 4(1-\alpha- \beta - \gamma)^2,
\end{align*}
where $L_{\kappa_{n,m}} := d\mathbb{P}_{\kappa_{n,m}}/d\mathbb{P}_{0,n,m}$ is the likelihood ratio and $\Db_{n+m}\sim \mathbb{P}_{0,n,m}$. Then, the minimax separation rate satisfies
\begin{align*}
    \rho^{\dagger}_{n,m}(\Cc,\alpha,\beta,w,M,N,L) \geq c_0 \rho^*_{n,m},
\end{align*}
for all sufficiently large $n,m$, and some constant $c_0=c_0(\Cc,\alpha,\beta,w,M,N,L)>0$.
\end{lem}
\begin{proof}[Proof of Lemma \ref{Appendix::lem::base of minimax lowerbound}]
Fix an arbitrary sequence of tests $\mathbf{\Delta}=(\Delta_{n,m})\in \Psi_{\alpha}$. 
For each $n,m$, denote 
\begin{align*}
    \Ac_{n,m}:= \Pc^{\otimes 2}_{\Fc_{\rho^*_{n,m}}(\Cc,w)},
\end{align*}
and define the rejection event $E_{n,m}:= \left\{\Delta_{n,m}=1\right\}$.
Set 
\begin{align*}
    R_{n,m}(\mathbf{\Delta}):=\sup_{(P,Q)\in \Ac_{n,m}}\Pb_{P,Q}^{n,m}\left(E^c_{n,m}\right).
\end{align*}
Since $\kappa_{n,m}(\Ac_{n,m})\geq 1-\gamma$, we have 
\begin{align*}
    \Pb_{\kappa_{n,m}}(E^c_{n,m})&=\int_{\Pc\times \Pc}\Pb_{P,Q}^{n,m}\left(E^c_{n,m}\right) d\kappa_{n,m}(P,Q)\\
    &\leq \int_{\Ac_{n,m}}\Pb_{P,Q}^{n,m}\left(E^c_{n,m}\right) d\kappa_{n,m}(P,Q) +\gamma\\
    &\leq R_{n,m}(\mathbf{\Delta})+\gamma.
\end{align*}
It follows that 
\begin{align*}
    R_{n,m}(\mathbf{\Delta}) \geq \Pb_{\kappa_{n,m}}(E^c_{n,m})-\gamma.
\end{align*}
Since $\Pb_{0,n,m}$ is generated by a pair in the null class $\Pc_0^{\otimes 2}$, we have 
\begin{align*}
    \limsup_{n,m\to \infty}\Pb_{0,n,m}(E_{n,m})\leq  \limsup_{n,m\to \infty}\sup_{(P,Q)\in \Pc_0^{\otimes 2}} \Pb_{P,Q}^{n,m}\left(E_{n,m}\right) \leq \alpha.
\end{align*}
Moreover, it holds that
\begin{align*}
    \Pb_{\kappa_{n,m}}(E_{n,m})\leq \Pb_{0,n,m}(E_{n,m}) +\| \Pb_{\kappa_{n,m}}-\Pb_{0,n,m}\|_{\mathrm{TV}},
\end{align*}
where $\|\cdot \|_{\mathrm{TV}}$ denotes the total variation. 
Consequently, we have 
\begin{align*}
     \liminf_{n,m\to \infty}R_{n,m}(\mathbf{\Delta})\geq 1-\alpha -\gamma - \limsup_{n,m\to \infty}\| \Pb_{\kappa_{n,m}}-\Pb_{0,n,m}\|_{\mathrm{TV}}.
\end{align*}
By absolute continuity, we can write 
\begin{align*}
    \| \Pb_{\kappa_{n,m}}-\Pb_{0,n,m}\|_{\mathrm{TV}} &=\frac{1}{2}\Eb_{\Pb_{0,n,m}}\left[\left|L_{\kappa_{n,m}}(\Db_{n+m})-1 \right|\right]\\
    &\leq \frac{1}{2}\left(\Eb_{\Pb_{0,n,m}}\left[L^2_{\kappa_{n,m}}(\Db_{n+m})\right]-1 \right)^{1/2}.
\end{align*}
This implies 
\begin{align*}
    \liminf_{n,m\rightarrow \infty}R_{n,m}(\mathbf{\Delta}) \geq 1-\alpha -\gamma-\frac{1}{2}\left(\limsup_{n,m\rightarrow \infty}\Eb_{\Pb_{0,n,m}}\left[L^2_{\kappa_{n,m}}(\Db_{n+m})\right]-1 \right)^{1/2}. 
\end{align*}
By the assumed likelihood-ratio bound, we obtain 
\begin{align*}
    \liminf_{n,m\rightarrow \infty}R_{n,m}(\mathbf{\Delta})  >\beta.
\end{align*}
Equivalently, for all sufficiently large $n,m$, 
\begin{align*}
   \sup_{(P,Q)\in \Ac_{n,m}}\Pb_{P,Q}^{n,m}\left(E^c_{n,m}\right) >\beta.
\end{align*}
Since, $\Ac_{n,m}\subseteq \Pc^{\otimes 2}_{\Fc_\rho(\Cc,w)}$, for every $0<\rho\leq \rho_{n,m}^*$, it follows that 
\begin{align*}
    \rho_{n,m}(\Delta_{n,m})\geq \rho_{n,m}^*,
\end{align*}
for all sufficiently large $n,m$. Since $\mathbf{\Delta}\in \Psi_{\alpha}$ is arbitrary, 
\begin{align*}
    \inf_{\mathbf{\Delta}=(\Delta_{n,m})\in \Psi_{\alpha}} \liminf_{n,m\to \infty} \frac{\rho_{n,m}(\Delta_{n,m})}{\rho_{n,m}^*} \geq 1.
\end{align*}
By the definition of $\rho_{n,m}^\dagger$, there exists a constant $C\in(0,\infty)$ and $\mathbf{\Delta}^*=(\Delta_{n,m}^*)\in \Psi_{\alpha}$ such that 
\begin{align*}
    \limsup_{n,m\rightarrow \infty}\frac{\rho_{n,m}(\Delta_{n,m}^*)}{\rho_{n,m}^{\dagger}} \leq C.
\end{align*}
Consequently, there exists a constant $c_0>0$ such that, for all sufficiently large $n,m$,
\begin{align*}
    \rho_{n,m}^{\dagger }\geq c_0\rho_{n,m}^*.
\end{align*}
This completes the proof of Lemma \ref{Appendix::lem::base of minimax lowerbound}.
\end{proof}

By Lemma \ref{Appendix::lem::base of minimax lowerbound}, it is sufficient to construct the probability measures $\Pb_{\kappa_{n,m}}$ and $\mathbb{P}_{0,n,m}$ satisfying the conditions in Lemma \ref{Appendix::lem::base of minimax lowerbound} with adequately chosen $\rho^*_{n,m}$. We proceed with the construction in the following three steps.
\begin{enumerate}[
    label=\textbf{Step \arabic*.},
    leftmargin=*,
    align=left
]

\item 
We reconstruct the collection of probability density functions used in the proof of \citet[Theorem 4]{AlbertLaurentMarrelMeynaoui2022}. 

\item 
We construct probability measures on $\mathbf{PD}$ induced by the density functions introduced in {\bf{Step 1}}. 

\item 
We verify that the measures constructed in {\bf{Step 2}} satisfy the conditions of Lemma \ref{Appendix::lem::base of minimax lowerbound}.

\end{enumerate}

In {\bf Step 1}, we construct the collection of probability density functions on $\Omega(M)$. This construction is adapted from \cite{ingster1993asymptotically}. Define a continuous function $G:\Rb \rightarrow \Rb$ by 
\begin{align*}
    G(t)=\exp\left(\frac{-1}{1-(4t+3)^2}\right)\Ib_{(-1,-1/2)}(t) -\exp\left(\frac{-1}{1-(4t+1)^2}\right)\Ib_{(-1/2,0)}(t).
\end{align*} Note that the support of $G$ is contained in $(-1,0)$ and $\int_{\Rb}G(t)dt=0$. Let 
\begin{align*}
    \Rc(M)\coloneqq\left(\frac{M}{4},\frac{M}{2}\right)\times\left(\frac{3M}{4},M\right),    
\end{align*}
which is an open rectangle contained in $\Omega(M)$. Define  
\begin{align*}
 f_0(x,y) := \frac{16}{M^2}\Ib_{\Rc(M)}(x,y),  
\end{align*}
which is the probability density function of the uniform distribution in $\Rc(M)$.
   We obtain a collection of  perturbed versions of $f_0$ by using the function $G$. To do this, let $\lambda_{i}\in (0,1]$ satisfy that $Z_{i}:=1/\lambda_{i}$ is an integer for $i=1,2$, and let $ I=\{1,\dots,Z_{1}\} \times \{1,\dots,Z_{2}\}$. For $s_1,s_2>0$, let 
  \begin{align}\label{Appendix::eq::stilde and sbar}
    \bar{\mathbf{s}}:= \frac{2}{1/s_1+1/s_2}, \quad \text{and} \quad \tau:=\left(1+\frac{1}{4s_1}+\frac{1}{4s_2}\right)^{-1}=\frac{2\bar{\mathbf{s}}}{2\bar{\mathbf{s}}+1}.    
  \end{align}
  We further define a quantity $\tilde{\lambda}$ by 
  \begin{align*}
       \tilde{\lambda} :=\left(\frac{1}{\lambda_1^{2s_1}}+\frac{1}{\lambda^{2s_2}_2}\right)^{-1/2}.
  \end{align*}
  For each $\theta:= \left\{\theta_{(i,j)}:(i,j)\in I\right\}\in \{-1,1\}^{Z_{1}\times Z_{2} }$, define the perturbed function $f_\theta: \Rb^2 \rightarrow \Rb$ by 
 \begin{align}\label{Appendix::eq::ftheta}
f_\theta(&x,y)
:= f_0(x,y) \\
&\quad  +C_0 \tilde{\lambda}\lambda_{1}\lambda_2
\sum_{(i,j)\in I} \theta_{(i,j)} 
G_{\lambda_1}\!\left(\frac{x - M/4}{M/4} - i\lambda_1\right)
G_{\lambda_2}\!\left(\frac{y - 3M/4}{M/4} - j\lambda_2\right),\notag
\end{align}
 where $G_{h}(\cdot):=(1/h)G(\cdot/h)$ for any $h>0$ and $C_0$ is a constant independent of $n$ and $m$ that will be specified later. 
 For each $\theta$, the function $f_\theta$ is supported in $\Rc(M)$. 
 The density functions $f_\theta$ and $f_0$ are obtained by transporting the support of the functions used in \citet[Section 4]{AlbertLaurentMarrelMeynaoui2022} from $(0,1)^2$ to $\Rc(M)$, via an affine map, and by allowing the parameters $\lambda_1$, $\lambda_2$ to vary across dimensions. Therefore, their theoretical properties established in \citet[Section 4]{AlbertLaurentMarrelMeynaoui2022} remain unchanged. This completes {\bf{Step 1}}.
 
In {\bf{Step 2}}, we construct  probability distributions $P$ and $Q_\theta$ on $\mathbf{PD}$ whose persistence intensity functions are $f_0$ and $f_\theta$, respectively. To do this, we use the following lemma. 

\begin{lem}\label{Appendix::the lem for the first step} Let $f$ be a probability density function on $\Omega=\{(x,y)\in \Rb^2: y>x\geq0\}$ with respect to the two-dimensional Lebesgue measure. Then there is a probability distribution $P$ on $\mathbf{PD}$, induced by $f$, whose persistence intensity function is $f$.
\end{lem}
\begin{proof}
The proof is motivated by \citet{wu2024estimation}.
Consider the Dirac measure $\delta_x \in \mathbf{PD}$ for $x\in \Omega$. Define a map $\Phi:\Omega\to\mathbf{PD}$ by
\[
    \Phi(x)=\delta_x.
\]
Then, the map $\Phi$ is measurable. 
Recall that $\mathbf{PD}$ is equipped with the sigma-field
\[
\mathcal B(\mathbf{PD})
=
\sigma\{D\mapsto D(B):B\in\Bc(\Omega),\ 
B\text{ relatively compact}\},
\]
where $\Bc(\Omega)$ is the Borel sigma-field of $\Omega$.
For every relatively compact set
$B\subset\Omega$, let
$\operatorname{ev}_B:D\in \mathbf{PD}\mapsto D(B)\in \Rb$ be an evaluation map. Note that $\mathcal B(\mathbf{PD})$ is the smallest sigma-field on which such evaluation maps are measurable. It holds that, for each $x\in \Omega$,
\[
    (\operatorname{ev}_B\circ\Phi)(x)
    =
    \delta_x(B)
    =
    \mathbf 1_B(x).
\]
Since the maps $1_B=\operatorname{ev}_B\circ \Phi$ and $\operatorname{ev}_B$ are measurable, $\Phi$ is also measurable.

Define the probability measure $P$ on $\mathbf{PD}$ as the pushforward measure induced by $f$, $\ie$, for any measurable set $\Yc\subset \mathbf{PD}$,
\begin{align}\label{Appendix::eq::construction of dist from density}
    P(\Yc) := \int_{\Phi^{-1}(\Yc)}f(x)dx.
\end{align}
Then, by a change-of-variables argument, for any measurable function $h$,
\begin{align*}
    \int_{\Yc}h(\mu)dP(\mu) = \int_{\Phi^{-1}(\Yc)}h(\Phi(x))f(x)dx.
\end{align*}
Let $\Ac \subset\Omega$ be a measurable set. By definition, the expected measure $\Eb[\mu]$ for $\mu \sim P$ satisfies 
\begin{align*}
    \Eb(\mu)(\Ac) &= \E[\mu(\Ac)]=\int_{\mathbf{PD}}\mu(\Ac)dP(\mu) \\
    &=\int_{\Phi^{-1}(\mathbf{PD})}\Phi(x)(\Ac)f(x)dx\\
    &=\int_{\Omega}\delta_x (\Ac)f(x)dx\\
    &=\int_\Omega \Ib(x\in \Ac)f(x)dx\\
    &=\int_{\Ac} f(x)dx
\end{align*}
This implies the intensity function of $P$ is $f$. This completes the proof of Lemma \ref{Appendix::the lem for the first step}.
\end{proof}

By Lemma \ref{Appendix::the lem for the first step}, there exist the probability distributions $P$ and $Q_\theta$ induced by $f_0$ and $f_\theta$, respectively. To define the probability measure $\kappa_{\rho^*_{n,m}}$, let 
$\Theta=\{\Theta_{(i,j)}:(i,j)\in I\}$ be a random vector whose components $\Theta_{(i,j)}$ are i.i.d. Rademacher variables and let $\pi$ be the corresponding distribution of $\Theta$. 
Let $\kappa_{\rho^*_{n,m}}$ be the distribution of $(P, Q_\Theta)$ induced by $\pi$. Let us now concretize the associated probability measures $\Pb_{\kappa_{\rho^*_{n,m}}}$ and $\mathbb{P}_{0,n,m}$, which are introduced in Lemma \ref{Appendix::lem::base of minimax lowerbound}. It suffices to consider them on measurable rectangles $\Ac = A_1 \times \cdots \times A_{n+m}$, 
where each $A_i$ is a measurable subset of $\mathbf{PD}$.
For such $\Ac$, we have that 
\begin{align}\label{Appendix::eq::concrete Ptheta}
    \Pb_{\kappa_{\rho^*_{n,m}}}(\Ac) &= \int P^{\otimes n}\otimes Q_{\theta}^{\otimes m}(\Ac)d\pi(\theta)\notag\\
    &=\int
    \prod_{i=1}^nP(A_i)\cdot \prod_{j=n+1}^{n+m}Q_{\theta}(A_j)
    d\pi(\theta)\notag\\
    &=
    \prod_{i=1}^n\int_{\Phi^{-1}(A_i)}f_0(x)dx\cdot \int \prod_{j=n+1}^{n+m}\int_{\Phi^{-1}(A_j)}f_\theta(x)dx
    d\pi(\theta),
\end{align}
where, in the last equality, we use the property \eqref{Appendix::eq::construction of dist from density} in the proof of Lemma \ref{Appendix::the lem for the first step}. 
Let $\mathbb{P}_{0,n,m}$ be $P^{\otimes (n+m)}$. Then, by the above arguments, we also have that 
\begin{align}\label{Appendix::eq::concrete P0}
    \Pb_{0,n,m}(\Ac) =\prod_{i=1}^{n+m}\int_{\Phi^{-1}(A_i)}f_0(x)dx.
\end{align} 
This completes {\bf{Step 2}}.

In {\bf{Step 3}}, we show that the conditions in Lemma \ref{Appendix::lem::base of minimax lowerbound} hold with $\Pb_{\kappa_{\rho^*_{n,m}}}$ and $\mathbb{P}_{0,n,m}$. We split {\bf{Step 3}} into two parts: Part A  and Part B. In Part A, we show that the distributions $P$ and $Q_\theta$ are contained in our model $\Pc$. The main issue in this part is to show that the function $\Phi$, defined in the proof of Lemma \ref{Appendix::the lem for the first step}, is implementable in TDA settings. 
In Part B, we verify the remaining conditions in Lemma \ref{Appendix::lem::base of minimax lowerbound} hold.

We begin with Part B. As mentioned in {\bf{Step 1}}, the density functions $f_0$ and $f_\theta$ satisfy the properties established in \citet[Lemma 6, Lemma 7, and Proposition 5]{AlbertLaurentMarrelMeynaoui2022}. Lemma \ref{Appendix::lem::properties of ftheta and f0} below shows key properties of intensity functions $f_0$ and $f_\theta$.

\begin{lem}
\label{Appendix::lem::properties of ftheta and f0} Let $w$ be a weight function, $s_1,s_2>0$, $\gamma\in (0,1)$, $R>0$, and $N>\frac{\|w^2\|_\infty}{|\Rc(M)|}$. Let $\lambda_i\in (0,1]$ such that $\frac{1}{\lambda_i}$ is an integer, for $i=1,2$. Consider $f_0$ and $f_\theta$ defined in {\bf{Step 1}}. Then, we have that 
\begin{enumerate}
    \item If $C_0\leq \min\left\{\frac{N|\Rc(M)|-\|w^2\|_\infty}{\|w^2\|_\infty |\Rc(M)|},e^{-2}|\Rc(M)|\right\}$, then 
    $\max{\{\|w^2\cdot f_0\|_\infty,\|w^2\cdot f_\theta\|_\infty\}}\leq N$ and $f_\theta$ is nonnegative.
    \item The functions $f_0$ and $f_\theta$ also satisfy $\| (f_0-f_\theta)\|_2=  C_0\|G\|^2_2\frac{M}{4}\tilde{\lambda}$, hence it holds that 
    \begin{align}\label{Appendix::L2 weighted discrepancy lower bound}
        \|w\cdot (f_0-f_\theta)\|_2\geq c^*\tilde{\lambda},
    \end{align}
    for $c^*\coloneq\underset{(x,y)\in \Rc(M)}{\inf}w(x,y)\cdot  C_0\|G\|^2_2\frac{M}{4}$.
    \item There exists a constant $C(s_1,s_2,\gamma,M,w)>0$ such that, if
\begin{align*}
C_0^2 \leq \frac{(2\pi)^2 R^2}{C(s_1,s_2,\gamma,M,w)},
\end{align*} then we have that
    \begin{align*}
        P_{\pi}\left(w\cdot \left(f_0-f_{\Theta}\right)\in \Sc^{s_1,s_2}_{2}(R)\right)\geq 1-\gamma,
    \end{align*}
\end{enumerate}
where $\theta$ in $f_\theta$ is replaced by the random vector $\Theta$ in the third statement of the lemma. Here, $P_\pi$ denotes the distribution induced by the law $\pi$ of $\Theta$.    
\end{lem}
\begin{proof} We prove only the third statement, since the proofs of the remaining statements follow from the same arguments as those in \citet{AlbertLaurentMarrelMeynaoui2022}.
In the proof, our target is to show that
\begin{align*}
    \Pb_\pi\left(\int_{\Rb^2}\left(|u|^{2s_1}+|v|^{2s_2}\right) \left|\reallywidehat{w (f_\Theta - f_0)}(u,v)\right|^2dudv\leq (2\pi)^2R^2\right)\geq 1-\gamma.
\end{align*}
To do this, we use the following lemma. 
\begin{lem}\label{lem::weighted young}
Let $s_1,s_2>0$. Let $K \in L^1(\mathbb{R}^2)$ satisfy
\[
\|K\|_{L^1} + \||x|^{s_1}K\|_{L^1} + \||y|^{s_2}K\|_{L^1} < \infty,
\]
where $\||x|^{s_1}K\|_{L^1}=\int_{\Rb^2}|x|^{s_1}|K(x,y)|d(x,y)$ and $\||y|^{s_2}K\|_{L^1}=\int_{\Rb^2}|y|^{s_2}|K(x,y)|d(x,y)$.
Let $h \in L^2(\Rb^2)$ satisfy
\[
\int_{\mathbb{R}^2} \left(|u|^{2s_1} + |v|^{2s_2}\right) |h(u,v)|^2 \, dudv < \infty.
\]
Then there exists a constant $C>0$, depending only on $K,s_1,s_2$, such that
\begin{align*}
&\int_{\mathbb{R}^2} \left(|u|^{2s_1} + |v|^{2s_2}\right)
|(K*h)(u,v)|^2 \, dudv \\
&\qquad \le C \left(
\|h\|_{L^2}^2
+
\int_{\mathbb{R}^2} \left(|u|^{2s_1} + |v|^{2s_2}\right)
|h(u,v)|^2 \, dudv
\right).
\end{align*}
\end{lem}
\begin{proof}
We only prove for the $u$-component; the $v$-component follows similarly.

For any $s_1>0$, there exists $C_{s_1}>0$ such that
\[
|u|^{s_1} \le C_{s_1} \big( |u-x|^{s_1} + |x|^{s_1} \big).
\]
Hence,
\begin{align*}
|u|^{s_1}|(K*h)(u,v)|
&= |u|^{s_1} \left| \int_{\mathbb{R}^2} K(x,y)\,h(u-x,v-y)\,dx\,dy \right| \\
&\le \int_{\mathbb{R}^2} |u|^{s_1} |K(x,y)|\,|h(u-x,v-y)|\,dx\,dy \\
&\le C_{s_1} \int_{\mathbb{R}^2}
\left( |u-x|^{s_1} + |x|^{s_1} \right)
|K(x,y)|\,|h(u-x,v-y)|\,dx\,dy.
\end{align*}
Thus we obtain,
\[
|u|^{s_1}|(K*h)|
\le C_{s_1}
\Big(
(|x|^{s_1}|K|) * |h|
+
|K| * (|u|^{s_1}|h|)
\Big).
\]
Taking the $L^2$-norm and applying Young's inequality for convolution,
\begin{align*}
\||u|^{s_1}(K*h)\|_{L^2}
&\le C_{s_1} \left(
\| |x|^{s_1}K \|_{L^1}\|h\|_{L^2}
+
\|K\|_{L^1}\||u|^{s_1}h\|_{L^2}
\right).
\end{align*}
Squaring both sides yields
\[
\||u|^{s_1}(K*h)\|_{L^2}^2
\le C \left(
\|h\|_{L^2}^2 + \||u|^{s_1}h\|_{L^2}^2
\right).
\]
An analogous proof holds for the $v$-component. Summing the two bounds gives
\begin{align*}
&\int_{\mathbb{R}^2} \left(|u|^{2s_1} + |v|^{2s_2}\right)
|(K*h)(u,v)|^2 \, dudv \\
&\qquad \le C \left(
\|h\|_{L^2}^2
+
\int_{\mathbb{R}^2} \left(|u|^{2s_1} + |v|^{2s_2}\right)
|h(u,v)|^2 \, dudv
\right),
\end{align*}
which completes the proof of Lemma \ref{lem::weighted young}.
\end{proof}
The Fourier transform $\reallywidehat{w\cdot (f_{\Theta}-f_0)}$ can be written as 
\begin{align*}
    \reallywidehat{w\cdot (f_{\Theta}-f_0)}= \frac{1}{(2\pi)^2} \widehat{w}*\reallywidehat{f_{\Theta}-f_0}.
\end{align*}
Therefore, by applying Lemma \ref{lem::weighted young} with $K=\widehat{w}$, $h=\reallywidehat{f_{\Theta}-f_0}$, we can obtain that
\begin{align*}
&\int_{\Rb^2}\left(|u|^{2s_1}+|v|^{2s_2}\right) \left|\reallywidehat{w (f_\Theta - f_0)}(u,v)\right|^2dudv\\
&=\frac{1}{(2\pi)^4}\int_{\mathbb{R}^2} \left(|u|^{2s_1} + |v|^{2s_2}\right)
|(\widehat{w}*\reallywidehat{f_\Theta - f_0})(u,v)|^2 \, dudv \\
&\qquad \le C_w(s_1,s_2) \left(
\|\reallywidehat{f_\Theta - f_0})\|_{L^2}^2
+
\int_{\mathbb{R}^2} \left(|u|^{2s_1} + |v|^{2s_2}\right)
|(\reallywidehat{f_\Theta - f_0})(u,v)|^2 \, dudv
\right),
\end{align*}
for some positive constant $C_w(s_1,s_2)$ depending only on $w$, $s_1$, and $s_2$. Note that the first summand in the last term $\|\reallywidehat{f_\Theta - f_0})\|_{L^2}^2$ can be bounded above as 
\begin{equation}\label{ineq::L^2ofwidehatftheta-f0}
\begin{aligned}
    \|(\reallywidehat{f_\Theta - f_0})\|_{L^2}^2&= (2\pi)^2\|f_\Theta - f_0\|_{L^2}\\
    &=(2\pi)^2\left(C_0\|G\|_2\frac{M}{4}\tilde{\lambda}\right)^2\\
    &\leq (2\pi)^2C_0^2\|G\|_2^2\left(\frac{M}{4}\right)^2=:C_1(M)C_0^2,
\end{aligned}
\end{equation}
by Plancherel's identity and the second statement of Lemma \ref{Appendix::lem::properties of ftheta and f0}. 

Furthermore, following the same arguments in the proof of \citet[Lemma 7]{AlbertLaurentMarrelMeynaoui2022},
we obtain that there exists a constant $C_2(s_1,s_2,\gamma,M)>0$ such that
\begin{align*}
    \int_{\Rb^2} (|u|^{2s_1}+|v|^{2s_2})\left|\reallywidehat{(f_{\Theta}-f_0)}(u,v)\right|^2 du dv \leq C_0^2C_2(s_1,s_2,\gamma,M), 
\end{align*}
with probability greater than $1-\gamma$. Therefore, by defining 
\[C(s_1,s_2,\gamma,M,w):=C_w(s_1,s_2)\left(C_1(M)+C_2(s_1,s_2,\gamma,M)\right),\] one can conclude that
\begin{align*}
    \Pb_\pi\left(\int_{\Rb^2}\left(|u|^{2s_1}+|v|^{2s_2}\right) \left|\reallywidehat{w (f_\Theta - f_0)}(u,v)\right|^2dudv\leq (2\pi)^2R^2\right)\geq 1-\gamma,
\end{align*}
since $ C_0^2 \leq \frac{(2\pi)^2 R^2}{C(s_1,s_2,\gamma,M,w)}$.
This completes the proof of Lemma \ref{Appendix::lem::properties of ftheta and f0}.
\end{proof}

Let us now consider the likelihood ratio $L_{\kappa_{\rho^*_{n,m}}}$ defined in Lemma \ref{Appendix::lem::base of minimax lowerbound}. Let $D_1,\dots D_{n+m}$ be i.i.d. observations from the distribution $P$ and define $X_i=\Phi^{-1}(D_i)$ for each $i=1,\dots,n+m$, where $\Phi$ is defined in the proof of Lemma \ref{Appendix::the lem for the first step}. By \eqref{Appendix::eq::concrete Ptheta} and \eqref{Appendix::eq::concrete P0}, and by standard measure-theoretic arguments, we obtain the following equality:
\begin{align}\label{Appenidx::eq::liklihood}
    L_{\kappa_{\rho^*_{n,m}}}(D_1,\dots,D_{n+m}) = \int \prod_{j=n+1}^{n+m}\frac{f_\theta(X_j)}{f_0(X_j)}d\pi(\theta).
\end{align}
In other words, the likelihood ratio $L_{\kappa_{\rho^*_{n,m}}}$ coincides with the one defined by $f_\theta$ and $f_0$. 
Using this, we verify the upper bound on the squared likelihood ratio in Lemma \ref{Appendix::lem::base of minimax lowerbound}. 
Before formally stating this bound, let us remind our target. In this proof, we aim to show the lower bound $\rho^{*}_{n,m}$ in Lemma \ref{Appendix::lem::base of minimax lowerbound} has the form:
\begin{align*}
    C\cdot (n+m)^{\frac{-\bar{\mathbf{s}}}{2\bar{\mathbf{s}}+1}},
\end{align*}
for some constant $C$ independent of $n$ and $m$.
To achieve this rate, we need to appropriately choose the parameter $\lambda_i$ so that $\tilde{\lambda}\asymp(n+m)^{\frac{-\bar{\mathbf{s}}}{2\bar{\mathbf{s}}+1}}$. The following lemma also captures this consideration.
\begin{lem}\label{Appendix::lem::likelihood condition} Let $\alpha$, $\beta$, $\gamma$ $\in (0,1)$ satisfy $\alpha+\beta+\gamma<1$, and let $s_1,s_2>0$, and $R>0$. Let $C_0(s_1,s_2,\gamma,w,M,N,R)$ satisfy the conditions in Lemma \ref{Appendix::lem::properties of ftheta and f0}. Then, there exists a constant $C(\alpha,\beta,s_1,s_2,\gamma,w,M,N,R)>0$ such that, if the quantities  $Z_i$ and $\lambda_i$ are set as
\begin{align*}
    Z_i=
     \left\lceil \frac{1}{C(\alpha,\beta,s_1,s_2,\gamma,w,M,N,R)\, (n+m)^{-\tau/(2s_i)}} \right\rceil \quad \text{and} \quad \lambda_i=\frac{1}{Z_i},\quad i=1,2,
\end{align*}
    then, for $\rho_{n,m}^*:= c^*\tilde{\lambda}$ where $c^*$ is a constant defined in \eqref{Appendix::L2 weighted discrepancy lower bound}, 
    \begin{align*}
    \rho_{n,m}^* \asymp(n+m)^{\frac{-\bar{\mathbf{s}}}{2\bar{\mathbf{s}}+1}}\quad \text{and}\quad  \kappa_{\rho_{n,m}^{*}}\left(\Pc^{\otimes2}_{\Fc_{\rho_{n,m}^{*}}(\Sc^{\mathbf{s}}_2(R))}\right)\geq 1-\gamma.   
    \end{align*}
     Moreover, 
    \begin{align*}
        \limsup_{n,m\to\infty}\Eb_{\Pb_{0,n,m}} \left[L^2_{\kappa_{\rho_{n,m}^{*}}}(\Db_{n+m})\right] <1+4(1-\alpha-\beta-\gamma)^2.
    \end{align*} 
\end{lem}
\begin{proof}[Proof of Lemma \ref{Appendix::lem::likelihood condition}]
It is sufficient to prove the last argument in the lemma.

For \(u=(i,j)\in I\), define
\begin{align*}
    g_u(x,y)
:=
G_{\lambda_1}
\left(
\frac{x-M/4}{M/4}-i\lambda_1
\right)
G_{\lambda_2}
\left(
\frac{y-3M/4}{M/4}-j\lambda_2
\right)
\end{align*}
and
\begin{align*}
\Delta_u(x,y)
:=
C_0\widetilde{\lambda}\lambda_1\lambda_2 g_u(x,y).    
\end{align*}
Then, we have 
\begin{align*}
    f_\theta=f_0+\sum_{u\in I}\theta_u\Delta_u.
\end{align*}
For $u\in I$, the supports of the functions $g_u$ are pairwise disjoint. Moreover, since $\int_{\mathbb{R}}G(t)\,dt=0$, it holds that
\begin{align}
    \int_{\Rc(M)}\Delta_u(x,y)\,dx\,dy=0
\end{align}
and
\begin{align}
    \int_{\Rc(M)}
\frac{\Delta_u(x,y)\Delta_v(x,y)}{f_0(x,y)}
\,dx\,dy
=0,
\qquad u\neq v\in I.
\end{align}
Since $f_0=16/M^2$ on $\Rc(M)$, a change of variables yields
\begin{align*}
v_\lambda
&:=
\int_{\Rc(M)}
\frac{\Delta_u(x,y)^2}{f_0(x,y)}
\,dx\,dy\\
&=
\frac{C_0^2M^4}{256}\,
\|G\|_2^4\,
\widetilde{\lambda}^{\,2}\lambda_1\lambda_2\\
&=:C_G\widetilde{\lambda}^{\,2}\lambda_1\lambda_2.
\end{align*}
We now observe the expectation of the squared likelihood ratio. 
Let $\Theta'$ be an independent copy of $\Theta$. By \eqref{Appenidx::eq::liklihood}, the first $n$ observations cancel from the likelihood ratio, and independence of the remaining $m$ observations yields
\begin{align*}
\Eb_{\Pb_{0,n,m}}
\left[
L_{\kappa_{\rho_{n,m}^{*}}}^{\,2}
(\mathbb{D}_{n+m})
\right]
&=
\Eb_{\Theta,\Theta'}
\left[
\int_{\Rc(M)}
\frac{f_\Theta(x,y)f_{\Theta'}(x,y)}
     {f_0(x,y)}
\,dx\,dy
\right]^m\\
&=
\Eb_{\Theta,\Theta'}
\left[
1+
v_\lambda
\sum_{u\in I}\Theta_u\Theta'_u
\right]^m.
\end{align*}
The quantity in square brackets is nonnegative, since it is an integral of
$f_\Theta f_{\Theta'}/f_0$. Therefore, using the inequality $1+x\leq e^x$, for $x\in \Rb$, and the fact that
$
\varepsilon_u:=\Theta_u\Theta'_u$, for $u\in I$, are i.i.d. Rademacher variables, we obtain 
\begin{align*}
\Eb_{\Pb_{0,n,m}}
\left[
L_{\kappa_{\rho_{n,m}^{*}}}^{\,2}
(\mathbb{D}_{n+m})
\right]&=
\Eb_{\Theta,\Theta'}
\left[
1+
v_\lambda
\sum_{u\in I}\Theta_u\Theta'_u
\right]^m\\
&\leq
\Eb
\left[
\exp\left(
mv_\lambda\sum_{u\in I}\varepsilon_u
\right)
\right]\\
&= \prod_{u\in I}\Eb\left[\exp(mv_{\lambda}\varepsilon_u)\right]\\
&=
(\cosh(mv_\lambda))^{|I|}\\
&\leq
\exp\left(\frac{|I|m^2v_\lambda^2}{2}\right),
\end{align*}
where, in the last inequality, we use $\cosh(x)\leq \exp(x^2/2)$ for any $x\in \Rb$. Moreover,
since
$
|I|=Z_1Z_2=\frac{1}{\lambda_1\lambda_2},
$
we can conclude that
\begin{align}
    \Eb_{\Pb_{0,n,m}}
\left[
L_{\kappa_{\rho_{n,m}^{*}}}^{\,2}
(\mathbb{D}_{n+m})
\right]
\leq
\exp\left(
\frac{C_{G}^2}{2}m^2\widetilde{\lambda}^{4}
\lambda_1\lambda_2
\right).
\end{align}
Let 
\begin{align*}
    Z_i= \left\lceil C_*^{-1}(n+m)^{\tau/(2s_i)}\right\rceil,\quad \lambda_i=\frac{1}{Z_i},\quad i=1,2.
\end{align*}
Then, we have 
\begin{align*}
    m^2 \tilde{\lambda}^4\lambda_1\lambda_2 &\leq (n+m)^{\frac{-4\bar s-2}{2\bar s+1}+2} \cdot (C_*^{-2s_1-1}+C_*^{-2s_2-1})^{-2}\\
    &=(C_*^{-2s_1-1}+C_*^{-2s_2-1})^{-2}.
\end{align*}
Since $1-\alpha-\beta-\gamma>0$, we may choose $C_*>0$, depending only on
$\alpha,\beta,\gamma,s_1,s_2,w,M,N$, and $R$, sufficiently small so that
\begin{align}
\frac{C_G^2}{2}(C_*^{-2s_1-1}+C_*^{-2s_2-1})^{-2}< 
\log\left\{1+4(1-\alpha-\beta-\gamma)^2\right\}.
\end{align}
With this choice, we have 
\begin{align}
    \limsup_{n,m\to\infty}\Eb_{\Pb_{0,n,m}}
\left[
L_{\kappa_{\rho_{n,m}^{*}}}^{\,2}
(\mathbb{D}_{n+m})
\right]
< 1+4(1-\alpha-\beta-\gamma)^2.
\end{align} 
Taking
\begin{align*}
C(\alpha,\beta,s_1,s_2,\gamma,w,M,N,R):=C_*
\end{align*}
completes the proof of the third argument in Lemma \ref{Appendix::lem::likelihood condition}.
\end{proof}

This completes Part B of {\bf{Step 3}}. Once Part A is proved, the proof of Theorem \ref{minimax lower bound}, with $C_{\star}(M,N,L,w,s_1,s_2,R,\alpha,\beta,\gamma):=c_0c_*$ which are defined in Lemma~\ref{Appendix::lem::base of minimax lowerbound} and \eqref{Appendix::L2 weighted discrepancy lower bound}, is completed by  combining Lemmas \ref{Appendix::lem::base of minimax lowerbound} and \ref{Appendix::lem::likelihood condition}. 

In Part A, we check whether the distributions $P$ and $Q_\theta$ are contained in $\Pc$; see Section \ref{sec::Assumptions and Notations} of the main text, for the definition of $\Pc$. We show that $P$ and $Q_\theta$ satisfy the conditions \ref{assump::randomness of diagram}, \ref{assump::Bounded persistence condition}, \ref{assump::Bounded cardinality condition}, and \ref{assump::Existence of intensity function}. 

The condition \ref{assump::Bounded persistence condition} is satisfied, since the supports of $f_\theta$ and $f_0$ are contained in $\Omega(M)$. 
The condition \ref{assump::non degeneracy} is also satisfied, since 
\begin{align*}
    \min \left(\|w^2f_0\|_2, \|w^2f_\theta\|\right)\geq \inf_{(x,y)\in \Rc(M)}w(x,y)^2 \min \left(\|f_0\|_2,\|f_\theta\|_2\right) \geq \frac{4\underline{w}^2}{M}\geq L. 
\end{align*}
Since the cardinality of each random persistence diagram following $P$ or $Q_\theta$ is $1$, the condition \ref{assump::Bounded cardinality condition} holds. By the first statement of Lemma \ref{Appendix::lem::properties of ftheta and f0}, the condition \ref{assump::Existence of intensity function} holds as well. Then, it remains to verify the condition \ref{assump::randomness of diagram}. 
This is equivalent to asking whether the image of $\Phi$, defined in the proof of Lemma \ref{Appendix::the lem for the first step}, can be realised as a persistence diagram for some point cloud with some filtration function.
\paragraph*{Point cloud construction}
Let a point $(b,d) \in \Omega$ with $b>0$ be given. 
We construct a point cloud $\Xc(b,d)$ on the circle $S^{1}(d)$, defined by
\begin{align*}
S^{1}(d) = \{ x \in \mathbb{R}^{2} \mid \|x\|_2 = d \}.
\end{align*}
The point cloud $\Xc(b,d)$ is constructed as follows: the points are arranged on $S^{1}(d)$ in cyclic order so that the Euclidean distance between any two adjacent points is exactly $2b$, except for the final pair, for which the distance is at most $2b$ if necessary. Figure \ref{Appendix::fig:cech_point_cloud} presents visual examples of point clouds constructed in this manner. With this convention, the map  $(b,d) \mapsto \Xc(b,d)$ is measurable when point clouds are viewed as counting measures on $\Rb^2$.

By Proposition \ref{prop::phofcircle} in the main text, we obtain that
\begin{align*}
    \mathrm{PD}_1\left(\mathrm{\check{C}ech}(\Xc(b,d))\right)=\{(b,d)\},
\end{align*}
implying that the function $\Phi:x\mapsto \delta_x$ can be realised by the \v{C}ech complex on the point cloud $\Xc(x)$. 
This completes Part A and the proof of Theorem \ref{minimax lower bound}. \qed

\begin{figure}[h]
    \centering
\includegraphics[width=\textwidth]{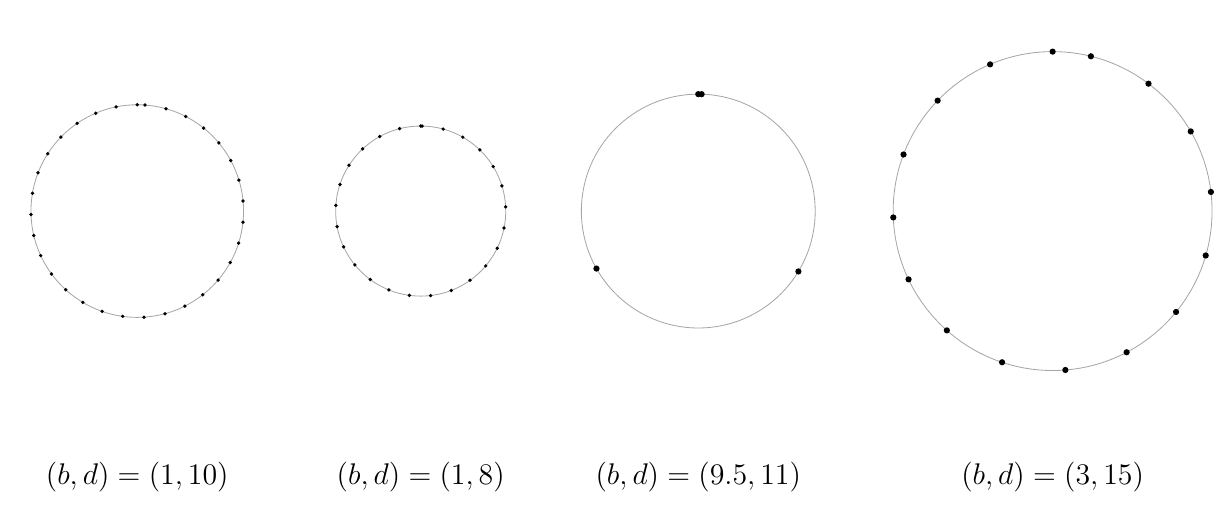}
\caption{
Examples of point clouds constructed on $S^1(d)$ with
 $(b,d)=(1,10)$,
 $(b,d)=(1,8)$,
 $(b,d)=(9.5,11)$, and
 $(b,d)=(3,15)$.
}
\label{Appendix::fig:cech_point_cloud}
\end{figure}
\subsection{Proof of Proposition \ref{prop::aggtest properties}}
We use the same notation used in the proofs of Theorems \ref{Theorem::uniform-validity} and  \ref{thm::upperbound_unifseparation_rate}.

\begin{proof}[Proof of the first statement]
Define 
\begin{align*}
    S_{n,m} :=\frac{U_{n,m,\lambda}}{\sigma_{n,m}}\quad \xi_{r,n,m}:=\frac{D_r}{\sigma_{n,m}},
\end{align*}
for $r=1,\dots,n+m$,
and 
\begin{align*}
    \widetilde{V}_{n,m}:=V_{n,m}/\sigma_{n,m}^2, \quad \text{where}\quad V_{n,m}=\sum_{r=1}^{n+m}\Eb \left[D_r^2\mid \Fc_{r-1}\right].
\end{align*}
By the martingale Berry--Esseen inequality of
\citet{HeydeBrown1970}; see also \citet{Fan2019}, applied with $p=2$, we have 
\begin{align}\label{BerryEsseeninequality} 
\sup_{t\in\mathbb R}
\left|
\Pr(S_{n,m}\le t)-\Phi(t)
\right|
\leq
C
\left\{
\Eb\left|\widetilde{V}_{n,m}-1\right|^2
+
\sum_{r=1}^{n+m}\Eb|\xi_{r,n,m}|^4
\right\}^{1/5},
\end{align}
for some constant $C>0$,
where $\Phi(t)$ is the cdf of $\Nc(0,1)$.
By \eqref{Vnmsigmanm Lsquare} and \eqref{order of scaled forth moment sum} in the proof of Theorem \ref{Theorem::uniform-validity}, we have 
\begin{align*}
    \Eb\left[\frac{V_{n,m}}{\sigma_{n,m}^2}-1\right]^2 \leq \widetilde{C}\left(\frac{1}{n+m}+\nu\right),\quad \sum_{r=1}^{n+m}\Eb \left|\xi_{r,n,m}\right|^4 \le \widetilde{C}\left(\frac{1}{(n+m)^2\nu}+\frac{1}{n+m}\right).
\end{align*}
Substituting these two bounds into \eqref{BerryEsseeninequality} yields 
\begin{align}\label{step1 Snm dist order}
    \sup_{t\in\mathbb R}
\left|
\Pr(S_{n,m}\le t)-\Phi(t)
\right|
\leq C_1\left(\nu + \frac{1}{n+m}+\frac{1}{(n+m)^2\nu}\right)^{1/5},
\end{align}
for some constant $C_1>0$.

Our statistic is not $S_{n,m}$ but $T_{n,m,\lambda}=U_{n,m,\lambda}/\hat{\sigma}_{n,m}$. Define $R_{n,m}:= \hat{\sigma}^2_{n,m}/\sigma^2_{n,m}$. Then, we can write $T_{n,m,\lambda}=S_{n,m}/\sqrt{R_{n,m}}$. By \eqref{RatioL2order} in the proof of Theorem \ref{Theorem::uniform-validity}, we have 
\begin{align}\label{the ratio of rnm}
    r_{n,m}:=\Eb\left[R_{n,m}-1\right]^2\leq C'\left(\frac{1}{n+m}+\frac{1}{(n+m)^2\nu}\right),
\end{align}
for some constant $C'>0$.
For any $\delta\in(0,1/2)$, define 
\begin{align*}
    \Ec_{\delta}:=\left\{|R_{n,m}-1|\leq \delta\right\}.
\end{align*}
We can derive 
\begin{align*}
    \sup_{t\in\mathbb R}
\left|
\Pb(T_{n,m,\lambda}\le t)-\Phi(t)
\right| \leq \sup_{t\in\mathbb R}
\left|
\Pb(S_{n,m}\le t)-\Phi(t)
\right|+C''r_{n,m}^{1/3},
\end{align*}
for some constant $C''>0$.
Indeed, note that, on $\Ec_\delta $, we have $R_{n,m}\geq1-\delta>0$, and hence
\begin{align*}
    T_{n,m,\lambda}\leq t \iff S_{n,m}\leq t\sqrt{R_{n,m}}.
\end{align*}
For $t\geq 0$, if 
   $\sqrt{1-\delta}\leq \sqrt{R_{n,m}}\leq \sqrt{1+\delta}$,
then 
    $t\sqrt{1-\delta}\leq t\sqrt{R_{n,m}}\leq t\sqrt{1+\delta}$.
Therefore, we have that, on $\Ec_{\delta}$, 
\begin{align*}
    \left\{S_{n,m}\leq t\sqrt{1-\delta }\right\}\subseteq \left\{T_{n,m,\lambda}\leq t\right\}\subseteq \left\{S_{n,m}\leq t\sqrt{1+\delta }\right\},
\end{align*}
which implies 
\begin{align*}
    \Pb\left(S_{n,m}\leq t\sqrt{1-\delta }\right)-\Pb\left(\Ec^c_{\delta}\right) \leq \Pb\left(T_{n,m,\lambda}\leq t\right) \leq \Pb\left(S_{n,m}\leq t\sqrt{1+\delta }\right) +\Pb\left(\Ec_{\delta}^c\right).
\end{align*}
For $t<0$, by the same argument, we can obtain 
\begin{align*}
    \Pb\left(S_{n,m}\leq t\sqrt{1+\delta }\right)-\Pb\left(\Ec^c_{\delta}\right) \leq \Pb\left(T_{n,m,\lambda}\leq t\right) \leq \Pb\left(S_{n,m}\leq t\sqrt{1-\delta }\right) +\Pb\left(\Ec_{\delta}^c\right).
\end{align*}
For simplicity, put 
\begin{align*}
    &t_{-}:= \min\left\{t\sqrt{1-\delta},t\sqrt{1+\delta}\right\},\\
    &t_{+}:= \max\left\{t\sqrt{1-\delta},t\sqrt{1+\delta}\right\},\\
    &d^{\mathrm{K}}_{n,m}:=\sup_{t\in\mathbb R}
\left|
\Pr(S_{n,m}\le t)-\Phi(t)
\right|.
\end{align*}
Then, for any $t\in \Rb$, it holds that 
\begin{align*}
    \Pb\left(S_{n,m}\leq t_{-}\right)-\Pb\left(\Ec^c_{\delta}\right) \leq \Pb\left(T_{n,m,\lambda}\leq t\right) \leq \Pb\left(S_{n,m}\leq t_{+}\right) +\Pb\left(\Ec_{\delta}^c\right).
\end{align*}
By the definition of $d^{\mathrm{K}}_{n,m}$, we obtain 
\begin{align}\label{Tnm dist order intermediate step}
    |\Pb(T_{n,m,\lambda}\leq t)-\Phi(t) | \leq d^{\mathrm{K}}_{n,m} +\max_{a\in \{\sqrt{1-\delta},\sqrt{1+\delta}\}}|\Phi(at)-\Phi(t)|+\Pb\left(\Ec_{\delta}^c\right).
\end{align}
Fix $a\in \{\sqrt{1-\delta},\sqrt{1+\delta}\}$. By the mean value theorem, there exists $b$ between $a$ and $1$ such that 
\begin{align*}
    |\Phi(at)-\Phi(t)| = |a-1||t|\phi(bt),
\end{align*}
where $\phi(x) =\frac{1}{\sqrt{2\pi}}e^{-x^2/2}$.
Since $0<\delta <1/2$, $b\geq \sqrt{1-\delta}\geq 1/\sqrt{2}.$ This implies $|t|\phi(bt) \leq |t| \phi\left(|t|/\sqrt{2}\right)$.
Note that 
\begin{align*}
   \Cc:=\sup_{t\in \Rb} |t| \phi\left(\frac{|t|}{\sqrt{2}}\right)<\infty. 
\end{align*}
Therefore, we can write 
\begin{align*}
    |\Phi(at)- \Phi(t) |\leq \Cc|a-1|. 
\end{align*}
Moreover, we have $|a-1|\leq \delta$. This implies 
\begin{align*}
    \sup_{t\in \Rb} |\Phi(at)-\Phi(t)|\leq \Cc \delta, \quad a \in \left\{\sqrt{1-\delta},\sqrt{1+\delta}\right\}.
\end{align*}
Substituting this result into \eqref{Tnm dist order intermediate step} yields 
\begin{align*}
     |\Pb(T_{n,m,\lambda}\leq t)-\Phi(t) | &\leq d^{\mathrm{K}}_{n,m} +\Cc\delta+\Pb\left(\Ec_{\delta}^c\right)\\
     &\leq d^{\mathrm{K}}_{n,m} +\Cc\delta+\frac{r_{n,m}}{\delta^2}.
\end{align*}
For sufficiently large $n$ and $m$, $r_{n,m}^{\frac{1}{3}}\in (0,1/2)$. Putting $\delta = r_{n,m}^{\frac{1}{3}}$ yields 
\begin{align*}
      |\Pb(T_{n,m,\lambda}\leq t)-\Phi(t) | \leq d^{\mathrm{K}}_{n,m} +C''r_{n,m}^{\frac{1}{3}},
\end{align*}
for some constant $C''>0$ and all sufficiently large $n$ and $m$. 
Therefore, by \eqref{step1 Snm dist order} and \eqref{the ratio of rnm}, we obtain 
\begin{align}\label{step 2 Tnm dist order}
     |\Pb(T_{n,m,\lambda}\leq t)-\Phi(t) | \leq C_2\left[\left(\nu+\frac{1}{n+m}+\frac{1}{(n+m)^2 \nu}\right)^{1/5}+ \left(\frac{1}{n+m}+\frac{1}{(n+m)^2 \nu}\right)^{1/3}\right],
\end{align}
for some constant $C_2>0$ and all sufficiently large $n$ and $m$.

We now verify that the conditions that $\nu \rightarrow 0$ and $\nu(n+m)^2 \rightarrow \infty $, as $n,m\rightarrow \infty$, hold for any bandwidth $\lambda \in \Lambda$. 
By definitions of  $L_{n,m}$ and $K_{n,m}$, we have 
\begin{align*}
    \frac{1}{2s_{n,m}}\lesssim 2^{-L_{n,m}}\lesssim \frac{1}{s_{n,m}},\quad \left(\frac{s_{n,m}}{n+m}\right)^2 \lesssim 2^{-K_{n,m}}\lesssim\left(\frac{2s_{n,m}}{n+m}\right)^2.
\end{align*}
Therefore, for any $\lambda=(\lambda_1,\lambda_2)\in \Lambda$, we have 
\begin{align*}
    \lambda_1\lambda_2 \gtrsim 2^{-K_{n,m}}\gtrsim \left(\frac{s_{n,m}}{n+m}\right)^2,
\end{align*}
which yields 
\begin{align*}
   (n+m)^2 \lambda_1\lambda_2 \gtrsim s_{n,m}^2 \longrightarrow \infty,\quad \text{as}\quad n+m \longrightarrow \infty.
\end{align*}
Moreover, 
\begin{align*}
    \lambda_1\lambda_2 \leq 2^{-2L_{n,m}}\lesssim \frac{1}{s_{n,m}^2} \longrightarrow 0,\quad \text{as}\quad n+m\longrightarrow \infty.
\end{align*}
Consequently, it holds that 
\begin{align*}
    \sup_{\lambda \Lambda}\nu(\lambda )\lesssim s_{n,m}^{-2}\longrightarrow 0,\quad \inf_{\lambda \in \Lambda} \nu(\lambda)(n+m)^2 \gtrsim s_{n,m}^2 \longrightarrow \infty,
\end{align*}
as $n,m\longrightarrow\infty$.
Furthermore, this implies 
\begin{align*}
    \sup_{\lambda \in \Lambda} \left(\nu+\frac{1}{(n+m)^2\nu}\right)\lesssim s_{n+m}^{-2}. 
\end{align*}
Substituting this result into \eqref{step 2 Tnm dist order} yields 
\begin{align}\label{step 3 result Tlambda uniform order}
    \sup_{(P,Q)\in \Pc_0^{\otimes 2}}\sup_{\lambda \in \Lambda }\sup_{t\in \Rb}|\Pb(T_{n,m,\lambda}\leq t)-\Phi(t) | \lesssim s_{n,m}^{-2/5} \leq(\log_2(n+m))^{-2a/5}.
\end{align}

We are now ready to prove the validity of the aggregated test. 
Our aggregated test $\Delta^{\mathrm{Bonf}}_{\alpha}$ is defined by 
\begin{align*}
    \Delta^{\mathrm{Bonf}}_{\alpha,\Lambda}= \Ib\left\{\max_{\lambda \in \Lambda }T_{n,m,\lambda}>z_{1-\alpha/|\Lambda| }\right\}.
\end{align*}
Note that, for each $\lambda \in \Lambda$,  
\begin{align*}
    \sup_{(P,Q)\in \Pc_0^{\otimes 2}}\Pb \left(T_{n,m,\lambda} > z_{1-\alpha/|\Lambda| }\right)\leq \frac{\alpha}{|\Lambda|}+\sup_{(P,Q)\in \Pc_0^{\otimes 2}}\sup_{\lambda \in \Lambda }\sup_{t\in \Rb}|\Pb(T_{n,m,\lambda}\leq t)-\Phi(t) | ,
\end{align*}
since $z_{1-\alpha/|\Lambda|}= \Phi^{-1}\left(1-\alpha/ |\Lambda|\right)$. This result yields
\begin{align*}
    \sup_{(P,Q)\in \Pc_0^{\otimes 2}}\Pb\left(\Delta^{\mathrm{Bonf}}_{\alpha,\Lambda}=1 \right)&\leq \sup_{(P,Q)\in \Pc_0^{\otimes 2}}\sum_{\lambda \in \Lambda}\Pb\left(T_{n,m,\lambda} > z_{1-\alpha/|\Lambda|}\right)\\
    &\leq  |\Lambda| \left(\frac{\alpha}{|\Lambda|}+\sup_{(P,Q)\in \Pc_0^{\otimes 2}}\sup_{\lambda \in \Lambda }\sup_{t\in \Rb}|\Pb(T_{n,m,\lambda}\leq t)-\Phi(t) |\right)\\
    &\leq \alpha + |\Lambda|\cdot \sup_{(P,Q)\in \Pc_0^{\otimes 2}}\sup_{\lambda \in \Lambda }\sup_{t\in \Rb}|\Pb(T_{n,m,\lambda}\leq t)-\Phi(t) |.
\end{align*}
By  $|\Lambda| = O\left(\log(n+m)^2\right)$ and \eqref{step 3 result Tlambda uniform order},
we finally obtain 
\begin{align*}
    \sup_{(P,Q)\in \Pc_0^{\otimes 2}}\Pb\left(\Delta^{\mathrm{Bonf}}_{\alpha,\Lambda}=1 \right) \leq  \alpha+ O\left(\log_2(n+m)^{2-2a/5}\right)=\alpha +o(1),\quad \text{as}\quad n,m\longrightarrow \infty, 
\end{align*}
since $a>5$. We can write 
$    \limsup_{n,m\to \infty}\sup_{(P,Q)\in \Pc_0^{\otimes 2}}\Pb\left(\Delta^{\mathrm{Bonf}}_{\alpha,\Lambda}=1 \right)\leq \alpha$.
This completes the proof of the first statement of Proposition \ref{prop::aggtest properties}.
\end{proof}
\begin{proof}[Proof of the second statement]
We next prove the second statement: 
\begin{align*}
    \rho_{n,m}
\left(
\Delta_{\alpha,\Lambda}^{\mathrm{Bonf}},
\Sc_2^{s_1,s_2}(R),
\beta,w,M,N,L
\right)
\leq
C(M,N,w,k,s_1,s_2,R,\alpha,\beta)
\left(
\frac{\sqrt{\log\log(n+m)}}{n+m}
\right)^{\frac{\bar{s}}{2\bar{s}+1}}.
\end{align*}
Since the Bonferroni test rejects the null when one of tests rejects the null, the uniform separation rate of $\Delta_{\alpha,\Lambda}^{\mathrm{Bonf}}$ can be bounded as follows:
\begin{align}\label{Bonf foundation}
    \rho_{n,m}
\left(
\Delta_{\alpha,\Lambda}^{\mathrm{Bonf}},
\Sc_2^{s_1,s_2}(R),
\beta,w,M,N,L
\right) \leq \min_{\lambda \in \Lambda }\rho_{n,m}
\left(
\Delta_{\alpha/|\Lambda|,\lambda},
\Sc_2^{s_1,s_2}(R),
\beta,w,M,N,L
\right) .
\end{align}
By \eqref{nonrejection prob upper bound} in the proof of Theorem \ref{thm::upperbound_unifseparation_rate}, we have 
\begin{align*}
     &\Pb\left(T_{n,m,\lambda}\leq z_{1-\alpha/|\Lambda|}\right)\\&\leq 4(1+z_{1-\alpha/|\Lambda|}^2)C_{\mathrm{Var}}\left\{\frac{1}{(n+m)\theta_\lambda}+\frac{1}{(n+m)^2\lambda_1\lambda_2\theta_\lambda^2}\right\}+\Pb(\hat{\sigma}_{n,m}=0). 
\end{align*}
By Lemma~\ref{lem::variance estimator zero prob conv}, this inequality yields that there exists a constant $C_{\beta}>0$ such that 
if 
\begin{align}\label{Bonf power individual power 1}
    \theta_{\lambda}\geq C_{\beta}\frac{\sqrt{1+z^2_{1-\alpha/|\Lambda|}}}{(n+m)\sqrt{\lambda_1\lambda_2}},
\end{align}
then $\Pb(T_{n,m,\lambda}\leq z_{1-\alpha/|\Lambda|})\leq \beta$, for sufficiently large
$n,m$.
Note that  
\begin{align}\label{Bonf power 1+zsquare bound}
    1+z_{1-{\alpha/ |\Lambda|}}^2\leq C_{\alpha}\log\log (n+m),
\end{align}
for some $C_\alpha>0$. 
Indeed, when $Z\sim \Nc(0,1)$ and $x\geq 0$, it holds that
\begin{align*}
    \Pb(Z>x) \leq \exp \left(-\frac{x^2}{2}\right).
\end{align*}
This implies 
\begin{align*}
    \alpha /|\Lambda| \leq \exp \left(-\frac{z^2_{1-\alpha/|\Lambda|}}{2}\right).
\end{align*}
Taking logarithms on both sides yields
\begin{align*}
    z_{1-\alpha/|\Lambda|}^2 \leq 2\log |\Lambda |+ 2\log \left(\frac{1}{\alpha}\right). 
\end{align*}
Therefore, for some constant $C_{\alpha}'>0$, we have
\begin{align*}
   1+ z_{1-\alpha/|\Lambda|}^2 \leq C_{\alpha}'\log|\Lambda|. 
\end{align*}
Furthermore, since $|\Lambda| =O((\log n+m)^2)$, we have 
\begin{align*}
    \log |\Lambda|  = O(\log\log n+m).
\end{align*}
Therefore, 
\begin{align*}
    1+ z_{1-\alpha/|\Lambda|}^2= O\left(\log\log n+m\right).
\end{align*}
For notational simplicity put
$\ell_{n+m}:=\log\log (n+m).$
Substituting \eqref{Bonf power 1+zsquare bound} into \eqref{Bonf power individual power 1} yields a sufficient condition for type II error control:
\begin{align*}
       \theta_{\lambda}\geq \frac{C_{\beta}C_{\alpha}\sqrt{\ell_{n+m}}}{(n+m)\sqrt{\lambda_1\lambda_2}}.
\end{align*}
Using Lemma \ref{Appendix::lem::Anisotropic bias bound}, we can write the condition in terms of $\|w\cdot (p-q)\|_{2}^2$ as follows: 
\begin{align}\label{Bonf power suff cond}
    \|w\cdot (p-q)\|_{2}^2 \gtrsim \left(\lambda_1^{2s_1}+\lambda_2^{2s_2}+\frac{\sqrt{\ell_{n+m}}}{(n+m)\sqrt{\lambda_1\lambda_2}}\right).
\end{align}
The bandwidth $\lambda=(\lambda_1,\lambda_2)$ balancing the term on the right-hand side of \eqref{Bonf power suff cond} is 
\begin{align*}
    \lambda_j^{\circ} =\left(\frac{\sqrt{\ell_{n+m}}}{(n+m)}\right)^{\frac{\bar{s}}{s_j(2\bar{s}+1)}},\quad j=1,2.
\end{align*}
Let $k_{j}^\circ $ be defined by 
\begin{align*}
    k_k^{\circ }:= \left\lceil \log_2\left(\frac{1}{\lambda_j^\circ}\right)\right\rceil,\quad j=1,2, 
\end{align*}
and define 
\begin{align*}
    \lambda^{*}_j := 2^{-k_{j}^o},\quad j=1,2.
\end{align*}
Then, we have 
\begin{align*}
    \frac{1}{2}\lambda_j^\circ \leq \lambda_j^*\leq \lambda_j^\circ.
\end{align*}

We now check this balancing bandwidth $\lambda^*=(\lambda^*_1,\lambda^*_2)$ is contained in $\Lambda$. Note that for each $j=1,2$, it holds that
\begin{align*}
    k_{j}^\circ =\frac{\bar{s}}{s_j(2\bar{s}+1)}\log_2\left(\frac{n+m}{\ell_{n+m}}\right)+O(1).
\end{align*}
Since $L_{n,m}=\lceil \log_2 s_{n,m}\rceil+O(1)=a\log_2\log_2(n+m)+O(1)$, it holds that, for all sufficiently large $n+m$,
\begin{align*}
    k_{j}^\circ \geq L_{n,m}.
\end{align*}
Moreover, 
\begin{align*}
    k_{1}^\circ +k_{2}^\circ &=\frac{\bar{s}}{2\bar{s}+1}\left(\frac{1}{s_1}+\frac{1}{s_2}\right)\log_2\left(\frac{n+m}{\ell_{n+m}}\right)+O(1)\\
    &<2\log_2\left(\frac{n+m}{\ell_{n+m}}\right)+O(1).
\end{align*}
Since $K_{n,m}=2\log_2(n+m)-2a\log_2\log_2(n+m)+O(1)$, we have that, for all sufficiently large $n+m$,
\begin{align*}
    k_{1}^\circ +k_{2}^\circ  \leq K_{n,m}.
\end{align*}
Therefore, $\lambda^*$ belongs to $\Lambda$.

Substituting $\lambda^*=(\lambda_1^*,\lambda_2^*)$ into \eqref{Bonf power suff cond} yields 
\begin{align*}
    \|w\cdot (p-q)\|_{2}^2 \gtrsim \left(\frac{\sqrt{\log\log(n+m)}}{n+m}\right)^{\frac{2\bar{s}}{2\bar{s}+1}}.
\end{align*}
Consequently, by \eqref{Bonf foundation}, we have 
\begin{align*}
    \rho_{n+m}
\left(
\Delta_{\alpha,\Lambda}^{\mathrm{Bonf}},
\Sc_2^{s_1,s_2}(R),
\beta,w,M,N,L
\right) \lesssim \left(\frac{\sqrt{\log\log(n+m)}}{n+m}\right)^{\frac{\bar{s}}{2\bar{s}+1}}.
\end{align*}
This completes the proof of the second statement of Proposition \ref{prop::aggtest properties}.
\end{proof}

\section{Proofs for Sections~\ref{sec::Topological properties of cech on the circle},  \ref{Appendix::Homotopy equivalence Cech Circle}}
\label{Appendix::proof topology}
\subsection{Proof of Proposition \ref{Appendix::homotopy_cech_of_circle}}
We first introduce the notation used in the proof.
\paragraph*{Notation used in this proof.}
\begin{enumerate}
    \item Recall that \(\Xc\) is a nonempty finite subset of \(S^1(r)\).
    
    \item For a set \(A\subseteq \mathbb{R}^2\), \(\partial A\) is the boundary
    of \(A\), and \(\mathrm{int}(A)\) is the interior of \(A\).
    
    \item For \(x,y\in \mathbb{R}^2\), 
    \[
        [x,y]:=\{tx+(1-t)y: t\in [0,1]\},
    \]
    and 
    \[
    (x,y):=\{tx+(1-t)y: t\in (0,1)\}.
    \]
    \item Let \(r>0\). For \(x,y \in S^1(r)\), \(x\sim y\) denotes that
    \(x\) and \(y\) are adjacent in the finite set \(\Xc\); for the definition,
    see Definition \ref{defn::adjacent} in the main text.
    
    \item For a set \(A\subset \mathbb{R}^2\), \(\mathrm{conv}(A)\) is the
    convex hull of \(A\), $\ie$,
    \[\mathrm{conv}(A):= \left\{\sum_{i=1}^n\lambda_ix_i: x_i\in A,\sum_{i=1}^n\lambda_i=1,\lambda_i\geq 0,n\in \Nb\right\}.\]
    Moreover, we let
    \[
        P_{\Xc}:=\mathrm{conv}(\Xc).
    \]
    
    \item For three points \(x,y,z\in \mathbb{R}^2\) such that
    \(\{x-z,y-z\}\) is linearly independent, let
    \[
        \triangle xyz := \mathrm{conv}(\{x,y,z\}).
    \]
    
    \item For three points \(x,y,z\in\mathbb{R}^2\), we denote by
\(\angle xyz\) the angle at \(y\) formed by the two segments
\([y,x]\) and \([y,z]\), \ie,
\[
    \angle xyz
    :=
    \arccos\left(
    \frac{\langle x-y,z-y\rangle}
    {\|x-y\|\,\|z-y\|}
    \right).
\]
    
    \item The two-dimensional closed ball centred at \(x\in \mathbb{R}^2\)
    with radius \(r>0\) is defined by
    \[
        \mathcal{B}(x,r):= \{y\in \mathbb{R}^2: \|x-y\|_2\leq r\}.
    \]
\end{enumerate}
We begin the proof by stating the claims and the lemmas that will be used throughout.
\begin{claim}\label{Appendix::half plan claim}
For $x\neq y\in S^{1}(r)$, define a closed subset $H_{x,y}\subseteq \Rb^2$ by
\[
H_{x,y}\coloneqq\left\{ v\in\mathbb{R}^{2}:(x+y)^{\top}v\leq r^{2}+x^{\top}y\right\}. 
\]
Then, we have that
\begin{enumerate}
\item $[x,y]\subseteq\partial H_{x,y}$,
\item If $x+y\neq0$, then $0\in \mathrm{int}(H_{x,y})$,
\item If $\mathcal{X}\subseteq H_{x,y}$, then  $P_{\mathcal{X}}\subseteq H_{x,y}$,
\item If $x,y\in \Xc$ satisfy $x\sim y$, then $\mathcal{X}\subseteq H_{x,y}$ and $P_{\mathcal{X}}\subseteq H_{x,y}$.
\end{enumerate}
\end{claim}
\begin{proof}\noindent
\begin{enumerate}
    \item Both $x$ and $y$ satisfies $(x+y)^{\top}x=r^{2}+x^{\top}y$ and $(x+y)^{\top}y=r^{2}+x^{\top}y$, so $x,y\in\partial H$. Then, $(x+y)^T(tx+(1-t)y)=r^2+x^Ty,$ for any $t\in [0,1].$ This implies $[x,y]\subseteq \partial{H_{x,y}}$.
    \item The condition $x+y\neq0$ ($x,y$ are not an antipodal pair) implies that $x^{\top}y>-r^{2}$, therefore, we have $0\in \mathrm{int}(H_{x,y})$.
    \item Since $H_{x,y}$ is convex and $\Xc \subseteq H_{x,y}$, we have $P_{\Xc}=\mathrm{conv}(\Xc)\subseteq H_{x,y}.$
    \item Suppose that there exists an element $z \in \Xc$ such that $z\notin H_{x,y}$. By simple computations, one can show that $\|x-y\|_{_2}> \max \{\|x-z\|_2,\|y-z\|_2\}.$ This implies that the point $z$ lies in the arc joining $x$ and $y$. Therefore, it yields a contradiction with the assumption that $x\sim y$.
Then, $P_{\mathcal{X}}\subset H_{x,y}$ as well, by the third statement.
\end{enumerate}
This completes the proof of Claim \ref{Appendix::half plan claim}.
\end{proof}
For each nonzero $x\in P_{\mathcal{X}}\setminus\{0\}$, let $\mathcal{I}_{x}\coloneqq\{\lambda\geq1:\lambda x\in P_{\mathcal{X}}\}$.
Since both $\{\lambda x\in\mathbb{R}^{2}:\lambda\geq1\}$ and $P_{\mathcal{X}}$
are closed and convex, 
\[
\{\lambda x\in P_{\mathcal{X}}:\lambda\geq1\}=\{\lambda x\in\mathbb{R}^{2}:\lambda\geq1\}\cap P_{\mathcal{X}}
\]
is closed and convex as well. Further, since $P_{\mathcal{X}}$ is
bounded, this set is also bounded. Then this set corresponds to $\mathcal{I}_{x}$
under the linear map $\lambda\mapsto\lambda x$ (which is bijective
since $x\neq0$), so $\mathcal{I}_{x}$ is closed, convex, and bounded.
Since $\mathcal{I}_{x}$ is a subset of $\mathbb{R}$, $\mathcal{I}_{x}$
is a closed and bounded interval.
Define the radial projection map $\rho_{\Xc}:P_{\mathcal{X}}\setminus\{0\}\to P_{\mathcal{X}}\setminus\{0\}$
by
\[
\rho_{\Xc}(x)\coloneq\left(\max\mathcal{I}_{x}\right)x.
\]
Note that $\max\mathcal{I}_{x}\geq1$, so $x\neq0$ implies $\rho_{\Xc}(x)\neq0$.
Let
\[
B_{\mathcal{X}}\coloneqq \mathrm{Im}(\rho_{\Xc}),
\]
where $\mathrm{Im}(f)$ denotes the image set of a function $f$. Figure \ref{Appendix::fig:imageofrho} provides visual illustrations  for the map $\rho_{\Xc}$.
\begin{figure}[h]
    \centering
\includegraphics[width=\textwidth]{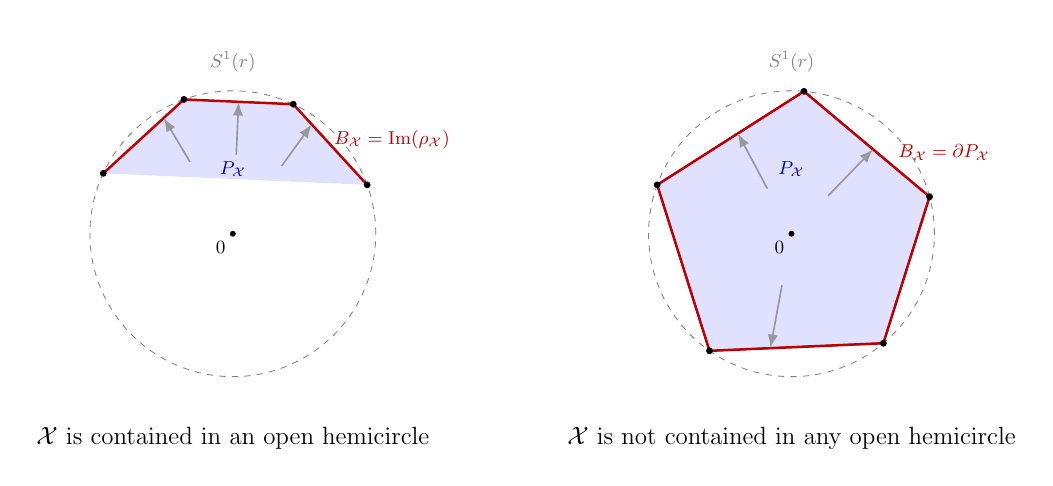}
    \caption{
    The radial projection map $\rho_{\Xc}$ when
(Left) $\Xc$ is contained in an open hemicircle and
(Right) $\Xc$ is not contained in any open hemicircle. 
    }
\label{Appendix::fig:imageofrho}
\end{figure}



\begin{claim}\label{Appendix::Claim::x1x2}
Let $x_{0}\in S^{1}(r)$ be satisfying $x_{0}\notin\mathcal{X}$ and
$(0,x_{0})\cap P_{\mathcal{X}}\neq\emptyset$. Let $w_{0}\neq0\in\mathbb{R}^{2}$
be a vector satisfying $w_{0}\perp x_{0}$, and let $M_{1}\coloneqq\left\{ v\in\mathbb{R}^{2}:w_{0}^{\top}v<0\right\} $,
$M_{2}\coloneqq\left\{ v\in\mathbb{R}^{2}:w_{0}^{\top}v>0\right\} $.
Then, we have that
\begin{enumerate}[label=(\alph*)]
\item\label{Appendix::Claime2::1}  Both $M_{1}\cap\mathcal{X}$ and $M_{2}\cap\mathcal{X}$ are nonempty.
\end{enumerate}
Furthermore, for 
\[
x_{1}\coloneqq \underset{x\in\mathcal{X}\cap M_{1}}{\mathrm{argmin}}\left\{ \left\Vert x-\rho_{\Xc}(z)\right\Vert _{2}\right\},\qquad x_{2}\coloneqq\underset{x\in\mathcal{X}\cap M_{2}}{\mathrm{argmin}}\left\{ \left\Vert x-\rho_{\Xc}(z)\right\Vert _{2}\right\},
\]
it holds that
\begin{enumerate}[resume,label=(\alph*)]
\item\label{Appendix::Claime2::2}  The path-connected component of $S^{1}(r)\setminus\{x_{1},x_{2}\}$
containing $x_{0}$ has empty intersection with $\mathcal{X}$. 
\item\label{Appendix::Claime2::3}  $x_{0}\in H_{x_{1},x_{2}}^{\complement}$ and $P_{\mathcal{X}}\subset H_{x_{1},x_{2}}$.
\item\label{Appendix::Claime2::4} $(0,x_{0})\cap P_{\mathcal{X}}\subset\triangle0x_{1}x_{2}$.
\item\label{Appendix::Claime2::5}  $x_{1}\sim x_{2}$.

\end{enumerate}
\end{claim}
\begin{proof}
\ref{Appendix::Claime2::1}.
Suppose either $M_{1}\cap\mathcal{X}$ or $M_{2}\cap\mathcal{X}$ is empty,
say $M_{2}\cap\mathcal{X}=\emptyset$. Then, $\mathcal{X}\subset M_{1}\cup\{\lambda z:\lambda\leq0\}$. 
Since $M_{1}\cup\{\lambda z:\lambda\leq0\}$ is convex, we have that 
\[
P_{\mathcal{X}}\subset M_{1}\cup\{\lambda z:\lambda\leq0\}.
\]
It implies that $z\notin P_{\mathcal{X}}$, since $z\notin M_{1}\cup\{\lambda z:\lambda\leq0\} $. This is a contradiction. 
Therefore, $M_{2}\cap\mathcal{X}$ is
nonempty. Similarly, we have $M_{1}\cap\mathcal{X}\neq\emptyset$ as well.

\ref{Appendix::Claime2::2}. Let $A_{x_{0}}$ be the path-connected component of $S_{1}\backslash\{x_{1},x_{2}\}$
containing $x_{0}$. Since for all $y\in A_{x_{0}}\cap M_{1}$, $\left\Vert y-x_{0}\right\Vert \leq\left\Vert x_{1}-x_{0}\right\Vert $,
the definition of $x_{1}$ implies that $A_{x_{0}}\cap M_{1}\cap\mathcal{X}=\emptyset$.
Similarly, $A_{x_{0}}\cap M_{2}\cap\mathcal{X}=\emptyset$ holds as
well.  Since $(M_{1}\cup M_{2}\cup\{x_{0}\})\cap A_{x_{0}}=A_{x_{0}}$,
we have that 
\[
A_{x_{0}}\cap\mathcal{X}=\left(A_{x_{0}}\cap M_{1}\cap\mathcal{X}\right)\cup\left(A_{x_{0}}\cap M_{2}\cap\mathcal{X}\right)\cup\left(A_{x_{0}}\cap\{x_{0}\}\cap\mathcal{X}\right)=\emptyset.
\]

\ref{Appendix::Claime2::3}.
The definition of $x_{1}$ and $x_{2}$ implies either path-connected
component of $S_{1}\setminus\{x_{1},x_{2}\}$ containing $x_{0}$
has empty intersection with $\mathcal{X}$; in other words, $x_{0}\in H_{x_{1},x_{2}}^{\complement}$
implies $\mathcal{X}\subset H_{x_{1},x_{2}}$, and $x_{0}\in \mathrm{int}(H_{x_{1},x_{2}})$
implies $\mathcal{X}\subset\overline{H_{x_{1},x_{2}}^{\complement}}$.

Now, suppose $x_{0}\in \mathrm{int}(H_{x_{1},x_{2}})$, so that $\mathcal{X}\subset\overline{H_{x_{1},x_{2}}^{\complement}}$.
Then $P_{\mathcal{X}}\subset\overline{H_{x_{1},x_{2}}^{\complement}}$, due to convexity of $\overline{H_{x_{1},x_{2}}^{\complement}}$, i.e.,
$P_{\mathcal{X}}\cap \mathrm{int}(H_{x_{1},x_{2}})=\emptyset$. Since $x_{0}\in \mathrm{int}(H_{x_{1},x_{2}})$
implies $(0,x_{0})\subset \mathrm{int}(H_{x_{1},x_{2}})$, this contradicts the assumption $(0,x_{0})\cap P_{\mathcal{X}}\neq\emptyset$.
Hence, it is indeed the case that $x_{0}\in H_{x_{1},x_{2}}^{\complement}$
and $\mathcal{X}\subset H_{x_{1},x_{2}}$, and due to convexity,
$P(\mathcal{X})\subset H_{x_{1},x_{2}}.$

\ref{Appendix::Claime2::4}.
We will first show that $(0,x_{0})\cap[x_{1},x_{2}]\neq\emptyset$.
Since $0\in H_{x_{1},x_{2}}$ and $x_{0}\in H_{x_{1},x_{2}}^{\complement}$,
this implies that 
\[
(0,x_{0})\cap\partial H_{x_{1},x_{2}}\neq\emptyset.
\]
Moreover, since $(0,x_{0})\subset\mathcal{B}(0,r)$, we have that 
\[
(0,x_{0})\cap\partial H_{x_{1},x_{2}}\subset\mathcal{B}(0,r)\cap\partial H_{x_{1},x_{2}}=[x_{1},x_{2}],
\]
which implies that $(0,x_{0})\cap[x_{1},x_{2}]\neq\emptyset$. Now, let
\[
\bar{x}_{0}\coloneqq(0,x_{0})\cap[x_{1},x_{2}].
\]
The two facts that $P_{\mathcal{X}}\subset H_{x_{1},x_{2}}$ and $x_{0}\in H_{x_{1},x_{2}}^{\complement}$,
imply that 
\[
(0,x_{0})\cap P_{\mathcal{X}}=(0,\bar{x}_{0}]\cap P_{\mathcal{X}}.
\]
Now, since $\bar{x}_{0}\in[x_{1},x_{2}]$, there exists $\xi\in[0,1]$
such that 
\[
\bar{x}_{0}=\xi x_{1}+(1-\xi)x_{2}.
\]
Then, for any $y\in(0,\bar{x}_{0}]$, there exists $\lambda\in[0,1]$
such that $y=\lambda\bar{x}_{0}$. Therefore, we have that 
\[
y=\lambda\xi x_{1}+\lambda(1-\xi)x_{2}+(1-\lambda)0,\quad\text{with }\lambda\xi+\lambda(1-\xi)+(1-\lambda)=1.
\]
Hence, $y\in\triangle0x_{1}x_{2}$, which implies that 
\[
(0,x_{0})\cap P_{\mathcal{X}}\subset\triangle0x_{1}x_{2}.
\]
\ref{Appendix::Claime2::5}
We will first show that $x_{1}+x_{2}\neq0$. If $x_{1}+x_{2}=0$,
then $0\in\partial H_{x_{1},x_{2}}$. Then since $x_{0}\in H_{x_{1},x_{2}}^{\complement}$,
we have 
\[
(0,x_{0})\subset H_{x_{1},x_{2}}^{\complement}.
\]
Since $P_{\mathcal{X}}\subset H_{x_{1},x_{2}}$ from (b), this implies
that 
\[
(0,x_{0})\cap P_{\mathcal{X}}=\emptyset,
\]
which contradicts the assumption $(0,x_{0})\cap P_{\mathcal{X}}\neq\emptyset$.
Hence we have 
\[
x_{1}+x_{2}\neq0,
\]
and this implies that $0\in \mathrm{int}(H_{x_{1},x_{2}})$.

Now from \ref{Appendix::Claime2::2}
the path-connected component of $S_{1}\setminus\{x_{1},x_{2}\}$ containing
$x_{0}$ has empty intersection with $\mathcal{X}$. Since $0\in \mathrm{int}(H_{x_{1},x_{2}})$,
the component contained in $H_{x_{1},x_{2}}^{\complement}$ is the shorter
arc. 
Since $x_{0}\in H_{x_{1},x_{2}}^{\complement}$, the shorter arc is the one containing $x_{0}$. 
Therefore, we have that 
\[
x_{1}\sim x_{2},
\] 
because otherwise the component containing $x_0$ has non-empty intersection with $\Xc$.
The proof of Claim \ref{Appendix::Claim::x1x2} is completed.
\end{proof}

\begin{lem}\label{Appendix::lem convexhull inclusion} For the convex hull $P_{\Xc}$, the following relation holds:
\[
P_{\mathcal{X}}\subseteq \left(\bigcup_{x\in\mathcal{X}}[0,x]\right)\bigcup\left(\bigcup_{x,y\in\mathcal{X},x\sim y}\triangle0xy\right).
\]
\end{lem}
\begin{proof}
Suppose $z\in P_{\mathcal{X}}\setminus\bigcup_{x\in\mathcal{X}}[0,x]$.
Then, $\rho_{\mathcal{X}}(z)\notin\mathcal{X}$. Claim \ref{Appendix::Claim::x1x2} implies
that there exist $x,y\in\mathcal{X}$ with $x\sim y$ such that 
\[
z\in\triangle0xy.
\]
This completes the proof of Lemma \ref{Appendix::lem convexhull inclusion}.
\end{proof}
\begin{claim}\label{Appendix::Claim characterization of the boundary} For any $x\in P_{\Xc}\setminus \{0\}$, 
$\rho_{\Xc}(x)=x$ if and only if $x\in B_{\mathcal{X}}$.
\end{claim}

\begin{proof}
If $x\in B_{\mathcal{X}}$, then $x=\left(\max\mathcal{I}_{y}\right)y$
for some $y\in P_{\mathcal{X}}$. Then, for any $\xi>\max\mathcal{I}_{y}$,
$\xi y\notin P_{\mathcal{X}}$. Therefore, for any $\lambda>1$, 
\[
\lambda x=\left(\lambda\max\mathcal{I}_{y}\right)y\notin P_{\mathcal{X}}.
\]
This implies that $\max\mathcal{I}_{x}=1$, and $\rho_{\Xc}(x)=x$.
Next, if $x\notin B_{\mathcal{X}},$ then $x\neq\rho_{\Xc}(y)$
for any $y\in P_{\mathcal{X}}\setminus\{0\}$. Hence, we obtain that $x\neq\rho_{\Xc}(x)$. This completes the proof of Claim \ref{Appendix::Claim characterization of the boundary}.
\end{proof}
\begin{lem}\label{Appendix::lem::Boundary equality}
When $\mathcal{X}$ is regarded as a set of vertices, $B_{\mathcal{X}}$ is
a union of vertices and neighbor edges, i.e.,
\[
B_{\mathcal{X}}=\mathcal{X}\bigcup\left(\bigcup_{x,y\in\mathcal{X},x\sim y}[x,y]\right).
\]
\end{lem}

\begin{proof}
From Claim \ref{Appendix::Claim characterization of the boundary},
we have that $x\in B_{\mathcal{X}}$ if and only if $\rho_{\Xc}(z)=z$.
Hence it is equivalent to show that 
\begin{equation}\label{eq:radial_function_dichotomy}
\begin{aligned}
\rho_{\mathcal{X}}(z)=z 
& \qquad\text{if } 
z\in\mathcal{X}\cup
\left(\bigcup_{x,y\in\mathcal{X},\,x\sim y}[x,y]\right),\\
\rho_{\mathcal{X}}(z)\neq z 
& \qquad\text{otherwise}.
\end{aligned}
\end{equation}

First, we assume $z\in\mathcal{X}\bigcup\left(\bigcup_{x,y\in\mathcal{X},x\sim y}[x,y]\right)$,
and we will show $\rho_{\mathcal{X}}(z)=z$. Note that either $z\in\mathcal{X}$
or $z\in[x,y]$ for some $x,y\in\mathcal{X}$ with $x\sim y$ holds.

Suppose $z\in\mathcal{X}$, then $\left\Vert z\right\Vert _{2}=r$,
and for any $v\in P_{\mathcal{X}}$, $\left\Vert v\right\Vert _{2}\leq r$.
Hence, $\max\mathcal{I}_{z}=1$, and 
\[
\rho_{\Xc}(z)=z.
\]

Suppose $z\in[x,y]$, for some $x,y\in\mathcal{X}$ with $x\sim y$.
Since $x\sim y$, we have that $x\neq y$, which implies that $x^{\top}y>-r^{2}$.
Recall
\[
H_{x,y}=\left\{ v\in\mathbb{R}^{2}:(x+y)^{\top}v\leq r^{2}+x^{\top}y\right\} ,
\]
then, by Claim \ref{Appendix::half plan claim}, $H_{x,y}$ is a closed
half plane, $0\in\mathrm{int}(H_{x,y})$, and $[x,y]\subset\partial H_{x,y}$.
Furthermore, $x\sim y$ implies that 
\[
P_{\mathcal{X}}=\mathrm{conv}(\mathcal{X})\subset H_{x,y}.
\]
Since $0\in\mathrm{int}(H_{x,y})$ and $z\in\partial H_{x,y}$, for
any $\lambda>1$, $\lambda z\notin H$, which implies $\lambda z\notin P_{\mathcal{X}}$.
Hence $\max\mathcal{I}_{z}=1$, and 
\[
\rho_{\Xc}(z)=z.
\]

Now, we assume $z\notin\mathcal{X}\bigcup\left(\bigcup_{x,y\in\mathcal{X},x\sim y}[x,y]\right)$,
and we will show $\rho_{\mathcal{X}}(z)\neq z$. Note that either $z\in[0,x]$
for some $x\in\mathcal{X}$, or $z\in\triangle0xy$ for some $x,y\in\mathcal{X}$
with $x\sim y$, by Lemma \ref{Appendix::lem convexhull inclusion}. 

Suppose $z\in[0,x]$, for some $x\in\mathcal{X}$. Since $z\neq x$,
then there exists some $\lambda_{*}>1$ such that $\lambda_{*}z=x$.
Hence, $\max\mathcal{I}_{z}>1$, and 
\[
\rho_{\Xc}(z)\neq z.
\]

Suppose $z\in\triangle0xy$, for some $x,y\in\mathcal{X}$ with $x\sim y$.
Since $z\notin[x,y]$, then there exists some $\lambda_{*}>1$ such
that $\lambda_{*}z\in[x,y]$. Hence, $\max\mathcal{I}_{z}>1$, and
\[
\rho_{\Xc}(z)\neq z.
\]

We have shown \eqref{eq:radial_function_dichotomy}. Consequently, by
Claim \ref{Appendix::Claim characterization of the boundary}, this
completes the proof of Lemma \ref{Appendix::lem::Boundary equality}. 
\end{proof}

\begin{claim}\label{Appendix::Claim::projection distance comparison}
For any $x\in P_{\mathcal{X}}\setminus\{0\}$, let $x_{0}\in\mathcal{X}$
be one of the closest points to $x$ among $\mathcal{X}$ (if there are multiple such points, choose any one of them). Then either of the following holds: 
\[
\rho_{\mathcal{X}}(x)=x_{0},\qquad\text{or}\qquad\angle x\rho_{\mathcal{X}}(x)x_{0}\in\left[\frac{\pi}{2},\pi\right].
\]
Correspondingly, we have 
\[
\left\Vert \rho_{\Xc}(x)-x_{0}\right\Vert _{2}\leq\left\Vert x-x_{0}\right\Vert _{2}.
\]
Visual illustrations of Claim \ref{Appendix::Claim::projection distance comparison} are presented in Figure \ref{Appendix::fig:rho projection}.
\end{claim}
\begin{figure}[h]
    \centering
\includegraphics[width=\textwidth]{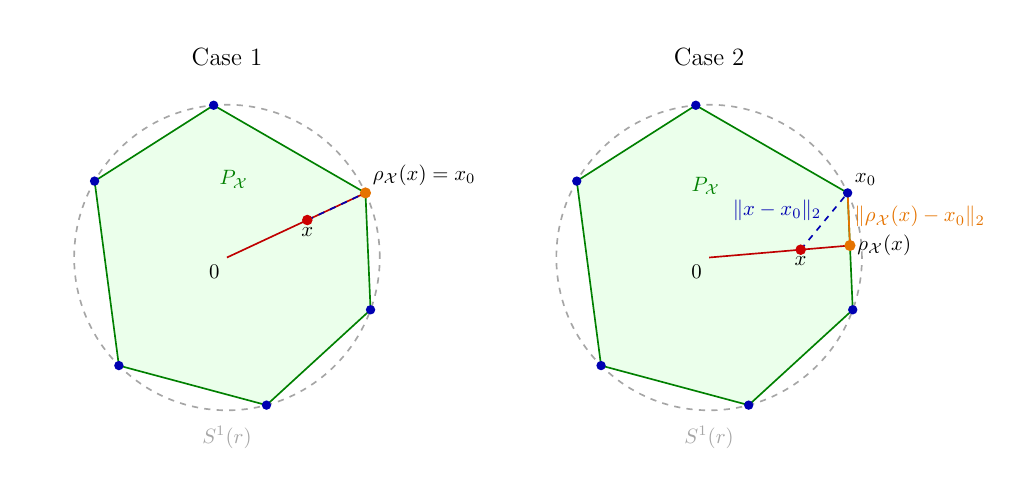}
    \caption{The two cases used in Claim 4:
(Case 1) $\rho_{\Xc}(x)=x_0$ and
(Case 2) $\rho_{\Xc}(x)\neq x_0$.
}
\label{Appendix::fig:rho projection}
\end{figure}

\begin{proof}
Since $\rho_{\Xc}(x)\in B_{\mathcal{X}}=\mathcal{X}\bigcup\left(\bigcup_{x,y\in\mathcal{X},x\sim y}[x,y]\right)$, by Lemma \ref{Appendix::lem::Boundary equality},
either $\rho_{\Xc}(x)\in\mathcal{X}$ or $\rho_{\Xc}(x)\in\bigcup_{x,y\in\mathcal{X},x\sim y}[x,y]$ holds.

If $\rho_{\Xc}(x)\in\mathcal{X}$, then $\rho_{\Xc}(x)=x_{0}$. For this case, we obviously have 
\[
\left\Vert \rho_{\Xc}(x)-x_{0}\right\Vert _{2}\leq\left\Vert x-x_{0}\right\Vert _{2}.
\]

Suppose $\rho_{\Xc}(x)\in\bigcup_{x,y\in\mathcal{X},x\sim y}[x,y]$.
Then, there exists some $x_{1}\in\mathcal{X}$ such that $\rho_{\Xc}(x)\in[x_{0},x_{1}]$.
This implies that
\[
\left\Vert x-x_{0}\right\Vert _{2}\leq\left\Vert x-x_{1}\right\Vert _{2}.
\]
Then, $x^{\top}x_{0}\leq x^{\top}x_{1}$, which implies that $\angle x_{0}0\rho_{\Xc}(x)\leq\angle x_{1}0\rho_{\Xc}(x)$.
Since $\angle0x_{0}x_{1}=\angle0x_{1}x_{0}$, we have that 
\[
\angle0\rho_{\Xc}(x)x_{0}\geq\angle0\rho_{\Xc}(x)x_{1}.
\]
Therefore, we can conclude that  
\[
\angle x\rho_{\Xc}(x)x_{0}=\angle0\rho_{\Xc}(x)x_{0}\geq\frac{\pi}{2},
\]
where the inequality holds, since $\angle0\rho_{\Xc}(x)x_{0}+\angle0\rho_{\Xc}(x)x_{1}=\pi$. Lastly, by the law of cosines, the fact that $\angle x\rho_{\Xc}(x)x_{0}\in [ \frac{\pi}{2},\pi]$ implies
\begin{align*}
    \left\Vert \rho_{\Xc}(x)-x_{0}\right\Vert _{2}\leq\left\Vert x-x_{0}\right\Vert _{2}.
\end{align*}
This completes the proof of Claim \ref{Appendix::Claim::projection distance comparison}.
\end{proof}
We will define three homotopy maps as follows. For the definition of homotopy maps, see Section \ref{Appendix::Background_homotopy}. 

Define $F_{1}:\mathbb{R}^{2}\times[0,1]\to\mathbb{R}^{2}$
by
\[
F_{1}(x,t)\coloneq(1-t)x+t\Pi_{P_{\mathcal{X}}}(x),
\]
where $\Pi_{P_{\Xc}}:\Rb^{2}\rightarrow P_{\Xc}$ is the projection map onto $P_{\Xc}$. Since $P_{\Xc}$ is nonempty closed convex set, the projection map is well-defined. Define 
$F_{2}:P_{\mathcal{X}}\setminus\{0\}\times[0,1]\to P_{\mathcal{X}}\setminus\{0\}$
by 
\[
F_{2}(x,t)\coloneq(1-t)x+t\rho_{\Xc}(x).
\]
Then, we define a homotopy map 
$F:\mathbb{R}^{2}\setminus\Pi_{P_{\mathcal{X}}}^{-1}\{0\}\times[0,1]\to\mathbb{R}^{2}$
by $F\coloneq F_{1}*F_{2}$, i.e., 
\begin{align}\label{Appendix::eq:: the homotopy map F}
F(x,t)\coloneq\begin{cases}
F_{1}(x,2t), & 0\leq t\leq\frac{1}{2},\\
F_{2}(\Pi_{P_{\mathcal{X}}}(x),2t-1), & \frac{1}{2}\leq t\leq1.
\end{cases}
\end{align}
\begin{claim}\label{Appendix::Claim::Continuity and well definedness of F}
The map $F$ in \eqref{Appendix::eq:: the homotopy map F} is well-defined
and continuous.
\end{claim}
\begin{proof}
To show that $F$ is well defined and continuous, we need to show
that for all $x\in\mathbb{R}^{2}\backslash\Pi_{P_{\mathcal{X}}}^{-1}\{0\}$,
$\Pi_{P_{\mathcal{X}}}(x)\neq0$ and 
\[
F_{1}(x,1)=F_{2}(\Pi_{P_{\mathcal{X}}}(x),0).
\]
The first is obvious since $x\in\mathbb{R}^{2}\backslash\Pi_{P_{\mathcal{X}}}^{-1}\{0\}$
is equivalent to that $\Pi_{P_{\mathcal{X}}}(x)\neq0$. The second also
holds, since 
\[
F_{1}(x,1)=\Pi_{P_{\mathcal{X}}}(x)=F_{2}(\Pi_{P_{\mathcal{X}}}(x),0).
\]
Therefore, $F$ is well-defined. 
Furthermore, by the pasting lemma, $F$ is continuous.
This completes of Claim \ref{Appendix::Claim::Continuity and well definedness of F}.
\end{proof}
\begin{claim}\label{Appendix::Claim::projection line inclusion}
Fix $t>0$. For any $y\in\bigcup_{x\in\mathcal{X}}\mathcal{B}(x,t)$,
\[
[y,\Pi_{P_{\mathcal{X}}}(y)]\subset\bigcup_{x\in\mathcal{X}}\mathcal{B}(x,t).
\]
\end{claim}

\begin{proof}
If $y\in P_{\mathcal{X}}$, then it is trivial. Therefore, we only need
to show the case where $y\in\mathbb{R}^{2}\setminus P_{\mathcal{X}}$.
Note that $\mathbb{R}^{2}\setminus P_{\mathcal{X}}$ factorizes as
follows: 
\[
\mathbb{R}^{2}\setminus P_{\mathcal{X}}=\left(\bigcup_{x,x'\in\mathcal{X},(x,x')\subset\partial P_{\mathcal{X}}}\Pi_{P_{\mathcal{X}}}^{-1}\left(x,x'\right)\setminus P_{\mathcal{X}}\right)\bigcup\left(\bigcup_{x\in\mathcal{X}}\Pi_{P_{\mathcal{X}}}^{-1}\{x\}\setminus P_{\mathcal{X}}\right).
\]

First, suppose $y\in\bigcup_{(x,x')\subset\partial P_{\mathcal{X}}}\Pi_{P_{\mathcal{X}}}^{-1}\left(x,x'\right)\setminus P_{\mathcal{X}}$.
Then, $y\in\Pi_{P_{\mathcal{X}}}^{-1}\left(x_0,x_1\right)\setminus P_{\mathcal{X}}$
for some $(x_0,x_1)\subset\partial P_{\mathcal{X}}$, so either
$y\in\mathcal{B}(x_{0},t)$ or $y\in\mathcal{B}(x_{1},t)$. Since $\Pi_{P_{\Xc}}(y)$ is the same as the orthogonal projection of $y$ onto $[x,x']$, we have that
\[
y-\Pi_{P_{\mathcal{X}}}(y)\perp x_{i}-\Pi_{P_{\mathcal{X}}}(y),\qquad i=0,1.
\]
Therefore, we obtain that
\[
\left\Vert \Pi_{P_{\mathcal{X}}}(y)-x_{i}\right\Vert _{2}\leq\left\Vert y-x_{i}\right\Vert _{2},\qquad i=0,1.
\]
This implies that $\Pi_{P_{\mathcal{X}}}(y)\in\mathcal{B}(x_{i},t)$. Since $\mathcal{B}(x_{i},t)$ is convex, it holds that 
\[
[y,\Pi_{P_{\mathcal{X}}}(y)]\subset\mathcal{B}(x_{i},t)\subset\bigcup_{x\in\mathcal{X}}\mathcal{B}(x,t).
\]

Second, suppose $y\in\bigcup_{x\in\mathcal{X}}\Pi_{P_{\mathcal{X}}}^{-1}\{x\}\setminus P_{\mathcal{X}}$,
and let $x_{0}\coloneqq\Pi_{P_{\mathcal{X}}}(y)$. Then 
\[
\Pi_{P_{\mathcal{X}}}(y)=x_{0}\in\mathcal{B}(x_{0},t),
\]
and since $\mathcal{B}(x_{0},t)$ is convex, this implies that 
\[
[y,\Pi_{P_{\mathcal{X}}}(y)]\subset\mathcal{B}(x_{0},t)\subset\bigcup_{x\in\mathcal{X}}\mathcal{B}(x,t).
\]

Consequently, in any case, $[y,\Pi_{P_{\mathcal{X}}}(y)]\subseteq\bigcup_{x\in\mathcal{X}}\mathcal{B}(x,t)$. This completes the proof of Claim \ref{Appendix::Claim::projection line inclusion}.
\end{proof}
\begin{claim}\label{Appnedix::Claim:: line segement contained in balls}
Fix $t\in(0,r)$. We have that for any $y\in P_{\mathcal{X}}\cap\left(\bigcup_{x\in\mathcal{X}}\mathcal{B}(x,t)\right)$,
\[
[y,\rho_{\Xc}(y)]\subset P_{\mathcal{X}}\cap\left(\bigcup_{x\in\mathcal{X}}\mathcal{B}(x,t)\right).
\]
\end{claim}

\begin{proof}
Let $y\in P_{\Xc}\cap \Bc(x_0,t)$, for some $x_0 \in \Xc$.
The condition $t< r$ implies that $y\neq0$. Hence, Claim \ref{Appendix::Claim::projection distance comparison} implies that 
\[
\rho_{\Xc}(y)\in\mathcal{B}(x_{0},t).
\]
Note that $\rho_{\Xc}(y)\in B_{\mathcal{X}}\subset P_{\mathcal{X}}$.
Since both $\mathcal{B}(x_{0},t)$ and $P_{\mathcal{X}}$ are convex,
$P_{\mathcal{X}}\cap\mathcal{B}(x_{0},t)$ is convex as well, and
this implies that 
\[
[y,\rho_{\Xc}(y)]\subset P_{\mathcal{X}}\cap\mathcal{B}(x_{0},t)\subset P_{\mathcal{X}}\cap\left(\bigcup_{x\in\mathcal{X}}\mathcal{B}(x,t)\right).
\]
This completes the proof of Claim \ref{Appnedix::Claim:: line segement contained in balls}.
\end{proof}
\begin{lem}\label{Appendix::lem::codomain of restricted homopoty F}
Suppose $t<r$. 
The image of the restricted map $F$ to $\bigcup_{x\in\mathcal{X}}\mathcal{B}(x,t)\times[0,1]$ is contained in $\bigcup_{x\in\mathcal{X}}\mathcal{B}(x,t)$, $\ie$,
\[
F|_{\bigcup_{x\in\mathcal{X}}\mathcal{B}(x,t)}:\bigcup_{x\in\mathcal{X}}\mathcal{B}(x,t)\times[0,1]\to\bigcup_{x\in\mathcal{X}}\mathcal{B}(x,t).
\]
\end{lem}

\begin{proof}
It needs to show that 
\begin{equation}
\left(\bigcup_{x\in\mathcal{X}}\mathcal{B}(x,t)\right)\cap\Pi_{P_{\mathcal{X}}}^{-1}\{0\}=\emptyset,\label{eq:unionball_emptyintersection}
\end{equation}
and 
\begin{align}
F\left(\bigcup_{x\in\mathcal{X}}\mathcal{B}(x,t)\times[0,1]\right)\subset\bigcup_{x\in\mathcal{X}}\mathcal{B}(x,t).\label{Appendix::inclusion 1}
\end{align}
From Claim~\ref{Appendix::Claim::projection line inclusion}, we
have that, for each $y\in\bigcup_{x\in\mathcal{X}}\mathcal{B}(x,t)$,
\begin{equation}
\left[y,\Pi_{P_{\mathcal{X}}}(y)\right]\subset\bigcup_{x\in\mathcal{X}}\mathcal{B}(x,t).\label{eq:project_line_inclusion}
\end{equation}

We will first show \eqref{eq:unionball_emptyintersection}. Note that
the condition $t<r$ implies that $0\notin\bigcup_{x\in\mathcal{X}}\mathcal{B}(x,t)$.
Suppose some $y\in\bigcup_{x\in\mathcal{X}}\mathcal{B}(x,t)$ satisfies
$\Pi_{P_{\mathcal{X}}}(y)=0$. Then $\left[y,0\right]\subset\bigcup_{x\in\mathcal{X}}\mathcal{B}(x,t)$
from \eqref{eq:project_line_inclusion}, and this contradicts $0\notin\bigcup_{x\in\mathcal{X}}\mathcal{B}(x,t)$.
Hence we have 
\[
\left(\bigcup_{x\in\mathcal{X}}\mathcal{B}(x,t)\right)\cap\Pi_{P_{\mathcal{X}}}^{-1}\{0\}=\emptyset.
\]

Next, we will show \eqref{Appendix::inclusion 1}. Claim~\ref{Appnedix::Claim:: line segement contained in balls}
implies that, for each $y\in\bigcup_{x\in\mathcal{X}}\mathcal{B}(x,t)$,
\[
\left[\Pi_{P_{\mathcal{X}}}(y),\rho_{\Xc}(\Pi_{P_{\mathcal{X}}}(y))\right]\subset\bigcup_{x\in\mathcal{X}}\mathcal{B}(x,t).
\]
Combining this and \eqref{eq:project_line_inclusion} yields the inclusion
\[
F\left(\bigcup_{x\in\mathcal{X}}\mathcal{B}(x,t)\times[0,1]\right)\subset\bigcup_{x\in\mathcal{X}}\mathcal{B}(x,t).
\]

Hence, the map $F|_{\bigcup_{x\in\mathcal{X}}\mathcal{B}(x,t)}$ can
be regarded as a map: 
\[
F|_{\bigcup_{x\in\mathcal{X}}\mathcal{B}(x,t)}:\bigcup_{x\in\mathcal{X}}\mathcal{B}(x,t)\times[0,1]\to\bigcup_{x\in\mathcal{X}}\mathcal{B}(x,t).
\]
This completes the proof of Lemma \ref{Appendix::lem::codomain of restricted homopoty F}. 
\end{proof}

\begin{cor}\label{Appendix::cor::homotopy equivalence}
Suppose $t< r$, then $\bigcup_{x\in\mathcal{X}}\mathcal{B}(x,t)$
deformation retracts to $\bigcup_{x\in\mathcal{X}}\mathcal{B}(x,t)\cap B_{\mathcal{X}}$
via $F$, and in particular, $\bigcup_{x\in\mathcal{X}}\mathcal{B}(x,t)$
and $\bigcup_{x\in\mathcal{X}}\mathcal{B}(x,t)\cap B_{\mathcal{X}}$
are homotopy equivalent.
\end{cor}

\begin{proof}
By Lemma \ref{Appendix::lem::codomain of restricted homopoty F}, we can consider $F$ at \eqref{Appendix::eq:: the homotopy map F} as a homotopy map 
\[
F:\bigcup_{x\in\mathcal{X}}\mathcal{B}(x,t)\times[0,1]\to\bigcup_{x\in\mathcal{X}}\mathcal{B}(x,t).
\]
Note that the following identities hold:
\begin{align*}
 F\left(\bigcup_{x\in\mathcal{X}}\mathcal{B}(x,t)\times\{0\}\right)=\bigcup_{x\in\mathcal{X}}\mathcal{B}(x,t) \quad \text{and}\quad    F\left(\bigcup_{x\in\mathcal{X}}\mathcal{B}(x,t)\times\{1\}\right)=\bigcup_{x\in\mathcal{X}}\mathcal{B}(x,t)\cap B_{\mathcal{X}}.
\end{align*}
Moreover, the restriction of $F$ to $\bigcup_{x\in\mathcal{X}}\mathcal{B}(x,t)\cap B_{\mathcal{X}}\times \{\alpha\}$ is the identity map on the set $\bigcup_{x\in\mathcal{X}}\mathcal{B}(x,t)\cap B_{\mathcal{X}}$, for any $\alpha \in [0,1]$.
Consequently, we have that  $\bigcup_{x\in\mathcal{X}}\mathcal{B}(x,t)$ deformation retracts
onto $\bigcup_{x\in\mathcal{X}}\mathcal{B}(x,t)\cap B_{\mathcal{X}}$
via $F$, implying that they are homotopy equivalent.
This completes the proof of Corollary \ref{Appendix::cor::homotopy equivalence}.
\end{proof}
We are now ready to specify the homotopy type of $\textrm{\v{C}ech}(\Xc,t)$.
\begin{proof}[Proof of Proposition~\ref{Appendix::homotopy_cech_of_circle}]
From Nerve Theorem in \citet[p.~71]{EdelsbrunnerHarer2010}, we have that 
\[
\textrm{\v{C}ech}(\Xc,t)\simeq\bigcup_{x\in\mathcal{X}}\mathcal{B}(x,t).
\]

For $t<r_{0}$, 
$\bigcup_{x\in\mathcal{X}}\mathcal{B}(x,t)$
consists of finitely many connected components, and each of them is contractible.
Therefore, we have that
\[
\textrm{\v{C}ech}(\Xc,t)\simeq\{x_{1},\ldots,x_{n_{t}}\}.
\]

When $r_{0}\leq  t< r$, by Corollary \ref{Appendix::cor::homotopy equivalence}, it holds that
\begin{align*}
\bigcup_{x\in\mathcal{X}}\mathcal{B}(x,t)& \simeq\bigcup_{x\in\mathcal{X}}\mathcal{B}(x,t) \cap B_{\mathcal{X}}\\
&=B_{\mathcal{X}}\\
& \simeq
\begin{cases}
S^{1},\quad \text{if } \Xc \text{ is not contained in any hemicircle},\\
\{x_{1}\},\quad \text{otherwise}.
\end{cases}
\end{align*}

When $t\geq r$, $P_{\mathcal{X}}\subset\bigcup_{x\in\mathcal{X}}\mathcal{B}(x,t)$,
and then define a homotopy map $\bar{F}_{1}:\bigcup_{x\in\mathcal{X}}\mathcal{B}(x,t)\times[0,1]\to\bigcup_{x\in\mathcal{X}}\mathcal{B}(x,t)$
as $F_{1}$ on $\bigcup_{x\in\mathcal{X}}\mathcal{B}(x,t)$, i.e.,
\[
\bar{F}_{1}(x,t)=(1-t)x+t\Pi_{P_{\mathcal{X}}}(x).
\]
Claim \ref{Appendix::Claim::projection line inclusion} gives that $\bar{F}_{1}$ is well defined. Then, $\bar{F}_{1}$
provides a deformation retraction from $\bigcup_{x\in\mathcal{X}}\mathcal{B}(x,t)$
to $P_{\mathcal{X}}$. Hence, we have that
\[
\bigcup_{x\in\mathcal{X}}\mathcal{B}(x,t)\simeq P_{\mathcal{X}}\simeq\{x_{1}\}.
\]
Consequently, the proof of Proposition \ref{Appendix::homotopy_cech_of_circle} is completed. \end{proof}

\subsection{Proof of Proposition \ref{prop::phofcircle}}\label{Appendix::Proofofproposition phcircles}
We adopt the same notation and definitions used in the proof of Proposition \ref{Appendix::homotopy_cech_of_circle}. 
In Proposition \ref{Appendix::homotopy_cech_of_circle},
we compute the homotopy type of $\textrm{\v{C}ech}(\Xc,t)$, for each $t>0$. 
Therefore, when $\Xc$ is not contained in any open hemicircle, the homology groups $\mathrm{H}_{1}\left(\bigcup_{x\in \Xc}\Bc(x,t_1)\right)$ and $\mathrm{H}_{1}\left(\bigcup_{x\in \Xc}\Bc(x,t_2)\right)$ have dimension one, 
for $r_0\leq t_1<t_2< r$. 
To obtain a persistence diagram from the homology groups, the persistence of each generator in the groups must be computed. 
Concretely, if the generator $g_1$ of $\mathrm{H}_{1}\left(\bigcup_{x\in \Xc}\Bc(x,t_1)\right)$ becomes zero at the homology group $\mathrm{H}_{1}\left(\bigcup_{x\in \Xc}\Bc(x,t_2)\right)$, then the persistence diagram has at least two points representing $g_1$ and the generator $g_2$ of $\mathrm{H}_{1}\left(\bigcup_{x\in \Xc}\Bc(x,t_2)\right)$.
Therefore, in this proof, we aim to show that all homology groups $\mathrm{H_1}\left(\bigcup_{x\in \Xc}\Bc(x,t)\right)$ are isomorphic, for $r_0\leq t< r$.

Let $\mathcal{B}(t)\coloneqq\bigcup_{x\in\mathcal{X}}\mathcal{B}(x,t)$
and $\mathcal{B}_{\partial B}(t)\coloneqq\bigcup_{x\in\mathcal{X}}\mathcal{B}(x,t)\cap\partial B_{\mathcal{X}}$, for notational simplicity.
Denote by $\imath_{A\to B}$ the inclusion map from $A$ to $B$, for sets $A$ and $B$ such that $A\subseteq B$, and let $\rho:= \rho_{\mathcal{X}}$. Furthermore, for a continuous map \(f:X\to Y\), we denote by
\[
f_*:\mathrm{H}_1(X)\to \mathrm{H}_1(Y),
\]
the homomorphism induced by \(f\) on the first homology groups.
\begin{lem}\label{Appendix::lem::inclusion diagram commutes} 
For $r_0\leq t_{1}<t_{2}< r$, the following diagram commutes:

\[
\begin{tikzcd}[
  column sep=7.5em,
  row sep=5.5em,
  cells={nodes={inner sep=2pt}}
]
\mathcal{B}(t_1)
  \arrow[r, "\iota_{\mathcal{B}(t_1)\to \mathcal{B}(t_2)}"]
  \arrow[d, shift left=1.0ex, "\rho|_{\mathcal{B}(t_1)}"]
&
\mathcal{B}(t_2)
  \arrow[d, shift left=1.0ex, "\rho|_{\mathcal{B}(t_2)}"]
\\
\mathcal{B}_{\partial B}(t_1)
  \arrow[u, shift left=1.0ex, "\iota_{\mathcal{B}_{\partial B}(t_1)\to \mathcal{B}(t_1)}"]
  \arrow[r, "\iota_{\mathcal{B}_{\partial B}(t_1)\to \mathcal{B}_{\partial B}(t_2)}"']
&
\mathcal{B}_{\partial B}(t_2)
  \arrow[u, shift left=1.0ex, "\iota_{\mathcal{B}_{\partial B}(t_2)\to \mathcal{B}(t_2)}"]
\end{tikzcd}
\]
\end{lem}

\begin{proof}
We need to show that $\rho|_{\mathcal{B}(t_{2})}\circ\imath_{\mathcal{B}(t_{1})\to\mathcal{B}(t_{2})}=\imath_{\mathcal{B}_{\partial B}(t_{2})\to\mathcal{B}(t_{2})}\circ\rho|_{\mathcal{B}(t_{1})}$
and $\imath_{\mathcal{B}_{\partial B}(t_{2})\to\mathcal{B}(t_{2})}\circ\imath_{\mathcal{B}_{\partial B}(t_{1})\to\mathcal{B}_{\partial B}(t_{2})}=\imath_{\mathcal{B}(t_{1})\to\mathcal{B}(t_{2})}\circ\imath_{\mathcal{B}_{\partial B}(t_{1})\to\mathcal{B}(t_{1})}$.
First, for all $y\in\mathcal{B}(t_{1})$, it holds that 
\begin{align*}
 & \rho|_{\mathcal{B}(t_{2})}\circ\imath_{\mathcal{B}(t_{1})\to\mathcal{B}(t_{2})}(y)=\rho|_{\mathcal{B}(t_{2})}(y)=\rho(y),\\
 & \imath_{\mathcal{B}_{\partial B}(t_{2})\to\mathcal{B}(t_{2})}\circ\rho|_{\mathcal{B}(t_{1})}(y)=\imath_{\mathcal{B}_{\partial B}(t_{2})\to\mathcal{B}(t_{2})}(\rho(y))=\rho(y),
\end{align*}
and, hence, we have that 
\[
\rho|_{\mathcal{B}(t_{2})}\circ\imath_{\mathcal{B}(t_{1})\to\mathcal{B}(t_{2})}=\imath_{\mathcal{B}_{\partial B}(t_{2})\to\mathcal{B}(t_{2})}\circ\rho|_{\mathcal{B}(t_{1})}.
\]
Second, for all $y\in\mathcal{B}_{\partial B}(t_1)$, it holds that
\begin{align*}
 & \imath_{\mathcal{B}_{\partial B}(t_{2})\to\mathcal{B}(t_{2})}\circ\imath_{\mathcal{B}_{\partial B}(t_{1})\to\mathcal{B}_{\partial B}(t_{2})}(y)=\imath_{\mathcal{B}_{\partial B}(t_{2})\to\mathcal{B}(t_{2})}(y)=y,\\
 & \imath_{\mathcal{B}(t_{1})\to\mathcal{B}(t_{2})}\circ\imath_{\mathcal{B}_{\partial B}(t_{1})\to\mathcal{B}(t_{1})}(y)=\imath_{\mathcal{B}(t_{1})\to\mathcal{B}(t_{2})}(y)=y,
\end{align*}
and, hence, we have that 
\[
\imath_{\mathcal{B}_{\partial B}(t_{2})\to\mathcal{B}(t_{2})}\circ\imath_{\mathcal{B}_{\partial B}(t_{1})\to\mathcal{B}_{\partial B}(t_{2})}=\imath_{\mathcal{B}(t_{1})\to\mathcal{B}(t_{2})}\circ\imath_{\mathcal{B}_{\partial B}(t_{1})\to\mathcal{B}(t_{1})}.
\]
This completes the proof of Lemma \ref{Appendix::lem::inclusion diagram commutes}.
\end{proof}
From Lemma \ref{Appendix::lem::inclusion diagram commutes}, we can induce an isomorphism between the homology groups, as stated in the following lemma.
\begin{lem}\label{Appendix::lem::an induced isomorphism} 
For $r_{0}\leq t_{1}<t_{2}<r$, the inclusion map $\imath_{\mathcal{B}(t_{1})\to\mathcal{B}(t_{2})}:\mathcal{B}(t_{1})\to\mathcal{B}(t_{2})$
induces an isomorphism $\left(\imath_{\mathcal{B}(t_{1})\to\mathcal{B}(t_{2})}\right)_*:\mathrm{H}_{1}(\mathcal{B}(t_{1}))\to \mathrm{H}_{1}(\mathcal{B}(t_{2}))$.
\end{lem}

\begin{proof}
The commutative diagram in the proof of Lemma \ref{Appendix::lem::inclusion diagram commutes} induces the following commutative diagram
on homology groups:

\[
\begin{tikzcd}[
  column sep=7.5em,
  row sep=5.5em,
  cells={nodes={inner sep=2pt}}
]
\mathrm{H}_1(\mathcal{B}(t_1))
  \arrow[r, "\left(\iota_{\mathcal{B}(t_1)\to \mathcal{B}(t_2)}\right)_*"]
  \arrow[d, shift left=1.0ex, "\left(\rho|_{\Bc(t_1)}\right)_*"]
&
\mathrm{H}_1(\mathcal{B}(t_2))
  \arrow[d, shift left=1.0ex, "\left(\rho|_{\Bc(t_2)}\right)_*"]
\\
\mathrm{H}_1(\mathcal{B}_{\partial B}(t_1))
  \arrow[u, shift left=1.0ex, "\left(\iota_{\mathcal{B}_{\partial B}(t_1)\to \mathcal{B}(t_1)}\right)_*"]
  \arrow[r, "\left(\iota_{\mathcal{B}_{\partial B}(t_1)\to \mathcal{B}_{\partial B}(t_2)}\right)_*"']
&
\mathrm{H}_1(\mathcal{B}_{\partial B}(t_2))
  \arrow[u, shift left=1.0ex, "\left(\iota_{\mathcal{B}_{\partial B}(t_2)\to \mathcal{B}(t_2)}\right)_*"]
\end{tikzcd}
\]
By Lemma \ref{Appendix::cor::homotopy equivalence}, 
the map  $\imath_{\mathcal{B}_{\partial B}(t_{1})\to\mathcal{B}(t_{1})}$ is a homotopy equivalence of $\Bc_{\partial B}(t_1)$ and $\Bc(t_1)$ with the inverse of 
 $\rho|_{\mathcal{B}(t_{1})}$ via the map $F$. From this, we have that the induced homomorphisms $\left(\iota_{\mathcal{B_{\partial B}}(t_1)\to \mathcal{B}(t_2)}\right)_*$ and $\left(\rho|_{\Bc(t_1)}\right)_*$ are isomorphisms.
The same holds for $\left(\imath_{\mathcal{B}_{\partial B}(t_{2})\to\mathcal{B}(t_{2})}\right)_*$
and $\left(\rho|_{\mathcal{B}(t_{2})}\right)_*$. Now, note that when $r_{0}\leq t_{1}<t_{2}<r$,
\[
\bigcup_{x\in\mathcal{X}}\mathcal{B}(x,t_{1})\cap\partial B_{\mathcal{X}}=\bigcup_{x\in\mathcal{X}}\mathcal{B}(x,t_{2})\cap\partial B_{\mathcal{X}}=\partial B_{\mathcal{X}}.
\]
Hence, $\left(\imath_{\mathcal{B}_{\partial B}(t_{1})\to\mathcal{B}_{\partial B}(t_{2})}\right)_*:\mathrm{H}_{1}(\partial B_{\mathcal{X}})\to \mathrm{H}_{1}(\partial B_{\mathcal{X}})$
is an isomorphism as well. Therefore, the homomorphism
\[
\left(\imath_{\mathcal{B}(t_{1})\to\mathcal{B}(t_{2})}\right)_*=\left(\imath_{\mathcal{B}_{\partial B}(t_{2})\to\mathcal{B}(t_{2})}\right)_*\circ\left(\imath_{\mathcal{B}_{\partial B}(t_{1})\to\mathcal{B}_{\partial B}(t_{2})}\right)_*\circ\left(\rho|_{\mathcal{B}(t_{1})}\right)_*
\]
is an isomorphism as well.
This completes the proof of Lemma \ref{Appendix::lem::an induced isomorphism}.
\end{proof}

\begin{proof}[Proof of Proposition~\ref{prop::phofcircle}]
The results of Proposition \ref{Appendix::homotopy_cech_of_circle} imply that 
\[
\mathrm{H}_{1}\left(\textrm{\v{C}ech}(\Xc,t)\right)=\begin{cases}
\{0\}, & t< r_{0},t\geq r,\\
\mathbb{Z}, & t\in[r_{0},r),
\end{cases}
\]
when $\Xc$ is not contained in any open hemicircle lying on $S^1(r)$, and 
\[
\mathrm{H}_{1}\left(\textrm{\v{C}ech}(\Xc,t)\right)=\{0\}, \text{ for all } t>0,
\]
when $\Xc$ is contained in some open hemicircle lying on $S^1(r)$.
Combining the above homology groups and Lemma \ref{Appendix::lem::an induced isomorphism} yields the persistence diagrams presented in Proposition \ref{prop::phofcircle}. Therefore, the proof of Proposition~\ref{prop::phofcircle} is completed.
\end{proof}
\newpage
\section{Testing algorithms}\label{Appendix::Algorithms}
        
        
        

\begin{algorithm}[H]
    \caption{\textbf{The studentized two-sample test}}
    \label{Appendix::studentized-test-algorithm}
    \begin{algorithmic}
        \STATE \textbf{Hypotheses:}
        $\mathrm{H}_0:p=q$ versus $\mathrm{H}_1:p\neq q$.

        \STATE \textbf{Input:}
        \begin{itemize}[noitemsep,topsep=0pt]
            \item Persistence diagrams
            $X_1,\ldots,X_n\overset{\mathrm{i.i.d.}}{\sim}P$ and
            $Y_1,\ldots,Y_m\overset{\mathrm{i.i.d.}}{\sim}Q$.
            \item A kernel function $k_\lambda$ with bandwidth $\lambda$.
            \item A weight function $w$.
            \item A significance level $\alpha\in(0,1)$.
        \end{itemize}

        \STATE \textbf{Procedure:}

        \STATE \textbf{1}. Compute the studentized statistic
        $T_{n,m,\lambda}$ defined in \eqref{Eq::studentized-test-statistic}.

        \STATE \textbf{2}. Compute the critical value
        $z_{1-\alpha}:=\Phi^{-1}(1-\alpha)$, where $\Phi$ is the
        cumulative distribution function of $\mathcal N(0,1)$.

        \RETURN Reject $\mathrm{H}_0$ if
        $T_{n,m,\lambda}>z_{1-\alpha}$; otherwise, do not reject
        $\mathrm{H}_0$.
    \end{algorithmic}
\end{algorithm}

\begin{algorithm}[H]
    \caption{\textbf{The Bonferroni bandwidth-aggregated two-sample test}}
    \label{Appendix::Bonferroni-test-algorithm}
    \begin{algorithmic}
        \STATE \textbf{Hypotheses:}
        $\mathrm{H}_0:p=q$ versus $\mathrm{H}_1:p\neq q$.

        \STATE \textbf{Input:}
        \begin{itemize}[noitemsep,topsep=0pt]
            \item Persistence diagrams
            $X_1,\ldots,X_n\overset{\mathrm{i.i.d.}}{\sim}P$ and
            $Y_1,\ldots,Y_m\overset{\mathrm{i.i.d.}}{\sim}Q$.
            \item A collection $\Lambda$ of bandwidths.
            \item A kernel function $k_\lambda$ for each
            $\lambda\in\Lambda$.
            \item A weight function $w$.
            \item A significance level $\alpha\in(0,1)$.
        \end{itemize}

        \STATE \textbf{Procedure:}

        \STATE \textbf{1}. For every $\lambda\in\Lambda$, compute the
        studentized statistic $T_{n,m,\lambda}$ defined in
        \eqref{Eq::studentized-test-statistic}.

        \STATE \textbf{2}. Compute the Bonferroni critical value
            $z_{1-\alpha/|\Lambda|}
            :=
            \Phi^{-1}
            \left(
                1-\frac{\alpha}{|\Lambda|}
            \right)$.

        \RETURN Reject $\mathrm{H}_0$ if
            $\max_{\lambda\in\Lambda}
            T_{n,m,\lambda}
            >
            z_{1-\alpha/|\Lambda|};$
        otherwise, do not reject $\mathrm{H}_0$.
    \end{algorithmic}
\end{algorithm}

\section{Takens' embedding}\label{Appendix::Taken's embedding}
Takens' embedding provides a method for reconstructing the state space of a time series from observed data. Using this reconstruction, one can analyse the underlying structure of the time series via persistent homology. 

Periodicity refers to the presence of repeating patterns at regular intervals over time and is one of the key characteristics of time series data. Such periodic behaviour can be effectively analysed through Takens' embedding combined with persistent homology. To be specific, let us define the Takens' embedding.

For a connected and bounded subset $\Ac \subseteq \Rb$, let $C(\Ac)$ denote the set of continuous functions from $\Ac$ to $\Rb$. 
Given a dimension parameter $d \in \Nb$ and a time-delay parameter $\tau \in \Rb$, define the embedding map $T_{d,\tau}: C(\Ac) \to C(\Ac, \Rb^d)$ by
\begin{align*}
T_{d,\tau}f(t) =
\begin{bmatrix}
f(t) \\
f(t+\tau) \\
\vdots \\
f(t+(d-1)\tau)
\end{bmatrix}
\in \Rb^d,
\end{align*}
which is referred to as Takens' embedding.

This embedding induces a point cloud in $\Rb^d$ from the observed time series. 
Periodic behaviour in the original time series manifests as loop structures in the embedded trajectory. 
Consequently, one-dimensional topological features in the persistence diagram constructed from this point cloud capture the periodicity of the time series, \citep{Takens1981Detecting,Takensembedding, PereaHarer2013SlidingWindows}.

\section{Additional simulation studies}\label{Appendix::Additional Simulation Study}
This section contains the hyperparameter settings for each test method, additional simulation results, and data descriptions that are not included in the main text. 

\paragraph*{The hyperparameter setup}
The hyperparameters for each method are set as follows. 
For the PD method, we use $1{,}000$ permutations and the 1-Wasserstein distance between 
persistence diagrams. 
For the PL method, we also use $1{,}000$ permutations and the $L_2$ distance between  the sample mean persistence landscapes. 
For the PI method, we use $40 \times 40$ pixels with Gaussian smoothing, and apply a 
two-sample $t$-test at each pixel. 
To adjust the resulting $p$-values, we adopt the Bonferroni procedure. 
Other hyperparameters such as the cut-off threshold, weight function, kernel bandwidth, 
and image range vary across simulations. 
Specifically, we use a cut-off value of $0.2$ for the torus and ORBIT5K experiments. 
The weight function is chosen as the square weight $w(x,y) = (y - x)^2$ for the 
torus data.
For the image range, we use $[0,2]^2$ for the torus and $[0,0.2]^2$ for the ORBIT5K data.
Lastly, for the kernel bandwidth, we use $0.2$ and $0.05$ for the torus and ORBIT5K data, respectively. 

\paragraph*{Bandwidth collection} The choice of bandwidth matters, since the normal approximation of the test statistics can be inaccurate for poorly chosen bandwidths in finite samples. In particular, the largest bandwidth in $\Lambda$, which is defined in \eqref{eq::bandwidth collection}, may be too large for this approximation to reliable in practice; see, e.g., Section~\ref{sec::Bandwidth collection}.
Moreover, $\Lambda$ does not account for the scale of the data. Therefore, we introduce a data-adaptive bandwidth collection for practical use.  

Let \(Z_1,\ldots,Z_{n+m}=(X_1,\ldots,X_n,Y_1,\ldots,Y_m)\) be the pooled persistence diagrams, and let
\begin{align*}
\Ac
=
\{(i,j):1\leq i<j\leq n+m,\ |Z_i|>0,\ |Z_j|>0\}.
\end{align*}
For each coordinate \(u\in\{b,d\}\), define
\[
\widehat F_u(t;\Xb_n,\Yb_m)
=
\frac{1}{|\Ac|}
\sum_{(i,j)\in\Ac}
\frac{1}{|Z_i||Z_j|}
\sum_{z\in Z_i}\sum_{z'\in Z_j}
\mathbf 1\bigl\{|z_u-z'_u|\leq t\bigr\},
\qquad t\geq 0.
\]
Here, \(z_b\) and \(z_d\) are the birth and death coordinates of \(z\). Thus, we compare points from distinct diagrams, give every pair of nonempty diagrams equal total weight, and distribute that weight equally over its point pairs. We define the reference birth and death scales as the medians of these two distributions:
\begin{align}\label{eq::meidan_birth_death}
r_u(\Xb_n,\Yb_m)
=
\frac{1}{2}
\left[
\inf\{t:\widehat F_u(t;\Xb_n,\Yb_m)\geq\tfrac12\}
+
\inf\{t:\widehat F_u(t;\Xb_n,\Yb_m)>\tfrac12\}
\right],
\quad u\in\{b,d\}.
\end{align}
The averaging convention agrees with the usual sample median when all diagrams have the same number of points. Given positive reference scales, the experimental bandwidth collection is defined by 
\begin{align}\label{eq::practical bandwidth collection}
\Lambda^{\mathrm{prac}}_{n,m}(\Xb_n,\Yb_m)
=\left\{\left(\frac{2^j r_b(\Xb_n,\Yb_m)}{\sqrt{n+m}},
\frac{2^k r_d(\Xb_n,\Yb_m)}{\sqrt{n+m}}
\right)
:j,k\in\{-1,0,1\}
\right\}.
\end{align} 
The cardinality of $\Lambda^{\mathrm{prac}}_{n,m}(\Xb_n,\Yb_m)$ is 9. Note that each bandwidth $\lambda_{n,m}=(\lambda_{1,n,m},\lambda_{2,n,m}) $ in $\Lambda^{\mathrm{prac}}_{n,m}(\Xb_n,\Yb_m)$ satisfies 
\begin{align*}
    \max \left(\lambda_{1,n,m}, \lambda_{2,n,m}\right)\to 0,\quad  (n+m)^2\lambda_{1,n,m}\lambda_{2,n,m}\to \infty,\quad \text{as}\quad n+m\to \infty,
\end{align*}
provided that $r_b$ and $r_d$ converge to a positive number. In Section~\ref{sec::Bandwidth collection},
we provide histograms of the test statistics obtained using bandwidths from $\Lambda^{\mathrm{prac}}_{n,m}(\Xb_n,\Yb_m)$.

\subsection{The ORBIT5K data simulation}\label{Appendix::ORBI5K}
In this section, we conduct the ORBIT5K data simulation. The dataset arises from a linked twist map, a discrete
dynamical system modelling fluid flow. The linked twist map was used in \citet{Hertzsch2007DNA} to model flows in DNA microarrays with a particular interest in understanding turbulent mixing. This data was first used in TDA by \citet{persistenceimage}. The linked twist map is called a Poincar\'e section. This Poincar\'e section is given by 
\begin{align*}
    x_{n+1} = x_n +ry_{n}(1-y_n) \text{ mod 1}\\
    y_{n+1} = y_n +rx_{n}(1-x_n) \text{ mod 1},
\end{align*}
where $r$ is a positive parameter. The orbits $\{(x_n,y_n): n=0,\dots \infty\}$ are dense in the domain $[0,1]^2$ for some values of $r$. However, for other values of $r$, circular structures emerge. According to the value of $r$, the truncated orbits $\{(x_n,y_n):n=0,\dots, N\}$, where $N\in \Nb$, exhibit different complex structures.

We choose a set of parameter values, $r=2.5,4.0$ and $4.3$ which produce  different orbit circles. For each parameter value $r$, we randomly sample an initial value $(x_1,y_1)$ and use $1,000$ iterations of the linked twist map to generate point clouds in $\Rb^2$. In generating point clouds, we use the Python code from \citet{KimKimZaheer2020PLLay}. To generate a persistence diagram from each point cloud, we use the Vietoris-Rips filtration. We use only the $1$-dimensional topological features. The figures of the point clouds and corresponding persistence diagrams are contained in Figure \ref{Appendix::Orbit}.

We conduct the PD, PL, PI, and the Bonferroni procedures for each pair of parameter groups. There are $6$ scenarios. Sc.1 : $r=2.5$ vs $r=2.5$, Sc.2 : $r=4.0$ vs $r=4.0$, Sc.3 : $r=4.3$ vs $r=4.3$, Sc.4 : $r=2.5$ vs $r=4.0$, Sc.5 : $r=2.5$ vs $r=4.3$, Sc.6 : $r=4.0$ vs $r=4.3$. Scenarios Sc.1, Sc.2, Sc.3 represent distributional-null configurations, in which the two groups are sampled from the same dynamical system.
Scenarios Sc.4–Sc.6 are designed to induce differences in the persistence intensity functions by using point clouds generated from different dynamical systems. 
For each scenario, we select $50$ persistence diagrams from each group. Therefore, we use a total of $100=50+50$ persistence diagrams for each scenario. The p-values of PD test and PL test and the adjusted p-value of PI test and the Bonferroni test are contained in Table \ref{Table::orbit5k}.
As shown in Table \ref{Table::orbit5k}, the proposed test appropriately fails to reject the distributional null in Sc.1, Sc.2, and Sc.3 and rejects the null in the remaining scenarios. 

\begin{table}[H]
\centering
\caption{p-values of the four testing methods in the ORBIT5K simulation study.}
\label{Table::orbit5k}
\setlength{\tabcolsep}{11pt}
\renewcommand{\arraystretch}{0.8}
\begin{tabular}{lcccccc}
\toprule
 & Sc.1 & Sc.2 & Sc.3 & Sc.4 & Sc.5 & Sc.6 \\
\midrule
\multicolumn{6}{l}{\textbf{Bonf Intensity}} \\
\quad Linear & 1.000 & 1.000 & 1.000 & <0.001 & <0.001 & <0.001 \\
\quad Arctan & 0.570 & 0.398 & 1.000 & <0.001 & <0.001 & <0.001 \\
\quad Constant & 1.000 & 1.000 & 1.000 & <0.001 & <0.001 & <0.001 \\

\multicolumn{6}{l}{\textbf{PI}} \\
\quad Linear & 1.000 & 1.000 & 1.000 & <0.001 & <0.001 & <0.001 \\
\quad Arctan & 1.000 & 1.000 & 1.000 & <0.001 & <0.001 & <0.001 \\
\quad Constant & 1.000 & 1.000 & 1.000 & 1.000 & 0.146 & <0.001 \\

\textbf{PD} & 0.783 & 0.172 & 0.587 & <0.001 & <0.001 & <0.001 \\

\textbf{PL} & 0.671 & 0.844 & 0.729 & <0.001 & <0.001 & <0.001 \\

\bottomrule
\end{tabular}
\end{table}

\begin{figure}[H]
    \centering
    \includegraphics[height=0.75\textheight]{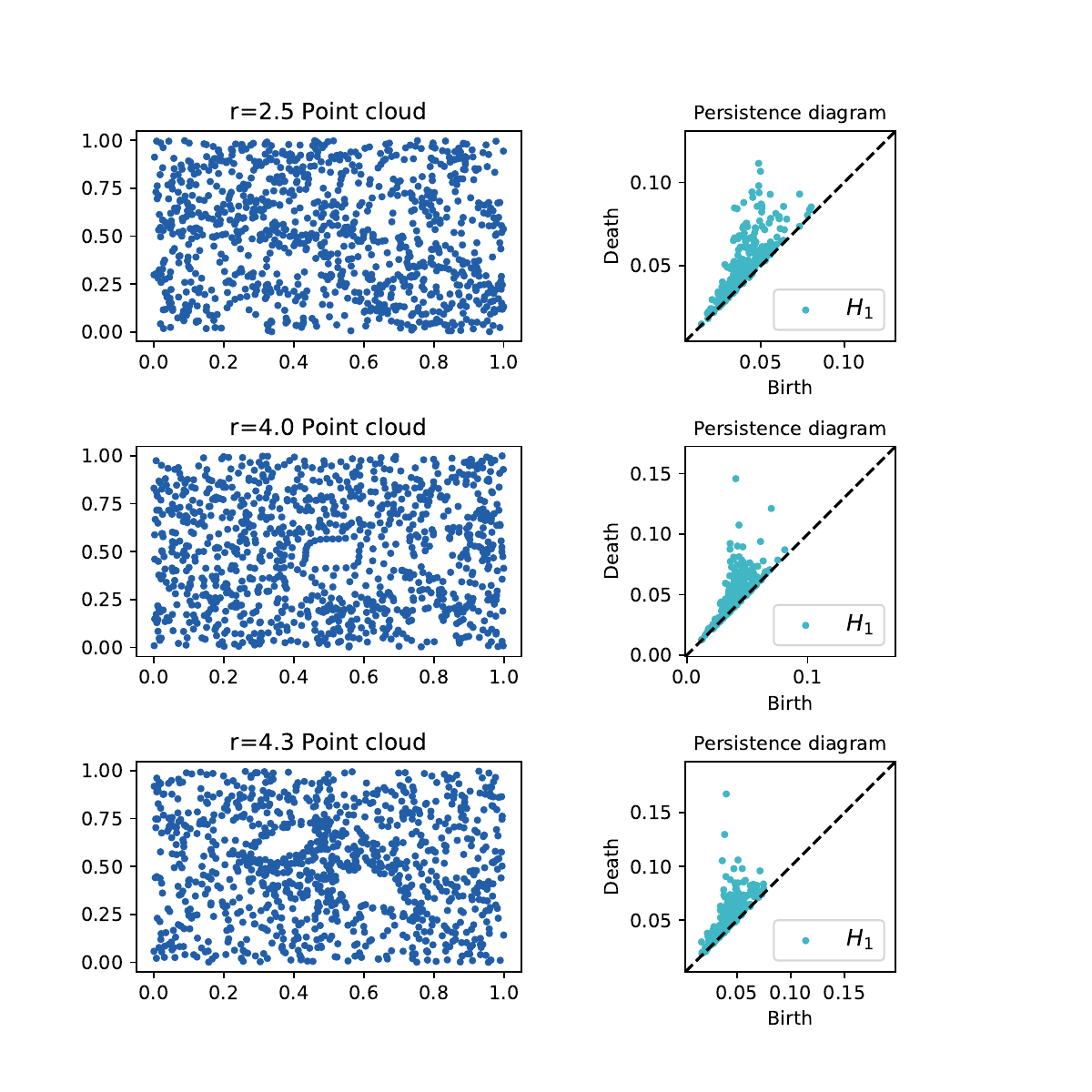}
    \caption{Examples of the truncated orbits $\{(x_n,y_n):n=1,\dots, 1,000\}$ of the linked twist map and their corresponding persistence diagrams for
 $r=2.5$, $r=4.0$, and $r=4.3$. }
    \label{Appendix::Orbit}
\end{figure}
\subsection{Additional torus simulation}\label{Appendix::Additional torus data simulation}
\paragraph*{Empirical validity}
Under the same setting of the torus simulation in Section \ref{Sec::Simulation Study} of the main text, we also provide empirical Type I error. We use independently generated persistence diagrams from the unexpanded torus, i.e., $\beta=0$. We conduct $250$ independent replications, and the empirical Type I error is defined as the rejection rate over the 250 replications. The result is provided in Figure~\ref{Fig::type 1 error torus simul}
\begin{figure}[H]
    \centering
    \includegraphics[width=0.9\textwidth]{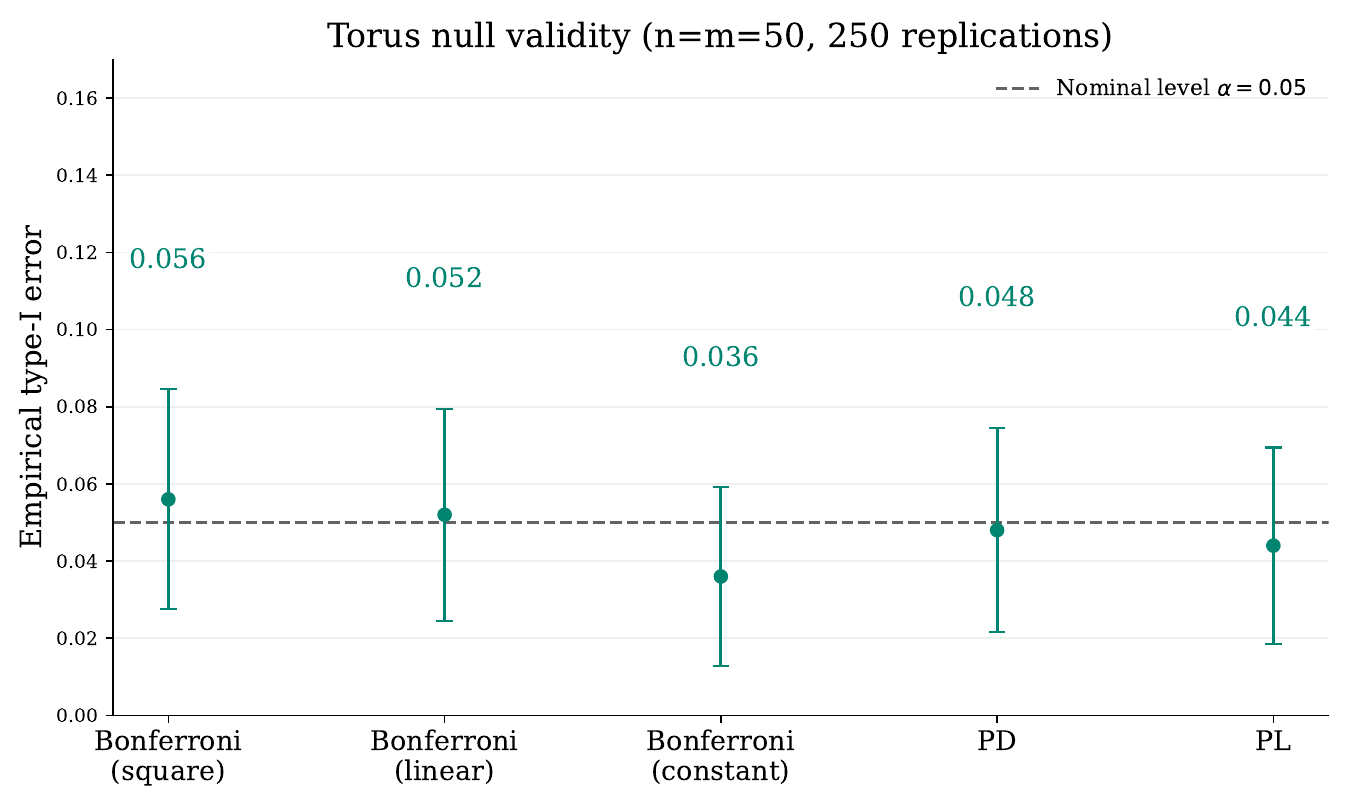}
    \caption{Empirical Type I errors in the torus simulation. The error bars represent approximate 95\% Monte Carlo confidence intervals, calculated as
$\widehat{\pi}\pm 1.96
\left\{
\frac{\widehat{\pi}(1-\widehat{\pi})}{250}
\right\}^{1/2}$,
where $\widehat{\pi}$ denotes the empirical rejection proportion.}
    \label{Fig::type 1 error torus simul}
\end{figure}
\paragraph*{Empirical power over different sample-sizes}
Under the same setting of the torus simulation in Section \ref{Sec::Simulation Study} of the main text, we investigate the change of empirical powers when the sample size $n+m$ varies. The simluation results are contained in Figures \ref{fig:torus-radius-power-weights} and \ref{fig:torus-radius-power-methods}.
\begin{figure}[!htbp]
    \centering
    \includegraphics[width=0.9\linewidth, height=0.4\textheight]
        {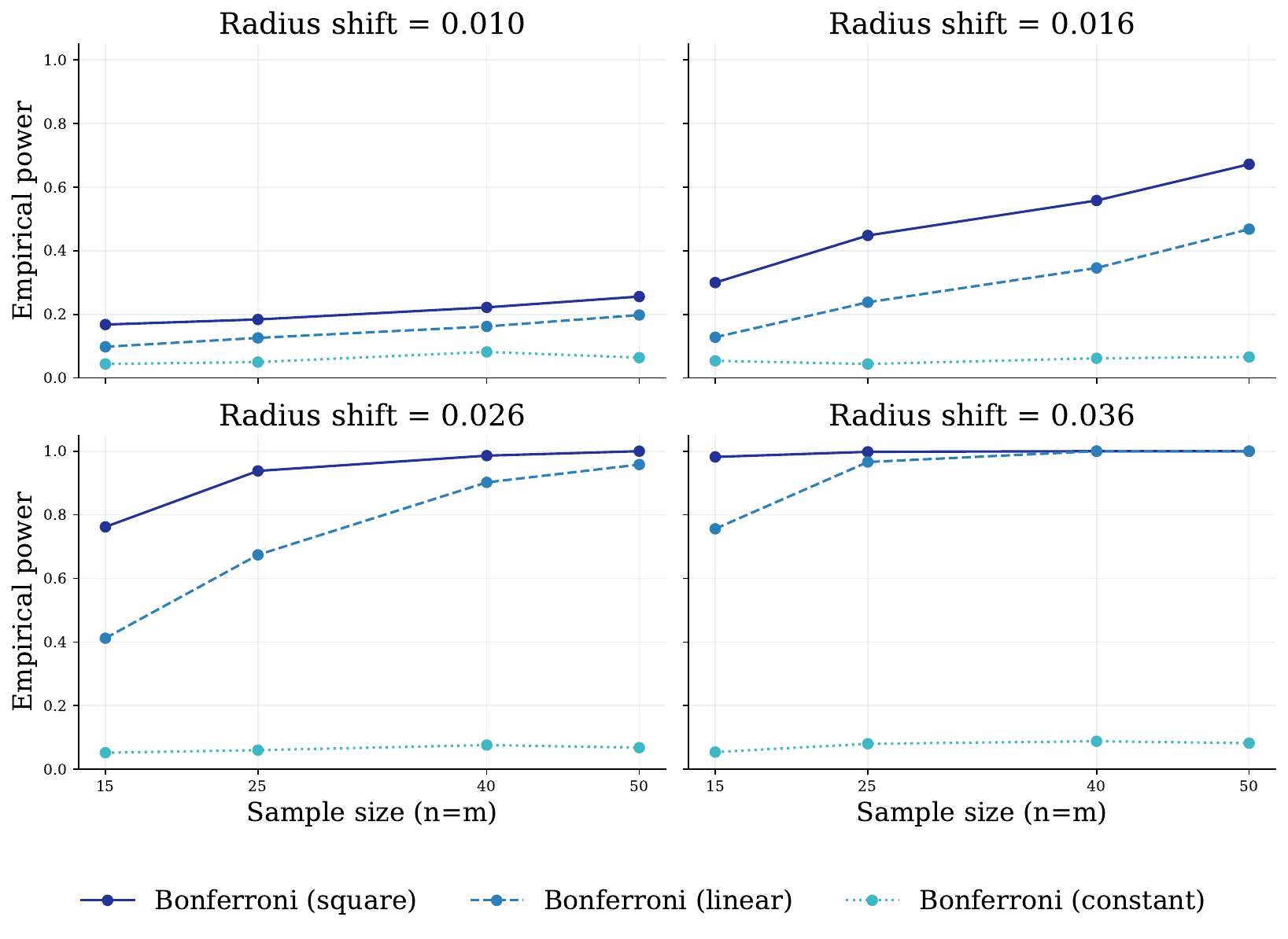}
    \caption{Empirical power versus group sample size $n=m$ in the
    torus radius-shift experiment. The square, linear, and constant
    weights are compared at four fixed radius shifts.
    Each estimate is based on 500 replications.}
    \label{fig:torus-radius-power-weights}
\end{figure}
The empirical power of the proposed tests with square and linear weights increases as $n+m$ grows. By contrast, the empirical power with the constant weight does not increase substantially over the sample sizes considered. This suggests that the choice of weight function affects how gains in sample size translate into power for this radius-shift alternative. 
\begin{figure}[!htbp]
    \centering
    \includegraphics[width=0.9\linewidth, height=0.4\textheight]
        {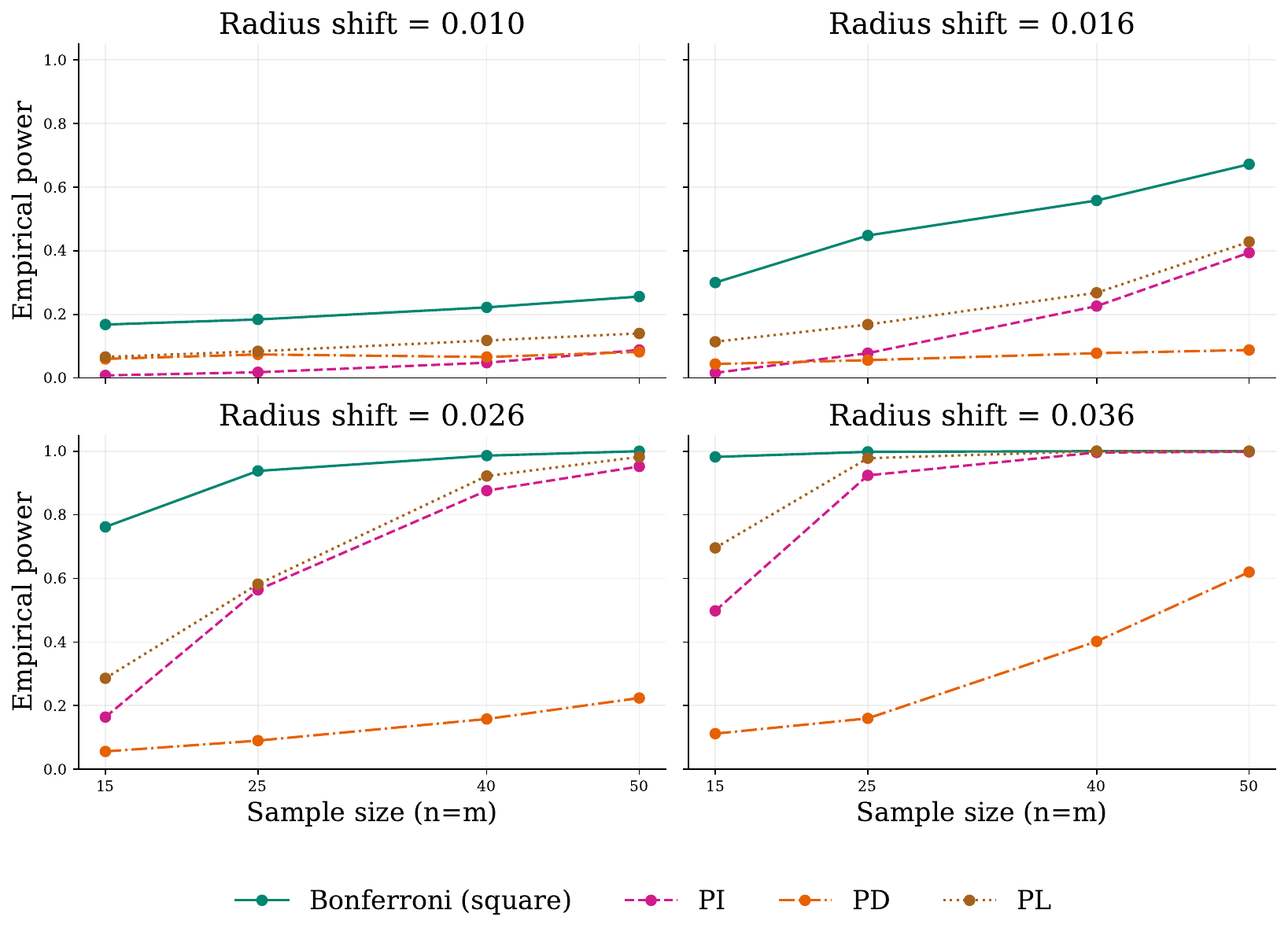}
    \caption{Empirical power versus group sample size $n=m$ in the
    torus radius-shift experiment. Bonferroni with the square weight,
    PI, PD, and PL are compared at four fixed radius shifts.
    Each estimate is based on 500 replications.}
    \label{fig:torus-radius-power-methods}
\end{figure}

\paragraph*{Tables of the torus data simulation}\label{Appendix::Tables of torus data simulation}
The specific numerical information of the torus simulation is provided here. 
\begin{table}[H]
\centering
\caption{Empirical powers of the tests at different expanding levels
in the torus data simulation.}
\label{tab:torus-power-by-beta}
\begin{tabular}{lcccc}
\toprule
Method $\backslash$ $\beta$
& $0.01$ & $0.016$ & $0.026$ & $0.036$ \\
\midrule
Bonf test & 0.256 & 0.672 & 1.000 & 1.000 \\
PI test   & 0.088 & 0.394 & 0.952 & 0.998 \\
PD test   & 0.082 & 0.088 & 0.224 & 0.620 \\
PL test   & 0.140 & 0.428 & 0.982 & 1.000 \\
\bottomrule
\end{tabular}
\end{table}

\begin{table}[H]
\centering
\caption{Empirical powers of the proposed test with different weights at different expanding levels in the torus data simulation.}
\begin{tabular}{lccccc}
\toprule
Weight\ $\backslash$\ $\beta$ & $0.01$ & $0.016$ & $0.026$ & $0.036$ \\
\midrule
$(y-x)^{2}$ & 0.256 & 0.672 & 1.000 & 1.000 \\
$y-x$ & 0.198 & 0.468 & 0.958 & 1.000 \\
constant & 0.064 & 0.066 & 0.068 & 0.082 \\
\bottomrule
\end{tabular}
\end{table}

\begin{table}[H]
\centering
\caption{Empirical power at radius shift $\delta=0.010$, based on 500 replications. }
\label{tab:torus-power-shift-010}
\setlength{\tabcolsep}{2pt}
\begin{tabular}{@{}rcccccc@{}}
\toprule
$n=m$ & Bonf$_{\mathrm{sq}}$ & Bonf$_{\mathrm{lin}}$ & Bonf$_{\mathrm{const}}$ & PI & PD & PL \\
\midrule
15 & 0.168 & 0.098 & 0.044 & 0.008 & 0.060 & 0.066 \\
25 & 0.184 & 0.126 & 0.050 & 0.018 & 0.074 & 0.084 \\
40 & 0.222 & 0.162 & 0.082 & 0.048 & 0.066 & 0.118 \\
50 & 0.256 & 0.198 & 0.064 & 0.088 & 0.082 & 0.140 \\
\bottomrule
\end{tabular}
\end{table}

\begin{table}[H]
\centering
\caption{Empirical power at radius shift $\delta=0.016$, based on 500 replications.}
\label{tab:torus-power-shift-016}
\setlength{\tabcolsep}{2pt}
\begin{tabular}{@{}rcccccc@{}}
\toprule
$n=m$ & Bonf$_{\mathrm{sq}}$ & Bonf$_{\mathrm{lin}}$ & Bonf$_{\mathrm{const}}$ & PI & PD & PL \\
\midrule
15 & 0.300 & 0.128 & 0.054 & 0.016 & 0.044 & 0.114 \\
25 & 0.448 & 0.238 & 0.044 & 0.078 & 0.056 & 0.168 \\
40 & 0.558 & 0.346 & 0.062 & 0.226 & 0.078 & 0.268 \\
50 & 0.672 & 0.468 & 0.066 & 0.394 & 0.088 & 0.428 \\
\bottomrule
\end{tabular}
\end{table}

\begin{table}[H]
\centering
\caption{Empirical power at radius shift $\delta=0.026$, based on 500 replications.}
\label{tab:torus-power-shift-026}
\setlength{\tabcolsep}{2pt}
\begin{tabular}{@{}rcccccc@{}}
\toprule
$n=m$ & Bonf$_{\mathrm{sq}}$ & Bonf$_{\mathrm{lin}}$ & Bonf$_{\mathrm{const}}$ & PI & PD & PL \\
\midrule
15 & 0.762 & 0.412 & 0.052 & 0.164 & 0.056 & 0.286 \\
25 & 0.938 & 0.674 & 0.060 & 0.564 & 0.090 & 0.582 \\
40 & 0.986 & 0.902 & 0.076 & 0.876 & 0.158 & 0.922 \\
50 & 1.000 & 0.958 & 0.068 & 0.952 & 0.224 & 0.982 \\
\bottomrule
\end{tabular}
\end{table}

\begin{table}[H]
\centering
\caption{Empirical power at radius shift $\delta=0.036$, based on 500 replications.}
\label{tab:torus-power-shift-036}
\setlength{\tabcolsep}{2pt}
\begin{tabular}{@{}rcccccc@{}}
\toprule
$n=m$ & Bonf$_{\mathrm{sq}}$ & Bonf$_{\mathrm{lin}}$ & Bonf$_{\mathrm{const}}$ & PI & PD & PL \\
\midrule
15 & 0.982 & 0.756 & 0.054 & 0.498 & 0.112 & 0.696 \\
25 & 0.998 & 0.966 & 0.080 & 0.924 & 0.160 & 0.978 \\
40 & 1.000 & 1.000 & 0.088 & 0.996 & 0.402 & 1.000 \\
50 & 1.000 & 1.000 & 0.082 & 0.998 & 0.620 & 1.000 \\
\bottomrule
\end{tabular}
\end{table}

\subsection{Practical guideline for choosing the bandwidth collection and the weight function}\label{sec::practical guideline}
In this section, we provide a practical guideline for choosing the bandwidth collection and the weight function. 
\subsubsection{Bandwidth collection}\label{sec::Bandwidth collection}
Using the bandwidth collection $\Lambda$ defined in Section \ref{Sec:: Bandwidth-Aggregated Two Sample Test} in finite samples can be problematic. 
Its bandwidths are determined without accounting for the scale of the observed data. Even after the collection is rescaled, the standard normal approximation of $T_{n,m,\lambda}$ may remain inaccurate for some bandwidths in the rescaled collection.
To illustrate this phenomenon, we conduct the following simulation. Let $F$ be the uniform distribution on $[1.5,2.5]\times[4.5,5.5]$. Under $P$, each persistence diagram $X_i$ contains two copies of a single point drawn from $F$. Under $Q$, each diagram $Y_j$ contains two independent points drawn from $F$. Consequently, $P\ne Q$, while their intensity functions satisfy $p=q=2f$, where $f$ is the density of $F$.

We set $n=m=30$, use the linear weight $w(b,d)=d-b$, and take $a=6$ in the theoretical bandwidth collection $\Lambda$. This gives $L_{n,m}=2$, $K_{n,m}=8$, and 15 candidate bandwidths. In each of 1,000 Monte Carlo repetitions, we evaluate $T_{n,m,\lambda}$ at every bandwidth using the same two samples. Figure \ref{Fig::largest_fail} compares the resulting histograms with the standard normal density for six selected bandwidths: three diagonal choices with $(k_1,k_2)\in\{(2,2),(3,3),(4,4)\}$, and three choices satisfying $k_1+k_2=8$, namely $(2,6)$, $(3,5)$, and $(6,2)$, where $\lambda=(2^{-k_1},2^{-k_2})$.

\begin{figure}[t]
    \centering
\includegraphics[width=1\textwidth,height=0.4\textheight]{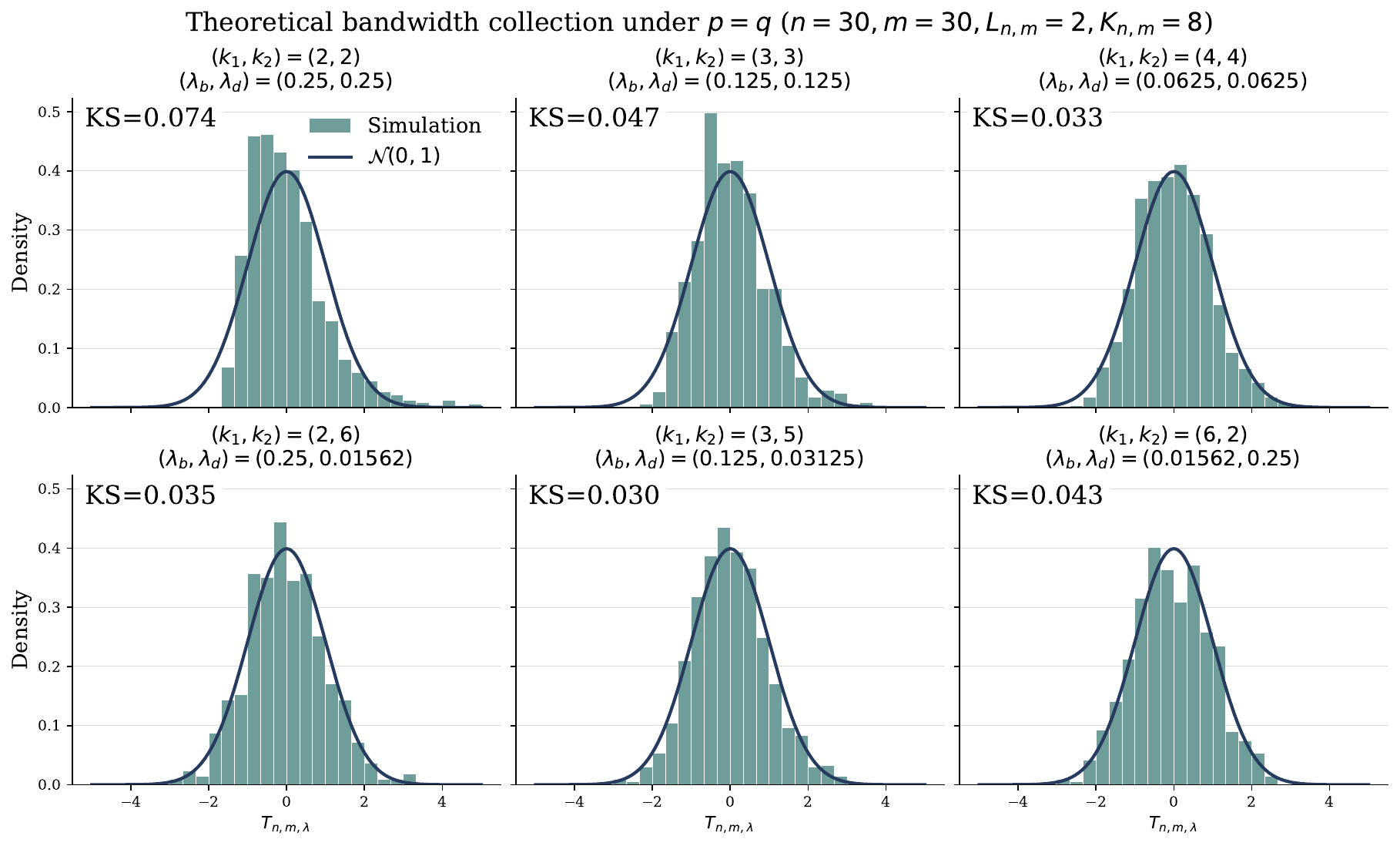}
    \caption{Empirical null distributions of $T_{n,m,\lambda}$ for six bandwidths in the theoretical collection, based on 1,000 repetitions with $n=m=30$ and $p=q$. KS denotes the Kolmogorov–Smirnov distance between the empirical distribution and the standard normal distribution.}
    \label{Fig::largest_fail}
\end{figure}

The Kolmogorov–Smirnov distance from the standard normal distribution is $0.074$ at the largest bandwidth, $\lambda=(0.25,0.25)$, compared with $0.047$ and $0.033$ at the next two diagonal choices. For the three displayed choices with the smallest bandwidth product, the distances are $0.035$, $0.030$, and $0.043$, respectively. Thus, in this setting, the finite-sample normal approximation is less accurate at the largest bandwidth, while the displayed small-product bandwidths remain closer to the standard normal distribution.

One may rescale the theoretical bandwidth collection by multiplying each bandwidth by $2^{-c_{\mathrm{shift}}}$, where $c_{\mathrm{shift}}$ is a nonnegative integer, to reduce the largest bandwidth. We examine the effect of this rescaling under the same null model with $n=m=30$, taking $c_{\mathrm{shift}}\in\{0,\ldots,5\}$. In each of 1,000 repetitions, the same two samples are used across all shifts. Figure \ref{Fig::shifted_fail} shows the null distribution at the boundary choice $(k_b,k_d)=(6+c_{\mathrm{shift}},2+c_{\mathrm{shift}})$, whose birth bandwidth is the smallest in the shifted collection. As $c_{\mathrm{shift}}$ increases, its Kolmogorov–Smirnov distance from the standard normal distribution rises from $0.043$ at $c_{\mathrm{shift}}=0$ to $0.217$ at $c_{\mathrm{shift}}=4$ and $0.337$ at $c_{\mathrm{shift}}=5$. Thus, rescaling the entire collection toward smaller bandwidths can substantially worsen the finite-sample normal approximation at this boundary.
\begin{figure}[H]
    \centering
\includegraphics[width=1\textwidth,height=0.4\textheight]{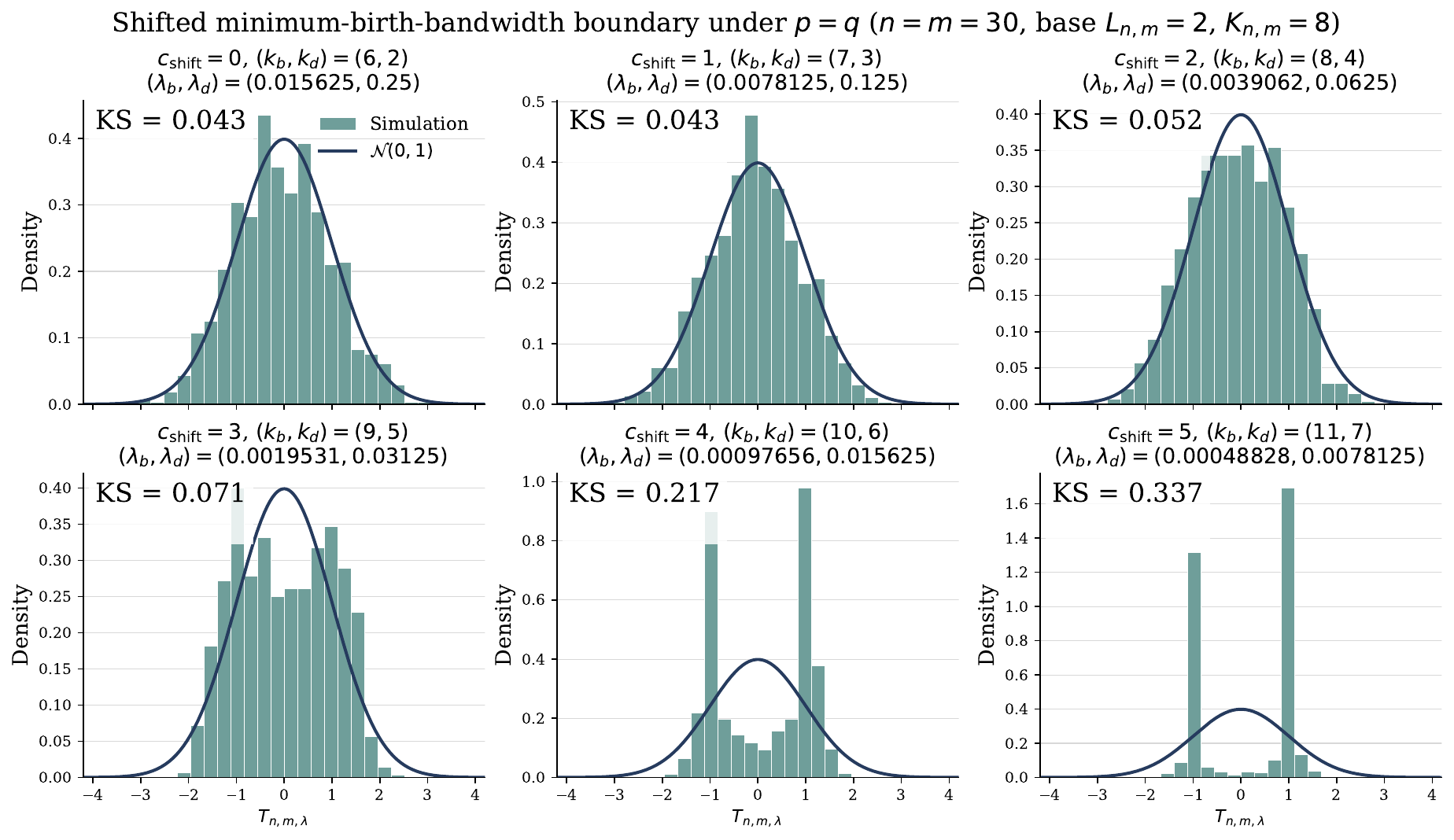}
    \caption{Empirical null distributions of $T_{n,m,\lambda}$ for $\lambda=(2^{-(6+c_{\mathrm{shift}})},2^{-(2+c_{\mathrm{shift}})})$, with $c_{\mathrm{shift}}=0,\ldots,5$, based on 1,000 repetitions with $n=m=30$ and $p=q$. KS denotes the Kolmogorov–Smirnov distance between the empirical distribution and the standard normal distribution.}
    \label{Fig::shifted_fail}
\end{figure}

The previous two simulations show that
 the distribution $T_{n,m,\lambda}$ may not closely approximate the standard normal distribution for some boundary bandwidths. To circumvent this issue, we use the bandwidth collection $\Lambda^{\mathrm{prac}}_{n,m}$, defined in \eqref{eq::practical bandwidth collection}, when implementing the simulation study and the real-data analysis. 
 
We next examine the normal approximations of test statistics over $\Lambda^{\mathrm{prac}}_{n,m}$. 
Under the same setting with the previous simulations, we have the histograms in Figure \ref{Fig::our_choice_adaptive_n30_m30}. Since the bandwidths vary across the data, the median bandwidth over the 1,000 replications is reported. Figure \ref{Fig::our_choice_adaptive_n30_m30} shows that the empirical null distributions of all test statistics closely approximate the standard normal distribution. This phenomenon remains under $n=m=50$ and $n=m=100$, as shown in Figures \ref{Fig::our_choice_adaptive_n50_m50} and \ref{Fig::our_choice_adaptive_n100_m100}.

We emphasize that $\Lambda^{\mathrm{prac}}_{n,m}$ is used as a practical bandwidth collection in our numerical experiments. The theoretical results for $\Lambda$ do not directly establish the validity of this data-adaptive collection $\Lambda^{\mathrm{prac}}_{n,m}$.
\begin{figure}[H]
    \centering
\includegraphics[width=1\textwidth,height=0.6\textheight,keepaspectratio]{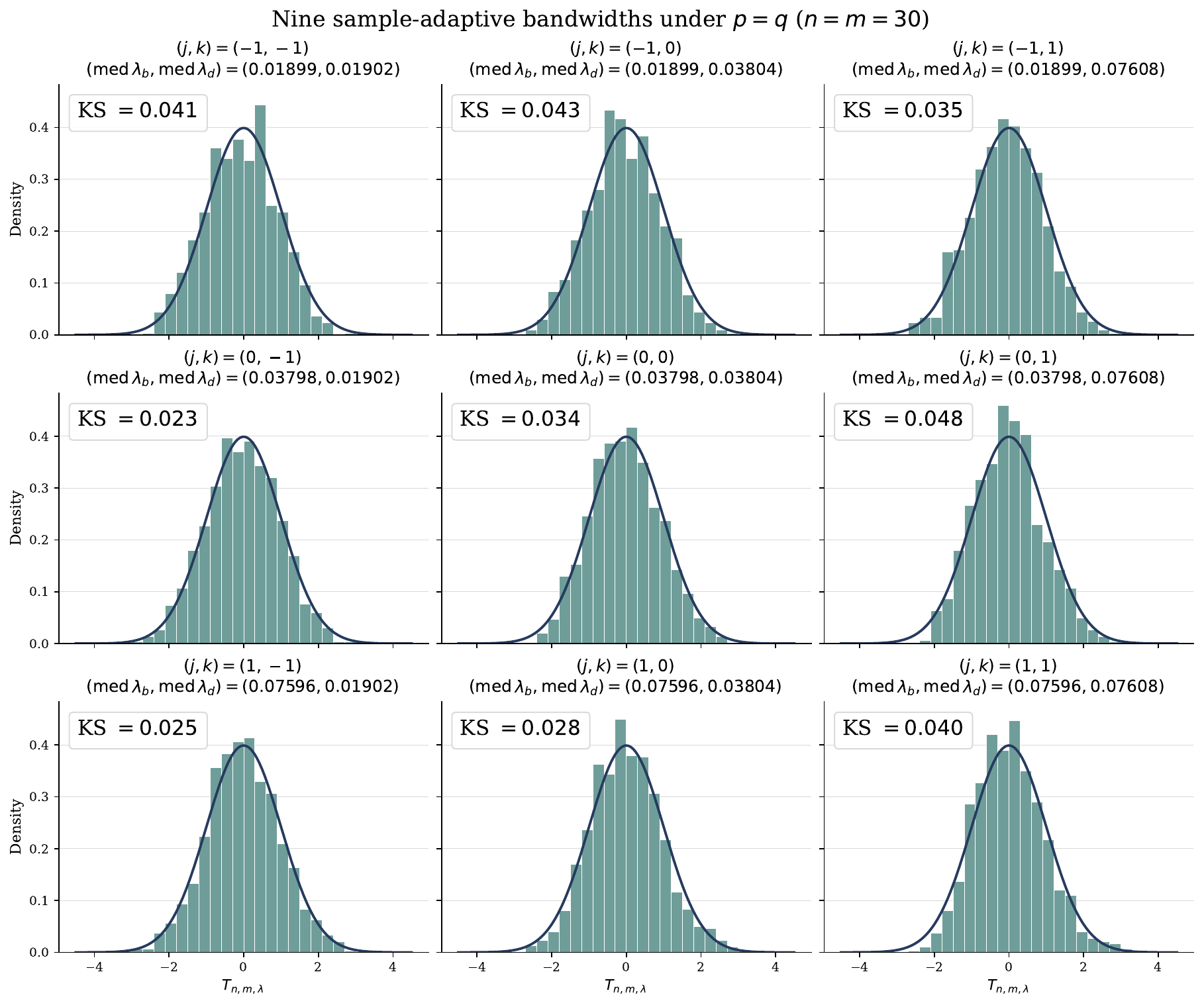}
    \caption{Empirical null distributions of $T_{n,m,\lambda}$ for the nine data-adaptive bandwidth choices in $\Lambda^{\mathrm{prac}}_{n,m}$, based on 1,000 repetitions with $n=m=30$ and $p=q$. Each panel reports the Kolmogorov–Smirnov (KS) distance to the standard normal distribution and the coordinate-wise median of $(\lambda_b,\lambda_d)$ across repetitions.}
    \label{Fig::our_choice_adaptive_n30_m30}
\end{figure}

\begin{figure}[H]
    \centering
    \includegraphics[width=\textwidth,height=0.6\textheight,keepaspectratio]
        {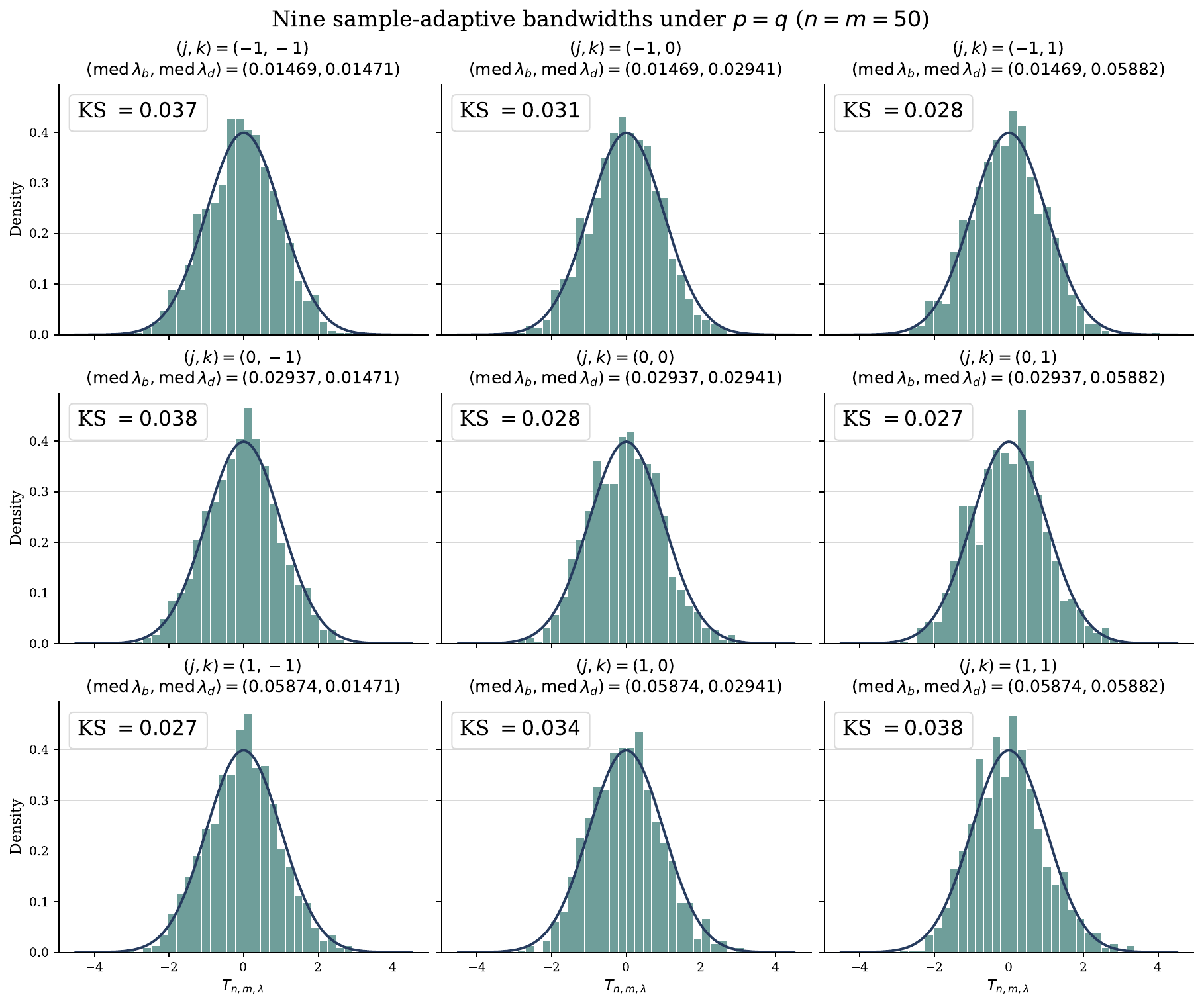}
    \caption{Empirical null distributions of $T_{n,m,\lambda}$ for the nine data-adaptive bandwidth choices in $\Lambda^{\mathrm{prac}}_{n,m}$, based on 1,000 repetitions with $n=m=50$ and $p=q$. Each panel reports the Kolmogorov–Smirnov (KS) distance to the standard normal distribution and the coordinate-wise median of $(\lambda_b,\lambda_d)$ across repetitions.}
    \label{Fig::our_choice_adaptive_n50_m50}
\end{figure}

\begin{figure}[H]
    \centering
\includegraphics[width=\textwidth,height=0.6\textheight,keepaspectratio]
        {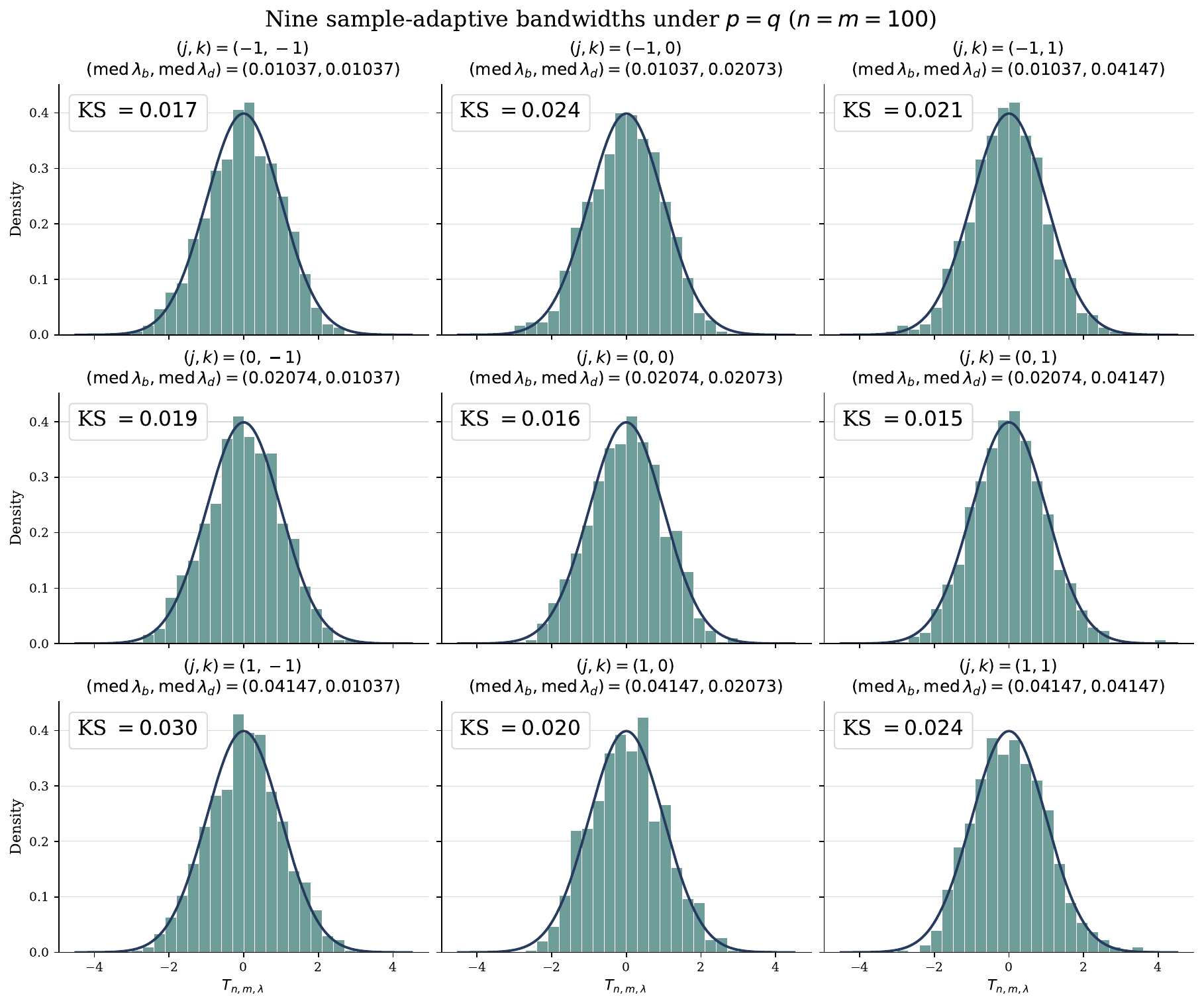}
    \caption{Empirical null distributions of $T_{n,m,\lambda}$ for the nine data-adaptive bandwidth choices in $\Lambda^{\mathrm{prac}}_{n,m}$, based on 1,000 repetitions with $n=m=100$ and $p=q$. Each panel reports the Kolmogorov–Smirnov (KS) distance to the standard normal distribution and the coordinate-wise median of $(\lambda_b,\lambda_d)$ across repetitions.}
    \label{Fig::our_choice_adaptive_n100_m100}
\end{figure}

\subsubsection{Weight function}
The choice of the weight function determines which persistence features receive greater emphasis and should therefore reflect the features regarded as informative in a given application. A useful practical guideline is to consider the scientific question together with the overall, unlabeled structure of the observed persistence diagrams and identify the relevant range of persistence. For example, if points close to the diagonal are expected to contain meaningful information for distinguishing the two persistence intensity functions, a constant weight may be appropriate. In contrast, when the relevant topological features are expected to occur at non-negligible persistence scales, a weight such as the square weight $w(b,d)=(d-b)^2$ or the linear weight $w(b,d)=d-b$ may be useful, as it emphasizes features farther from the diagonal while reducing the contribution of short-lived features. When no single weight function is clearly justified, we recommend prespecifying a finite collection of plausible weight functions and extending the Bonferroni procedure to aggregate jointly over the resulting weight–bandwidth pairs.

Importantly, the weight function should not be selected by running the
test with several candidate weights and subsequently reporting the
weight that yields the most favourable testing result. Such a procedure
implicitly performs a data-dependent selection among multiple tests,
and the resulting test does not in general preserve the nominal Type I
error level unless the selection step is explicitly incorporated into
the calibration procedure. 
\section{Real data analysis}\label{Appendix::Sec::Real Data Analysis}
In this section, we perform a hypothesis test on sampled sound data from two wind instruments—the flute and the clarinet. This experiment was originally conducted using the two-stage persistence image test~\citep{MoonLazar2023}. In our study, we apply the proposed Bonferroni test under the same experimental conditions
as those in~\citet{MoonLazar2023}.

Following the experimental setup of~\citet{MoonLazar2023}, we have the 20 independent persistence diagrams in each group. The original data repository: \url{https://github.com/Matt-OR/TDA-TimeSeriesAnalysis} is no longer accessible; the original sound data used in our analysis are archived in our repository for reproducibility.

For both the flute and the clarinet, the note A4 is recorded at a sampling rate of 44{,}100~Hz. 
The clarinet sound is recorded for 4.75 seconds and the flute sound for 1.9 seconds, with a time resolution of approximately $50\,\mu s$. A plot of the sampled sound data is shown in the upper panel of Figure~\ref{Appendix::fig:threeimages}. 

The sampled sound data are transformed into point clouds via Takens' embedding~\citep{Takens1981Detecting,Takensembedding} using the parameters $d=2$ and $\tau=3$; see Section~\ref{Appendix::Taken's embedding}, for the definition of Takens' embedding. The resulting point clouds are displayed in the lower panel of Figure~\ref{Appendix::fig:threeimages}. 
To obtain persistence diagrams from these point clouds, we use the Vietoris–Rips complex. For our testing procedure, we focus only on one-dimensional topological features. 
Figure \ref{Appendix::fig:instrument persistence diagrams} contains an example pair of these persistence diagrams.
The visibly different persistence diagram patterns in Figure \ref{Appendix::fig:instrument persistence diagrams} motivate this setting as an alternative expected to exhibit different persistence intensity functions.

We examine three scenarios: Scenario~1 (clarinet vs.~clarinet), Scenario~2 (flute vs.~flute), and Scenario~3 (flute vs.~clarinet). For each scenario, we apply the proposed test with a constant weight and level $0.05$, using 20 persistence diagrams from each group. The result is non-rejection in Scenarios~1 and~2 and rejection in Scenario~3. Therefore, the test fails to reject the null hypothesis when the two samples come from the same instrument and rejects it when the samples come from different instruments.
\begin{figure}[H]
    \centering

    \includegraphics[width=1\textwidth]{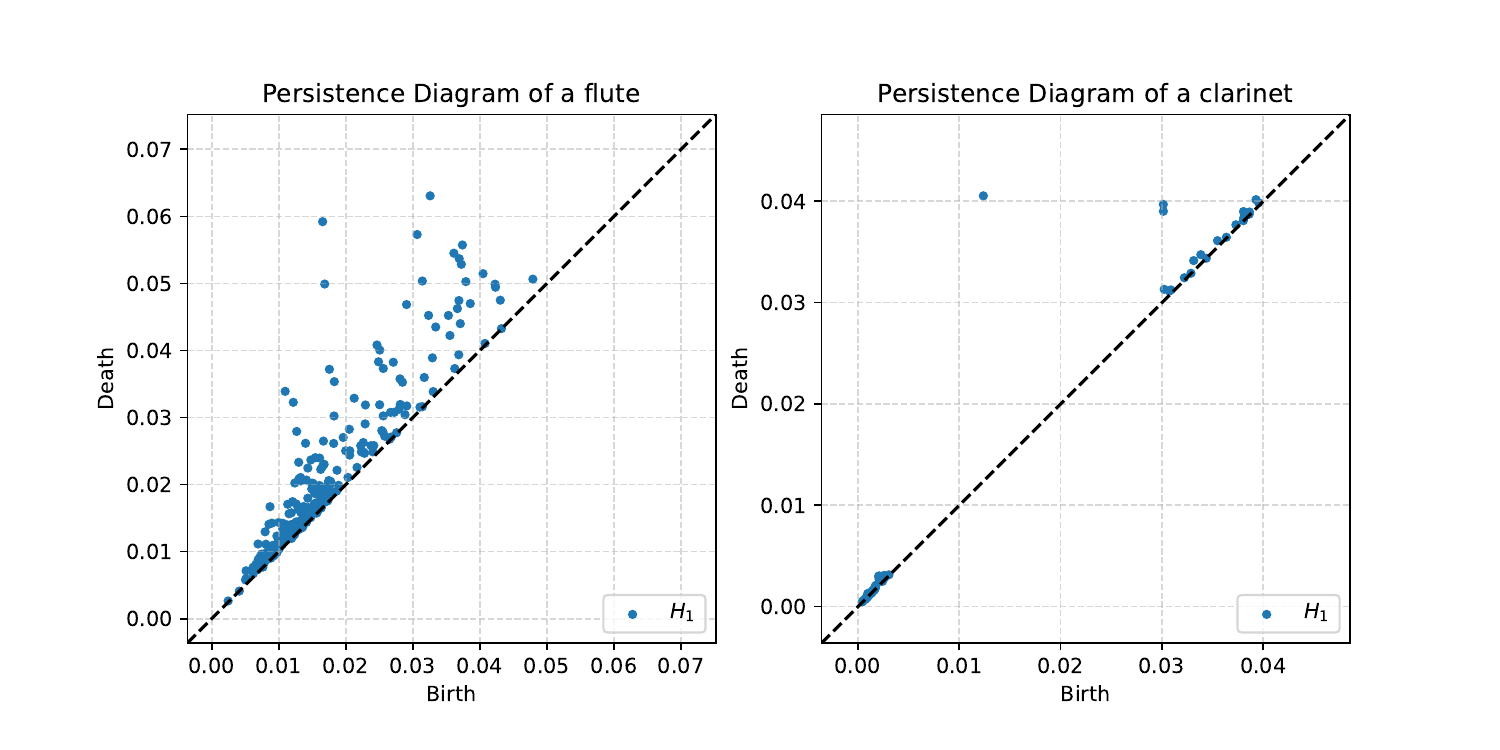}
    \caption{Persistence diagrams obtained using the Vietoris--Rips filtration from point clouds constructed from (Left) the flute sound and (Right) the clarinet sound. 
}
     \label{Appendix::fig:instrument persistence diagrams}
\end{figure}
\begin{figure}[H]
    \centering
    \includegraphics[width=0.8\textwidth]{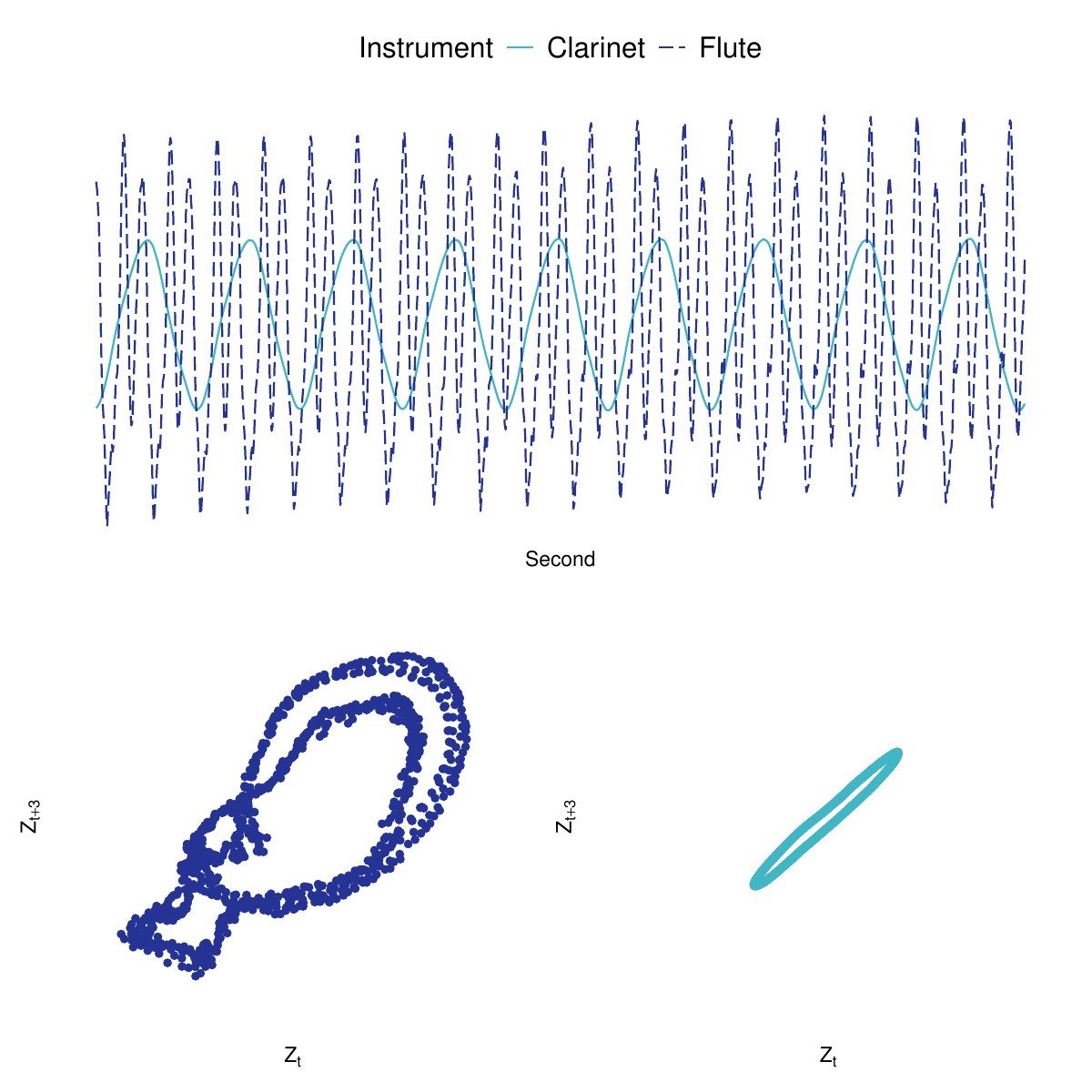}
    \caption{(Upper) Raw time series of two instruments. (Lower left) Point cloud of \textbf{flute sound}. (Lower right) Point cloud of \textbf{clarinet sound}. The illustration is adapted from \citet{MoonLazar2023} using
independently generated data. 
}
     \label{Appendix::fig:threeimages}
\end{figure}

\bibliographystyle{agsm}
\bibliography{reference}

@article{Chazal2019, title={The density of expected persistence diagrams and its kernel based estimation}, volume={10},  journal={Journal of Computational Geometry}, author={Divol, Vincent and Chazal, Frédéric}, year={2020}, pages={127–153} }

@article{JMLR:v13:gretton12a,
  author = {Gretton, A. and Borgwardt, K. M. and Rasch, M. J. and Sch{\"o}lkopf, B. and Smola, A.},
  title   = {A Kernel Two-Sample Test},
  journal = {Journal of Machine Learning Research},
  year    = {2012},
  volume  = {13},
  number  = {25},
  pages   = {723--773}
}

@article{kusano2018expectation,
  title={On the expectation of a persistence diagram by the persistence weighted kernel},
  author={Genki Kusano},
  journal={Proceedings of the 35th International Conference on Machine Learning},
  volume={80},
  pages={2805--2814},
  year={2018},
  publisher={PMLR}
}

@book{lee1990ustatistics,
  title     = {{`U-Statistics: Theory and Practice'}},
  author    = {Lee, A. J.},
  series    = {Statistics: A Series of Textbooks and Monographs},
  volume    = {110},
  publisher = {Marcel Dekker},
  year      = {1990},
  address   = {New York},
  isbn      = {9780824782535}
}

@article{schrab2023mmdagg,
  title   = {{MMD} Aggregated Two-Sample Test},
  author = {Schrab, A. and Kim, I. and Albert, M. and Laurent, B. and Guedj, B. and Gretton, A.},
  journal = {Journal of Machine Learning Research},
  volume  = {24},
  number  = {194},
  pages   = {1--81},
  year    = {2023}
}

@article{MoonLazar2023,
  author    = {Chul Moon and Nicole A. Lazar},
  title     = {Hypothesis Testing for Shapes using Vectorized Persistence Diagrams},
  journal   = {Journal of the Royal Statistical Society: Series C },
  volume    = {72},
  number    = {3},
  pages     = {628--648},
  year      = {2023},
  doi       = {10.1093/jrsssc/qlad024}
}

@article{kim2022minimax,
  title={Minimax optimality of permutation tests},
  author={Kim, I. and Balakrishnan, S. and Wasserman, L.},
  journal={The Annals of Statistics},
  volume={50},
  number={1},
  pages={225--251},
  year={2022},
  publisher={Institute of Mathematical Statistics},
  doi={10.1214/21-AOS2120}
}

@article{li2019optimality,
  author    = {Tong Li and Ming Yuan},
  title     = {On the Optimality of Gaussian Kernel Based Nonparametric Tests against Smooth Alternatives},
  journal   = {Journal of Machine Learning Research},
  volume    = {25},
  number    = {334},
  pages     = {1--62},
  year      = {2024}
}

@article{kusano2017kernel,
  title = {Kernel Method for Persistence Diagrams via Kernel Embedding and Weight Factor},
  author = {Genki Kusano and Kenji Fukumizu and Yasuaki Hiraoka},
  journal = {Journal of Machine Learning Research},
  year = {2018},
  volume = {18},
  number = {189},
  pages = {1--41},
}

@article{kusano2016persistence,
  author  = {Kusano, Genki and Hiraoka, Yasuaki and Fukumizu, Kenji},
  title   = {Persistence Weighted Gaussian Kernel for Topological Data Analysis},
  journal = {Proceedings of the 33rd International Conference on Machine Learning},
  year    = {2016},
  volume  = {48},
  pages   = {2004--2013}
}

@article{wu2024estimation,
  author  = {Weichen Wu and Jisu Kim and Alessandro Rinaldo},
  title   = {On the Estimation of Persistence Intensity Functions and Linear Representations of Persistence Diagrams},
  journal = {Proceedings of the 27th International Conference on Artificial Intelligence and Statistics},
  year    = {2024},
  volume  = {238},
  pages   = {3610--3618}
}

@book{vandervaart1998asymptotic,
  title        = {Asymptotic Statistics},
  author       = {van der Vaart, A. W.},
  year         = {1998},
  publisher    = {Cambridge University Press},
  series       = {Cambridge Series in Statistical and Probabilistic Mathematics},
  volume       = {3},
  doi          = {10.1017/CBO9780511802256}
}

@article{RobinsonTurner2017,
  author    = {Andrew Robinson and Katharine Turner},
  title     = {Hypothesis Testing for Topological Data Analysis},
  journal   = {Journal of Applied and Computational Topology},
  year      = {2017},
  volume    = {1},
  number    = {2},
  pages     = {241--261},
  doi       = {10.1007/s41468-017-0008-z}
}

@article{Bubenik2015,
  author  = {Peter Bubenik},
  title   = {Statistical Topological Data Analysis Using Persistence Landscapes},
  journal = {Journal of Machine Learning Research},
  year    = {2015},
  volume  = {16},
  pages   = {77--102}
}

@article{Hertzsch2007DNA,
author = {Hertzsch, Jan-Martin and Sturman, Rob and Wiggins, Stephen},
title = {{DNA} Microarrays: Design Principles for Maximizing Ergodic, Chaotic Mixing},
journal = {Small},
volume = {3},
number = {2},
year = {2007}
}

@article{KimKimZaheer2020PLLay,
  author  = {Kwangho Kim and Jisu Kim and Manzil Zaheer and Joon Sik Kim and Frederic Chazal and Larry Wasserman},
  title   = {{PLLay: Efficient Topological Layer Based on Persistence Landscapes}},
  journal = {Advances in Neural Information Processing Systems},
  year    = {2020},
  volume  = {33},
  pages   = {},  
}

@article{NiyogiSmaleWeinberger2008,
  author = {Partha Niyogi and Stephen Smale and Shmuel Weinberger},
  title     = {Finding the Homology of Submanifolds with High Confidence from Random Samples},
  journal   = {Discrete \& Computational Geometry},
  volume    = {39},
  number    = {1/3},
  pages     = {419--441},
  year      = {2008},
  doi       = {10.1007/s00454-008-9053-2}
}

@article{ChazalCohenSteinerLieutier2009,
  author = {Chazal, Frédéric and David Cohen-Steiner and André Lieutier},
  title     = {A Sampling Theory for Compact Sets in Euclidean Space},
  journal   = {Discrete \& Computational Geometry},
  volume    = {41},
  number    = {3},
  pages     = {461--479},
  year      = {2009},
  doi       = {10.1007/s00454-009-9144-8}
}

@misc{conf/icml/DivolL21,
  author    = {Divol, Vincent and Leclerc, Th{\'e}o},
  title     = {Estimation and Quantization of Expected Persistence Diagrams},
  year      = {2021},
  note      = {In Proceedings of the 38th International Conference on Machine Learning (ICML), PMLR, pp. 2760--2770.}
}

@article{Takensembedding,
author = {Stark, J. and Broomhead, D. S. and Davies, M. E. and Huke, J.},
title = {Takens embedding theorems for forced and stochastic systems},
year = {1997},
issue_date = {Dec., 1997},
publisher = {Elsevier Science Ltd.},
address = {GBR},
volume = {30},
number = {9},
issn = {0362-546X},
doi = {10.1016/S0362-546X(96)00149-6},
journal = {Nonlinear Anal.},
month = dec,
pages = {5303–5314},
numpages = {12}
}

@book{Takens1981Detecting,
  author    = {Takens, Floris},
  title     = {Detecting strange attractors in turbulence},
  booktitle = {Dynamical Systems and Turbulence, Warwick 1980},
  series    = {Lecture Notes in Mathematics},
  volume    = {898},
  pages     = {366--381},
  year      = {1981},
  publisher = {Springer-Verlag},
  doi       = {10.1007/BFb0091924},
}

@article{PereaHarer2013SlidingWindows,
  author  = {Perea, Jose and Harer, John},
  title   = {Sliding windows and persistence: An application of topological methods to signal analysis},
  journal = {Foundations of Computational Mathematics},
  volume  = {15},
  year    = {2015},
  doi     = {10.1007/s10208-014-9206-z}
}

@article{
amorphoussolids,
author = {Yasuaki Hiraoka  and Takenobu Nakamura  and Akihiko Hirata  and Emerson G. Escolar  and Kaname Matsue  and Yasumasa Nishiura },
title = {Hierarchical structures of amorphous solids characterized by persistent homology},
journal = {Proceedings of the National Academy of Sciences},
volume = {113},
number = {26},
pages = {7035-7040},
year = {2016},
doi = {10.1073/pnas.1520877113}
}

@inproceedings{NIPS2006_e9fb2eda,
 author = {Gretton, Arthur and Borgwardt, Karsten and Rasch, Malte and Sch\"{o}lkopf, Bernhard and Smola, Alex},
 booktitle = {Advances in Neural Information Processing Systems},
 editor = {B. Sch\"{o}lkopf and J. Platt and T. Hoffman},
 pages = {},
 publisher = {MIT Press},
 title = {A Kernel Method for the Two-Sample-Problem},
 volume = {19},
 year = {2006}
}

@inproceedings{SchrabKSD,
author = {Schrab, Antonin and Guedj, Benjamin and Gretton, Arthur},
title = {{KSD} aggregated goodness-of-fit test},
year = {2022},
isbn = {9781713871088},
publisher = {Curran Associates Inc.},
address = {Red Hook, NY, USA},
booktitle = {Proceedings of the 36th International Conference on Neural Information Processing Systems},
articleno = {2364},
numpages = {15},
location = {New Orleans, LA, USA},
series = {NIPS '22}
}

@article{Cheng2024,
author = {Xiuyuan Cheng and Yao Xie},
title = {{Kernel two-sample tests for manifold data}},
volume = {30},
journal = {Bernoulli},
number = {4},
publisher = {Bernoulli Society for Mathematical Statistics and Probability},
pages = {2572 -- 2597},
year = {2024},
doi = {10.3150/23-BEJ1685}
}

@article{WynneDuncan2022,
  title        = {A Kernel Two-Sample Test for Functional Data},
  author       = {George Wynne and Andrew B. Duncan},
  journal      = {Journal of Machine Learning Research},
  volume       = {23},
  number       = {73},
  pages        = {1--51},
  year         = {2022},
  eprint       = {https://arxiv.org/abs/2008.11095},
}

@book{IngsterSuslina2003,
  author    = {Ingster, Yuri I. and Suslina, Irina A.},
  title     = {Nonparametric Goodness-of-Fit Testing Under Gaussian Models},
  series    = {Lecture Notes in Statistics},
  volume    = {169},
  publisher = {Springer},
  address   = {New York},
  year      = {2003},
  doi       = {10.1007/978-1-4612-0533-1}
}

@article{GLUZBERG2023107216,
title = {Topological data analysis of noise: Uniform unimodal distributions},
journal = {Communications in Nonlinear Science and Numerical Simulation},
volume = {121},
pages = {107216},
year = {2023},
issn = {1007-5704},
doi = {https://doi.org/10.1016/j.cnsns.2023.107216},
author = {Victor E. Gluzberg and Yuri A. Katz}}

@article{AlbertLaurentMarrelMeynaoui2022,
  author    = {Mélisande Albert and Béatrice Laurent and Amandine Marrel and Anouar Meynaoui},
  title     = {Adaptive test of independence based on {HSIC} measures},
  journal   = {Annals of Statistics},
  volume    = {50},
  number    = {2},
  pages     = {858--879},
  year      = {2022},
  doi       = {10.1214/21-aos2129}
}

@article{persistenceimage,
  author  = {Henry Adams and Tegan Emerson and Michael Kirby and Rachel Neville and Chris Peterson and Patrick Shipman and Sofya Chepushtanova and Eric Hanson and Francis Motta and Lori Ziegelmeier},
  title   = {Persistence Images: A Stable Vector Representation of Persistent Homology},
  journal = {Journal of Machine Learning Research},
  year    = {2017},
  volume  = {18},
  number  = {8},
  pages   = {1--35}
}

@misc{you2022comparingmultiplelatentspace,
  title         = {Comparing Multiple Latent Space Embeddings Using Topological Analysis},
  author        = {You, Kisung and Kim, Ilmun and Jin, Ick Hoon and Jeon, Minjeong and Shung, Dennis},
  year          = {2022},
  eprint        = {2208.12435},
  archivePrefix = {arXiv},
  primaryClass  = {stat.ME},
  note   = {arXiv:2208.12435}
}

@book{PaulsenRaghupathi2016, place={Cambridge}, series={Cambridge Studies in Advanced Mathematics}, title={An Introduction to the Theory of Reproducing Kernel Hilbert Spaces}, publisher={Cambridge University Press}, author={Paulsen, Vern I. and Raghupathi, Mrinal}, year={2016}, collection={Cambridge Studies in Advanced Mathematics}}

@article{ingster1993asymptotically,
  title={Asymptotically minimax hypothesis testing for nonparametric alternatives. I, II, III},
  author={Ingster, Yuri I},
  journal={Math. Methods Statist},
  volume={2},
  number={2},
  pages={85--114},
  year={1993}
}

@book{Hatcher,
  address = {Cambridge},
  author = {Hatcher, Allen},
  isbn = {0-521-79160-X; 0-521-79540-0},
  mrclass = {55-01 (55-00)},
  mrnumber = {1867354 (2002k:55001)},
  mrreviewer = {Donald W. Kahn},
  pages = {xii+544},
  publisher = {Cambridge University Press},
  title = {Algebraic topology},
  username = {mwpb479},
  year = 2002
}

@book{EdelsbrunnerHarer2010,
  author    = {Edelsbrunner, Herbert and Harer, John},
  title     = {Computational Topology: An Introduction},
  publisher = {American Mathematical Society},
  year      = {2010},
  isbn      = {978-0-8218-4925-5}
}

@article{hiraoka2018limit,
  title={Limit theorems for persistence diagrams},
  author={Duy, Trinh Khanh and Hiraoka, Yasuaki and Shirai, Tomoyuki},
  journal={The Annals of Applied Probability},
  volume={28},
  number={5},
  pages={2745--2780},
  year={2018},
  publisher={Institute of Mathematical Statistics}
}

@misc{chen2015statisticalanalysispersistenceintensity,
      title={Statistical Analysis of Persistence Intensity Functions}, 
      author={Yen-Chi Chen and Daren Wang and Alessandro Rinaldo and Larry Wasserman},
      year={2015},
      eprint={1510.02502},
      archivePrefix={arXiv},
      note={arXiv:1510.02502}
}

@article{Topaz2015,
  author  = {Topaz, Chad M. and Ziegelmeier, Lori and Halverson, Tom},
  title   = {Topological Data Analysis of Biological Aggregation Models},
  journal = {PLOS ONE},
  volume  = {10},
  number  = {5},
  year    = {2015},
  doi     = {10.1371/journal.pone.0126383}
}

@article{Wilding2021,
  author  = {Wilding, Georg and Nevenzeel, Keimpe and van de Weygaert, Rien
             and Vegter, Gert and Pranav, Pratyush and Jones, Bernard J. T.
             and Efstathiou, Konstantinos and Feldbrugge, Job},
  title   = {Persistent homology of the cosmic web -- I. Hierarchical topology
             in {$\Lambda$CDM} cosmologies},
  journal = {Monthly Notices of the Royal Astronomical Society},
  volume  = {507},
  number  = {2},
  pages   = {2968--2990},
  year    = {2021},
  doi     = {10.1093/mnras/stab2326}
}

@article{Hoeffding1948,
  author  = {Hoeffding, Wassily},
  title   = {A Class of Statistics with Asymptotically Normal Distribution},
  journal = {The Annals of Mathematical Statistics},
  year    = {1948},
  volume  = {19},
  number  = {3},
  pages   = {293--325}
}

@article{Bochner1933,
  author  = {Bochner, Salomon},
  title   = {Monotone Funktionen, Stieltjessche Integrale und
             harmonische Analyse},
  journal = {Mathematische Annalen},
  volume  = {108},
  pages   = {378--410},
  year    = {1933},
  doi     = {10.1007/BF01452844}
}

@article{HeydeBrown1970,
  author  = {Heyde, C. C. and Brown, B. M.},
  title   = {On the Departure from Normality of a Certain Class of Martingales},
  journal = {The Annals of Mathematical Statistics},
  year    = {1970},
  volume  = {41},
  number  = {6},
  pages   = {2161--2165},
  doi     = {10.1214/aoms/1177696722}
}

@article{Fan2019,
title = {Exact rates of convergence in some martingale central limit theorems},
journal = {Journal of Mathematical Analysis and Applications},
volume = {469},
number = {2},
pages = {1028-1044},
year = {2019},
issn = {0022-247X},
author = {Xiequan Fan}
}

@book{hall1980martingale,
  author    = {Hall, Peter and Heyde, Christopher C.},
  title     = {Martingale Limit Theory and Its Application},
  publisher = {Academic Press},
  address   = {New York},
  year      = {1980}
}

@book{HsingEubank2015,
  author    = {Hsing, Tailen and Eubank, Randall},
  title     = {Theoretical Foundations of Functional Data Analysis,
               with an Introduction to Linear Operators},
  year      = {2015},
  publisher = {John Wiley \& Sons},
  series    = {Wiley Series in Probability and Statistics}
}

\end{document}